\documentclass[12pt, a4paper,leqno]{amsart}
\usepackage{amsmath,amsthm,amscd,amssymb,amsfonts, amsbsy}
\usepackage{latexsym}
\usepackage{txfonts}
\usepackage{exscale}
\usepackage{graphicx}
\usepackage{bbm}
\usepackage{enumerate}
\usepackage{soul}
\usepackage[utf8]{inputenc}

\usepackage[colorlinks,citecolor=red,pagebackref,hypertexnames=false]{hyperref}
\usepackage{color}

\calclayout
\allowdisplaybreaks

\theoremstyle{plain}
\newtheorem{theorem}[equation]{Theorem}
\newtheorem{lemma}[equation]{Lemma}
\newtheorem{corollary}[equation]{Corollary}
\newtheorem{proposition}[equation]{Proposition}

\theoremstyle{definition}
\newtheorem{definition}[equation]{Definition}
\newtheorem{condition}[equation]{Condition}

\theoremstyle{remark}
\newtheorem{remark}[equation]{Remark}

\numberwithin{equation}{section}

\newcommand{\eps}{\varepsilon}

\newcommand{\dist}{\operatorname{dist}}

\newcommand{\ds}{\displaystyle}

\newcommand{\teps}{{\tilde{\eps}}}
\newcommand{\re}{\mathbb{R}}
\newcommand{\rn}{\mathbb{R}^n}

\newcommand{\ree}{\mathbb{R}^{n+1}}

\newcommand{\La}{\Lambda}

\newcommand{\dd}{\mathbb{D}}

\newcommand{\om}{\Omega}

\newcommand{\cB}{\mathcal{B}}
\newcommand{\cG}{\mathcal{G}}

\newcommand{\F}{\mathcal{F}}
\newcommand{\cH}{\mathcal{H}}

\newcommand{\LL}{\mathcal{L}}

\newcommand{\nn}{\mathcal{N}}

\newcommand{\W}{\mathcal{W}}

\newcommand{\B}{\mathcal{B}}

\newcommand{\sbf}{{\bf S}}

\newcommand{\G}{\mathcal{G}}

\newcommand{\pom}{\partial\Omega}

\newcommand{\hm}{\omega}

\newcommand{\sem}{\setminus}

\newcommand{\ba}{{\bf H}}
\newcommand{\bb}{{\bf K}}

\newcommand{\ch}{{\tt ch}}

\renewcommand{\P}{\mathcal{P}}

\renewcommand{\emptyset}{\text{\textup{\O}}}

\DeclareMathOperator{\supp}{supp}

\DeclareMathOperator{\diam}{diam}

\DeclareMathOperator{\interior}{\tt int}

\newcommand{\vertiii}[1]{{\left\vert\kern-0.15ex\left\vert\kern-0.15ex\left\vert #1
		\right\vert\kern-0.15ex\right\vert\kern-0.15ex\right\vert}}

\def\Xint#1{\mathchoice
{\XXint\displaystyle\textstyle{#1}}%
{\XXint\textstyle\scriptstyle{#1}}%
{\XXint\scriptstyle\scriptscriptstyle{#1}}%
{\XXint\scriptscriptstyle%
\scriptscriptstyle{#1}}%
\!\int}
\def\XXint#1#2#3{{\setbox0=\hbox{$#1{#2#3}{%
\int}$ }
\vcenter{\hbox{$#2#3$ }}\kern-.6\wd0}}
\def\barint{\,\Xint -} % \, corrects the \! used in the definition
\def\bariint{\barint_{} \kern-.4em \barint}
\def\bariiint{\bariint_{} \kern-.4em \barint}
\renewcommand{\iint}{\int_{}\kern-.34em \int} %\, minor space between the integrals
\renewcommand{\iiint}{\iint_{}\kern-.34em \int} %\, minor space between the integrals

\renewcommand{\d}{d}

\begin{document}
\allowdisplaybreaks

\title[Solvability of the $L^p$ Dirichlet problem implies parabolic uniform rectifiability]{Solvability of the $L^p$ Dirichlet problem for the heat equation implies parabolic uniform rectifiability}

\author{S. Bortz}
\address{Department of Mathematics
\\
University of Alabama
\\
Tuscaloosa, AL, 35487, USA}
\email{sbortz@ua.edu}

\author{S. Hofmann}
\address{
Department of Mathematics
\\
University of Missouri
\\
Columbia, MO 65211, USA}
\email{hofmanns@missouri.edu}

\author{J.M. Martell}
\address{Instituto de Ciencias Matematicas CSIC-UAM-UC3M-UCM, E-28049 Madrid, Spain}
\email{chema.martell@icmat.es}

\author{K. Nystr\"om}
\address{Department of Mathematics, Uppsala University, S-751 06 Uppsala, Sweden}
\email{kaj.nystrom@math.uu.se}

\thanks{S.B. was supported by the Simons Foundation’s Travel Support for Mathematicians program and an AWI-CONSERVE Fellowship. S.H. is supported by NSF grant DMS-2349846.
J.M.M. acknowledges financial support from MCIN/AEI/10.13039/501100011033 grants CEX2023-001347-S and PID2022-141354NB-I00.
	K.N. was partially supported by grant  2022-03106  from the Swedish research council (VR)}


\date{\today}

%\keywords{put keywords here}

\begin{abstract}  
 Let $\Omega \subset \mathbb{R}^{n+1}$ be an open set in space-time with boundary $\Sigma = \partial \Omega$. Under minimal and natural background assumptions - namely, that $\Sigma$ is time-symmetrically parabolic Ahlfors--David regular and that $\Omega$ satisfies an interior corkscrew condition - we treat a one-phase parabolic free boundary problem which establishes the necessity of parabolic uniform rectifiability for $L^p(d\sigma)$ solvability of the Dirichlet problem for the heat equation. More precisely, we prove that if the caloric measure associated with $\Omega$ satisfies a weak-$A_\infty$ condition with respect to the surface measure $\sigma = \mathcal{H}_{\mathrm{par}}^{n+1}\!\lfloor_{\Sigma}$, then $\Sigma$ is parabolically uniformly rectifiable, hence equivalently, that
 solvability of the Dirichlet problem for the heat (or adjoint heat) equation in $\Omega$ with boundary data in $L^p(d\sigma)$, for some $p \in (1,\infty)$, implies parabolic uniform rectifiability. Our main theorem thus 
 identifies parabolic uniform rectifiability as the correct geometric framework for boundary regularity, and $L^p$ solvability, in the parabolic setting.

\end{abstract}

\maketitle
\setcounter{tocdepth}{1}
%Shortens the table of contents to only show sections.
\tableofcontents

%\newpage

\section{Introduction and statement of the main result}

{
In this paper, we treat a 1-phase parabolic free boundary problem
which in turn is equivalent to proving, under natural and
essentially minimal background hypotheses
(see Theorem \ref{main.thrm} just below), that if
$\Omega \subset \ree$ is an open set in space and time,
then the solvability of the Dirichlet problem for the heat equation in $\Omega$ with
data in $L^p(\d\sigma)$, for some
$p\in (1,\infty)$,  implies that $\Sigma=\pom$ is
parabolic uniformly rectifiable. Here, $\sigma$ is the ``surface measure" on $\Sigma$,
i.e., the restriction to $\Sigma$ of $\cH_{\text{par}}^{n+1}$, the
parabolic Hausdorff measure of homogeneous dimension $n+1$ (see \eqref{sigdef}
and Definition \ref{parHausdef} below).

Solvability of the $L^p(\d\sigma)$ Dirichlet problem is fundamentally connected to quantitative absolute continuity of caloric measure with respect to surface measure $\sigma$. In fact (see \cite{GH1}), this solvability is equivalent to the weak-$A_\infty$  property of caloric measure with respect to $\sigma$,
and it is through this connection that we shall obtain our result, and through
which the connection with free boundary problems is realized. The following is our main theorem. The terminology and notions used will be defined in the sequel. }

\begin{theorem}\label{main.thrm}
Assume that $\Omega \subset \ree$ is an open set. Assume that $\Sigma=\partial\Omega$ is time symmetric Ahlfors-David regular in the sense of  Definition \ref{ADR.def}
with constant $M$, that $\Omega$ satisfies the corkscrew condition
in the sense of Definition \ref{CS.def}  with constant $\gamma$, and
that caloric measure satisfies a local weak-$A_\infty$ condition with respect to
$\sigma=\cH_{\text{\em par}}^{n+1}|_\Sigma$
in the sense of Definition \ref{defweakAinfty.def} with constants $p$ and $C$.
Then $\Sigma$ is parabolic uniformly rectifiable in the sense of Definition \ref{def1.UR} with
constants $(M,c)$ where $c$ only depends on $n,M,\gamma,p$ and $C$.
\end{theorem}

\begin{remark}
In the previous theorem, as will be apparent from the
definitions to be given in Section \ref{Sec2},
we implicitly assume that $\Omega$ and $\Sigma$ are both
unbounded, and are not contained in any ``time-limited" strip $\{(X,t): a<t<b\}$ with
either $a$ or $b$ finite.  We do this only for the sake of simplicity, and in the sequel
(see Remark \ref{boundeddomain}), we shall briefly indicate the routine changes that
one would make to treat the case of time-limited or bounded domains.
\end{remark}

Our background assumptions, namely, that $\Omega$ satisfies an (interior)
corkscrew condition, and that its boundary is time symmetric Ahlfors-David
regular, are both natural. In the absence of the former assumption, one may
construct an example for which the main hypothesis (the weak-$A_\infty$ property)
holds, but the conclusion (uniform rectifiability) fails, even in the elliptic setting:
see \cite[Appendix A]{AHMMT}. The time symmetric Ahlfors-David
regularity condition provides, first of all, an underlying doubling surface measure
on the boundary, and moreover, the time symmetry itself provides sufficient
thermal capacity, both forwards and backwards in time, to yield local H\"older
continuity up to the boundary for both solutions and adjoint solutions (we use both)
vanishing on a surface ball (see Lemmas \ref{Bourgain} and \ref{continuity} below).

%\cite{LS,LewMur,H95,H,HL96,HL}

%Assume that $\Omega \subset \ree$ is an open set, with time symmetric Ahlfors-David regular boundary $\Sigma = \partial \Omega$ of homogeneous (parabolic) dimension $n+1$, satisfying the corkscrew condition. If the caloric measure satisfies a local weak-$A_\infty$ condition with respect to
%$\cH_{\text{par}}^{n+1}|_\Sigma$, then $\Sigma$ is parabolic uniformly rectifiable.

The notions of parabolic uniform rectifiability, and  parabolic uniformly rectifiable sets, were originally introduced by the second and fourth author together with J. Lewis \cite{HLN1, HLN2}, in the context of free boundary problems for the heat equation. The notion gives structure to time-varying boundaries/sets which lack differentiability, and which are locally not necessarily given by graphs. Instead, {geometry is controlled by a local geometric square function and a Carleson measure, based on which key geometric information and structure can be extracted: this is captured in the notion of parabolic uniform rectifiability.  The local geometric square function
quantifies, on each scale,
how the underlying set deviates from time-independent
hyperplanes in a local, scale invariant $L^2$-sense (mean square sense); thus, it is
the parabolic analogue of the $\beta$-numbers of P.~Jones \cite{J}.
Parabolic uniform rectifiability is the dynamic counterpart of the uniform rectifiability studied in the  monumental works of G. David and S. Semmes \cite{DS1, DS2}.

In the context of graphs, parabolic uniform rectifiability captures and extracts the geometrical theoretical essence of the regular parabolic Lipschitz graphs introduced and studied in \cite{LS, LewMur, H95, H, HL}. By definition, a parabolic Lipschitz graph, or
a  Lip(1,1/2) graph, is the graph of a function which is Lipschitz with respect to the parabolic metric. A parabolic Lipschitz graph is a {\em regular} parabolic Lipschitz graph, or a
{\em regular} Lip(1,1/2) graph, if, in addition, the half-order time derivative of the function defining the graph is in (parabolic) BMO (see Definition \ref{goodgraph.def}). Combining the results of Section 2 in \cite{HLN1,HLN2} with the developments in \cite{LewMur,H,HL96},
it follows that
a Lip(1,1/2) graph is parabolic uniformly rectifiable in the sense of  \cite{HLN1, HLN2} if and only if it is a regular Lip(1,1/2) graph.
As proved in \cite{LS, KWu}, the class of regular Lip(1,1/2) graphs is strictly smaller than the class of  Lip(1,1/2) graphs.

For the heat equation in time-varying graph domains, {\em regularity}
of a Lip(1,1/2) graph (equivalently, its parabolic uniform rectifiability) is
deeply rooted in the study of the Dirichlet problem,
as it represents, in the graph setting,
{\it a necessary and sufficient geometric condition
for the $L^p(\d\sigma)$ solvability of the Dirichlet problem for the
heat/adjoint heat equation}, and equivalently and consequently,
for the $A_\infty$ property of parabolic/caloric measures.
As regards sufficiency, while it was
 shown in \cite{KWu} that there are Lip(1,1/2) graph domains for which caloric measure and surface measure are mutually singular, nonetheless in \cite{LewMur} it is proved that for {\em regular} Lip(1,1/2) graph domains, the caloric measure and the surface measure are quantitatively related in the sense that they are mutually absolutely continuous, and the associated parabolic Poisson kernel satisfies a scale-invariant reverse Hölder inequality in $L^q(\d\sigma)$ for some $q\in (1,\infty)$.  Thus, the authors obtained solvability of the Dirichlet problem for the heat equation with
data in $L^p(\d\sigma)$, for $p = q/(q-1)$ perhaps large and unspecified, i.e.,
they showed that  in the context of Lip(1,1/2) graph domains, the regular
Lip(1,1/2) condition is sufficient for the $L^p(\d\sigma)$ solvability (for some finite $p \geq 1$)
 of the Dirichlet problem for the heat equation.
Furthermore, the importance and relevance of regular Lip(1,1/2) graph domains, from the perspective of parabolic singular integrals, layer potentials,
boundary value problems and inverse problems, is made clear in
\cite{LS, LewMur, H95, H, HL96, HL}.
In particular, in \cite{HL96}  the solvability of the $L^2(\d\sigma)$ Dirichlet
problem (and of the $L^2(\d\sigma)$ Neumann and regularity problems) was proved  for the heat equation, by the way of layer potentials, in domains defined as regions above regular Lip(1,1/2) graphs. In \cite{HL96} the $L^2$ results were established under the restriction that the half-order time derivative (measured in BMO) of the function defining the graph is small. For solvability with data in $L^2$ (as opposed to $L^p$ with $p<\infty$ possibly large),
this smallness is sharp in the
sense that in \cite{HL96}, it is proved that there are regular Lip(1,1/2) graph domains for which the $L^2(\d\sigma)$ Dirichlet problem is not solvable.

Given the result of \cite{LewMur} outlined above, which
established the sufficiency of parabolic uniform rectifiability in the
context of parabolic graphs, and in light of  the example of \cite{KWu},
the question of necessity naturally arose:
was {\it parabolic uniform rectifiability
also necessary for the $L^p(\d\sigma)$ solvability of the Dirichlet problem for the heat equation}. In \cite{BHMN},
 we recently resolved  this long-standing open problem in the affirmative,
 for domains defined as regions above graphs of Lip(1,1/2) functions.
 In particular, for such domains,
 combining the results of \cite{LewMur,BHMN}, we
 can now conclude that $L^p(\d\sigma)$ solvability, for some
finite exponent $p$, is equivalent to parabolic uniform rectifiability of the
boundary, i.e., equivalent to
the property that for the function $\psi$ whose graph defines the boundary, the
half-order time derivative of $\psi$ lies in the space of (parabolic) bounded mean oscillation.

{
 The main result of this paper, Theorem \ref{main.thrm}, together with the techniques developed to prove it, represent a
 far-reaching extension of \cite{BHMN} and its analysis, to domains not necessarily given as graphs. In particular, Theorem \ref{main.thrm} provides, under essentially minimal background assumptions on $\Omega$ and $\Sigma$, the final step in establishing the necessity of parabolic uniform rectifiability for the $L^p(d\sigma)$ solvability of the Dirichlet problem for the heat/adjoint heat equation, thereby identifying parabolic uniform rectifiability as the natural geometric framework for boundary regularity in the parabolic setting. }

Concerning sufficient conditions for the $L^p(\d\sigma)$ solvability of the Dirichlet problem for the heat/adjoint heat equation in domains not necessarily given by graphs, we note that in \cite{HLN1, HLN2}, parts of the analysis on domains above regular Lip(1,1/2) graphs  mentioned above, was extended beyond the  setting of graphs. In \cite{HLN1} the existence of big pieces of regular Lip(1,1/2) graphs, under the  assumptions that $\Sigma$ is parabolic uniform rectifiable and Reifenberg flat in the parabolic sense, is established. The results in \cite{HLN1, HLN2}, were the first of their kind in the context of parabolic problems, and the studies in \cite{HLN1, HLN2}, were motivated by the study of parabolic/caloric measure in rough time-varying domains. However, other than \cite{HLN1, HLN2}, only a few notable works in rougher parabolic settings have appeared \cite{BHHLN1,E,GH1,GH2,N2006,NS,MP}. Notably, in
\cite{BHHLN1} the first, second and fourth author, together with J. Hoffman and J. Garcia-Luna, recently obtained,  by expanding on \cite{NS}, a flexible parabolic analogue of the work of David and Jerison \cite{DJ}. As a consequence, under appropriate connectivity assumptions (parabolic NTA), the local $L^p(\d\sigma)$ solvability of the Dirichlet problem is established. Without connectivity, instead assuming other conditions, the same conclusion is proved in \cite{GH1} in settings complementing \cite{BHHLN1,NS}. It is fair to say that \cite{BHHLN1,GH1,NS} represent the current state of the art concerning sufficient conditions for the solvability of the  $L^p(\d\sigma)$ Dirichlet problem for the heat equation in
domains whose boundary is % parabolic uniformly rectifiable but
not necessarily given (locally) as a graph.

More generally, finding alternative and equivalent characterization of
parabolic uniform rectifiability, % not necessarily in terms of caloric measure but
along the lines of the
corresponding characterizations of uniform rectifiability in
\cite{DS1}, \cite{DS2}, is a central issue.
Up to recently this has also been unchartered territory, but through
the works of the first, second and fourth author, again together with J. Hoffman and J. Garcia-Luna, see \cite{BHHLN-BP,BHHLN-Corona,BHHLN-CME}, a theory is emerging.  Indeed,  combining the results in \cite{BHHLN-BP, BHHLN-Corona}, one can conclude that if $\Sigma \subset \ree$ is  parabolic Ahlfors-David regular (see Definition \ref{ADR.def}), then the following statements are equivalent:
\begin{list}{$(\theenumi)$}{\usecounter{enumi}\leftmargin=1cm \labelwidth=1cm \itemsep=0.2cm \topsep=.2cm \renewcommand{\theenumi}{\alph{enumi}}}
\item $\Sigma$ is parabolic uniformly rectifiable.

\item $\Sigma$ admits a Corona decomposition with respect to regular Lip(1,1/2) graphs.

\item $\Sigma$ has big pieces squared of regular Lip(1,1/2) graphs.
\end{list}
In addition, as mentioned above, in the case of a Lip(1,1/2) graph, by \cite{LewMur,BHMN}, $(a)$ is also equivalent to the $L^p(\d\sigma)$ solvability of the Dirichlet problem for the heat equation.

%(d)&\mbox{\quad  All sufficiently regular  convolution type
%parabolic Calderon-Zygmund}\notag\\
%&\mbox{\quad  operators  with odd kernels are $L^2$ bounded on $\Sigma$.}

Let us observe also that our results can be viewed in the context of
evolutionary free boundary problems, i.e., the field of free boundary
problems in which the free boundary is moving in both space and time.
To be more precise, consider a
1-phase caloric free boundary problem, in which one is given a
solution $u$ of the heat equation (or adjoint heat equation),
and a domain $\om$ such that
\begin{equation}\label{upos}
\om = \{ u>0\}\,,
\end{equation}
so that $u$ vanishes on $\Sigma = \pom$ (the free boundary).  One further
assumes some knowledge of $|\nabla u|$ restricted to $\Sigma$
(thus, the boundary conditions are over-determined), and the goal is to then
deduce geometric information about the free boundary.
Our work treats a local, scale-invariant version of this problem,
in which $u$ is a Green function with some fixed pole,
and in that case, our hypothesis that
caloric measure belongs to weak-$A_\infty$ with
respect to $\sigma$, can be used to obtain
some quantitative, scale invariant
control on the oscillations of $|\nabla u|$ (thus $u$ is in some sense
``close-to-linear").  In turn,
the latter control allows us to show that $\Sigma$ is
parabolic uniformly rectifiable, which itself can be viewed
as a quantitative, scale-invariant sort of flatness of the free boundary.

We mention that the idea of using some version of
almost linearity of the Green function to deduce
geometric flatness has a long history in the theory of free boundary problems,
beginning with \cite{AC} in the elliptic ``small constant" setting,
in which connection see also \cite{Jerison-Calpha, KT1,KT2,KT3},
and in the elliptic ``large constant" setting
(the analogue of our results here), for which see \cite{LV,HLMN}; see also \cite{DM}.
Previous parabolic work in the small constant setting also used this idea, see
\cite{HLN1,HLN2,E}
(these works present parabolic versions of
\cite{AC,Jerison-Calpha, KT1,KT2,KT3}).
We refer the reader to the introduction to
\cite{BHMN} for a more detailed review of this history.

As regards \cite{E},
we point out that the {\em conclusion} in the large constant case treated
in the present paper
(namely that $\pom$ is uniformly rectifiable),
is a {\em hypothesis} in \cite{E},
and for that matter, in all the other works, mentioned above,
treating the small constant case in either the elliptic\footnote{In the
small constant elliptic case,
the hypothesis of uniform rectifiability is imposed implicitly: it is
a consequence of the fact that a
Reifenberg flat domain with Ahlfors regular boundary is,
in particular, a chord-arc domain by the results of \cite{DJ},
and thus has a uniformly
rectifiable boundary.}
or parabolic setting.
The assumption of parabolic uniform rectifiability
is made explicitly in \cite{HLN2,E}, to rule out
the case of a non-regular Lip(1,1/2) graph with
vanishing constant (see the example in \cite[p 384]{HLN1}).
Our results here therefore sharpen those of \cite{E}
and \cite{HLN2}, since one now sees that the
hypothesis of parabolic uniform rectifiability in those works is redundant:  indeed,
parabolic uniform rectifiability is now a consequence of
our Theorem \ref{main.thrm}, since
the assumptions made in \cite{HLN2,E} imply in particular the hypotheses
of our theorem (at least at all sufficiently small scales).
For related developments in the elliptic
(in particular, the harmonic) case,
we refer to \cite{HLMN,MT,AHMMT} and the references therein.

{
In light of the outline given above, the results of this paper provide a definitive contribution to the theory. Under the
essentially minimal background hypotheses that the domain satisfies an interior corkscrew condition, and has a time-symmetric parabolic Ahlfors-David regular boundary, we solve a 1-phase caloric free boundary problem that in turn
is equivalent to
showing that solvability of the Dirichlet problem for the heat equation in $\Omega$ with boundary data in $L^p(d\sigma)$, for some $p \in (1,\infty)$, implies that
$\Sigma $ is parabolically uniformly rectifiable. }

\subsection{Outline of the proof} The proof of Theorem \ref{main.thrm} is rather intricate,
so to assist the reader we shall now provide a brief outline of the argument, which
proceeds in several steps:  a preliminary Step 0, and then the main Steps 1-4.  It is
worth mentioning that Steps 1-4 already give a more self-contained
approach to the corresponding elliptic
result, for the Laplace operator, treated in \cite{HLMN}, as our argument in the present paper
circumvents the use of any of the ``Carleson set" geometric characterizations of uniform rectifiability
developed in \cite{DS2}.
% In particular, we also give a different route to the main result in \cite{HLMN}.
More important, in contrast to the elliptic setting, the
(parabolic) weak half-space approximation condition, WHSA,
which we establish as Step 1, need {\em not} imply
parabolic uniform rectifiability of the boundary,
see \cite[Observation 4.19]{BHHLN-BP}. This obstruction is
one of the main reasons that the parabolic case is significantly
more challenging than the elliptic one.

\smallskip

\noindent
{\bf Step 0 - The Corona decomposition implies parabolic uniform rectifiability.}
In this preliminary step, we observe that the conclusion
of Theorem \ref{main.thrm} can be reduced to establishing the existence of a
(semi-coherent) Corona
decomposition in terms of {\em regular} Lip(1,1/2) graphs; indeed this is
precisely the content of Theorem \ref{coronaisUR.thrm}.
Thus, the goal in the main steps 1-4 is to construct such a Corona decomposition.

\noindent
{\bf Step 1 - The parabolic weak half-space approximation condition.} In the first step,
summarized as Theorem \ref{WHSA.thrm} below,
we prove that the hypotheses of Theorem \ref{main.thrm} imply that $\Sigma=\pom$ satisfies the parabolic {\em weak half-space approximation} (WHSA)
condition (see Definition \ref{def2.14});
in particular we thus prove a parabolic version of the main conclusion in \cite{HLMN}
(see Theorem \ref{WHSA.thrm}). In the elliptic case treated in \cite{HLMN},
we were effectively done at this point, since (as observed
in \cite{HLMN}) %, with the use of some ideas of \cite{LV})
elliptic WHSA already implies uniform rectifiability of
the boundary.  As noted in \cite[Observation 4.19]{BHHLN-BP} (and above),
the latter implication fails in the
parabolic setting, and this failure thus forces
the extensive additional analysis in Steps 2-5.

To carry out Step 1, we first construct a preliminary Corona decomposition of
$\dd = \dd(\Sigma)$ (the dyadic cubes on $\Sigma$), see Lemma \ref{initialcorona1.lem},
in which we produce a disjoint decomposition $\dd = \dd(\Sigma)=\cG\cup\cB$ such that the cubes in $\cB$ (bad cubes) satisfy a Carleson packing condition. Furthermore, the collection
$\cG$ (good cubes) can be subdivided into a collection
of disjoint stopping time trees, $\{\sbf\}$, each having a
maximal cube $Q(\sbf)$, and such that the cubes $\{Q(\sbf)\}$ also
satisfy a Carleson packing condition. By construction, each
tree $\sbf$ is semi-coherent (see Definition \ref{d3.11}) and has the property that
\[
\mu(Q)\approx {\sigma(Q)},\quad  \text{for all $Q \in \sbf$},
\]
where $\mu$ is a normalized caloric measure.
%  taking the (physical) size of $\sbf$, $\sigma(Q(\sbf))$.
Moreover, the spatial gradient of the normalized Green function is
non-degenerate at certain corkscrew/reference point associated
to $Q$. For each stopping time tree $\sbf$, we then
split the cubes in $\sbf$ % $\dd_{\F,Q_0}$
into subclasses, which either pack in the Carleson sense, or which enjoy the property that
the spatial gradient of the normalized Green function is non-degenerate, and
does not oscillate too much, and hence that the Green function is roughly a linear function in the spatial variable in (a component) of the Whitney region associated to a cube $Q$.
This approximate linearity then yields the geometric flatness of the boundary, modulo packing, that is entailed by the WHSA.  We remark that Step 1 is significantly more delicate than its
elliptic analogue in \cite{HLMN}, due  to the time-directedness of estimates.

The next three steps are jointly
summarized in Proposition \ref{finalprop} (see Remark \ref{prop8.2proofoutline}),
from which the conclusion of Theorem \ref{main.thrm} follows almost immediately.
% (as we note in Section \ref{Sec7}, just after the statement of the proposition).

\noindent
{\bf Step 2 - Construction of Lip(1,1/2) graphs for the Corona decomposition.}
% To proceed, our strategy to the proof of Theorem \ref{main.thrm} is, building on
% Theorem \ref{WHSA.thrm} and Step 1,
In the second step, we will
 show that, modulo a set of bad cubes $\mathcal{B}_{\sbf}$
 enjoying a Carleson packing condition, we can
sub-divide each $\sbf$ from Step 1 into
a collection of disjoint
stopping time trees $\{\sbf^*\}$ such that for each $\sbf^*$,
there is a Lip(1,1/2) graph $\Gamma_{\sbf^*}$ with certain
properties. % adapted to $\sbf^*$.
In particular (see Lemma \ref{graphlip.lem} and Lemma \ref{graphclosecrna.lem} below),
$\Gamma_{\sbf^*}$ is the graph of a Lip(1,1/2)
function $\psi_{\sbf^*}$ with small constant, and
$\dist(Q,\Gamma_{\sbf^*})\lesssim \diam(Q)$ for each $Q\in\sbf^*$.
Moreover, a localized portion of the region above the graph of $\psi_{\sbf^*}$ does
not meet $\Sigma$, is above $\Sigma$ in the coordinate system constructed,
and is contained in one component of $\Omega$.

The construction in this step is somewhat similar to that in \cite{BHHLN-Corona},
which in itself is modelled on \cite{DS1}, and in our setting relies heavily on
the fact that all cubes in $\sbf^*$ satisfy the WHSA condition.

\noindent
{\bf Step 3 -  The Lip(1,1/2) graph $\Gamma_{\sbf^*}$ is a
regular Lip(1,1/2) graph.} In the third step,
we prove that the  Lip(1,1/2) graph $\Gamma_{\sbf^*}$, constructed in Step 2,
is in fact a {\em regular} Lip(1,1/2) graph (recall that in the context of Lip(1,1/2) graphs,
 this means that the graph is parabolic
uniformly rectifiable). We achieve this by first
`pushing' the (weak)-$A_\infty$ estimates
for the caloric measure for $\Omega$ % $\mathbb R^{n+1}\setminus \Sigma$
onto the caloric measure $\hm_{\sbf^*}$ associated to
the domain above the approximating graph $\Gamma_{\sbf^*}$.
Since caloric measure for the
graph domain is doubling, we thus
obtain the full $A_\infty$ property for $\hm_{\sbf^*}$, and may
 then invoke the main result of \cite{BHMN}.

\noindent
{\bf Step 4 - The Corona decomposition: verifying the Carleson packing condition.}
In the fourth step, for each $\sbf$ as in Step 1, we prove that
a Carleson packing condition holds for
the set of bad cubes $\B_\sbf$, and also for
the collection $\{Q(\sbf^*)\}_{\sbf^*\subset \sbf}$
of  maximal cubes associated to the
stopping time trees $\{\sbf^*\}$.
This step is quite technical.  The proof will exploit the fact
that, by our results in Step 3, the non-tangential/square function
estimates of \cite{B89} are available to us
 in the regions above $\Gamma_{\sbf^*}$, and our argument is inspired by
the proof of the so-called ``$\eps$-approximability" property
 (see \cite[Chapter VIII, Section 6]{G}).  On the other hand, we are
 forced to overcome a significant technical obstacle, namely that we do not
 know that the corkscrew/reference points alluded to in Step 1, corresponding to
 different cubes, can be connected
 in a nice enough way, and this adds considerable complication to the
 $\eps$-approximability argument.
Matters would be significantly easier
 if we already knew that $\Sigma$ was uniformly rectifiable
  (see \cite[Section 5]{HMM}), but of
 course, that is what we are trying to prove.

\subsection{Organization of the paper} In Section \ref{Sec2}, which is of
a preliminary nature, we introduce some basic notation, terminology, and definitions
used throughout the paper, e.g.,
parabolic Ahlfors-David regular sets, the corkscrew condition,
the Dirichlet problem, and caloric measure.
Furthermore (and this is particularly important due to the technical complexity
of the paper), at the end of the section we outline our conventions concerning
floating constants. In Section \ref{Sec3}, we introduce further % dyadic
notation and definitions pertaining to dyadic cubes, Whitney regions, and chain
regions, and we introduce the (parabolic) weak half-space approximation
(WHSA) condition.  We note that here we use only the term ``chain
region", in contrast to \cite{HLMN} where the term ``Harnack chain region" is used.
The reason is that the term Harnack chain region, in the parabolic setting,
typically entails time directedness, whereas for our purposes
in the context of the parabolic WHSA property, the notion of (not necessarily time directed)
chain regions is the correct one.
In particular, the parabolic Harnack inequality will not be used in these
regions to make estimates.

In Section \ref{Sec4}, we define regular
Lip(1,1/2) graphs and parabolic uniform rectifiability, and we
define and discuss Corona decompositions.
We also state Theorem \ref{coronaisUR.thrm}, which provides the foundation for the
proof of Theorem \ref{main.thrm}.
In Section \ref{Sec4.5}, we establish a number of estimates for
the Green function, in domains having parabolic Ahlfors-David regular
boundary, to be used in the sequel.
In Section \ref{Sec5} we start the argument aiming at
establishing that $\Sigma$ has the parabolic weak half-space approximation condition, WHSA, in particular we start the argument towards completing Step 1 and the proof of Theorem \ref{WHSA.thrm}. Step 1 discussed above is carried out in Section \ref{Sec5} and Section \ref{Sec6}.  In Section \ref{Sec7}, we set up the structure for the proof of Theorem \ref{main.thrm} after having proved Theorem \ref{WHSA.thrm}. We here define the stopping rules based on which the stopping time trees $\{\sbf^*\}$ are constructed, and a number of preliminary reductions are carried out here. In particular, in Section \ref{Sec7} we
observe that  Theorem \ref{main.thrm} can be
reduced to  Proposition \ref{finalprop}, via Theorem \ref{coronaisUR.thrm}.
In the remaining sections we then
complete the proof of Proposition \ref{finalprop} along the following scheme.
First, in Section \ref{sec: graph}, proceeding along the lines of \cite{DS1, BHHLN-Corona},
we construct Lip(1,1/2) graphs for {every} stopping
time tree $\sbf^*$,
using that the $(\teps,K_0)$-WHSA condition (with $\teps=10\eps$) is
satisfied for every $Q\in\sbf^*$.
Thus, Step 2 outlined above is completed.
Second, in Section \ref{sec: reggraph}, using the weak-$A_\infty$ property and the main
result in \cite{BHMN}, we prove that the Lip(1,1/2) graphs constructed
are in fact regular Lip(1,1/2) graphs, thereby completing Step 3.
Third, in Section \ref{sec:pack}, we carry out Step 4, by proving
the cube packing property \eqref{graphclosepack}, which in turn
completes the proof of Proposition \ref{finalprop}, and hence also that of
Theorem \ref{main.thrm}.

\section{Notation and Definitions}\label{Sec2}

  Points in  Euclidean space-time
$ \mathbb R^{n+1}$ are denoted by $(X,t)$.
 When we wish to distinguish a particular
``vertical" direction in space, we shall use the notation $(X,t) = ( x_0,
 \dots,  x_{n-1},t)$,  where $ X = ( x_0, \dots,
x_{n-1} ) \in \mathbb R^{n } $, $n\geq 1$, and $t$ represents the time-coordinate.
We let  $\overline{E}$ and $\partial E$ be the closure and boundary of the set $ E \subset
\mathbb R^{n+1}$. $  \langle \cdot ,  \cdot  \rangle $  denotes  the standard inner
product on $ \mathbb R^{n} $ and we let  $  | X | = \langle X, X \rangle^{1/2} $ be
the  Euclidean norm of $ X\in\re^n$.  We let $\|(X,t)\|:=|X|+|t|^{1/2}$ denote the
parabolic length
of a space-time vector $(X,t)$. Given $(X,t), (Y,s)\in\mathbb R^{n+1}$
we let $$\dist(X,t,Y,s):= \|(X-Y,t-s)\| =
|X-Y|+|t-s|^{1/2},$$ and, more generally, we let
\[\dist(E_1,E_2) := \inf_{(X,t) \in E_1, (Y,s) \in E_2} \dist(X,t,Y,s),\]
denote the parabolic distance between  $E_1$ and $E_2$, where $E_1,E_2 \subseteq \mathbb R^{n+1}$. Also,
\[
\dist( X,t, E ) :=\inf_{(Y,s) \in E} \dist(X,t,Y,s)
\]
 is defined to equal the parabolic distance from  $  (X,t) \in \mathbb R^{n+1} $ to $ E$. We let $\diam(E)$ denote the parabolic diameter of $E$, i.e., the diameter of $E$ as measured using the parabolic distance function.

 Given $X\in\mathbb R^n$ we let $B(X,r)$ denote the open ball in $\mathbb R^n$, centered at $X$ and of radius $r$, and we let
\[ C( X, t,r ) \, := \, \{ ( Y, s )\in \mathbb R^{ n + 1 } :  Y\in B(X,r),\  | t - s | < r^2  \},\]
 whenever $(X,t)\in
\mathbb R^{n+1}$, $r>0$. We call $C(X,t,r) $ a parabolic cylinder ``of size $r$".
We also introduce the open time-forward and time-backward halves of $C(X,t,r)$,
\begin{align}\label{Cpm}
 C^+( X, t,r )&:= C( X, t,r )\cap \{ ( Y, s )\in \mathbb R^{ n + 1 } :s>t\},\notag\\
 C^-( X, t,r ) &:= C( X, t,r )\cap \{ ( Y, s )\in \mathbb R^{ n + 1 } : s<t\}.
 \end{align}
 We will also sometimes simply write $C_r$ (and $C^+_r, C^-_r$)
to denote a parabolic cylinder of size $r$ (and its time forward and backward halves)
whose center has been left implicit.

We let $ \d X $ denote  Lebesgue $n$-measure on  $\mathbb R^{n}$ and we
let $ \d t $ denote  Lebesgue $1$-measure on  $\mathbb R$.

\begin{definition}{\bf (Parabolic Hausdorff measure and the parabolic homogeneous
dimension).}
\label{parHausdef}
Given $\alpha \geq 0$, we let $\cH^\alpha$ denote the
 standard (Euclidean) $\alpha$-dimensional Hausdorff measure.
  We also define a {parabolic} Hausdorff measure of {homogeneous}
  dimension $\alpha$, denoted
  $\cH_{\text{par}}^\alpha$, in the same way that one defines standard (Euclidean) Hausdorff measure, but instead using the {parabolic} distance. That is,, for $\delta>0$, and for $E\subset \re^{n+1}$, we set
  \[ \cH_{\text{par},\delta}^\alpha(E):= \inf \sum_k (\diam(E_k))^\alpha,
  \]
  where the infimum runs over all countable such coverings of $E$, $\{E_k\}_k$, with $\diam(E_k)\leq \delta$ for all $k$. We then define
  \[
  \cH_{\text{par}}^\alpha (E) := \lim_{\delta\to 0^+} \cH_{\text{par},\delta}^\alpha(E)\,.
  \]
As is the case of classical Hausdorff measure, $ \cH_{\text{par}}^\alpha$ is a Borel regular measure.
We refer the reader to \cite[Chapter 2]{EG} for a discussion of the basic properties of standard
Hausdorff measure.  The arguments in \cite{EG} adapt readily to $  \cH_{\text{par}}^\alpha$.
In particular, one obtains a measure equivalent to   $\cH_{\text{par}}^\alpha$ if one defines
$\cH_{\text{par},\delta}^\alpha$ in terms of coverings by parabolic cylinders, rather than
arbitrary sets of parabolic diameter at most $\delta$.  As in the classical setting, we define the parabolic homogeneous dimension of a set
$A\subset \re^{n+1}$ by
\[\cH_{\text{par}, \text{dim}}(A):= \inf\left\{ 0\leq \alpha<\infty: \,\cH_{\text{par}}^\alpha(A)=0\right\}.
\]
We observe that $\cH_{\text{par}, \text{dim}}(\ree)=n+2$.
\end{definition}

For $\Sigma$, a fixed closed subset of $\ree$ of  homogeneous dimension
$\cH_{\text{par}, \text{dim}}(\Sigma)=n+1$, and
$ (X, t ) \in \Sigma $ and $r>0$, we define ``surface balls" (aka surface cylinders),
$$\Delta=\Delta(X,t,r):=\Sigma\cap C(X,t,r),$$
and we introduce its time-forward and time-backward halves
\begin{align*}
 \Delta^+=\Delta^+(X,t,r)&:= \Delta(X,t,r))\cap \{ ( Y, s )\in \mathbb R^{ n + 1 } :s> t\},\notag\\
\Delta^-=\Delta^-(X,t,r) &:= \Delta(X,t,r))\cap \{ ( Y, s )\in \mathbb R^{ n + 1 } :s<t\}.
 \end{align*}
As in the case of a cylinder $C_r=C(X,t,r)$, we will sometimes simply write
$\Delta_r,\Delta^+_r, \Delta^-_r$ when the center has been left implicit.

Given a closed set $\Sigma \subset \mathbb R^{n+1}$ we
define  a surface measure on
$\Sigma$ as the restriction of $\cH_{\text{par}}^{n+1}$ to $\Sigma$, i.e.,
\begin{equation}\label{sigdef}
\sigma = \sigma_\Sigma:=  \cH_{\text{par}}^{n+1}|_\Sigma\,.
\end{equation}

\begin{definition}[Parabolic Ahlfors-David regular]\label{ADR.def}
We say that a closed set $\Sigma\subset \rn\times \mathbb{R}$
is parabolic Ahlfors-David regular, parabolic ADR or simply ADR for short, with constant $M \ge 1$, if
\[M^{-1} r^{n+1} \le \sigma(\Delta(X,t,r)) \le  M r^{n+1},
\quad \forall (X,t) \in \Sigma, \,\, 0<r<\diam(\Sigma). \]
Assume now\footnote{This assumption suffices
for our purposes, but can be relaxed:
see Remark \ref{boundeddomain} and \cite[Definition 1.22]{GH1}.}
 that $\diam(\Sigma) =\infty$.  We
say that $\Sigma$ is {time-backward ADR} (or TBADR) if $\Sigma$ is ADR  and there exists a constant $M \ge 1$, such that
\[\sigma(\Delta^-(X,t,r)) \ge M^{-1} r^{n+1},
\quad \forall (X,t) \in \Sigma, \,\, 0<r< \infty.\]
We say that $\Sigma$ is {time-forward ADR} (or TFADR) if $\Sigma$ is ADR  and there exists a constant $M \ge 1$, such that
\[\sigma(\Delta^+(X,t,r)) \ge M^{-1} r^{n+1},  \quad
\forall (X,t) \in \Sigma,\,\, 0<r< \infty.\]
We say that $\Sigma$ is {time-symmetric ADR} (or TSADR) if $\Sigma$ is ADR, TBADR and TFADR. In this case, we will let $M$ be the maximum of the ADR, TBADR and TFADR constants (the constants in their definition), and we call this $M$ the time-symmetric ADR constant.
\end{definition}

Given a closed set $\Sigma\subset\mathbb R^{n+1}$ and $(Y,s)\in\mathbb R^{n+1}\setminus\Sigma$, we let $\delta(Y,s)$ denote the
\begin{equation}\label{dist}
\mbox{smallest $d$ such that }\Sigma\cap \overline{C(Y,s,d)}\neq\emptyset.
\end{equation}
Obviously $\delta(Y,s)\sim\dist(Y,s,\Sigma)$ with uniform constants for all $(Y,s)\in\mathbb R^{n+1}\setminus\Sigma$. When considering
the distance from $(Y,s)\in\mathbb R^{n+1}\setminus\Sigma$ to $\Sigma$,
we will use $\delta(Y,s)$ as our notion of parabolic distance from $(Y,s)$ to $\Sigma$.
 As may be seen by the following elementary lemma, this choice
has important implications when combined with the time-symmetric ADR assumption.
\begin{lemma}\label{touchp} Assume that $\Sigma\subset\mathbb R^{n+1}$ is time-symmetric ADR. Consider $(Y,s)\in\mathbb R^{n+1}\setminus\Sigma$ and let
$d:=\delta(Y,s)$ be defined as in \eqref{dist}. Assume that $(X,t)\in\Sigma$ is a point realizing $d$, i.e., $(X,t)\in\Sigma\cap \overline{C(Y,s,d)}$. Then
$(X,t)\in \partial B(Y,d)\times [s-d^2,s+d^2]$.
\end{lemma}
\begin{proof} $\partial C(Y,s,d)=\Gamma_1\cup\Gamma_2\cup\Gamma_3$ where $\Gamma_1:=\partial B(Y,d)\times [s-d^2,s+d^2]$, $\Gamma_2:=B(Y,d)\times \{s-d^2\}$, and
$\Gamma_3:=B(Y,d)\times \{s+d^2\}$. As $(X,t)\in \partial C(Y,s,d)$ we want to exclude the possibilities that $(X,t)\in\Gamma_2$ or $(X,t)\in\Gamma_3$. Assume that
$(X,t)\in\Gamma_2$. Then, by the construction of $C(Y,s,d)$, and the fact that $B(Y,s)$ is open,  we see that there exists $\eps=\eps(X,t)>0$ such that
$C^+(X,t,\tilde\eps)\cap\Sigma=\emptyset$ for all $0<\tilde\eps<\eps$. In particular, if $(X,t)\in\Gamma_2$ then $\Sigma$ can not be time-forward  ADR at $(X,t)$ and this conclusion violates one of the assumptions stated in the lemma. Similarly, if $(X,t)\in\Gamma_3$ then we conclude that $\Sigma$ can not be time-backward  ADR at $(X,t)$.
\end{proof}

\begin{definition}[Corkscrew condition]\label{CS.def}
Given an open set $\Omega \subset \ree$, we say that $\Omega$ satisfies the (interior) corkscrew condition if there exists a constant $\gamma \ge 1$ such that if $(X,t) \in \Sigma$,
$0<r<\infty$, then there exists a parabolic cylinder  $C(X_1,t_1,\rho)$, contained in $C(X,t,r)$, such that $C(X_1,t_1,\rho)\subset\Omega$  and with
\[\gamma^{-1} r \leq \rho < r.\]
\end{definition}

\begin{remark}\label{cp}
Given  $(X,t) \in \Sigma$,  $r>0$, we refer to the point $(X_1,t_1)$ in Definition \ref{CS.def} as a corkscrew point associated to $(X,t)$ and at scale $r$. Naturally, $(X_1,t_1)$ is not unique.
\end{remark}

Given the open set $\Omega\subset \ree$, we define its (forward) {parabolic
boundary} $\partial_p^+\Omega$ as
$$\partial_p^+\Omega:=\left\{(X,t)\in\partial\Omega: \forall r>0 , \,C^-(X,t,r)\,
\text{ meets }\, \ree\setminus \Omega\right\}.
$$
Similarly, we define its adjoint (or backward) {parabolic
boundary} $\partial_p^-\Omega$ as
$$\partial_p^-\Omega:=\left\{(X,t)\in\partial\Omega: \forall r>0 , \,C^+(X,t,r)\,
\text{ meets }\, \ree\setminus \Omega\right\}.$$
Note that if $\Sigma = \partial \Omega$ is TSADR,  then
$\partial_p^+\Omega=\Sigma=\partial_p^-\Omega$. In the following
we consistently assume that $\Sigma = \partial \Omega$ is TSADR (unless we explicitly say otherwise).

\begin{remark}\label{boundeddomain}
As just noted, $\Sigma =\pom=\partial_p^+\Omega=\partial_p^-\Omega$,
since we assume TSADR at all scales $0<r<\infty$.
Furthermore, the latter assumption implies that
$\diam(\Sigma) =\diam(\Omega) = \infty$, and $\Sigma$ is not contained
in any time-limited strip $\{(X,t):a<t<b\}$ with either $a$ or $b$ finite.
On the other hand, we have made this assumption only for the sake of simplicity, and in fact
our results can be extended to the case that $\Omega$ is
contained in a time-limited strip, % e.g., $\{(X,t):0<t<T\}$,
and in particular to the case that
$\Omega$ is a bounded space-time domain.
In this case, we define the ``bottom boundary" $\partial_{b}^+\Omega$
and ``adjoint (or time reversed) bottom boundary" $\partial_{b}^-\Omega$ (aka the
``abnormal" boundary) by
 \[
 \partial_{b}^{\pm}\Omega:=
 \left\{(x,t)\in\partial_p^{\pm}\Omega: \exists \,\eps>0 \,  \text{ such that } \,C_\eps^{\pm}(x,t)\,
\subset \Omega\right\}\,.
\]
Observe that $\partial_p^{\pm}\Omega =\pom\setminus  \partial_{b}^{\mp}\Omega$.
The ``lateral boundary" is then defined to be
\begin{equation}\label{lateral}
\partial_{\ell}\Omega := \partial_p^{+}\Omega\setminus  \partial_{b}^{+}\Omega
= \partial_p^{-}\Omega\setminus  \partial_{b}^{-}\Omega\,,
\end{equation}
and in this setting, we set $\Sigma= \partial_{\ell}\Omega$.  Let $T_{min}$ (resp. $T_{max}$) denote the supremum
(resp. infimum) of the numbers $a$ (resp. $b$) such that
$\Sigma \subset \{(X,t):a<t<b\}$.  We say that a surface cylinder
$\Delta=\Delta(X_0,t_0,r)$ is ``admissible" if
$r<\min(\sqrt{t_0-T_{Min}},\sqrt{T_{Max}-t_0})/8$, and we
then assume that TSADR, and the weak-$A_\infty$ condition
(Definition \ref{defweakAinfty.def} below) both hold for all
admissible\footnote{see \cite[Definition 1.22]{GH1}
for a precise definition of TBADR, etc., in this setting}
$\Delta$ (with uniform constants).
The conclusion that we draw is that
$\Sigma$ is locally parabolic uniformly rectifiable\footnote{see \eqref{eq1.sfll}}
 (again with uniform constants),
on all admissible $\Delta$.
We leave the details to the interested reader.
We remark that in this setting, if either $\diam(\Sigma)$ is finite, or
$T_{min} = -\infty$ and $\diam(\Sigma)=\infty$, then
the weak-$A_\infty$ condition on admissible $\Delta$
is equivalent to $L^p$ solvability of the initial-Dirichlet problem, with
vanishing initial data $u(X,T_0)=0$,
on subdomains $\Omega\cap \{t>T_0\}$, for $T_0>T_{min}$
(with constants depending on the ratio $(T_0-T_{min})/\diam(\Sigma)$
in the case that $\diam(\Sigma)$ is finite); see \cite[Theorem 2.10]{GH1}.

Finally, we mention that in the TSADR setting (although not necessarily in general),
the lateral boundary defined in
in \eqref{lateral} coincides with the ``quasi-lateral" boundary defined in
\cite{GH1}, as follows from \cite[Remark 1.26]{GH1} and the time reversibility of
the TSADR condition.
\end{remark}

\begin{definition}[Continuous Dirichlet problem]
Given $f\in C_c(\Sigma)$ (continuous with compact support in $\Sigma$),
we say that $u$ solves the continuous Dirichlet problem for the heat equation in $\Omega$ with boundary data $f$, if $(\partial_t-\Delta)u=0$ in $\Omega$, if
$u$ is continuous on the closure of $\Omega$, and if
\begin{align}\label{bc}
u|_{\Sigma}=f\mbox{ and }\lim_{\|(X,t)\|\to\infty}
u(X,t)=0.
\end{align}
The statement $u|_{\Sigma}=f$ means that
 \[\lim_{(X,t)\to (Y,s)} u(X,t) = f(Y,s),\]
 whenever $(Y,s) \in \Sigma$ and it is understood the limits above are taken within $\Omega$. Similarly, given $f\in C_c(\Sigma)$, we say that $u$ % $\widehat u$
 solves the continuous Dirichlet problem for the adjoint heat equation in
 $\Omega$ with boundary data $f$,
 if $(\partial_t+\Delta)u=0$ in $\Omega$, if $u$ % $\widehat u$
 is continuous on the closure of $\Omega$, and if
 \eqref{bc} holds. % with $u$ replaced by $\widehat u$.
\end{definition}

Given $\Omega:=\mathbb R^{n+1}\setminus\Sigma$ as above, let $u=u(X,t)$ be the
Perron-Wiener-Brelot (PWB) solution (see \cite{W1}, \cite[Chapter 8]{W2})
of the Dirichlet problem  for the heat equation, with data $f\in C_c(\Sigma)$. By the Perron construction,
for each point $(X,t) \in \Omega$,
the mapping
$f \mapsto u(X,t)$ is  bounded, and by
the resolutivity of functions $f\in C_c(\Sigma)$ (see \cite[Theorem 8.26]{W2}), %or resp., using the solvability assumption),
%By the construction of the Perron solution, and Wiener regularity,
it is also linear. The caloric %(resp. parabolic)
measure with pole $(X,t)$ is the unique probability measure $\omega^{X,t}(\cdot)$ on $\Sigma$, given by
the Riesz representation
theorem, such that
\begin{equation}\label{parmeasuredef}
u(X,t)=\iint_{\Sigma}f(Y,s)\, \d\omega^{X,t}(Y,s).
\end{equation}
Similarly, the adjoint caloric %(resp. parabolic)
measure with pole $(X,t)$ is the unique probability measure $\widehat\omega^{X,t}(\cdot)$ on $\Sigma$, given by
the Riesz representation
theorem, such that
\begin{equation*}%\label{parmeasuredef} \widehat
u(X,t)=\iint_{\Sigma}f(Y,s)\, \d\widehat\omega^{X,t}(Y,s)
\end{equation*}
is the PWB solution to the adjoint Dirichlet problem in $\Omega$.

The following lemma may be deduced by using the Wiener test proved in \cite{EG2}.
\begin{lemma}\label{Dir} Let $\Omega\subset\mathbb R^{n+1}$ be an open set. Assume that
$\Sigma = \partial \Omega$ is {time-backward ADR}. Then the continuous Dirichlet problem for the heat equation is uniquely solvable in $\Omega$, and given $f\in C_c(\Sigma)$,
the unique solution at $(X,t)\in\Omega$ is given as
\begin{equation*}%\label{parmeasuredefa}
u(X,t)=\iint_{\Sigma}f(Y,s)\, \d\omega^{X,t}(Y,s).
\end{equation*}
Similarly, if $\Sigma$ is {time-forward ADR}, then the continuous Dirichlet problem for the adjoint heat equation is uniquely solvable in $\Omega$, and given $f\in C_c(\Sigma)$, the unique solution at $(X,t)\in\Omega$ is given as
\begin{equation*}%\label{parmeasuredefa} \widehat
u(X,t)=\iint_{\Sigma}f(Y,s)\, \d\widehat{\omega}^{X,t}(Y,s).
\end{equation*}
\end{lemma}

We here define the hypothesis on caloric measure  stated in our main result. We let $\Omega\subset\mathbb R^{n+1}$ be an open set with time-symmetric ADR boundary $\Sigma = \partial \Omega$.

\begin{definition}[Weak-$A_\infty$]\label{defweakAinfty.def}
 We say the caloric measure $\hm$ is in weak-$A_\infty$ with respect to $\sigma$ if there exists constants $q > 1$, $C_1 \ge 1$, such that the following holds. For every $(X_0, t_0) \in \Sigma$, $0<r<\infty$, and $(Y_0, s_0) \in \Omega \setminus C(X_0,t_0, 4r)$, it holds that
 $$
 \hm^{Y_0,s_0}\big|_{\Sigma \cap C(X_0,t_0,r)} \ll \sigma,
 $$ and $k^{Y_0,s_0} := {\d\hm^{Y_0,s_0}}/{\d\sigma}$ satisfies,
\begin{equation}\label{wrhpdefeq.eq}
\left(\bariint_{\Delta(X,t, \rho)} (k^{Y_0,s_0})^{q} \, \d\sigma \right)^{1/q} \le
C_1 \bariint_{\Delta(X,t, 2\rho)} k^{Y_0,s_0} \, \d\sigma,
\end{equation}
whenever $(X,t) \in \Sigma$ and $C(X,t,2\rho) \subset C(X_0, t_0, r)$. If \eqref{wrhpdefeq.eq} holds, then we say that $k=k^{Y_0,s_0}$ satisfies a (local) weak reverse Hölder inequality (with exponent $q$).
\end{definition}

\subsection{Conventions concerning floating constants} Throughout the paper we will use the following conventions. We will write $c$ to denote, if not
otherwise stated, a constant satisfying $1\leq c <\infty$. We refer to $n$, the constant $M$ of Definition \ref{ADR.def}, the constant $\gamma$ of
 Definition \ref{CS.def}, and the constants $p$ and $C_1$ appearing in Definition \ref{defweakAinfty.def},  as {\it allowable parameters}. For all constants $A,B\in \mathbb R_+$, the notation $A\lesssim B$  means, unless otherwise stated, that $A/B$ is bounded from above by a positive constant depending at most on the allowable parameters. $A\gtrsim B$ should be interpreted similarly. We write $A\approx B$  if $A\lesssim B$ and  $B\lesssim A$. The notation $$A\lesssim_{\xi_1,...,\xi_k} B$$  means, unless otherwise stated, that $A/B$ is bounded from above by a positive constant depending at most on the allowable parameters and the constants $\xi_1$,..., $\xi_k$.  Finally, we note the we shall sometimes use $N$ to
 denote a large positive constant (at least 1),
 fixed for the duration of the proof of some lemma or proposition, but with a
possibly different value when used in a different proof.   Similarly, $\zeta$ will denote a small positive constant, with the same convention.
Specific constants greater than or equal to 1 will
sometimes be denoted as $N_0,N_1,N_2$, etc.~(or $M_0,M_1,M_2$, etc.).

\section{Dyadic cubes and associated geometrical notions}\label{Sec3}
In this section we introduce  dyadic cubes and associated notation to be used. We also define Whitney and chain type regions, as well as the
(parabolic) {weak  half-space approximation} condition.

\subsection{Dyadic cubes} We will work with
dyadic decompositions of the boundary $\Sigma$,
which in particular is a space of homogenenous type,
by virtue of the parabolic ADR property.
Such constructions are now ubiquitous in harmonic analysis,
and for the following lemma we refer to \cite{Ch,DS1,DS2,HK}.

\begin{lemma}\label{cubes}  Assume that $\Sigma  \subset \mathbb R^{n+1}$ is (parabolic) ADR  in the sense of Definition \ref{ADR.def} with constant $M$. Then $\Sigma$ admits a parabolic dyadic decomposition in the sense that there exist positive, finite
constants $c_0$, $\zeta$, and $c_*\,$, dpending only on dimension
and ADR, such that the following holds. For each $k \in \mathbb{Z}$,
there exists a collection of Borel sets, $\mathbb{D}_k$,  which we will call (dyadic) cubes, such that
$$
\mathbb{D}_k:=\{Q_{j}^k\subset \Sigma: j\in \mathfrak{I}_k\},$$ where
$\mathfrak{I}_k$ denotes some (countable)  index set depending on $k$, and such that the following hold:
\begin{list}{$(\theenumi)$}{\usecounter{enumi}\leftmargin=1cm \labelwidth=1cm \itemsep=0.2cm \topsep=.2cm \renewcommand{\theenumi}{\roman{enumi}}}
\item $\Sigma=\cup_{j}Q_{j}^k\,$ for each $k\in{\mathbb Z}$.

\item If $m\geq k$ then either $Q_{i}^{m}\subseteq Q_{j}^{k}$ or $Q_{i}^{m}\cap Q_{j}^{k}=\emptyset$.

\item For each $(j,k)$ and each $m<k$, there is a unique $i$ such that $Q_{j}^k\subset Q_{i}^m$.
\item $\diam\big(Q_{j}^k\big)\leq c_* 2^{-k}$.
\item Each $Q_{j}^k$ contains $\Sigma\cap \Delta(Z^k_{j},t^k_j, c_02^{-k})$ for some $(Z^k_{j},t^k_j)\in \Sigma$.
\item $\sigma(\{(Z,t)\in Q^k_j: \dist(Z,t,\Sigma\setminus Q^k_j)\leq \varrho \,2^{-k}\big\})\leq c_*\,\varrho^\zeta\,\sigma(Q^k_j),$
for all $k,j$, and all $\varrho\in (0,1)$.
\end{list}
\end{lemma}

\begin{remark}\label{remarkddcubes} We denote by
$\mathbb{D}=\mathbb{D}(\Sigma)$ the
collection of all $Q^k_j$, i.e.,
$$\mathbb{D} := \cup_{k} \mathbb{D}_k.$$ For a dyadic cube
$Q\in \mathbb{D}_k$, we set $\ell(Q) := 2^{-k}$, and we will
refer to this quantity as the size
of $Q$.  Evidently, by ADR,  $\ell(Q)\sim\diam(Q)$ with constants
of comparison depending at most on $n$ and $M$. Furthermore,
given $Q\in \mathbb{D}$ and $k\in\mathbb Z_+$, we let
$Q'\in \mathbb{D}_k(Q)$ if $Q'\subset Q$ and $\ell(Q') := 2^{-k}\ell(Q)$,
and we set $\dd(Q):=\{Q'\in\dd(\Sigma): Q'\subset Q\}$.
\end{remark}

\begin{remark}\label{remarkscube}
Set
$$M(Q):=\max\big(\ell(Q),\diam (Q)\big), \,\text{ and }\,
m(Q):=\min\big(\ell(Q),\diam (Q)\big)\,.$$
Note that $(iv)$ and $(v)$ of Lemma \ref{cubes}, and the fact that
$\ell(Q)\sim\diam(Q)$,
imply that there exists, for each cube $Q\in\mathbb{D}_k$, a point $(X_Q,t_Q)\in \Sigma$, a radius $r_Q>0$, and  a cylinder $C(X_Q,t_Q,r_Q)$, such that
%$r\approx2^{-k} \approx\diam (Q)$ and
\begin{equation}\label{cube-ball}
M(Q)\lesssim r_Q \leq m(Q), \text{ and } \Sigma\cap C(X_Q,t_Q,r_Q)\subset Q \subset \Sigma\cap C(X_Q,t_Q,cr_Q),\end{equation}
for some uniform constant $c$. We introduce the associated surface ball/cylinder
\begin{equation}\label{cube-ball2}
\Delta_Q:= \Sigma\cap C(X_Q,t_Q,r_Q),\end{equation}
and we will refer to the point $(X_Q,t_Q)$ as the ``center" of $Q$.
Given a cube $Q \in \mathbb{D}$ and $\lambda > 1$ we use the notation
\begin{equation}\label{dilatecube}
\lambda Q := \{(X,t) \in \Sigma : \dist(X,t, Q) < (\lambda - 1) \diam(Q)\}.
\end{equation}
\end{remark}

\begin{remark} For any $Q\in \mathbb{D}(\Sigma)$, we set
$\mathbb{D}_Q:= \{Q'\in \mathbb{D}:\,Q'\subseteq Q\}\,.$ Given $Q_0\in \mathbb{D}(\Sigma)$ and a family $\mathcal{F}=\{Q_j\}\subset \mathbb{D}$ of pairwise disjoint cubes, we set
\begin{equation}\label{eq:def-sawt}
\mathbb{D}_{\mathcal{F},Q_0}: =\mathbb{D}_{Q_0}\!\setminus\Big(\bigcup_{Q_j\in\mathcal{F}}\mathbb{D}_{Q_j}\Big).
\end{equation}
\end{remark}

Following \cite{DS2}, we will work with coherent and semi-coherent collections of dyadic cubes. Let $\Sigma  \subset \mathbb R^{n+1}$ be parabolic ADR,  in the sense of Definition \ref{ADR.def}, with dyadic cubes $\mathbb D(\Sigma )$.

\begin{definition}[Coherent and semi-coherent stopping time trees]\label{d3.11}  Let $\sbf\subset \mathbb D(\Sigma )$. We say that $\sbf$ is coherent if the following three conditions hold.
\begin{list}{$(\theenumi)$}{\usecounter{enumi}\leftmargin=1cm \labelwidth=1cm \itemsep=0.2cm \topsep=.2cm \renewcommand{\theenumi}{\alph{enumi}}}
\item  $\sbf$ contains a unique maximal element $Q(\sbf)$ which contains all other elements of $\sbf$ as subsets.
\item  If $Q$  belongs to $\sbf$, and if $Q\subset \widetilde{Q}\subset Q(\sbf)$, then $\widetilde{Q}\in {\bf S}$.
\item  Given a cube $Q\in \sbf$, either all of its children belong to $\sbf$, or none of them do.
\end{list}
We say that $\sbf$ is semi-coherent if only conditions $(a)$ and $(b)$ hold.
\end{definition}

\subsection{Whitney regions}\label{cubes.sect+} In the following we let $\mathcal{W}=\mathcal{W}(\Omega)$ denote a collection
of (closed) dyadic (parabolic)  Whitney cubes of $\Omega$, constructed so that the cubes in $\mathcal{W}$, denoted $\{I\}$,
form a covering of $\Omega$ with non-overlapping interiors, and  such that
\begin{equation}\label{eqWh1} 4\, {\rm{diam}}\,(I)\leq \dist(4 I,\pom) \leq  \dist(I,\pom) \leq 40 \, {\rm{diam}}\,(I),\end{equation}
and
\begin{equation}\label{eqWh2}\diam(I_1)\approx\diam(I_2), \mbox{ whenever $I_1$ and $I_2$ meet.}
\end{equation}
In \eqref{eqWh2}, the implicit constants only depend on $n$. Given a small, positive parameter $\tau$, and given $I\in\mathcal{W}$, we let
\begin{equation}\label{eq2.3*}I^* =I^*(\tau) := (1+2\tau)I
\end{equation}
denote the corresponding fattened Whitney cube.
We fix $\tau$ sufficiently small depending only on $n$, so that the cubes $\{I^*\}$  retain the usual properties of Whitney cubes,
in particular that
\begin{equation}\label{eq2.3*g}\diam(I) \approx\diam(I^*) \approx\dist(I^*,\Sigma) \approx\dist(I,\Sigma),
\end{equation}
where, strictly speaking, the implicit constants now also depend on $\tau$. In fact, by taking $\tau$ smaller, we can ensure that
\begin{equation}\label{eq2.3*g+}
\mbox{\eqref{eq2.3*g} holds with  $(1+ 100\tau)I$ in place of $I^*$.}
\end{equation}
Note also that taking $\tau$ small enough we can guarantee that
\begin{equation}\label{Whitneytouch}
(1+ 100\tau)I\cap (1+ 100\tau)J\neq\emptyset \ \Longleftrightarrow\ \partial I\cap\partial J\neq\emptyset.
\end{equation}

We next associate Whitney regions to the dyadic cubes on the boundary. Let $\mathcal{W}=\mathcal{W}(\Omega)$ denote the collection
of (closed) dyadic Whitney cubes of $\Omega$  introduced above, see \eqref{eqWh1} and \eqref{eqWh2}. Assuming that $\Sigma=\pom$ is parabolic ADR  and given $Q\in \mathbb{D}(\Sigma)$, we introduce, for a given parameter $K_0\gg 1$,
\begin{equation}\label{eq2.1}
\mathcal{W}_Q=\mathcal{W}_Q^{K_0}:= \left\{I\in \mathcal{W}:\,K_0^{-1} \ell(Q)\leq \ell(I)
\leq K_0\,\ell(Q),\, {\rm and}\, \dist(I,Q)\leq K_0\, \ell(Q)\right\}.
\end{equation}
Given a small, positive parameter $\tau$, and given $I\in\mathcal{W}$, the fattened Whitney cube $I^* =I^*(\tau)$ was introduced in \eqref{eq2.3*} and the cubes $\{I^*\}$  retain the usual properties of Whitney cubes stated in \eqref{eq2.3*g}. We will consistently assume that the construction is done so that
also \eqref{eq2.3*g+} holds. Given $I\in\W$ with $\ell(I)\lesssim \diam(\Sigma)$,  pick $(Y_I,s_I) \in \Sigma$ such that $\dist(Y_I,s_I,I) = \dist(I, \Sigma)$.
Let ${Q}_I\in\dd$ be such that $\ell(Q_I) = \ell(I)$ and $Q_I\ni(Y_I,s_I)$. Assuming that $K_0$ is large enough we can guarantee that $I\in \W_{Q_I}$ and also  that this property holds for any other possible choice of  $(Y_I,s_I)$ with $\dist(Y_I,s_I,I) = \dist(I, \Sigma)$.

We introduce Whitney regions
with respect to $Q$ by setting
\begin{equation}\label{eq2.3}
U_Q:= \bigcup_{I\in \mathcal{W}_Q}I^*\,.
\end{equation}
We observe that the Whitney region $U_Q$ may have more than one connected component,
but that the number of distinct components is uniformly bounded, depending only upon $n$, $K_0$ and $\tau$. We enumerate the components of $U_Q$
as $\{U_Q^i\}_i$. Note that by construction
\begin{equation}\label{cnpclsuq.eq}
\dist((Y,s), Q) \le c'_nK_0\ell(Q), \quad \forall (Y,s) \in U_Q
\end{equation}
for a purely dimensional constant $c'_n$.

\subsection{Chain regions}\label{cubes.sect++} Given $U_Q$ as in \eqref{eq2.3} we here introduce additional and  enlarged  Whitney regions. In the following $\eps>0$ is a small but fixed parameter.  Given $\eps>0$,
and given $Q\in\mathbb{D}(\Sigma)$,
we write
\begin{equation}\label{epschain}
(X,t)\sim_{\eps,Q} (Y,s)
\end{equation}
if $(X,t)$ may be connected to $(Y,s)$ by a chain of
at most $\eps^{-1}$ parabolic ``Whitney cylinders", i.e., cylinders of the form
$C_k:=C(Y^k,s^k,\delta(Y^k,s^k)/2)$, with
$1\leq k \leq N\leq \eps^{-1}$, such that
$C_k$ meets $C_{k+1}$, $1\leq k\leq N-1$, and
$$\eps^3\ell(Q)\leq\delta(Y^k,s^k)\leq \eps^{-3}\ell(Q).$$
\begin{definition}\label{def2.11a} Given $Q\in\mathbb{D}(\Sigma)$ and $\eps>0$ small, we  let
\begin{equation}\label{eq2.3a}
\tilde U^i_Q=\tilde U^{\eps,i}_Q:= \left\{(X,t) \in\Omega:\, (X,t)\sim_{\eps,Q} (Y,s)\  {\rm for\, some\,} (Y,s)\in U^{i}_Q\right\},
\end{equation}
and
$$\tilde U_Q=\bigcup_i\tilde U^i_Q.$$
\end{definition}

\begin{remark}\label{r2.5}
Since $\tilde  U^i_Q$ is
an enlarged version of the connected component $U_Q^i$, it may be that
there are some $i\neq j$ for which
$\tilde  U^i_Q$ meets $\tilde  U^j_Q$.  However, this overlap will be harmless.
\end{remark}

Note that in the parabolic context it would be misleading to
refer to the chain regions introduced above as Harnack chain regions.
Indeed, in Definition \ref{def2.11a}, the chains and the parabolic cylinders
do not take the direction of time into account, and the parabolic
context only enters through the use of parabolic cylinders.
By contrast, true parabolic Harnack chains are time directed.
Thus, our chains are not the usual ones considered
in the context of parabolic problems,
but for our purposes, the present notion of ``chain region"
will be the correct one. In particular, the parabolic Harnack inequality will
not be used to compare values of non-negative solutions (or adjoint solutions)
at different points globally within the chain region.

\subsection{The (parabolic) weak half-space approximation (WHSA) condition}\label{secWHSA} Let us consider  $Q\in \mathbb{D}(\Sigma)$ and recall that there is a point
$(X_Q,t_Q)\in Q$, and a parabolic cylinder $C_Q:=C(X_Q,t_Q,r_Q)$ such that
\eqref{cube-ball} holds, i.e.,
\begin{equation*}
r_Q\approx\diam (Q), \text{ and }\Sigma\cap C_Q\subset Q \subset \Sigma\cap cC_Q,\end{equation*}
for some uniform constant $c>1$, where $cC_Q:=C(X_Q,t_Q,cr_Q)$. Given $\eps>0$ small, we define a parabolic
dilate of $C_Q$,
\begin{equation}\label{eq2.bstarstar}
C_Q^{*}=C_{Q,\eps}^{*}:= \eps^{-2} C_Q = C(X_Q,t_Q,\eps^{-2}r_Q).
\end{equation}

\begin{definition}\label{def2.13} Given $\eps>0$ and $K_0\ge 1$,
we say that $Q\in\mathbb{D}(\Sigma)$ satisfies the $(\eps,K_0)$-{local weak half-space approximation} condition, the $(\eps,K_0)$-{local WHSA} condition for short, or
the $\eps$-local WHSA with parameter $K_0$, if there is a half-space
$ H_Q$ and a hyperplane  $P_Q =\partial H_Q$, containing a line parallel to the time axis, such that
\begin{list}{$(\theenumi)$}{\usecounter{enumi}\leftmargin=1cm \labelwidth=1cm \itemsep=0.2cm \topsep=.2cm \renewcommand{\theenumi}{\roman{enumi}}}
\item $\dist(Z,\tau,\Sigma)\leq\eps\ell(Q)$, for every $(Z,\tau)\in P_Q\cap C_{Q,\eps}^{*}$

\item $\dist(Q,P_Q)\leq c_n K_0 \ell(Q)$,

\item $H_Q\cap C_{Q,\eps}^{*}\cap \Sigma=\emptyset$.
\end{list}
The constant $c_n$ in item $(ii)$ is purely dimensional, and will be specified
momentarily  (see Remark \ref{cnchoic.rmk}).
\end{definition}

Given $K_0$, we will consistently
assume that $\eps>0$ is chosen small enough that
$ K_0 \ll \eps^{-1}$.
In this case we note that   Definition \ref{def2.13} part $(ii)$
states that % if $Q$ satisfies the $\eps$-local WHSA  then its corresponding
the hyperplane $P_Q$ has an ample intersection with the cylinder $C_{Q}^{*}=C_{Q,\eps}^{*}$ in the sense that
\begin{equation}\label{intersect-WHSA}
\dist(X_Q,t_Q,P_Q) \lesssim  K_0 \,\ell(Q) %\ll \eps^{-\frac14}\,\ell(Q)
\,\ll\eps^{-2}\ell(Q).
\end{equation}
\begin{remark}\label{cnchoic.rmk}
Let us give an explicit description of $c_n$. Let $(Y,s) \in U_Q$ and let $(X,t) \in \Sigma$ be the point realizing $\delta(Y,s)$ then using  \eqref{cnpclsuq.eq} it holds that
\[\dist(X,t,Q) \le c_n'' K_0 \ell(Q),\]
where $c_n''$ depends only on dimension.
We then choose $c_n := 10(c_n' + c''_n)$ so that any point $(Z,\tau)$ along the line segment from $(Y,s)$ to $(X,t)$ satisfies
\[\dist(Z,\tau, Q) \le c_n K_0 \ell(Q).\]
\end{remark}

\begin{definition} \label{def2.14} Given $\eps>0$ and $K_0>0$,
we say that  $\Sigma\subset \ree$
satisfies the (parabolic) $(\eps,K_0)$-{weak half-space approximation} condition, the $(\eps,K_0)$-WHSA for short, if the set $\cB$, defined as the set of cubes  in $\mathbb{D}(\Sigma)$ for which the
$(\eps,K_0)$-local  WHSA condition fails, satisfies the packing condition
\begin{equation}\label{eq2.pack2}
\sum_{Q\in\cB,\, Q\subseteq R} \sigma(Q)  \lesssim_{\eps,K_0} \,\sigma(R),\quad \forall\, R\in \mathbb{D}(\Sigma)\,.
\end{equation}
\end{definition}

Note that in the definition of the $(\eps,K_0)$-local  WHSA condition, we insist that the plane $P_Q$ should contain a line parallel to the time axis. We will refer to such planes as
``\textbf{$t$-independent}" planes (in the sense that the equation of
such a plane depends on spatial variables only, in
any coordinate system that leaves the $t$-axis fixed).

\section{Parabolic uniform rectifiability} \label{Sec4}
Let $\Sigma \subset \mathbb R^{n+1}$ be a closed set.
For the moment, we give a general definition, in which one may
allow $\Sigma$ to be bounded.
Assume that $\Sigma$ is {parabolic Ahlfors-David regular}
with constant $M\geq 1$.
We introduce the parabolic version of the ``beta"-numbers of P. Jones:
$$
\beta( Z, \tau, r  ):=\beta_\Sigma( Z, \tau, r  ):=
\inf_{P \in \mathcal{P}}  \biggl ( \, \bariint_{  \Delta ( Z, \tau,r) }
\, \biggl (\frac {\dist ( Y,s, P )}{r}\biggr )^2  \d \sigma (Y, s )\biggr )^{1/2},
$$
whenever $(Z,\tau)\in \Sigma $, $0<r<\diam(\Sigma)$,
and $\mathcal{P}$  is the set of all $ n $-dimensional
hyperplanes $ P $ containing a line
parallel to the $ t $ axis (that is, $t$-independent planes).
We also introduce the measure on $\Sigma\times(0,\diam(\Sigma))$
$$\d \nu( Z, \tau,
r  ) :=\d \nu_\Sigma ( Z, \tau,
r  )  = \big(\beta_\Sigma  ( Z, \tau, r)\big)^2 \, \d \sigma ( Z, \tau) \, r^{ - 1 }
\d r.
$$

\begin{definition}\label{def1.UR}
Assume that $\Sigma  \subset \mathbb R^{n+1}$ is parabolic ADR with constant $M$. Let $\nu=\nu_\Sigma$ be defined as above. Then
 $\Sigma$ is parabolic uniformly rectifiable with constants $(M,c)$ if
\begin{align}\label{eq1.sf}
\sup_{(X,t)\in\Sigma,\ 0<\rho<\diam(\Sigma) }  \rho^{ -(n+1) }\nu ( \Delta(X,t,\rho) \times ( 0, \rho) ) \leq
c.
\end{align}
Furthermore, if $(X_0,t_0)\in \Sigma$, $r_0>0$, then we say that $\Sigma\cap C(X_0,t_0,r_0)$ is locally parabolic uniform rectifiable with constants $(M,c)$, if
\begin{align}\label{eq1.sfll}
\sup_{(X,t)\in\Sigma,\ C(X,t,\rho)\subset C(X_0,t_0,r_0)}  \rho^{ -(n+1) }\nu ( \Delta(X,t,\rho) \times ( 0, \rho) ) \leq
c.
\end{align}
\end{definition}

\subsection{Lip(1,1/2) and Regular Lip(1,1/2) graph domains}\label{ssRLip}

A function $\psi:\mathbb R^{n-1}\times\mathbb R\to \mathbb R$ is called Lip(1,1/2) with constant $b_1$, if
\begin{align}\label{1.1}
|\psi(x,t)-\psi(y,s)|\leq b_1(|x-y|+|t-s|^{1/2}),
\end{align}
whenever $(x,t)\in\mathbb R^{n}$, $(y,s)\in\mathbb R^{n}$.
If $$\Sigma = \{(\psi(x,t), x,t): (x,t) \in \mathbb{R}^{n-1} \times \mathbb{R}\},$$
in the coordinates $P^\perp \times P$,
for some $t$-independent hyperplane
$P \in \mathcal{P}$ and for some Lip(1,1/2) function $\psi$,
then we say that $\Sigma$ is a Lip(1,1/2) graph. An open set $\Omega\subset\mathbb R^{n+1}$ is said to be a
(unbounded) Lip(1,1/2)   graph
domain, with constant $b_1$, if
\begin{align}\label{1.1a}
\Omega=\Omega_\psi=\{(x_0,x,t)\in\mathbb
R^{n-1}\times\mathbb R\times\mathbb R:x_n>\psi(x,t)\},
\end{align} for some Lip(1,1/2)  function $\psi$ having Lip(1,1/2)  constant bounded by $b_1$.

 Given a function $\psi:\mathbb R^{n-1}\times\mathbb R\to \mathbb R$,  we let $D_{1/2}^t \psi  (x, t) $ denote
the half derivative in $ t $ of $ \psi ( x, \cdot )$ with  $x $ fixed.
This half derivative in time can be defined by way of the Fourier
transform using the multiplier $|\tau|^{1/2}$, or by
\begin{align} \label{1.8}
 D_{1/2}^t  \psi (x, t):= \widehat c \int_{ \mathbb R }
\, \frac{ \psi ( x, s ) - \psi ( x, t ) }{ | s - t |^{3/2} } \, \d s,
\end{align} for $ \widehat c$  properly chosen. We let $ \| \cdot \|_* $ denote the
semi-norm in parabolic $\text{BMO}(\mathbb R^{n})$ (replace standard cubes by parabolic cubes/cylinders in the definition of $\text{BMO}$).

\begin{definition}\label{goodgraph.def}
 We say that $ \psi = \psi ( x, t ):\mathbb R^{n-1}\times\mathbb R\to \mathbb R$ is a {regular} Lip(1,1/2)  function
with parameters $b_1$ and $b_2$,
if $\psi$ satisfies \eqref{1.1} and if
 \begin{align} \label{1.7}
 D_{1/2}^t\psi\in \text{BMO}(\mathbb R^n), \ \ \|D_{1/2}^t\psi\|_*\leq b_2<\infty.
\end{align}
  If $\Sigma = \{(\psi(x,t),x, t): (x,t) \in \mathbb{R}^{n-1} \times \mathbb{R}\}$
 in the coordinates  $P^\perp \times P$,
  for some $t$-independent
  plane $P\in \mathcal{P}$ and for some {regular} Lip(1,1/2) function $\psi$,
  then we say that $\Sigma$ is a {regular} Lip(1,1/2) graph.
\end{definition}

\subsection{Corona decompositions and  big pieces squared} As mentioned in the introduction,  in  \cite{BHHLN-BP,BHHLN-Corona,BHHLN-CME} a theory concerning equivalent characterizations of parabolic uniform rectifiability is emerging. Indeed, combining \cite{BHHLN-BP}, \cite{BHHLN-Corona}, we can conclude that if $\Sigma \subset \ree$ is  parabolic Ahlfors-David regular, then $\Sigma$ is parabolic uniformly rectifiable if and only if $\Sigma$ has a Corona decomposition with respect to regular Lip(1,1/2) graphs (see Definition \ref{unilateralcorona.def}), and equivalently $\Sigma$ has big pieces squared of regular Lip(1,1/2) graphs.
To prove Theorem  \ref{main.thrm}, we will construct a
Corona decomposition of $\Sigma$ with respect to regular Lip(1,1/2) graphs, thus,
equivalently, $\Sigma$ has big pieces squared of regular Lip(1,1/2) graphs, and
hence is parabolic uniformly rectifiable.

Let us define these terms.

\begin{definition}\label{unilateralcorona.def}
Suppose that $\Sigma \subset \ree$ is parabolic ADR with constant $M$. We say that $\Sigma$ has a Corona decomposition by regular Lip(1,1/2) graphs if there exist constants
$\eta > 0$ and  $K > 1$, and a disjoint decomposition
$\mathbb D(\Sigma) = \cG\cup\cB$, such that the following hold:
\begin{list}{$(\theenumi)$}{\usecounter{enumi}\leftmargin=1cm \labelwidth=1cm \itemsep=0.2cm \topsep=.2cm \renewcommand{\theenumi}{\roman{enumi}}}

\item  The collection $\cG$ can be subdivided into a collection of disjoint stopping time re\-gimes, $\mathcal{S}=\{\sbf\}$,  $$\cG = \bigcup\limits_{\sbf \in \mathcal{S}} \sbf,$$ such that each such tree $\sbf$ is coherent.

\item The maximal cubes $\{Q({\sbf})\}$ and the cubes in $\cB$ satisfy a Carleson
packing condition,
$$\sum_{{\sbf}: Q({\sbf})\subseteq R}\sigma\big(Q({\sbf})\big)+\sum_{Q\in\cB,\, Q\subseteq R} \sigma(Q)\leq\, c(n,M,\eta,K)\, \sigma(R),
\quad \forall R\in \mathbb D(\Sigma).$$

\item For each ${\sbf}$, there exists a coordinate system, and a
regular  Lip(1,1/2)  function $ \psi_{\sbf}:=\psi = \psi ( x, t ) : \mathbb R^{n-1}\times\mathbb R\to \mathbb R$, with
 parameters $b_1\leq\eta $ and $b_2$, with $b_2$ depending at most on $n$, $M$, $\eta$, and $K$, such if we define
 $\Gamma_{\sbf}:=\{(\psi_{\sbf}(x,t),x,t): (x,t)\in \mathbb R^{n-1}\times\mathbb R\}$, then
\begin{equation}\label{eq2.2a}
\sup_{(X,t) \in KQ} \dist(X,t,\Gamma_{{\sbf}} )\leq\eta\,\diam(Q),
\end{equation}
 for every $Q\in {\sbf}$.
\end{list}
\end{definition}

\begin{definition}\label{semiCorona}
If $(i)$--$(iii)$ of Definition \ref{unilateralcorona.def} hold, then we say that $\Sigma$
has an $(\eta,K)$-Corona decomposition by regular Lip(1,1/2) graphs. If
in $(i)$ of Definition \ref{unilateralcorona.def}, each stopping time
tree $\sbf$ is only semi-coherent, then we say  that
$\Sigma$ has a semi-coherent $(\eta,K)$-Corona decomposition by regular Lip(1,1/2) graphs.
\end{definition}

We observe that Definition \ref{unilateralcorona.def} gives a notion of Corona
decompositions which differs in a subtle way
from the one used in \cite{BHHLN-Corona}:
\begin{theorem}[{\cite[Theorem 3.1]{BHHLN-Corona}}]\label{main.bil.thrm}  Suppose that
 $\Sigma\subset\mathbb R^{n+1}$ is parabolic uniformly rectifiable
 with constants $(M,c_0)$. Let $\delta\ll 1$
and $\kappa\gg 1$ be two given  positive constants. Then there exists a disjoint decomposition
$\dd(\Sigma) = \tilde\G\cup \tilde\B$, satisfying the following properties.
\begin{list}{$(\theenumi)$}{\usecounter{enumi}\leftmargin=1cm \labelwidth=1cm \itemsep=0.2cm \topsep=.2cm \renewcommand{\theenumi}{\alph{enumi}}}
\item  The collection $\tilde\cG$ can be subdivided into a
collection of disjoint stopping time trees, $\tilde{\mathcal{S}}:=\{\tilde{\sbf}\}$,  $$\tilde\cG = \bigcup\limits_{\tilde{\sbf} \in \tilde{\mathcal{S}}} \tilde{\sbf},$$ such that each such stopping time tree $\tilde{\sbf}$ is coherent.

\item The maximal cubes $\{Q({\tilde{\sbf}})\}$, and the cubes in $\tilde\cB$, satisfy a Carleson
packing condition,
$$\sum_{{\tilde{\sbf}}: Q({\tilde{\sbf}})\subseteq R}\sigma\big(Q({\tilde{\sbf}})\big)+\sum_{Q\in\tilde\cB,\, Q\subseteq R} \sigma(Q)\leq\, c(n,M,c_0,\delta,\kappa)\, \sigma(R),
\quad \forall R\in \mathbb D(\Sigma).$$

\item For each ${\tilde{\sbf}}$, there exists a coordinate system, and a
regular  Lip(1,1/2)  function $ \psi_{\tilde{\sbf}}:=\psi = \psi ( x, t ) : \mathbb R^{n-1}\times\mathbb R\to \mathbb R$, with
 parameters $b_1=c(n,M)\cdot\delta$ and $b_2=b_2(n,M,c_0)$, such if we define
 $\Gamma_{\tilde{\sbf}}:=\{(\psi_{\tilde{\sbf}}(x,t),x,t): (x,t)\in \mathbb R^{n-1}\times\mathbb R\}$, then
 \begin{equation}\label{graphapp}
\sup_{(X,t) \in \kappa Q} \dist(X,t,\Gamma_{{\tilde{\sbf}}} )\leq\delta\,\diam(Q),
\end{equation}
 for every $Q\in {\tilde{\sbf}}$.
\end{list}
\end{theorem}

Note that in Theorem \ref{main.bil.thrm}  we have the freedom to choose $\delta\ll 1$
and $\kappa\gg 1$ in the construction, while Definition \ref{unilateralcorona.def} simply states that there should exist $\eta > 0$ and  $K > 1$
such that $(i)$--$(iii)$ of Definition \ref{unilateralcorona.def} hold. In particular, in Definition \ref{unilateralcorona.def} the constant $\eta$ \textbf{does not have to be small}. Based on Theorem \ref{main.bil.thrm}, we see that if $\Sigma\subset\mathbb R^{n+1}$ is parabolic uniformly rectifiable, then $\Sigma$ has  a Corona decomposition in the sense of  Definition \ref{unilateralcorona.def}. The important fact for our purposes is that the existence of an
$(\eta,K)$-Corona decomposition by regular Lip(1,1/2) graphs in the sense of Definition \ref{unilateralcorona.def}, for some $\eta>0$ and $K>1$, is sufficient for parabolic uniformly rectifiability. The following result is a special case of  \cite[Theorem 1.2 ]{BHHLN-BP},  we refer to that reference for the precise definition of the notion of big pieces squared.

\begin{theorem}[{\cite[Theorem 1.2 ]{BHHLN-BP}}]\label{coronaisUR.thrm-}
Suppose that $\Sigma \subset \ree$ is parabolic ADR.
Assume that $\Sigma$ has an
$(\eta,K)$-Corona decomposition in the sense of
Definition \ref{unilateralcorona.def} by regular Lip(1,1/2) graphs, for some $\eta>0$ and $K>1$. Then $\Sigma$ has big pieces squared of regular Lip(1,1/2) graphs.
\end{theorem}

%In particular, to prove Theorem \ref{main.thrm} we will use the following theorem which is a special case of  Theorem 1.2 in \cite{BHHLN-BP}. We refer to \cite{BHHLN-BP} for the precise definition of the notion of big pieces squared.

We emphasize once more that to
apply this result, the parameter $\eta$
in the  $(\eta,K)$-Corona decomposition  of Definition \ref{unilateralcorona.def} need
not be small. To explain this heuristically, one can see that if  $\eta$ is `large', then this is compensated by the fact that the approximation in the Corona decomposition persist for many generations.
Furthermore, as stated in \cite[Theorem 4.16 ]{BHHLN-BP}, if  $\Sigma$ has big pieces squared of regular Lip(1,1/2) graphs, then $\Sigma$ is parabolic uniformly rectifiable. In particular, the following holds:
\begin{theorem}\label{coronaisUR.thrm--}
Suppose that $\Sigma \subset \ree$ is parabolic ADR with constant $M$. Assume that $\Sigma$ has, for some $\eta>0$ and $K>1$, an  $(\eta,K)$-Corona decomposition in the sense of Definition \ref{unilateralcorona.def} by regular Lip(1,1/2) graphs with constants
$b_1\leq\eta$ and $b_2$
(see Subsection \ref{ssRLip} and Definition \ref{goodgraph.def}).
Then $\Sigma$ is parabolic uniformly rectifiable with constants $(M,c)$ for some $c=c(n,M,\eta,K,b_1,b_2)$.
\end{theorem}

Let us note that in Theorem \ref{coronaisUR.thrm--},  % Definition \ref{unilateralcorona.def}
we assume that each
stopping time tree $\sbf$ is coherent (see Definition \ref{unilateralcorona.def}).
In particular, the results of \cite{BHHLN-BP} are
derived under this assumption.
However, in our proof of Theorem \ref{main.thrm} we
will establish directly only that $\Sigma$ has a
{\em semi-coherent} $(\eta,K)$-Corona
decomposition by regular Lip(1,1/2) graphs
in the sense of Definition \ref{semiCorona}, i.e.,
the stopping trees $\{\sbf\}$ we construct will be merely semi-coherent,
and not necessarily coherent.
Fortunately, this is not a serious obstacle:  indeed, in
\cite[p. 56]{DS2}, it is observed that any semi-coherent
Corona decomposition can be transformed into a coherent one, with the same values of
$\eta$ and $K$ (albeit with larger, but still uniform, packing constants, depending
only on allowable parameters).
Thus, we immediately obtain the following improvement of Theorem \ref{coronaisUR.thrm--}.

\begin{theorem}\label{coronaisUR.thrm}
Suppose that $\Sigma \subset \ree$ is parabolic ADR with constant $M$.
Assume that $\Sigma$ has, for some $\eta>0$ and $K>1$, a {\em semi-coherent}
$(\eta,K)$-Corona decomposition in the sense of
Definition \ref{semiCorona}, by regular Lip(1,1/2) graphs with constants
$b_1\leq\eta$ and $b_2$. Then $\Sigma$ is parabolic uniformly rectifiable with constants $(M,c)$ for some $c=c(n,M,\eta,K,b_1,b_2)$.
\end{theorem}

\section{Estimates for caloric measure and the Green function}\label{Sec4.5}

Let $\Omega \subset \ree$ be an open set with time-symmetric ADR boundary $\Sigma = \partial \Omega$. Recall that in Section \ref{Sec2} we introduced the Dirichlet problem and the associated caloric and adjoint caloric measures $\omega$ and $\widehat{\omega}$, respectively. Using $\omega^{X,t}$ we define the  Green function for $(\partial_t-\Delta)$ associated to $\Omega$, with pole at $(Y,s)\in \Omega$, as
\begin{align}\label{ghh1-}
G (X,t,Y,s)=\Gamma(X,t, Y,s)-\iint_{\Sigma}
\Gamma(\tilde X,\tilde t, Y,s) \, \d\omega^{X,t}(\tilde X,\tilde t),
\end{align}
where $\Gamma$ is the heat kernel. If we instead consider $(X,t)\in \Omega$ as fixed, then, for $(Y,s)\in
\Omega$,
\begin{align}\label{ghh1---}
G (X,t,Y,s)=\Gamma(X,t, Y,s)-\iint_{\Sigma}
\Gamma(X,t, \tilde X,\tilde t) \, \d\widehat\omega^{Y,s}(\tilde X,\tilde t).
\end{align}
Note that $G(X,t,Y,s)$ is caloric, i.e., a solution to the heat equation,
as a function of the variables $(X,t)$,
in the subdomain $\Omega \setminus \{(Y,s)\}$,
and $G(X,t,Y,s)$ is adjoint-caloric, i.e., a
solution to the adjoint heat equation, in the variables $(Y,s)$, in $\Omega \setminus \{(X,t)\}$.

In the following all implicit constants will depend only on $n$ and
the ADR constant $M$ unless otherwise stated.

\begin{lemma}[{\cite[Lemma 2.2]{GH1}}]\label{Bourgain} Assume that $\Omega\subset\mathbb R^{n+1}$ is such that
$\Sigma = \partial \Omega$ is time-backward ADR with constant $M$. Let $(X_0,t_0)\in\Sigma$, and let $0<r<\infty$.
Then there exists a constant
$c$, $1\leq c<\infty$, depending only on $n$ and $M$, such that
\begin{equation*}
\omega^{X,t}(\Delta(X_0,t_0,r))\gtrsim 1,
\end{equation*}
whenever $(X,t)\in \Omega\cap C(X_0,t_0,r/c)$. Similarly, if $\Sigma$ is time-forward  ADR with constant $M$, then the same conclusion
holds with $\omega$ replaced by $\widehat \omega$.
\end{lemma}

\begin{lemma}[{\cite[Lemma 2.5]{GH1}}]\label{continuity}
Assume that $\Omega\subset\mathbb R^{n+1}$ is such that
$\Sigma = \partial \Omega$ is time-backward ADR with
constant $M$. Let $(X_0,t_0)\in\Sigma$, let $0<r<\infty$, and for any $\kappa>0$, set
$C_{\kappa r}:= C(X_0,t_0,\kappa r)$. Assume that
$u$ is a bounded, non-negative solution to the
heat equation in $\Omega \cap C_{2r}\,$, vanishing continuously on
$\Delta(X_0,t_0,2r)$. Then
\begin{equation}\label{boundaryholder}
u(Y,s)\lesssim \left(\dfrac{\delta(Y,s)}{r}\right)^{\! \alpha}\sup_{\Omega\cap C_{2r}} u\,,
\qquad  (Y,s)\in \Omega\cap C_r\,,
\end{equation}
where the implicit constant and the exponent $\alpha>0$ depend only on
$n$ and $M$.
 The same conclusion remains true if $\Sigma$ is time-forward ADR with constant $M$,
and $u$ is a bounded non-negative solution to the
adjoint heat equation in $\Omega \cap C_{2r}\,$, vanishing continuously on
$\Delta(X_0,t_0,2r)$.
\end{lemma}

\begin{lemma}\label{CFMS1} Assume that $\Omega\subset\mathbb R^{n+1}$ is such that
$\Sigma = \partial \Omega$  is time-symmetric ADR with constant $M$.
Let $(Y,s)\in \Omega$ and let $(X_0,t_0)\in\Sigma\cap\overline{C(Y,s,d)}$ where $d=\delta(Y,s)$. Then there exists $c$, $1\leq c<\infty$, depending only on $n$ and $M$, such that if
$(X,t)\in \Omega\setminus C(Y,s,d/2)$, % $t\geq s+d^2$,
then
\begin{align*}
d^{n}G (X,t,Y,s)\lesssim\omega^{X,t}( \Delta(X_0,t_0,cd)).
\end{align*}
\end{lemma}
\begin{proof} Let $(Y,s)\in \Omega$ and let $\widehat\Omega:=\Omega\setminus \overline{C(Y,s,d/2)}$. Note that $G (X,t,Y,s)\equiv 0$ if $(X,t)\in \widehat\Omega$ and $t\leq s$. Furthermore,
$d^{n}G (X,t,Y,s)\lesssim 1$ whenever $(X,t)\in \partial C(Y,s,d/2)$ and $t\geq s$. This follows from the fact that $G (X,t,Y,s)\leq \Gamma(X,t,Y,s)$ and elementary estimates. Furthermore, using Lemma \ref{Bourgain} we see that there exists $c$, $1\leq c<\infty$, depending only on $n$ and the time-symmetric ADR constant $M$, such that
$\omega^{X,t}( \Delta(X_0,t_0,cd))\gtrsim 1$ for all $(X,t)\in \partial C(Y,s,d/2)$. The desired estimate now follows from the maximum principle in
$\widehat\Omega\cap\{(X,t): t>s\}$.
\end{proof}

\begin{remark}\label{remarkDYs}
We set some convenient notation.  Given $(Y,s)\in \Omega$,
let $(X_0,t_0)\in\Sigma\cap\overline{C(Y,s,d)}$, where $d=\delta(Y,s)$
(if there is more than one such $(X_0,t_0)$ we just pick one), and set
\begin{equation}\label{DYsdef}
\Delta_{Y,s} :=  \Delta(X_0,t_0,cd)\,,
\end{equation}
where $c$ is the constant in Lemma \ref{CFMS1}.  Thus, by the ADR property, the
conclusion of the lemma can be written as
\begin{equation}\label{CFMSDYs}
\frac{G (X,t,Y,s)}{\delta(Y,s)}\,\lesssim\,\frac{\omega^{X,t}( \Delta_{Y,s})}{\sigma(\Delta_{Y,s})}\,.
\end{equation}
\end{remark}

The proof of the next lemma is adapted from \cite{GHMN}.

\begin{lemma}[Caccioppoli at the boundary for Green's function]
\label{gensalsa-aEN}
% \label{cacc-bound}
% Let $(X_0,t_0)\in\Sigma$,
Assume that $\Sigma$ is time-symmetric ADR.
Let $C_r$ be a parabolic cylinder centered on $\Sigma$, and let
$C_{2r}$ denote its concentric
double (in the parabolic sense).
% \[Q_{4r}(x_0,t_0)\cap \partial\Omega = Q_{4r}(x_0,t_0)\cap
% \Sigma=:Q_{4r}\cap\Sigma=:\Delta_{4r}\,.\]
%  Fix $(X,t) \in \Omega\setminus C_{4r}$, and
 % set $u(Y,s) :=G(X,t,Y,s)$, so that
Let $u(Y,s):= G(X,t,Y,s)$, with $(X,t) \in \Omega\setminus C_{4r}$
% in $\Omega_{2r}:= \Omega\cap C_{2r}$, which vanishes continuously on
 % $\Delta_{2r}=\Sigma\cap C_{2r}$.
Then
\[ \iint_{\Omega_{r}} |\nabla u|^2 \,\lesssim \, r^{-2} \iint_{\Omega_{2r}} |u|^2\,. \]
\end{lemma}

\begin{proof}
Let $\eta\in C_0^\infty (C_{2r})$, with $0\leq \eta\leq 1$, and $\eta \equiv 1$ on $C_r$, such
 that $|\nabla \eta| \lesssim r^{-1}$, and $|\partial_t \eta|\lesssim r^{-2}$.
 Given $0<\eps\ll1$, let $\Phi=\Phi_\eps\in C^\infty (\Omega)$, with $0\leq \Phi\leq 1$,
 $\Phi(X,t) \equiv 1$ if $\delta(X,t)\geq \eps$, $\Phi(X,t) \equiv 0$ if $\delta(X,t)\leq  \eps/4$, satisfying
$|\nabla \Phi| \lesssim \eps^{-1}$, and $|\partial_t \Phi|\lesssim \eps^{-2}$.

Then
\[ \iint_{\Omega_{r}} |\nabla u|^2 =\lim_{\eps\to 0}  \iint_{\Omega_{r}} |\nabla u|^2 \Phi_\eps^2\,.
\]
In turn, by ellipticity and the definition of $\eta$, for fixed $\eps>0$,
\begin{multline*}\iint_{\Omega_{r}} |\nabla u|^2 \Phi_\eps^2 \lesssim
\iint_\Omega \nabla u\cdot \nabla u \left(\eta \Phi_\eps\right)^2\\[4pt]
= \iint_\Omega \nabla u\cdot \nabla \left( u \left(\eta \Phi_\eps\right)^2\right) -
 \iint_\Omega \nabla u\cdot \nabla (\eta^2)u  \Phi_\eps^2
 -  \iint_\Omega \nabla u\cdot \nabla ( \Phi_\eps^2)  u \eta^2\\[4pt]
=:\,  \mathbf I(\eps)+ \mathbf{II}(\eps)+ \mathbf{III}(\eps)\,.
\end{multline*}
We may handle term $\mathbf{II}(\eps)$ exactly as in the usual proof of Caccioppoli's inequality,
using Cauchy's inequality to hide a small term on the left hand side of the inequality, and then eventually letting
$\eps\to 0$.
For term $\mathbf I(\eps)$, we have
\[\mathbf I(\eps) = \iint_\Omega \partial_t u \,  u \left(\eta \Phi_\eps\right)^2
= - \iint_\Omega u^2\, \eta\, \Phi_\eps^2\, \partial_t \eta \, - \, \iint_\Omega u^2 \,\eta^2\, \Phi_\eps \,\partial_t \Phi_\eps
=: \mathbf I' (\eps)+ \mathbf I''(\eps)\,.
\]
By construction of $\eta$, the term $\mathbf I'(\eps)$
satisfies the desired bound, and we may let $\eps\to 0$.
To handle term $\mathbf I''(\eps)$, we let $\W$ denote a fixed collection of (parabolic) Whitney cubes of $\Omega$,
let $S_{\!\eps}:= \{(Y,s)\in \Omega:\, \eps/4 \leq \delta(Y,s)\leq\eps\}$, and we set
\[\W_\eps:=\left\{I \in \W:\, I\cap S_\eps \cap C_{2r}\neq \emptyset\right\}\,.\]
Note that by the adjoint version of \eqref{boundaryholder}
% (which uses TFADR in place of TBADR)
with $2r$ in place of $r$, and pointwise bounds for the Green function
\begin{equation}\label{eq.ubound}
u(Y,s) = G(X,t,Y,s) \lesssim \eps^\alpha r^{-n-\alpha}\,,
\end{equation}
for $(Y,s) \in C_{2r}\cap S_\eps$, and $(X,t) \in \Omega \setminus C_{4r}$.
Given $I \in \W_\eps$, and for each $(Y,s)\in I$, define $\Delta_{Y,s}$ as in
Remark \ref{remarkDYs}.  Let $(Y_I,s_I)$ be the center of $I$, and
set $\Delta_I:=  \kappa \Delta_{Y_I,s_I}$
with $\kappa$ chosen large enough that
$\Delta_{Y,s}\subset \kappa \Delta_{Y_I,s_I}$ for all $(Y,s)\in I$.
By construction of $\eta$ and $\Phi_\eps$, the definition of $u$, and \eqref{CFMSDYs}, \eqref{eq.ubound},
we then have
% \begin{multline*}
\begin{equation*}
\mathbf I''(\eps) \,\lesssim \,\eps^{-2+\alpha}r^{-n-\alpha} \sum_{I\in \W_\eps} \iint_I \hm^{X,t} (\Delta_I) \,\eps^{-n}
% \\[4pt]
\approx\, \eps^\alpha r^{-n-\alpha}  \sum_{Q\in \W_\eps}  \hm^{X,t} (\Delta_I) \,\lesssim \, \eps^\alpha r^{-n-\alpha}
\hm^{X,t}(\Delta_{3r}) \,,
\end{equation*}
% \end{multline*}
where we have used that $\ell(I) \approx \dist(I,\Sigma) \approx \eps$, and $|I|\approx \eps^{n+2}$, for $I\in \W_\eps$,
and that the sets $\Delta_I$ have bounded overlaps as $I$ ranges over $\W_\eps$.  Letting  $\eps\to 0$, we see that
$\mathbf{I}''(\eps) \to 0$.  Finally, we may handle term $\mathbf{III}(\eps)$ by using Cauchy's inequality to hide a small term on the left hand side of the inequality, plus another term (involving $|\nabla\Phi_\eps|^2$), which may be handle exactly like term
$\mathbf{I}''(\eps)$.  We omit the details.
\end{proof}

The proof of the next lemma is implicit in \cite{FGS}, and we follow their idea.

\begin{lemma}[Riesz Formula] \label{Representa}
Let $\varphi\in C_c^\infty(\ree)$. Then for all $(X,t)\in\Omega\,$,
\begin{equation}   \label{eqriesz}
\iint_{\Omega} \big(\nabla G(X,t,Y,s)\cdot\nabla \varphi(Y,s)
\,+\,G(X,t,Y,s) \partial_s\varphi(Y,s)\big)\,dYds
= \, \varphi(X,t)\, - \iint_{\Sigma} \varphi\, d\omega^{X,t} \,,
 \end{equation}
and for all $(Y,s)\in\Omega\,$,
  \begin{equation}  \label{eqriesz2}
\iint_{\Omega} \big(\nabla G(X,t,Y,s)\cdot\nabla \varphi(X,t)
- G(X,t,Y,s) \partial_t\varphi(X,t)\big)\,dXdt
= \, \varphi(Y,s)\, - \iint_{\Sigma} \varphi\, d\widehat{\omega}^{Y,s} \,.
\end{equation}
\end{lemma}

\begin{proof} We shall prove only \eqref{eqriesz}; the proof of \eqref{eqriesz2}
is analogous.

Recall that for $(X,t,Y,s)\in\Omega\times\Omega \setminus \{(X,t)=(Y,s)\}$,
we define $G$ by the formula
\[G(X,t,Y,s) := \Gamma(X,t,Y,s) - V_{Y,s}(X,t)\,,\]
where $\Gamma$ is the global fundamental solution, and
\begin{equation}\label{Vdef}
V_{Y,s}(X,t):= \iint_{\Sigma} \Gamma(Z,\tau,Y,s) \, d\hm^{X,t}(Z,\tau)\,.
\end{equation}
Define $u(Y,s):= G(X,t,Y,s)$ for all
$(Y,s) \in \ree\setminus\{(X,t)\}$
by setting $G(X,t,\cdot,\cdot) \equiv 0$ in $\ree\setminus\overline{\Omega}$.
This extension by zero then means that
$V_{Y,s}(X,t) =\Gamma(X,t,Y,s)$ for $(Y,s)\in \ree\setminus\overline{\Omega}$, and in that case
the formula \eqref{Vdef} remains valid for all $(X,t)\in\Omega$, and for all
$(Y,s)\in \ree\setminus \Sigma$.
% In particular, $G(X,t,\cdot,\cdot) \equiv 0$ in $Q_r \setminus \Omega$.
The left hand side of \eqref{eqriesz} then equals
\begin{multline*}
 \iint_{\ree} \big(\nabla G(X,t,Y,s)\cdot\nabla \varphi(Y,s)\,+\,G(X,t,Y,s) \partial_s\varphi(Y,s)\big)\,dYds\\[4pt]
 = \,
 \iint_{\ree} \big(\nabla \Gamma(X,t,Y,s)\cdot\nabla \varphi(Y,s)\,+\,\Gamma(X,t,Y,s) \partial_s\varphi(Y,s)\big)\,dYds
 \\[4pt]
 -\, \iint_{\ree} \big(\nabla V_{Y,s}(X,t)\cdot\nabla \varphi(Y,s)
 \,+\,V_{Y,s}(X,t) \partial_s\varphi(Y,s)\big)\,dYds
 \\[4pt] =:\,\ba(X,t) -\bb(X,t)\,.
\end{multline*}
By standard properties of the global fundamental solution,
\begin{equation}\label{eqH}
\ba (X,t)
= \varphi(X,t)\,,\qquad (X,t)\in\ree\,.
\end{equation}
Also, by the validity of \eqref{Vdef} for all $(Y,s)\in\ree\setminus\Sigma$,
\begin{multline*}
\bb(X,t) = % \\[4pt]
\iint_{\Sigma} \iint_{\ree} \big(\nabla \Gamma(Z,\tau,Y,s)\cdot\nabla \varphi(Y,s)\,+\,\Gamma(Z,\tau,Y,s) \partial_t\varphi(Y,s)\big)\,dYds\, d\omega^{X,t}(Z,\tau)
\\[4pt]
=\, \iint_{\Sigma} \ba(Z,\tau)\, d\omega^{X,t}(Z,\tau)\, =\, \iint_{\Sigma} \varphi(Z,\tau) \,d\omega^{X,t}(Z,\tau) \,,
\end{multline*}
where in the last step we have used \eqref{eqH} with $(Z,\tau)$ in place of $(X,t)$.
This proves \eqref{eqriesz}.
\end{proof}

The following is an immediate corollary of the previous results, combined with
\cite[Lemma 2.5]{GH1} (see Remark \ref{remarkcorollary} below).

\begin{corollary} [{\cite[Lemma 2.5]{GH1}}]
\label{continuity2}
Let $\Omega\subset\mathbb R^{n+1}$, and suppose that
$\Sigma = \partial \Omega$ is time-forward ADR with
constant $M$. Let $(X_0,t_0)\in\Sigma$, let $0<r<\infty$, and for any $\kappa>0$, set
$C_{\kappa r}:= C(X_0,t_0,\kappa r)$.
Let $u(Y,s)= G(X,t,Y,s)$, the Green function with pole
$(X,t)\in \Omega \setminus \overline{C_{2r}}$, with $t> t_0+ 4r^2$.  Then
% if $\Sigma$ is time-forward ADR with constant $M$, Then
\begin{equation}\label{boundaryholder2}
u(Y,s)\lesssim\left(\dfrac{\delta(Y,s)}{r}\right)^{\alpha}\bariiint_{\Omega\cap C_{2r}}
 u(Z,\tau)\, \d Z \d \tau\,,\qquad  (Y,s)\in \Omega\cap C_r\,.
\end{equation}
\end{corollary}
}

\begin{remark}\label{remarkcorollary}
 Lemma \ref{continuity}
(estimate \eqref{boundaryholder}) is proved in \cite[Appendix B]{GH1}.
The proof there is further extended to yield \eqref{boundaryholder2}, provided that
the extension of $u$ by zero in $C_{2r}\setminus \Omega$ is a subsolution (or adjoint subsolution)
in $C_{2r}$.  In the special case that $u$ is the Green function, one may verify the adjoint
subcaloric property in $C_{2r}$,
using the Riesz formula (Lemma \ref{Representa}).
In the rather general geometric context considered here, we do not know if the analogue of the
Riesz formula holds (that is, we do not know if there is a non-negative boundary Riesz measure)
for general non-negative solutions or adjoint solutions vanishing on a surface ball (i.e., for solutions other than Green functions).  On the other hand, for our purposes in this paper, it will be enough to know that \eqref{boundaryholder2}
holds for the Green function.
\end{remark}

\begin{lemma}\label{lemma:G-aver}  Assume that $\Omega\subset\mathbb R^{n+1}$ is such that
$\Sigma = \partial \Omega$  is time-symmetric ADR with constant $M$.   Let $(X_0,t_0)\in\Sigma$, $0<r<\infty$, consider $(X,t)\in \Omega$, and set
$u(Y,s):=G(X,t,Y,s)$.
For any $\kappa>0$, set $C_{\kappa r}:=C(X_0,t_0,\kappa r)$.
Then there exists $c$, $1\leq c<\infty$, depending only on $n$ and $M$, such that if $t\geq t_0+4cr^2$, then
\begin{equation}\label{eqn:aver-B}
\sup_{\Omega\cap C_r} u \lesssim
\bariiint_{\Omega\cap C_{2r}}  u(Y,s)\, \d Y\d s \lesssim r^{-n}{\omega^{X,t}(\Delta(X_0,t_0,cr))}.
\end{equation}
Furthermore, given $\xi \in (0,4]$ we have
\begin{equation}\label{eqn:aver-B2}
r^{-n-2}\iiint_{\Omega\cap C_{2r}\cap \{(Y,s):\,\delta(Y,s)<\xi r\}}  u(Y,s)\, \d Y\d s\lesssim \xi^2r^{-n}{\omega^{X,t}(\Delta(X_0,t_0,cr))}.
\end{equation}
\end{lemma}

\begin{proof} Let $(X,t)\in \mathbb R^{n+1}\setminus\Sigma$ and set
$ u(Y,s)=G(X,t,Y,s)$. The first inequality in \eqref{eqn:aver-B} follows immediately
from Corollary \ref{continuity2}, and
the second inequality in \eqref{eqn:aver-B} follows immediately from the case
$\xi=4$ of \eqref{eqn:aver-B2}.

% In the following we extend $ u(Y,s)$ to $\mathbb R^{n+1}$
% by defining it to equal $0$ outside of $\Omega$. Then
% $ u$ is an adjoint sub-caloric function in $C(X_0,t_0,2r)$.
% Therefore the first inequality in \eqref{eqn:aver-B} follows  immediately.

We therefore turn to the proof of \eqref{eqn:aver-B2}.
To begin, we let
$\mathcal{W}_r^\xi= \lbrace I\rbrace$ be the set of all Whitney cubes which have a nonempty intersection with the set
$$C_{2r}\cap \{(Y,s):\,\delta(Y,s)<\xi r\}.$$
Let $(X_I,t_I)\in\Sigma$ be a point on $\Sigma$ closest to $I$ as measured by $\delta(\cdot,\cdot)$.  For $(Y,s)\in I$, by Lemma \ref{CFMS1}, we have
\[
 u(Y,s) \lesssim \ell(I)^{-n}\omega^{X,t}(\Delta(X_I,t_I,c\ell(I))).
\]
Consequently, since $|I| =\ell(I)^{n+2}$, we have
\begin{multline*}
r^{-n-2}\iiint_{\Omega\cap C_{2r}\cap \{(Y,s):\,\delta(Y,s)<\xi r\}}  u(Y,s)\, \d Y\d s
\\
\le\, r^{-n-2}\sum_{I\in \mathcal{W}_r^\xi}
\iiint_{I} u(Y,s)\, \d Y\d s
\,\lesssim \, r^{-n-2}\sum_{I\in \mathcal{W}_r^\xi}\ell(I)^{2}
\omega^{X,t}\big(\Delta(X_I,t_I,c\ell(Q))\big)
\\
\approx \,r^{-n-2}\sum_{k:2^{-k}\lesssim \xi r}2^{-2\,k}\!\!\!
\sum_{I\in  \mathcal{W}_r^\xi: \ell(I)=2^{-k}}
\!\! \omega^{X,t}\big(\Delta(X_I,t_I,c\ell(I))\big)
\,\,\lesssim \,\,\xi^2r^{-n}{\omega^{X,t}\big(\Delta(X_0,t_0,cr)\big)},
\end{multline*}
where in the last step we have used that for each fixed $k$, the sets
$\{\Delta(X_I,t_I,c\ell(I))\}$ with $\ell(I)=2^{-k}$
have uniformly bounded overlaps, and are all contained in $\Delta(X_0,t_0,cr)$,
for $c$ chosen large enough.
% This proves \eqref{eqn:aver-B2} and this readily implies
% the inequality to the right in \eqref{eqn:aver-B} by taking $\xi=4$.
\end{proof}

\section{Proof of Theorem \ref{main.thrm}: the (parabolic) WHSA condition} \label{Sec5}

The (parabolic) weak half-space approximation (WHSA) condition was introduced and discussed in Subsection \ref{secWHSA}. Given $\eps>0$ and $K_0>0$,  $Q\in \mathbb{D}(\Sigma)$ satisfies the $(\eps,K_0)$-{local WHSA} condition if there is a half-space
$H_Q$ and a hyperplane  $P_Q =\partial H_Q$, containing a line parallel to the time axis, such that $(i)$--$(iii)$ of Definition \ref{def2.13} are satisfied. Given $\eps>0$ we will in the following consistently assume that $ K_0 \ll \eps^{-2}$.

The purpose of this section and the subsequent section is to prove the following theorem.

\begin{theorem}\label{WHSA.thrm}
Assume that $\Omega \subset \ree$ is an open set, that $\Sigma=\partial\Omega$ is time symmetric Ahlfors-David regular with constant $M$, that $\Omega$ satisfies the corkscrew condition  with constant $\gamma$, and
that caloric measure satisfies a local weak-$A_\infty$ condition with respect to
$\sigma=\cH_{\text{\em par}}^{n+1}|_\Sigma$ with constants $p$ and $C_1$. Then there exists $K_0>1$ and $\eps_0:=K_0^{-100}\ll 1$, depending only on $n,M,\gamma,p$ and $C_1$, such that $\Sigma$ satisfies the $(\eps,K_0)$-WHSA condition in the sense of Definition \ref{def2.14}, for all $\eps\leq\eps_0$.
\end{theorem}

Consequently, by proving Theorem \ref{WHSA.thrm} we prove that the assumptions of Theorem \ref{main.thrm} imply that $\Sigma$ satisfies the $(\eps,K_0)$-WHSA condition for all $\eps\leq\eps_0$. In this section we develop a number of important technical preliminaries for the proof of
Theorem \ref{WHSA.thrm} before the theorem is finally proved in Section \ref{Sec6}. From now on all implicit
constants will depend at most on the allowable parameters, i.e., on $n,M,\gamma,p$ and $C_1$,  unless otherwise stated.

\subsection{The construction of sawtooths based on the weak $A_\infty$ condition}

Recall that for any integer $N\geq 1$, and any cube $R\in\dd(\Sigma)$, we
set $\dd_N(R):=\{Q\in\dd: \ell(Q) =2^{-N} \ell(R)\}$.

\begin{lemma}\label{nondegcubelem.lem}
Consider $R \in \dd(\Sigma)$. There exist constants $N_0, N_1$, depending only on $n$, $M$ and $\gamma$,  and a cube $Q_0 \in \mathbb{D}_{N_0}(R)$ and a point $A^+_{Q_0}$, with the following properties.

\begin{list}{$(\theenumi)$}{\usecounter{enumi}\leftmargin=1cm \labelwidth=1cm \itemsep=0.2cm \topsep=.2cm \renewcommand{\theenumi}{\alph{enumi}}}
\item The point $A^+_{Q_0} =(Y,s)$ satisfies
\[
s - \tau \ge (500 \diam(Q_0))^2, \quad \forall (Z,\tau) \in Q_0,
\]
and
\[
N_1 \diam(Q_0)\geq \delta(Y,s) \ge 100 \diam(Q_0)\,.
\]

% where $N$ depends only on $n,M$, and $\gamma$.
\item The caloric measure with pole at $A^+_{Q_0}$ satisfies
\[\hm^{A_{Q_0}^+}(Q_0) \gtrsim 1,\]
with implicit constants depending only on $n$, $M$ and $\gamma$.
\end{list}
\end{lemma}
\begin{proof} Given $R \in \dd(\Sigma)$, let $C(X_{R},t_{R},r_{R})$, with
$r_{R} \approx\ell(R)$, denote
the corresponding cylinder introduced in Remark \ref{remarkscube}.
Let
$$
\Delta_{R} := \Sigma\cap
C(X_{R},t_{R},r_{R})
$$
 be the associated surface ball/cylinder contained in $R$. Let $(Y,s')$ be
the corkscrew point associated to $(X_{R},t_{R})$ and at scale $r_{R}/c_0$, see Remark \ref{cp},
where $c_0$ is the constant $c$ in
Lemma \ref{Bourgain}. Using Lemma \ref{Bourgain},
\begin{equation}\label{nondegcubelemeq1.eq}
\omega^{(Y,s')}(R) \ge \omega^{(Y,s')}(\Delta_{R}) \gtrsim 1.
\end{equation}
Let $N_0 \in \mathbb{N}$ be an integer to be chosen.
By the properties of the dyadic cubes we have
that $\#\mathbb{D}_{N_0}(R) \lesssim 2^{(n+1)N_0}$ with an implicit constant depending only on $n$ and $M$. Thus, using \eqref{nondegcubelemeq1.eq} there exists $Q_0 \in \mathbb{D}_{N_0}(R)$ with the property that
\begin{equation}\label{nondegcubelemeq1.eq+}\omega^{(Y,s')}(Q_0) \gtrsim 2^{-(n+1)N_0}.
\end{equation}
As $\omega^{(Y,s')}(Q_0) > 0$ it must be the case that $Q_0$ contains a point $(Z,\tau)$ with $\tau \le s'$. Thus, as $\diam(Q_0) \approx 2^{-N_0} \diam(R)$ we have that
\begin{equation}\label{nondegcubelemeq2.eq}
\begin{split}
(Z', \tau') \in Q_0 \implies \tau'& \le \tau +  (\diam(Q_0))^2\le s' + (\diam(Q_0))^2\le s' + c2^{-2N_0}r_{R}^2,
\end{split}
\end{equation}
for a constant $c$ depending only on $n$ and $M$. By definition of the corkscrew point there exists a constant $\tilde{c}< 1$, depending on $\gamma$, such that $C(Y,s', \tilde{c}r_{R}) \subset \Omega$. Thus, by the parabolic Harnack inequality and \eqref{nondegcubelemeq1.eq+}, we see that if
\[A_{Q_0}^+:=(Y,s):= (Y, s' + (\tilde{c}r_{R}/100)^2),\]
then
\begin{equation}\label{nondegcubelemeq3.eq}
 \omega^{A_{Q_0}^+}(Q_0) \gtrsim 2^{-(n+1)N_0}.
 \end{equation}
This proves $(b)$ provided our final choice of $N_0$ only depends on $n$, $\gamma$ and $M$. Furthermore, as $s = s' + (\tilde{c}r_{R}/100)^2$ we have by \eqref{nondegcubelemeq2.eq} that if  $N_0$ is sufficiently large, then
\begin{align}\label{No1}
s - \tau' &\ge  (\tilde{c}r_{R}/200)^2 \gtrsim 2^{2N_0} \diam(Q_0)^2\ge  (500^2)\diam(Q_0)^2, \quad \forall (Z',\tau') \in Q_0.
\end{align}
Moreover, since $C(Y,s', \tilde{c}r_{R}) \subset \Omega$, $s = s' + (\tilde{c}r_{R}/100)^2$, and $r_{R} \approx 2^{N_0}\diam(Q_0)$, it follows  that
\begin{align}\label{No2}\delta(Y,s) \approx \delta(Y,s') \approx \diam(R) \approx 2^{N_0}\diam(Q_0),
\end{align} provided $N_0$ is large enough. We now note that we can choose  $N_0$ large, but only depending on $n$, $M$ and $\gamma$, so that
\eqref{No1} and \eqref{No2} hold. This verifies $(a)$.
\end{proof}

Let  $\mathcal{M}$ denote the parabolic (Hardy-Littlewood) maximal function on $\Sigma$, defined with respect to $\sigma$.  Recall the notion of a fattened cube
$\lambda Q$ with $\lambda>1$, defined in \eqref{dilatecube}.

\begin{lemma}\label{l4.4}
Let $Q_0 \in \mathbb{D}(\Sigma)$ and assume that $(Y,s) \in \Omega \setminus C(X_{Q_0}, t_{Q_0}, 50 \diam(Q_0))$ is such that
\[\hm^{(Y,s)}(Q_0) \ge a_0 > 0.\]
Assume that $\hm$ is in weak-$A_\infty$ and let
\[\mu: = \frac{\sigma(Q_0)}{\hm^{(Y,s)}(Q_0)} \hm^{(Y,s)}.\]
Then there exist $a_1, a_2 > 0$, and a family, $\mathcal{F} = \{Q_j\}$ of non-overlapping dyadic subcubes of $Q_0$, with $a_1$, $a_2$, depending only on $n$, $M$, the weak-$A_\infty$ constants of $\hm$, and $a_0$,  such that
\begin{equation}\label{preeq4.8}
\sigma(Q_0 \setminus \cup_{\mathcal{F}} Q_j) \ge a_1 \sigma(Q_0),
\end{equation}
and such that
\begin{equation}\label{preeq4.9}
\frac12 \le \frac{\mu(Q)}{\sigma(Q)} \le \bariint_Q \mathcal{M}(\mu\big|_{2Q_0})(X,t) \, \d\sigma(X,t) \le a_2, \quad \forall Q \in \mathbb{D}_{\mathcal{F},Q_0},
\end{equation}
\end{lemma}

\begin{proof} Set $r_0:= \diam(Q_0)$, and for any $\kappa >0$, set
$C_{\kappa r_0} := C(X_{Q_0}, t_{Q_0}, \kappa r_0)$. Let
$\Delta:= \Sigma \cap C_{r_0}$, and more generally,
set $\kappa \Delta:= \Delta \cap C_{\kappa r_0}$.
% be the surface cylinder concentric with $Q_0$, with radius $\diam(Q_0)$,
Let $(Y,s) \in \Omega \setminus C_{50r_0}$.
Then, since $\hm$ is in weak-$A_\infty$ (see Definition \ref{defweakAinfty.def}),
it follows that
$\hm^{Y,s}\big|_{12\Delta} \ll \sigma$,
and there exists $q\in (1,\infty)$ such that
\begin{equation}\label{wrhpdefeq.eqrep}
\left(\bariint_{\Delta'} (k^{Y,s})^{q} \, \d\sigma \right)^{1/q} \lesssim \bariint_{2\Delta'} k^{Y,s} \, \d\sigma,
\end{equation}
for every $\Delta'=\Delta(X,t,\rho)$ with $(X,t)\in\Sigma$ and
$C(X,t, 2\rho)\subset C_{12r_0}$,
where $k^{Y,s}=\d\hm^{Y,s}/\d\sigma$. Using that $\mu\ll \sigma$ in $4Q_0$,
we write $k=\d\mu/\d\sigma$ and
\[
k(X,t):= \frac{\sigma(Q_0)}{\hm^{(Y,s)}(Q_0)} k^{Y,s}(X,t), \quad (X,t) \in 4Q_0.
\]
Then, using the $\mathrm{L}^p$ boundedness of the parabolic maximal function, Minkowski's inequality, and the fact that $2Q_0\subset 3\Delta$, and have comparable diameters, we obtain
\begin{multline}\label{eafvfrv}
\bariint_{Q_0} \mathcal{M}(k 1_{2Q_0}) \, \d\sigma
\le \left(\bariint_{Q_0} \mathcal{M}(k 1_{2Q_0})^q \, \d\sigma \right)^{1/q}
\lesssim
\left(\bariint_{2Q_0} k^q \, \d\sigma\right)^{1/q}
\\
\lesssim \left(\bariint_{3 \Delta} k^q \, \d\sigma\right)^{1/q}
\lesssim \bariint_{6\Delta} k \, \d\sigma
\lesssim \frac{\omega^{Y,s} (6\Delta)}{\hm^{(Y,s)}(Q_0)}\lesssim \tilde a,
\end{multline}
where $\tilde{a}>1$ depends only on $(n,M,C_1,p)$
and $a_0$.   Thus, by \eqref{eafvfrv} and
construction, \eqref{preeq4.9} holds for $Q = Q_0$, with
$c\tilde{a}$ in place of $a_2$.

We now perform a stopping time argument and extract a disjoint family $\mathcal{F} = \{Q_j\}\subset \dd(Q_0)\setminus \{Q_0\}$ which are maximal with respect to the property that either
\begin{equation}\label{stoptimebad1.eq}
\frac{\mu(Q_j)}{\sigma(Q_j)}<\frac12,
\end{equation}
or
\begin{equation}\label{stoptimebad2.eq}
\bariint_{Q_j} \mathcal{M}(\mu\big|_{2Q_0}) \, \d\sigma > \Upsilon\tilde{a}\,, % a_2.
\end{equation}
where $\Upsilon\gg 1$ is to be chosen.
Setting $a_2 := \Upsilon\tilde{a}$, we need only %Based on this, it remains to
prove that  \eqref{preeq4.8} holds provided $\Upsilon$ is chosen large enough depending only on the allowable parameters.
To this end, we say that a cube $Q_j$ is of type 1 if \eqref{stoptimebad1.eq} holds, and
of type 2 if \eqref{stoptimebad2.eq} holds (if both hold, we arbitrarily assign the cube to be of type 1). Let $\mathcal{F}_1$ be the
collection of type 1 cubes, and let $\mathcal{F}_2$ be the collection of type 2 cubes, so that
$\mathcal{F}:=\mathcal{F}_1\cup \mathcal{F}_2$.

Set
\[
A_1:= \cup_{\mathcal{F}_1} Q_j\,,\quad A_2:= \cup_{\mathcal{F}_2} Q_j\,,
\quad E:= Q_0 \setminus \cup_{\mathcal{F}} Q_j\,,
\]
and note that therefore, by our normalization for $\mu$,
\begin{equation}\label{EAsplit}
1\,=\,\frac{\mu(Q_0)}{\sigma(Q_0) }
\,=\,\frac1{\sigma(Q_0) }\Big( \mu(E) +\mu(A_1) +\mu(A_2)\Big)\,.
\end{equation}
By \eqref{stoptimebad1.eq},
\begin{align}\label{stoptimebad1controlpre.eq}
\mu(A_1) &= \sum_{Q_j \in \mathcal{F}_1} \mu(Q_j)\le \frac{1}{2} \sum_{Q_j \in \mathcal{F}_1} \sigma(Q_j)\le \frac{1}{2}\sigma(Q_0)\,. % =  \frac{1}{2}\mu(Q_0),
\end{align}
% by our normalization for $\mu$.
Note also that by \eqref{stoptimebad2.eq} and \eqref{eafvfrv}, since $a_2 = \Upsilon\tilde{a}$,
\begin{multline}\label{stoptimebad2control.eq}
\sigma(\cup_{\mathcal{F}_2} Q_j) = \sum_{Q_j \in \mathcal{F}_2} \sigma(Q_j)
 \le
 \sum_{Q_j \in \mathcal{F}_2} \frac{1}{a_2}  \iint_{Q_j} \mathcal{M}(\mu\big|_{2Q_0}) \, \d\sigma
\\
\le \frac{1}{a_2} \iint_{Q_0} \mathcal{M}(\mu\big|_{2Q_0}) \, \d\sigma \le \frac{c\sigma(Q_0)}{\Upsilon}\,.
\end{multline}
Using that $\hm^{(Y,s)}(Q_0) \ge a_0$ and that $\hm^{(Y,s)}(\Sigma)  \le 1$, we deduce
\[\mu(\Sigma) = \frac{\sigma(Q_0)}{\hm^{(Y,s)}(Q_0)} \hm^{(Y,s)}(\Sigma) \le a_0^{-1}\sigma(Q_0) = a_0^{-1} \mu(Q_0).\]
In particular,
\begin{align}\label{aa1}\mu(2Q_0) \le a_0^{-1} \mu(Q_0).
\end{align}
% Set $E := Q_0 \setminus \cup_{\mathcal{F}_1} Q_j$.
Consequently, a standard argument, using Hölder's inequality
and the reverse Hölder condition for $k$, shows that for any Borel subset $A$ of $Q_0$,
\begin{align}
\label{aa2}
\mu(A) \le c \left(\frac{\sigma(A)}{\sigma(Q_0)}\right)^{1/q'} \mu(2Q_0)
\le \frac{c}{a_0} \left(\frac{\sigma(A)}{\sigma(Q_0)}\right)^{1/q'} \mu(Q_0),
\qquad
A\subset Q_0,
\end{align}
where $c\geq 1$ depends on $n$, ADR, and the reverse H\"older constant.
In particular,
setting $A= A_2$, we can use \eqref{stoptimebad2control.eq} and the normalization for $\mu$
% \eqref{stoptimebad1controlpre.eq},   \eqref{aa2}, and \eqref{aa1}
to deduce that
\begin{align} \label{A2bound}
\mu(A_2)
\leq c \Upsilon^{\,-1/q'} \mu(Q_0) = c \Upsilon^{\,-1/q'} \sigma(Q_0) \leq \frac14 \sigma(Q_0)\,,
\end{align}
for $\Upsilon$ chosen large enough.

Furthermore,
setting $A = E=Q_0 \setminus \cup_{\mathcal{F}} Q_j$
in \eqref{aa2}, and using the normalization, we have
\begin{align}\label{muEest}
\mu(E)
\le ca_0^{-1} \left(\frac{\sigma(E)}{\sigma(Q_0)}\right)^{1/q'}  \sigma(Q_0).
\end{align}
Thus, combining \eqref{EAsplit}, \eqref{stoptimebad1controlpre.eq},
\eqref{A2bound}, and \eqref{muEest}, and hiding the two small terms $\frac12+\frac14$,
we see that
\[
\frac14 \,\leq\, ca_0^{-1} \left(\frac{\sigma(E)}{\sigma(Q_0)}\right)^{1/q'} \,.
\]
Setting $a_1 := (a_0/4c)^{q'}$, we therefore conclude that
\begin{equation}\label{stoptimebad1control.eq}
\sigma(E)\ge a_1 \sigma(Q_0),
\end{equation}
as desired.
\end{proof}

\subsection{A preliminary Corona decomposition}

\begin{lemma}\label{initialcorona1.lem}
There exist constants $a_0, a_2, N,  N_1>0$,  depending only on  $n$, $M$, $\gamma$, and  the weak-$A_\infty$ constants of $\hm$;
and a disjoint decomposition of $\dd = \dd(\Sigma)=\cG\cup\cB$, such that the following is true:

\begin{list}{$(\theenumi)$}{\usecounter{enumi}\leftmargin=1cm \labelwidth=1cm \itemsep=0.2cm \topsep=.2cm \renewcommand{\theenumi}{\alph{enumi}}}

\item The collection $\cG$ can be subdivided into a collection
of disjoint stopping time trees, $\mathcal{S}=\{\sbf\}$,
$$
\cG = \bigcup\limits_{\sbf \in \mathcal{S}} \sbf,
$$ such that each such tree $\sbf$ is semi-coherent.
\item The maximal cubes $\{Q(\sbf)\}$ and the cubes in $\cB$ satisfy a Carleson packing condition
\[
\sum_{Q(\sbf): Q(\sbf) \subseteq Q} \sigma(Q(\sbf))
\,+\,\sum_{Q' \in \mathcal{B}: Q' \subseteq Q } \sigma(Q')
\,\leq \,N \sigma(Q), \quad \forall Q \in \dd.
\]

\item For each stopping time tree $\sbf$ there exists a
time-forward corkscrew point $(Y_\sbf,s_\sbf)$ relative to the
maximal cube $Q(\sbf)$, such that the following conditions hold:
\begin{list}{$(\theenumi.\theenumii)$}{\usecounter{enumii}\leftmargin=1cm \labelwidth=1cm \itemsep=0.2cm \topsep=.2cm \renewcommand{\theenumii}{\arabic{enumii}}}

\item $s_\sbf - t \ge (500 \diam(Q(\sbf)))^2$, $\forall (Z,t) \in Q(\sbf)$, and
\[
N_1\diam(Q(\sbf))\ge\delta(Y_\sbf,s_\sbf) \ge 100 \diam(Q(\sbf))\,.
\]

\item $\hm^{(Y_\sbf,s_\sbf)}(Q(\sbf)) \ge a_0$.

\item The measure
\[\mu:=  \frac{\sigma(Q(\sbf))}{\hm^{(Y_\sbf,s_\sbf)}(Q(\sbf))} \hm^{(Y_\sbf,s_\sbf)} \]
satisfies
\begin{equation}\label{eq4.8}
1/2 \le \frac{\mu(Q)}{\sigma(Q)} \le \bariint_Q \mathcal{M}(\mu\big|_{2Q(\sbf)})(X,t) \, \d\sigma(X,t) \le a_2,  \quad \forall Q \in \sbf.
\end{equation}

\end{list}

% \smallskip

% \item For $\dd_\xi(Q)$ defined as above in \eqref{Depsdef}, it holds that
% $Q\in\sbf$ whenever $Q\in\dd_{\xi}(Q(\sbf))$.

\end{list}
\end{lemma}

\begin{proof} The lemma is a
variant of the folkloric result that
``Big Pieces implies Corona", and the proof will proceed accordingly.
% We first construct a preliminary decomposition satisfying properties $(a)-(c)$
% (with no dependence on $\xi$), and then given $\xi\in (0,1)$,
% we refine this decomposition to obtain also property $(d)$,

Given $R\in\dd(\Sigma)$, let $Q_0\in \dd_{N_0}(R)$ be as in
Lemma \ref{nondegcubelem.lem}, and we denote this cube
using the notation $Q_R:=Q_0$.
Applying Lemma \ref{l4.4} to $Q_R:=Q_0$, we introduce the following notation.
\begin{itemize}
\item Let $\F'(R)$ denote the family $\F=\{Q_j\}_j\subset\dd_{Q_R}$
constructed in Lemma \ref{l4.4}, let
$\F''(R)$ denote the collection of subcubes of $R$ that are maximal with respect to the property
that they do not meet $Q_R$, i.e.,
\[
\F''(R):= \text{maximal elements of} \left\{ Q\subset R: Q_R\cap Q=\emptyset\right\},
\]
and set
\[
\F(R):= \F'(R) \sqcup \F''(R)\,,
\]
where here and elsewhere we use $\sqcup$ to denote a union of sets that are pairwise disjoint.
\smallskip
\item Let $\B(R)$ denote the collection of subcubes of $R$ (including $R$ itself), that
{\em properly} contain $Q_R$, i.e.,
\[
\B(R):= \left\{Q\subset R: Q_R\subsetneq Q\right\}.
\]
\item Let
\[
\sbf(R):=\left\{Q\subset Q_R: Q \text{ is not contained in any } Q_j\in\F \right\}
\]
denote the stopping time tree, with maximal element
$Q(\sbf(R))=Q_R$, relative to the family $\F=\{Q_j\}$ constructed in Lemma \ref{l4.4}.
\end{itemize}

We now iterate this construction as follows.  Given $R_0\in\dd(\Sigma)$, we set
$\F_0:= \{R_0\}$, and define recursively
\[
\F_k:= \bigsqcup_{R\in \F_{k-1}} \F(R)\,,\qquad k\ge 1\,,
\]
and set
\[
\B_k:= \bigsqcup_{R\in \F_{k}} \B(R)\,,\quad \G_k:= \bigsqcup_{R\in \F_{k}} \sbf(R)\,,
\quad k\ge 0\,,
\]
We further define
\[
\F_*:= \bigsqcup_{k=0}^\infty\F_k\,\qquad
\B:=\bigsqcup_{k=0}^\infty\B_k\,,\qquad \G:=\bigsqcup_{k=0}^\infty\G_k\,,
\]
Property $(a)$, and the fact that
$\dd_{R_0} = \B\sqcup\G$, now follow immediately from the construction.
Property $(c)$ follows immediately from the construction
and Lemma \ref{nondegcubelem.lem}.
Furthermore, once the local version of the lemma has been
established on any given $R_0$, one may produce a global version with $\dd(\Sigma) =\B\sqcup\G$, by the argument in \cite[p.~38]{DS1}.

It therefore remains only to establish the packing condition $(b)$.  To this end, let us make the following pair of observations, each of which follows immediately from the construction.
\begin{enumerate}
\item For each $k$, the cubes in $\F_k$ are pairwise disjoint.
\smallskip
\item There exists a uniform $\eta = \eta(a_1,N_0,n,M)\in (0,1)$ (where $a_1$ is the constant
in \eqref{preeq4.8}, which in turn also depends only on allowable parameters), such that
if $R\in \F_k$, then
\[
\sum_{R'\in \F_{k+1}: \,R'\subset R}\sigma(R') = \sum_{R'\in\F(R)} \sigma(R') \leq (1-\eta)
\, \sigma(R)\,.
\]
\end{enumerate}
Iterating (2), we then have
for $R\in\F_k$, and $m\geq 1$, that
\begin{equation*}% \label{Fkpack}
\sum_{R'\in \F_{k+m}: \,R'\subset R}\sigma(R') \leq (1-\eta)^m
\, \sigma(R)\,.
\end{equation*}
Consequently, for any cube $Q\subset R_0$, letting $k_Q$ denote the smallest $k$ such that
$Q$ contains or is equal to a cube in $\F_k$, we then have
\begin{equation}\label{Fkpack}
 \sum_{k=k_Q}^\infty \sum_{R\in\F_k:\,R\subset Q} \sigma(R)\, \leq\,
 \sum_{k=k_Q}^\infty  (1-\eta)^{k-k_Q}\sigma(Q) \,=\,C_\eta \,\sigma(Q)\,.
\end{equation}
Note also that, for any $R\in\F_*$, and for any $Q\in \B(R)$
(thus, $Q_R\subsetneq Q\subset R$), by construction
\begin{equation}\label{Bpack}
\sum_{Q'\in\B(R): \, Q'\subset Q}\sigma(Q') \lesssim \sigma(Q)\,,
\end{equation}
where the implicit constant is uniform and depends only on $n$ and $M$.
In particular, when $Q=R$, we simply have $\sum_{Q'\in\B(R)}\sigma(Q') \lesssim \sigma(R)$.

Recall that $Q_R=Q(\sbf(R))$ is the maximal cube for the tree
$\sbf(R)$.  For $Q\in\dd_{R_0}$, we then have
\begin{multline*}
\sum_{Q(\sbf): Q(\sbf) \subseteq Q} \sigma(Q(\sbf))
\,+\,\sum_{Q' \in \mathcal{B}: Q' \subseteq Q } \sigma(Q')
\\[4pt]
\leq\, c\sigma(Q) +
\sum_{k=k_Q}^\infty \sum_{R\in\F_k:\,R\subset Q}
\left[\sigma(Q_R)+\sum_{Q'\in\B(R)} \sigma(Q')\right]
\\[4pt]
\lesssim \,\sigma(Q)+\sum_{k=k_Q}^\infty \sum_{R\in\F_k:\,R\subset Q} \sigma(R)
\,\lesssim\,\sigma(Q)\,,
\end{multline*}
where the term $c\sigma(Q)$ appears in the first inequality to account for possible
cubes contained in $Q$, but which also
belong either to $\sbf(R)$ or to $\B(R)$ for some $R \in \F_{k_Q-1}$ (thus, for some $R$
not contained in $Q$), and here we have used \eqref{Bpack} to handle the bad cubes;
in the second inequality we have used the trivial fact that $Q_R\subset R$, along with
\eqref{Bpack} in the special case $Q=R$; finally, the last step is just \eqref{Fkpack}.
Thus, we obtain the packing condition $(b)$.
\end{proof}

\subsection{Analysis in a single stopping time tree $\sbf$ from Lemma \ref{initialcorona1.lem}}
\label{initialstopanlys.sect}

In this section we work with a single stopping time tree $\sbf$ from Lemma \ref{initialcorona1.lem}. For notational convenience we set $Q_0 := Q(\sbf)$ and we recall that $\sbf = \mathbb{D}_{\F,Q_0}$ for a collection of pairwise disjoint subcubes of $Q_0$, denoted $\mathcal{F}$.  Based on Lemma \ref{initialcorona1.lem} we consider the point $(Y_\sbf,s_\sbf)$ and we set
\[
\mu:=  \frac{\sigma(Q_0)}{\hm^{(Y_\sbf,s_\sbf) }(Q_0)} \hm^{(Y_\sbf,s_\sbf) }.
\]
For future reference, we let $\widehat u$ denote the Green function with
the same normalization:
\begin{equation}\label{normgreeb}\widehat u(Y,s):=G(Y_\sbf,s_\sbf,Y,s)\sigma(Q_0)/\omega^{(Y_\sbf,s_\sbf) }(Q_0).
\end{equation}

\begin{lemma}\label{doublingintree.lem}
Consider $Q \in \sbf=\mathbb{D}_{\mathcal{F}, Q_0}$,
let $(X,t)\in Q$, and suppose that there is a positive constant $a\leq 1$ such that
$a\ell(Q) \leq r $. If $\Delta(X,t,r) \subset 2Q_0$, then
\begin{equation}\label{muaveragebound}
\frac{\mu(\Delta(X,t,r))}{\sigma(\Delta(X,t,r))} \lesssim_a a_2\,,
\end{equation}
where $a_2$ is the constant in \eqref{eq4.8}.  More precisely,
\begin{equation}\label{dbtreeeq1.eq}
\mu(\Delta(X,t,r)) \lesssim_{a,a_2} \frac{\sigma(\Delta(X,t,r))}{\sigma(Q)} \mu(Q).
\end{equation}
 If in addition, $r \le b \ell(Q)$ for some $b\geq a$ then
\begin{equation}\label{dbtreeeq2.eq}
\mu(\Delta(X,t,r)) \lesssim_{a,b,a_2} \mu(Q).
\end{equation}
\end{lemma}

\begin{proof}
Let $(X,t)\in Q\in\sbf$, with $\Delta(X,t,r) \subset 2Q_0$,
 and suppose first that $a\ell(Q) \leq r \leq 2\ell(Q)$.
 Then for some $c$ large enough depending on $n$ and the ADR constants, we have
%  \begin{equation*}
% \frac{\mu(\Delta(X,t,r))}{\sigma(\Delta(X,t,r))} \lesssim_a
% \inf_{(Z,\tau)\in Q}\frac{\mu(\Delta(Z,\tau,c\diam(Q))\cap 2Q_0)}{\sigma(\Delta(Z,\tau,c\diam(Q)))}
% \le \bariint_Q \mathcal{M}(\mu\big|_{2Q_0})(X,t) \, \d\sigma(X,t) \le a_2\,,
% \end{equation*}
 \begin{multline}  \label{muaveragebound2}
\frac{\mu(\Delta(X,t,r))}{\sigma(\Delta(X,t,r))} \lesssim_a
\inf_{(Z,\tau)\in Q}\frac{\mu(\Delta(Z,\tau,c\diam(Q))\cap 2Q_0)}{\sigma(\Delta(Z,\tau,c\diam(Q)))}
\\[4pt]
\lesssim_a \bariint_Q \mathcal{M}(\mu\big|_{2Q_0})(X,t) \, \d\sigma(X,t) \lesssim_a a_2\,,
\end{multline}
by property $(c.3)$ of  Lemma \ref{initialcorona1.lem}, which proves
\eqref{muaveragebound} in the present scenario.  On the other hand, if $r>2\ell(Q)$, then choose $Q'\in\sbf$ with $Q\subset Q'$, and $\ell(Q')<r\le 2\ell(Q')$. In this case,
\eqref{muaveragebound2} (thus also
\eqref{muaveragebound}) continues to hold, by the same argument, with $Q$ replaced by $Q'$
(and with $a=1$).
If there is no such $Q'\in\sbf$, then $r\approx \diam(Q_0)$, with $\Delta(X,t,r) \subset 2Q_0$,
and we may repeat the previous argument with $Q$ replaced by $Q_0$, to see that
\eqref{muaveragebound2}, and hence also \eqref{muaveragebound}, hold yet again.  Thus,
\eqref{muaveragebound} holds in all cases with $(X,t)\in Q\in\sbf$, $a\ell(Q)\leq r$,
and  $\Delta(X,t,r) \subset 2Q_0$.

Combining \eqref{muaveragebound} with the first inequality in \eqref{eq4.8},
we obtain \eqref{dbtreeeq1.eq}.  Finally, \eqref{dbtreeeq2.eq} follows immediately from \eqref{dbtreeeq1.eq} and the ADR property.
\end{proof}

In the following we let  $K_0 \gg 1$ be a large constant depending at most on the allowable parameters, i.e., on $n,M,\gamma$ and the weak-$A_\infty$ constants of $\hm$. We consistently assume that
\begin{equation}\label{consta} \eps \le \eps_0 =  K_0^{-100}.
\end{equation}

\subsubsection{Non-degeneracy of $\nabla_X \widehat u$}
Recall the normalized Green function $\widehat u$ defined in \eqref{normgreeb}.

 Given $Q_0$, let $\F=\{Q_j\}$ be the family of dyadic subcubes of
 $Q_0$, constructed in Lemma \ref{l4.4}, and set
 $\sbf=\mathbb{D}_{\F,Q_0}$, so that $Q_0=Q(\sbf)$.
 We let
\begin{equation}\label{subsetcubes}
\mathbb{D}_{\F,Q_0}^*:=\{Q\in \sbf: \ell(Q)\le \eps^{10} \,\ell(Q_0)\}.
\end{equation}
Recall that
\begin{equation*}
U_Q:= \bigcup_{I\in \W_Q}I^*, % \qquad U^{\rm big}_Q := \bigcup_{I\in \W^{\rm big}_Q}I^*\,.
\end{equation*}
where
\begin{equation*}
\W_Q:= \left\{I\in \W:\,K_0^{-1} \ell(Q)\leq \ell(I)
\leq K_0\,\ell(Q),\, {\rm and}\, \dist(I,Q)\leq K_0\, \ell(Q)\right\},
\end{equation*}
and
\begin{equation*}I^* =I^*(\tau) := (1+\tau)I,
\end{equation*}
for a given small, positive parameter $\tau$. Thus, if $Q \in \mathbb{D}_{\F,Q_0}^*$ then the pole of the (normalized) Green function $\widehat u$, $(Y_\sbf,s_\sbf)$, is well separated (at a scale
on the order of $\diam(Q_0$) from $U_Q$, provided $K_0$ is chosen sufficiently large, and consequently, that $\eps$ is sufficiently small. We will  frequently use that $\widehat u$ is a solution to the adjoint heat equation in and near $U_Q$.

\begin{lemma}\label{refp} If  $\eps\leq\eps_0$,  and $K_0$ is chosen sufficiently large depending only on the allowable parameters, then the following holds for all $Q\in \mathbb{D}_{\F,Q_0}^*$. There exists $(Y^0_Q,s^0_Q)\in U_Q$ such that
$$
\ell(Q)\lesssim \delta(Y^0_Q,s^0_Q)\leq
\dist(Y^0_Q,s^0_Q, X_Q,t_Q)\leq 20 \diam(Q),
$$
and such that
\begin{equation*}%\label{eq6.5}
1\lesssim |\nabla_X \widehat u(Y^0_Q,s^0_Q)|.
\end{equation*}
The implicit constants depend only on the allowable parameters,
but are in particular independent of $K_0$ and $\eps$.
\end{lemma}

\begin{proof} Fix $Q\in \mathbb{D}_{\F,Q_0}^*$, that is,
$Q\in\mathbb{D}_{\F, Q_0}$ with
$\ell(Q)\le \eps^{10}\,\ell(Q_0)$.   % $\ell(Q)\le K_0^{-1}\,\ell(Q_0)$.
Let $$\widehat{C}_Q:=C(X_Q,t_Q,10 \diam(Q)),$$  and recall that $(X_Q,t_Q)$ is defined as the center of $Q$. Let $\phi_Q\in C_c^\infty(\widehat{C}_Q)$, $0\leq\phi_Q\leq 1$, be such that $\phi_Q\equiv 1$ on $Q$ and  $$\|\nabla_X \phi_Q\|_\infty\lesssim \ell(Q)^{-1},\ \|\partial_t\phi_Q\|_\infty\lesssim \ell(Q)^{-2}.$$
In the following argument we will require that $\eps$ % $K_0$
is chosen small
enough to ensure that if $c' \ge 1$ is the % larger of the two constants
constant $c$ from  % Lemma \ref{gensalsa-aEN} and
Lemma \ref{lemma:G-aver}, then
\begin{equation}\label{bigCcontain}
C(X_Q,t_Q, 10^{10} (c')^2\diam(Q)) \cap \Sigma \subseteq 2Q_0\,,
\end{equation}
and the pole of $\widehat u$, $(Y_\sbf,s_\sbf)$, satisfies
\[ (Y_\sbf,s_\sbf) \not \in C\big(X_Q,t_Q, 10^{10}(c')^2 \diam(Q)\big),\]
as well as
\[s_\sbf - t_Q \ge 40^{1000}(c')^2 \diam(Q)^2.\]
To achieve this we can simply choose
$\eps_0^{-1} \gtrsim 40^{1000}c'$, % $K_0 \gtrsim 40^{1000}c'$,
with implicit constant depending only on $n$ and $M$, as the first property stated then
holds by properties of dyadic cubes, while the second and third properties are
possible by property $(c.1)$ of Lemma \ref{initialcorona1.lem} (and that $\diam(Q) \approx \ell(Q)$).
In the following we let $\zeta \in (0,10^{-10})$ be a small parameter to be chosen,  depending only on $n$, $M$, and the constant $a_2$ from Lemma \ref{initialcorona1.lem}, i.e., $\zeta$ will in the end depend only
on the allowable parameters.

After these preliminaries we see, using Lemma \ref{Representa}, that
\begin{align}\label{eqn:estw-}
\ell(Q)\,\mu(Q)
\le
\ell(Q)\,\iint \phi_Q\, d\mu \leq&\ell(Q)\iiint |\nabla_X \widehat u(Y,s)||\nabla_X \phi_Q(Y,s)|\, \d Y\d s\notag\\
&+ \ell(Q)\iiint\widehat u(Y,s)|\partial_s\phi_Q| \, \d Y\d s.
\end{align}
We let
\[\Sigma_\zeta := \Sigma_\zeta(Q) = \{(Y,s): \delta(Y,s) < \zeta \diam(Q) \},\]
and define $\Sigma_{\zeta/2}$ and $\Sigma_{\zeta/4}$ analogously. Introducing
\begin{align*}
I_1&:=\ell(Q)\iiint_{(\Omega\cap \widehat{C}_Q)\setminus \Sigma_\zeta} |\nabla_X \widehat u(Y,s)||\nabla_X \phi_Q(Y,s)|\, \d Y\d s,\notag\\
I_2&:=\ell(Q)\iiint_{\Omega\cap \Sigma_\zeta \cap \widehat C_Q} |\nabla_X \widehat u(Y,s)||\nabla_X \phi_Q(Y,s)|\, \d Y\d s,\notag\\
I_3&:=\ell(Q)\iiint_{(\Omega\cap \widehat{C}_Q)\setminus \Sigma_\zeta} \widehat u(Y,s)|\partial_s\phi_Q(Y,s)|\, \d Y\d s, \notag\\
I_4&:=\ell(Q)\iiint_{\Omega\cap \Sigma_\zeta \cap \widehat C_Q} \widehat u(Y,s)|\partial_s\phi_Q(Y,s)|\, \d Y\d s,
\end{align*}
we deduce from \eqref{eqn:estw-} that $\ell(Q)\,\mu(Q)\lesssim I_1+I_2+I_3+I_4$.
% \begin{align*}
% \ell(Q)\,\mu(Q)\lesssim I_1+I_2+I_3+I_4.
% \end{align*}
Clearly,
\begin{align}\label{I1est}
I_1&\lesssim\ell(Q)^{n+2}\sup_{(Y,s)\in \widehat{C}_Q\setminus \Sigma_\zeta}|\nabla_X \widehat u(Y,s)|.
\end{align}

Let $(Y,s) \in \Sigma_\zeta \cap \widehat C_Q$. Then $C(Y,s, \zeta \diam(Q))$ meets $\Sigma$ at a point $(\widehat{Y},\widehat{s})$ such that $\dist(Y,s, \widehat{Y},\widehat{s})<2 \zeta \diam(Q)$. Moreover, since $(Y,s) \in \widehat C_Q$ it holds that $\dist(Y,s, X_Q,t_Q) \leq 10\diam(Q)$, and hence
\[\dist(\widehat{Y},\widehat{s}, X_Q,t_Q) < 12 \diam(Q).
\]
In particular,
\[\Sigma_\zeta \cap \widehat C_Q \subseteq \bigcup_{(\widehat{Y},\widehat{s}) \in C(X_Q,t_Q, 12\diam(Q)) }  C(\widehat{Y},\widehat{s}, \zeta \diam(Q)).\]
By a straightforward argument
\[ \Sigma_\zeta \cap \widehat C_Q \subseteq \cup_k C_k, \]
where
\[C_k = C(Y_k,s_k, 5\zeta \diam(Q))\text{ with } (Y_k,s_k) \in  C(X_Q,t_Q, 12\diam(Q))\cap\Sigma,\]
and $\{C_k\}$ is a collection of cylinders chosen to have the same size and bounded overlap.

Using the notation introduced above, Schwarz's inequality,
Lemma \ref{gensalsa-aEN}, and Lemma \ref{lemma:G-aver}, we obtain
\begin{align}\label{I2est}
I_2&\leq
\sum_k \iiint_{C_k } |\nabla_X \widehat u(Y,s)|\, \d Y\d s\lesssim  \left(\frac{1}{\zeta\ell(Q)}\right)\,  \sum_k\iiint_{4C_k} \widehat u(Y,s)\, \d Y\d s \notag \\
&\lesssim
  \left(\frac{1}{\zeta\ell(Q)}\right)  \iiint_{C(X_Q,t_Q, 100 \diam(Q)) \cap \{\delta(Y,s) < 100\zeta \diam(Q)\}} \widehat u(Y,s)\, \d Y\d s\\
&\lesssim
   \left(\frac{1}{\zeta \ell(Q)}\right)  \zeta^2 \ell(Q)^2 \,
\mu (\Delta(X_Q,t_Q, 100 c'\diam(Q)))  \notag % \\ &
\,\lesssim\,
\zeta \ell(Q)\, \mu(Q),
\end{align}
where in the last step we have used \eqref{bigCcontain} and Lemma \ref{doublingintree.lem}.
% and that $C(X_Q,t_Q, 10^{10}c' \diam(Q)) \cap \Sigma \subseteq 2Q_0$.

In a similar manner (but simpler, since the boundary
Caccioppoli inequality of Lemma \ref{gensalsa-aEN} is not needed), we deduce that
\begin{align}\label{I4est}
I_4&\lesssim
\zeta^2 \ell(Q)\, \mu(Q).
\end{align}

We now consider $I_3$. Let
$(Y,s)\in (\Omega\cap \widehat{C}_Q) \setminus \Sigma_\zeta$.
Consider the continuous curve $\LL=\LL(\theta)$, from $(Y,s)$ to $(X_Q,t_Q)$, parameterized over $[0,2]$ by
\[\LL = \LL_1 \cup \LL_2\]
where
\[\LL_1 = ((1 -\theta) Y+ \theta X_Q,s), \quad \theta \in [0,1],\]
and
\[
\LL_2 = (X_Q, (2-\theta) s + (\theta-1) t_Q), \theta \in (1,2].
\]
By construction $\LL \subset \widehat{C}_Q$, and since $h = \delta \circ \LL$ is continuous with $h(0) \ge \zeta \diam(Q)$ and $h(2)  = 0$, there exists a smallest $\theta_0 \in (0,2)$ such that $h(\theta_0) = (\zeta/2)\diam(Q)$. To estimate $\widehat u(Y,s)$ we divide the argument  into two cases.
\\ \\
{\bf Case 1}: $\theta_0 \in (1,2)$. In this case we have $\LL_1 \subset \widehat{C}_Q \setminus \Sigma_{\zeta/2}$ and we estimate
\begin{align*}
\widehat{u}(Y,s) &\le \widehat{u}(\LL(\theta_0)) + |u(Y,s) - \widehat{u}(X_Q,s)| + |\widehat{u}(X_Q,s) - \widehat{u}(\LL(\theta_0))|
\\& =  \widehat{u}(\LL(\theta_0)) + J_1 + J_2.
\end{align*}
Using that $\delta(\LL(\theta_0))  = (\zeta/2)\diam(Q)$, along with
% Lemma \ref{continuity}
Corollary \ref{continuity2} and  Lemma \ref{lemma:G-aver}, we have
\begin{align*}
\widehat{u}(\LL(\theta_0)) &\lesssim  \,\zeta^\alpha \bariiint_{\Omega\cap C(X_Q,t_Q, 100c'\diam(Q))} \widehat u(Y,s)\, \d Y\d s
\\ & \lesssim \, \zeta^\alpha  \ell(Q)^{-n} \mu(\Delta(X_Q,t_Q, 100^2 (c')^2\diam(Q))) % \\ &
\,\lesssim \,\zeta^{\alpha} \ell(Q)^{-n} \mu(Q),
\end{align*}
where we in the last step  again have used Lemma \ref{doublingintree.lem}. Now for $J_1$, we recall that $\LL_1 \subset \widehat{C}_Q \setminus \Sigma_{\zeta/2}$ is a line segment in a spatial direction only, with length at most $c\ell(Q)$. Therefore, we can use
the fundamental theorem of calculus to obtain
\[J_1 \lesssim \ell(Q)\sup_{\widehat{C}_Q \setminus \Sigma_{\zeta/2}} |\nabla_X\widehat u|.\]
Furthermore, for $J_2$ have that  $\LL_2([1,\theta_0])$ is a line segment in the $t$-direction, contained in  $\widehat{C}_Q \setminus \Sigma_{\zeta/2}$ and with length at most $c\ell(Q)^2$. Hence, by the fundamental theorem of calculus,
\begin{align*}
J_2 & \lesssim \ell(Q)^2\sup_{\widehat{C}_Q \setminus \Sigma_{\zeta/2}} |\partial_t \widehat{u}|=  \ell(Q)^2\sup_{\widehat{C}_Q \setminus \Sigma_{\zeta/2}} |\Delta \widehat{u}|=   \ell(Q) \zeta^{-1}\sup_{\widehat{C}_Q \setminus \Sigma_{\zeta/4}} |\nabla_X\widehat{u}|\,,
\end{align*}
where in the last step we have used
interior estimates % to control $|\Delta \widehat{u}|$ by $|\nabla_X\widehat{u}|$
for solutions of the (adjoint) heat equation. % $\widehat u$ and its derivatives.
% are solutions to the adjoint heat equation).
Putting the estimates together we can conclude that in this case we have
\begin{equation}\label{goodcscase1.eq}
\widehat{u}(Y,s)  \lesssim    \zeta^{\alpha} \ell(Q)^{-n} \mu(Q)  +   \zeta^{-1} \ell(Q) \sup_{\widehat{C}_Q \setminus \Sigma_{\zeta/4}} |\nabla_X\widehat{u}|.
\end{equation}
{\bf Case 2}: $\theta_0 \in (0,1]$. In this case we have
\[\widehat{u}(Y,s) \le \widehat{u}(\LL(\theta_0)) + |u(Y,s) - \widehat{u}(\LL(\theta_0))| =\widehat{u}(\LL(\theta_0))  + J_1.\]
Note that in this case $\LL_1([0,\theta_0])$ is a line segment in a spatial direction only, with length at most $c\ell(Q)$ and contained in $\widehat{C}_Q \setminus \Sigma_{\zeta/2}$. We may therefore repeat the analysis above for $J_1$ and deduce that also in this case \eqref{goodcscase1.eq} holds.

By the analysis above, we have that \eqref{goodcscase1.eq} holds in either case,
for $(Y,s)\in (\Omega\cap \widehat{C}_Q) \setminus \Sigma_\zeta$.
Thus, using \eqref{goodcscase1.eq}, that $|\partial_s \phi| \lesssim \ell(Q)^{-2}$, and integrating over $(\Omega\cap \widehat{C}_Q)\setminus \Sigma_\zeta$, we obtain
\begin{align}\label{I3est}
I_3 &\lesssim \ell(Q)  \ell(Q)^{n+2} \ell(Q)^{-2}\big (\zeta^{\alpha} \ell(Q)^{-n} \mu(Q)  +    \zeta^{-1}\ell(Q) \sup_{\widehat{C}_Q \setminus \Sigma_{\zeta/4}} |\nabla_X\widehat{u}|\big)
\\ & \lesssim \ell(Q)\zeta^{\alpha}\mu(Q) + \zeta^{-1}\ell(Q)^{n+2}\sup_{\widehat{C}_Q \setminus \Sigma_{\zeta/4}} |\nabla_X\widehat{u}|. \notag
\end{align}
Collecting our estimates for $I_1,I_2,I_3$ and $I_4$ in \eqref{I1est}, \eqref{I2est}, \eqref{I3est},
\eqref{I4est},
, and using that $\zeta^{\alpha} > \zeta$, we obtain from \eqref{eqn:estw-} that
\[\ell(Q)\,\mu(Q) \lesssim \,\zeta^{\alpha}\ell(Q)\,\mu(Q) + \zeta^{-1}\ell(Q)^{n+2}\sup_{\widehat{C}_Q \setminus \Sigma_{\zeta/4}} |\nabla_X\widehat{u}|.\]
In particular, choosing $\zeta$ sufficiently small, depending only on the allowable parameters, we have
\[\ell(Q) \mu(Q) \lesssim \ell(Q)^{n+2}\sup_{\widehat{C}_Q \setminus \Sigma_{\zeta/4}} |\nabla_X\widehat{u}|.\]
As $Q \in \mathbb{D}_{\mathcal{F},Q_0}$, it holds that $\mu(Q) \ge (1/2)\sigma(Q)$, and since $\sigma(Q) \approx \ell(Q)^{n+1}$ we have
\[1\lesssim\sup_{\widehat{C}_Q \setminus \Sigma_{\zeta/4}} |\nabla_X\widehat{u}|.\]
Thus, there exists $(Y^0_Q,s^0_Q) \in \widehat{C}_Q \setminus \Sigma_{\zeta/4}$ with
\[1\lesssim|\nabla_X\widehat{u}(Y^0_Q,s^0_Q)|.\]
 Finally, insisting that $K_0$ is chosen sufficiently large,
 depending on $n$, $M$, and $\zeta$, we can
 ensure that $(Y^0_Q,s^0_Q) \in U_Q$. This proves the lemma.
\end{proof}

Recall that $\dd_Q:= \{Q'\in \mathbb{D}:\,Q'\subseteq Q\}.$
We introduce some notation that will be used in the next lemma.
Given $Q\in \dd(\Sigma)$, and a constant $\xi\in (0,1)$, set
\begin{equation}\label{Depsdef}
\dd^\xi(Q):=\left\{Q'\in\dd(\Sigma),\, Q'\subseteq Q:\,
\xi \ell(Q)\leq \ell(Q')\leq \ell(Q)  % \dist(Q',Q) \leq \ell(Q)
\right\}\,.
\end{equation}
As a consequence of the previous lemma, along with Lemma \ref{initialcorona1.lem}, we shall deduce the following.

\begin{lemma}\label{refp+}
Let $\xi\in (0,1)$.  Then for each  stopping time tree
$\sbf\in\mathcal{S}$ from
Lemma \ref{initialcorona1.lem}, there is a family
$\mathcal{S}'_\sbf = \mathcal{S}'_\sbf(\xi) % \mathcal{S}'  = =\mathcal{S}'(\xi)
=\{\sbf'\}$ of pairwise disjoint, semi-coherent sub-trees of $\sbf$,
and a set $\mathcal{B}'_\sbf % \mathcal{B}'= =\mathcal{B}_\sbf(\xi)
\subset \sbf$ such that
\[
\sbf = \left(\cup_{\sbf'\in\mathcal{S}'_\sbf}\,\sbf'\right) \cup \mathcal{B}'_\sbf\,,
\]
where the sub-trees $\sbf'$ have the property that
\begin{equation}\label{xicontain}
Q\in \sbf' \implies \dd^\xi(Q)\subset \sbf\,.
\end{equation}

Moreover, the maximal cubes $\{Q(\sbf')\}$ and the cubes in $\cB'_\sbf$ satisfy a Carleson packing condition
\begin{equation}\label{pack'}
\sum_{Q(\sbf'):\, \sbf'\in \mathcal{S}'_\sbf, \,Q(\sbf') \subseteq Q} \sigma(Q(\sbf'))
\,+\,\sum_{Q' \in \mathcal{B}'_\sbf: Q' \subseteq Q } \sigma(Q')
\,\lesssim_\xi \, \sigma(Q\cap Q(\sbf)), \quad \forall Q \in \dd\,,
\end{equation}
where the implicit constant depends only on $\xi$ and the allowable parameters.
Consequently, there is a refined Corona decomposition
$\dd(\Sigma) =\cG'\cup\cB'$, with % $\cG'=$ where
\[
\cG' = \cup_{\sbf'\in\mathcal{S}'}\,\sbf'
% \cup_{\sbf\in\mathcal{S}}\left(\cup_{\sbf'\in\mathcal{S}'_\sbf}\sbf'\right)\subset \cG
\,,\qquad
\cB'=\cB\cup\left( \cup_{\sbf\in\mathcal{S}}\cB'_\sbf\right)\,,
\]
where $\mathcal{S}'= \cup_{\sbf\in\mathcal{S}}\mathcal{S}'_\sbf\,$,
and $\dd=\cG\cup\cB$ is the Corona decomposition in
Lemma \ref{initialcorona1.lem}.  Furthermore, the refinement
satisfies the packing condition
\begin{equation}\label{pack'+}
\sum_{Q(\sbf'): Q(\sbf') \subseteq Q} \sigma(Q(\sbf'))
\,+\,\sum_{Q' \in \mathcal{B}': Q' \subseteq Q } \sigma(Q')
\,\lesssim_\xi \, \sigma(Q), \quad \forall Q \in \dd\,,
\end{equation}
where again the implicit constant
depends only on $\xi$ and the allowable parameters.
\end{lemma}

\begin{proof} We only sketch the proof, since the details are routine.
Given $\sbf\in\mathcal{S}$ as in
Lemma \ref{initialcorona1.lem}, we say that a cube $Q\in\sbf$ is ``$\xi$-good" if
  every $Q'\in\dd^\xi(Q)$ belongs to $\sbf$.
Starting with the maximal $\xi$-good cubes in $\sbf$,
we sub-divide dyadically, stopping the first time that we reach a
cube $Q$ that is not $\xi$-good.  Notice that such a cube $Q$ must contain a subcube $Q'$ of comparable size (depending on $\xi$), which either belongs to
$\cB$, or is a maximal cube for another stopping time tree in $\mathcal{S}$.  We then iterate this construction, starting with the maximal $\xi$-good sub-cubes of each of the previously selected stopping cubes, and so on.  In this way, we
organize the $\xi$-good cubes in $\sbf$ into semi-coherent trees $\{\sbf'\}$, while
all the non $\xi$-good cubes in $\sbf$ are placed into $\cB'_\sbf$.
The desired properties may be verified routinely, and we omit the details.
In particular, the packing condition \eqref{pack'} follows from the packing condition
in Lemma \ref{initialcorona1.lem},
and in turn, the latter two conditions imply \eqref{pack'+}.
\end{proof}

Recall that $r_Q\approx \ell(Q)$
is the ``radius" of the surface cylinder
$\Delta_Q=\Delta(X_Q,t_Q,r_Q)\subset Q$
defined in \eqref{cube-ball}, \eqref{cube-ball2}.

\begin{lemma}\label{refp++}
If  $\eps\leq\eps_0$, and $\xi>0$ are chosen sufficiently small, and $K_0$ is chosen sufficiently large, depending only on the allowable parameters, then given any stopping time tree $\sbf$ as in Lemma \ref{initialcorona1.lem}, it holds that for each
refined stopping time tree
$\sbf' =\sbf'(\xi)\subset \sbf$ as in Lemma \ref{refp+},
and for every
$Q\in \sbf'$ with $\ell(Q)\leq \eps^{10}\ell(Q(\sbf))$,
there exists $(Y^*_Q,s^*_Q)\in U_Q$ such that
\begin{equation}\label{Y*Qest}
\ell(Q)\lesssim \delta(Y^*_Q,s^*_Q)\leq
\dist(Y^*_Q,s^*_Q, X_Q,t_Q)\leq 10^{-100}r_Q\,,
% \leq 10^{-100}\diam(Q)\,,
\end{equation}
and such that
\begin{equation}\label{eq.nablaulowerbd}
1\,\lesssim \,|\nabla_X \widehat u(Y^*_Q,s^*_Q)|.
\end{equation}
The implicit constants depend only on the allowable parameters,
but are in particular independent of $K_0$ and $\eps$.
\end{lemma}

\begin{proof} Let $Q\in \sbf'=\sbf'(\xi)$,
where $\xi=2^{-N}$ with $N$ a large integer
to be chosen momentarily.
Fix $Q'\subset Q$, with $\ell(Q') = 2^{-N} \ell(Q)$, such that $(X_Q,t_Q)\in Q'$, and observe that $Q'\in \dd^\xi(Q)$.  Then by Lemma \ref{refp+}, $Q'\in\sbf$, so we may apply Lemma \ref{refp} to $Q'$. Set
\[
(Y^*_Q,s^*_Q):= (Y^0_{Q'},s^0_{Q'})\,,
\]
and observe that the claimed estimates follow immediately from
Lemma \ref{refp}, for $\xi$ chosen small enough (i.e., for $N$ chosen large enough),
depending only on allowable parameters.  Finally, choosing $K_0$ large enough, depending only on the implicit constant in the first inequality in
\eqref{Y*Qest} (thus, only on the allowable parameters),
we find that $(Y^*_Q,s^*_Q)\in U_Q$.
\end{proof}

\begin{corollary}\label{cor.deltahatu}
Let $Q$ be as in Lemma \ref{refp++},
and let $\tilde{U}_Q^i$, defined as in Definition \ref{def2.11a}, be
the connected component of $\tilde{U}_Q$ which
contains $({Y}^*_Q, {s}^*_Q)$.  Then
\[
1 \lesssim \inf_{(Y,s) \in \tilde{U}_Q^i} \,\delta(Y,s)^{-1}\,\widehat u(Y,s)\,,
\]
where the implicit constant may depend on $\eps, K_0$, and the allowable parameters.
\end{corollary}

\begin{proof}
The claimed estimate follows immediately from \eqref{eq.nablaulowerbd},
standard interior estimates for non-negative (adjoint) caloric functions,
Harnack's inequality, and the definition of $\tilde{U}_Q$.  We omit the routine details.
\end{proof}

We conclude this subsection by proving an upper bound for
$\nabla_X\widehat u$, independent of $\eps$.

\begin{lemma}\label{refp+++} Let $\sbf$
be as in Lemma \ref{initialcorona1.lem}, and let $Q\in \sbf$, with
$\ell(Q)\leq \eps^{10}\ell(Q(\sbf))$, and with $\eps$ small enough.
Then the normalized Green
function $\widehat u$ satisfies
\begin{equation*} % \label{upperb1a}
% \delta(Y,s)|\nabla_Y^2\widehat u(Y,s)|+
|\nabla_Y \widehat u(Y,s)| +(\delta(Y,s))^{-1}\widehat u(Y,s)\lesssim 1\,,
\quad\forall  (Y,s)\in U_Q,
\end{equation*}
where the implicit constant depends only on
the allowable parameters, including  $K_0$, but {\em not} on $\eps$.
\end{lemma}

Note that in particular, combining Lemma \ref{refp++} and
Lemma \ref{refp+++}, we obtain
\begin{equation}\label{eq6.5}
 |\nabla_X \widehat u(Y_Q^*,s_Q^*)| \approx 1,
\end{equation}
with implicit constants depending at most on
the allowable parameters, including  $K_0$, but {\em not} on $\eps$.

\begin{proof}[Proof of Lemma \ref{refp+++}]
By standard interior estimates for solutions of the (adjoint)
heat equation, it suffices to prove that
$\widehat u(Y,s)\lesssim_{K_0} \delta(Y,s)$, for
$(Y,s)\in V_Q$, where $V_Q$ is a slightly fattened version of $U_Q$.
% \Omega\cap C_{r_1}(X_Q,t_Q)$,

Set $r_1:= NK_0\ell(Q)$, where $N\ge 1$ is a sufficiently large fixed constant, depending only on allowable parameters, chosen so that
$V_Q\subset \Omega\cap C(X_Q,t_Q,r_1)$.
% contains this fattened version of $U_Q$.
Let $c\in [1,\infty)$, be the constant in Lemma \ref{lemma:G-aver}, and
note that for $\eps$ small, $s_\sbf\geq t_Q+4cr_1^2$, and
$\Delta(X_Q,t_Q,cr_1)\subset 2Q(\sbf)$.
Since $\mu$ and $\widehat u$ have the same normalization, we may use
Lemma \ref{lemma:G-aver} to deduce that for $(Y,s)\in V_Q$,
\begin{equation*} % \label{6.56}
\widehat u (Y,s) \lesssim
% \bariiint_{\Omega\cap C_{2r}}  u(Z,\tau)\, \d Z\d \tau \lesssim
r_1^{-n}\mu\big(\Delta(X_Q,t_Q,cr_1)\big) \lesssim r_1 \approx
K_0 \ell(Q) \approx_{K_0} \delta(Y,s)\,,
\end{equation*}
where in the second step, we have used ADR and Lemma \ref {doublingintree.lem}/\eqref{muaveragebound} with $a=1$.
\end{proof}

\subsection{An important square function estimate}\label{cutoffsawtoothsection.sect} For a suitably small $\eps_0 = K_0^{-100}$,
we fix $\eps$,  $0<\eps\leq\eps_0$, and we let $L$ be a large (fixed) positive constant  to be chosen, depending at most on the allowable parameters, in particular, $L$ is independent of $K_0$ and $\eps$. Given $(Y,s)\in\Omega$ we introduce closed cylinders
\begin{equation}\label{eq5.1}
C_{Y,s}:= \overline{C(Y,s,(1-\eps^{2L/\alpha})\delta(Y,s))},\qquad \widetilde{C}_{Y,s}:= \overline{C\big(Y,s,\delta(Y,s)\big)},
\end{equation}
where $0<\alpha<1$ is the exponent appearing in Lemma \ref{continuity}. Then $C_{Y,s}\subset \widetilde{C}_{Y,s}$ and $\Sigma\cap \widetilde{C}_{Y,s}\neq\emptyset$. We also introduce the intermediate closed cylinder
\begin{equation}\label{eq5.1bp}
\widehat C_{Y,s}:= \overline {C(Y,s,(1-\eps^{2L/\alpha}/10)\delta(Y,s))},
\end{equation}
and we have
\begin{equation}\label{eq5.1bg}
C_{Y,s}\subset \widehat C_{Y,s}\subset \widetilde{C}_{Y,s}.
\end{equation}

Recall the set  $\tilde U^i_{Q}$ introduced in Definition \ref{def2.11a}. In the following we also need to augment $\tilde U^i_{Q}$ as follows. For each $i$ we set
\begin{equation}\label{Ustar}
\W^{i,*}_{Q}:= \bigl\{I\in \W: I^*\,\, {\rm meets}\, C_{Z,\tau}\, {\rm for\, some\,} (Z,\tau) \in \bigcup_{(X,t)\in\tilde U^i_{Q}}C_{X,t}\bigr\}
\end{equation}
and
\[
\W^{*}_{Q}:= \bigcup_{i} \W^{i,*}_{Q}.
\]
We also set
\begin{equation}\label{Ustar2}
U^{i,*}_{Q} := \bigcup_{I\in\W^{i,*}_{Q}} I^{**},\qquad
U^{*}_{Q} :=
\bigcup_{I\in\W^{*}_{Q}} I^{**}
=
\bigcup_i U^{i,*}_{Q},
\end{equation}
where $I^{**}=(1+2\tau)I$ is a fattened Whitney cube, with $\tau$ small and previously fixed. By construction,
$$
\tilde U^i_{Q}\,
\subset\,
\bigcup_{(X,t)\in\tilde U^i_{Q}} C_{X,t}\,
\subset
\bigcup_{(Z,\tau) \in\bigcup_{(X,t)\in\tilde U^i_{Q}}  C_{X,t}} C_{Z,\tau}
\subset\, U^{i,*}_{Q},
$$
and for all $(Z,\tau) \in U_{Q}^{i,*}$, we have that $\delta(Z,\tau)\approx_\eps \ell(Q)$.
Moreover, by construction there is a  path connecting any pair of points in
$U^{i,*}_{Q}$ (with characteristics depending on $\eps$).
Furthermore, for every $I\in\W^{i,*}_{Q}$,
\begin{equation}\label{defsforlwbnd.eq}
\eps^{\upsilon}\,\ell(Q)\lesssim \ell(I) \lesssim \eps^{-3}\,\ell(Q),
\qquad
\dist(I,Q)\lesssim \eps^{-4} \,\ell(Q),  % \quad \forall \, I\in\W^{i,*}_{Q}\,,
\end{equation}
where $\upsilon\geq 3$ depends on $n$, $M$, and $L$. In particular, for
$Q_0 \in \dd(\Sigma)$ fixed, and  for any pairwise disjoint collection $\F = \{Q_j\}$ of cubes in $\dd_{Q_0}$,  the above construction yields, for every
 $Q\in \mathbb{D}_{\F,Q_0}^*:=
\{Q\in \mathbb{D}_{\F,Q_0}: \ell(Q)\le \eps^{10} \,\ell(Q_0)\}$, enlarged Whitney regions $\tilde U_Q$ and $U_Q^\ast$, such that
 \begin{equation}\label{UQcontain}
 U_{Q}\subset \tilde U_Q\subset U_Q^\ast.
\end{equation}
For us, the case of interest is that
 $\mathbb{D}_{\F,Q_0}=\sbf$,
 with $\sbf$ as in Lemma \ref{initialcorona1.lem}, so that $Q_0=Q(\sbf)$, and
 $\mathbb{D}_{\F,Q_0}^*$ is as in \eqref{subsetcubes}.
Consider now any semi-coherent subtree $\sbf_1\subset\sbf$, with maximal cube
$Q_1:=Q(\sbf_1)\in \mathbb{D}_{\F,Q_0}^*$ (thus,
$\sbf_1\subset \mathbb{D}_{\F,Q_0}^*$).
 For such $\sbf_1$, we introduce subdomains
 \begin{equation}\label{sawtooth-appen-HMM--}
\Omega^{*}_{\sbf_1} :=
\interior\bigg(\bigcup_{Q\in \sbf_1} U_Q^*\bigg)
=
\interior\bigg(\bigcup_{Q\in \sbf_1 } \bigcup_{I\in\W^{*}_{Q}} I^{**} \bigg)
\end{equation}
% \begin{equation}\label{sawtooth-appen-HMM--}
% \Omega^{*}_{\F,Q_0} :=
% \interior\bigg(\bigcup_{Q\in \mathbb{D}_{\F,Q_0}^* } U_Q^*\bigg) =
% \interior\bigg(\bigcup_{Q\in \mathbb{D}_{\F,Q_0}^* }
% \bigcup_{I\in\W^{*}_{Q}} I^{**} \bigg)
% \end{equation}
(where as usual, $\interior(E)$ denotes the topological interior of a set $E$),
and we let
\begin{equation}\label{sawtooth-appen-HMM--+}
\widetilde{\mathcal{W}}_{\sbf_1}
:=
\bigcup_{Q\in \sbf_1 } \W^{*}_{Q}
\end{equation}
We also introduce
\begin{equation}\label{moredomains}
\Omega^{**}_{\sbf_1}:=
\interior\bigg(\bigcup_{I\in \widetilde{\mathcal{W}}_{\sbf_1}} I^{***}\bigg), \quad  \Omega^{***}_{\sbf_1}:= \interior\bigg(\bigcup_{I\in \widetilde{\mathcal{W}}_{\sbf_1}} I^{****}\bigg),
\end{equation}
where $\tau>0$ is small and fixed, and $I^{***} = (1 + 4\tau)I$ and $I^{****} = (1 + 8\tau)I$.

Given a tree $\sbf_1$ as above, with maximal cube $Q_1:=Q(\sbf_1)$,
we introduce its ``truncated" version (which is itself
also a semi-coherent tree):
\[
\sbf_1^N:= \{Q\in\sbf_1: \ell(Q) \geq 2^{-N}\ell(Q_1)\}
\]
where $N$ is an arbitrary large integer, and for any such $N$, we define the corresponding subdomains $\Omega^{*}_{\sbf_1^N}$, etc., as in
\eqref{sawtooth-appen-HMM--} and \eqref{moredomains}, but of course with
$\sbf_1^N$ in place of $\sbf_1$.  By the nature of the Whitney regions,
the subdomains $\Omega^{*}_{\sbf_1^N}$, etc., are effectively truncated
so that the distance from the resulting region to
$\Sigma$ is roughly of the order $2^{-N}\ell(Q_1)$.
Similarly, we define
$\widetilde{\mathcal{W}}_{\sbf_1^N}$ as in \eqref{sawtooth-appen-HMM--+},
but again with $\sbf_1^N$ in place of $\sbf_1$.
Then clearly
\begin{align}\label{doma++}
\Omega_{\sbf_1}^* \subset
\Omega^{**}_{\sbf_1}\subset \Omega^{***}_{\sbf_1},\quad
\Omega_{\sbf^N_1}^* \subset
 \Omega^{**}_{\sbf^N_1}\subset \Omega^{***}_{\sbf^N_1},
\end{align}
and
\begin{align}\label{doma++a}
 \Omega^{*}_{\sbf^N_1}\subset \Omega^{*}_{\sbf^{N'}_1},\quad
 \Omega^{**}_{\sbf^N_1}\subset \Omega^{**}_{\sbf^{N'}_1},
 \quad \Omega^{***}_{\sbf^N_1}\subset \Omega^{***}_{\sbf^{N'}_1},
\end{align}
whenever $N\leq N'$. Consequently, by monotone convergence,
\begin{equation}
\iiint_{\Omega^{**}_{\sbf_1}} F(Y,s)\,\d Y\d s
=\lim_{N\to\infty} \iiint_{\Omega^{**}_{\sbf^N_1}} F(Y,s)\,\d Y\d s,
\label{eq:Omega_N-TCM}
\end{equation}
for any locally integrable, non-negative function $F$ on $\Omega^{**}_{\sbf_1}$.
The conclusion remains valid with $\Omega^{**}_{\sbf^N_1}$, $\Omega^{**}_{\sbf_1}$, replaced by $\Omega^{*}_{\sbf^N_1}$, $\Omega^{*}_{\sbf_1}$
or $\Omega^{***}_{\sbf^N_1}$, $\Omega^{***}_{\sbf_1}$.  Finally, we introduce
 \begin{align}\label{doma+-ha}
\Gamma_{\sbf^N_1}&:=\{I\in \widetilde{\mathcal{W}}_{\sbf^N_1}: \exists J \in \W\setminus\widetilde{\mathcal{W}}_{\sbf^N_1} \text{ such that } \partial J\cap \partial I \neq\emptyset\}.
\end{align}
%and
%\begin{align}\label{doma+-ha+}
%I_{\F,Q_0}^N:=\widetilde{\mathcal{W}}_{\F,Q_0}\setminus \Gamma_{\F,Q_0}^N.
%\end{align}

\begin{lemma}\label{lemma:approx-saw} For each $N$ (large),
there exists
$\Psi_N\in C_c^\infty(\Omega^{***}_{\sbf^N_1})$, % supported in $\Omega^{***}_{\sbf^N_1}$,
 such that
\begin{align}\label{Coral}
\quad 1_{\Omega^{**}_{\sbf^N_1}}\lesssim \Psi_N\le 1_{\Omega^{***}_{\sbf^N_1}},
\end{align}
and such that
\begin{align}\label{Corb}
\quad |\nabla_X\Psi_N(Y,s)|\delta(Y,s)+|\partial_t\Psi_N(Y,s)|\delta^2(Y,s)\lesssim 1,
\end{align}
for all $(Y,s)\in\mathbb R^{n+1}\setminus\Sigma$, with implicit constants depending only on $n$, $M$,  $\eps$, and $L$, but not on $N$. Furthermore,
\begin{equation}
|\nabla_X\Psi_N(Y,s)|+|\partial_t\Psi_N(Y,s)|\equiv 0, \quad  \forall (Y,s) \in \mathbb{R}^{n+1} \setminus \bigcup_{I \in \Gamma_{\sbf^N_1}} I^{****}. %\mbox{ whenever }(Y,s)\in\bigcup_{I\in I_{\F,Q_0}^N}I^{*}.
\label{eq:fregtgtr}
\end{equation}
\end{lemma}
\begin{proof}
We follow the argument in \cite[Lemma 4.44]{HMT}. 	Recall that given $I$, any closed dyadic parabolic cube in $\mathbb R^{n+1}$, we have define $I^*=(1+\tau)I$, $I^{**} = (1 + 2\tau)I$, $I^{***} = (1 + 4\tau)I$, $I^{****} = (1 + 8\tau)I$ . We here also introduce $\widetilde I=(1+6\tau)I$ so that
\begin{equation}
I^{***}
\subsetneq
\interior(\widetilde I)
\subsetneq \interior(I^{****}).
\label{eq:56y6y6}
\end{equation}
Given $I_0:=[-1/2,1/2]^{n}\times [-1/2,1/2]\subset\mathbb R^{n+1}$, fix
$\phi_0\in C_c^\infty(\interior(\widetilde I_0))$ such that
$$
 1_{I^{***}_0} \le \phi_0\le 1_{\widetilde I_0},
$$
and $|\nabla_X \phi_0|+|\partial_t \phi_0|\lesssim 1$. Here the implicit constant will also depend on the parameter $\tau$. Given $I\in \W=\W(\Omega)$, we set $$
\phi_I(Y,s)=\phi_0((Y-X(I))/\ell(I),(s-t(I))/\ell^2(I)),
$$
 where $(X(I),t(I))$ denotes the center of $I$. Then $\phi_I\in C^\infty(\interior(\widetilde I\,))$,
 $$
 1_{I^{***}}\le \phi_I\le 1_{\widetilde I},
 $$
 and
 $$
 \ell(I)|\nabla_X \phi_I|+\ell^2(I)|\partial_t \phi_I|\lesssim 1.
 $$
 To start the construction of $\Psi_N$ leading to \eqref{Coral}, we let,
 for every $(Y,s)\in\Omega$, % =\mathbb R^{n+1}\setminus\Sigma$,
 $$
 \phi(Y,s):=\sum_{I\in \W} \phi_I(Y,s).
 $$
 It then follows that $\phi\in C^\infty(\Omega)$ since for every compact subset of $\Omega$ the stated sum has finitely many non-vanishing terms. Also, $1\le \phi(Y,s)\leq c_{\tau}$ for every $(Y,s)\in \Omega$, since the family $\{\widetilde I\}_{I\in \W}$ has bounded overlaps by our choice of $\tau$, and the Whitney boxes cover $\Omega$. Hence, letting $\eta_I:=\phi_I/\phi$ we see that
 $\eta_I\in C_c^\infty(\mathbb R^{n+1})$,
 $$
 c_\tau^{-1}1_{I^{***}}\le \eta_I\le 1_{\widetilde I},
 $$
 and that $$\ell(I)|\nabla_X \eta_I|+\ell^2(I)|\partial_t \eta_I|\lesssim 1.$$ Using this, and recalling the definition of $\widetilde{\mathcal{W}}_{\sbf^N_1}$ above, we set
\begin{align}\label{cuttoff}
\Psi_N(Y,s):=\sum_{I\in \widetilde{\mathcal{W}}_{\sbf^N_1}} \eta_I(Y,s)\,=\,
\phi(Y,s)^{-1} \sum\limits_{I\in \widetilde{\mathcal{W}}_{\sbf^N_1}} \phi_I(Y,s)\,.
% \frac{\sum\limits_{I\in \widetilde{\mathcal{W}}_{\sbf^N_1}} \phi_I(Y,s)}{\sum\limits_{I\in \W} \phi_I(Y,s)},
\end{align}
for all $(Y,s)\in\Omega$. Note that the number of terms in the sum defining $\Psi_N$ is bounded depending on $N$, hence $\Psi_N\in C_c^\infty(\mathbb R^{n+1})$
since each $\eta_I\in C_c^\infty(\mathbb R^{n+1})$.
% imply that $\Psi_N\in C_c^\infty(\mathbb R^{n+1})$.
By construction
\begin{align}\label{PsiNsupport}
\supp \Psi_N
\subset \bigcup_{I\in \widetilde{\mathcal{W}}_{\sbf^N_1}} \widetilde I
\subset\interior\Big(
\bigcup_{I\in \widetilde{\mathcal{W}}_{\sbf^N_1}} I^{****}\big)= \Omega^{***}_{\sbf^N_1}.
\end{align}
% This, the fact that $\widetilde{\mathcal{W}}_{\sbf^N_1}\subset \W$,
% and the definition of $\Psi_N$, immediately give that
Consequently, since $\Psi_N\leq 1$ by construction, we have
$\Psi_N\le 1_{\Omega^{***}_{\sbf^N_1}}$. On the other hand if
$(Y,s)\in \Omega^{**}_{\sbf^N_1}$,
then $(Y,s)\in I^{***}$ for some
$I\in \widetilde{\mathcal{W}}_{\sbf^N_1}$, so that
$\Psi_N(Y,s)\ge \eta_I(Y,s)\ge c_\tau^{-1}$. This completes the proof of \eqref{Coral}.

To prove \eqref{Corb} we note that
$$
|\nabla_X \Psi_N(Y,s)|
\le
\sum_{I \in \widetilde{\mathcal{W}}_{\sbf^N_1}} |\nabla_X\eta_I(Y,s)|
\lesssim
\sum_{I\in \W} \ell(I)^{-1}\,1_{\widetilde I}\,(Y,s)
\lesssim
\delta(Y,s)^{-1},
$$
for every $(Y,s)\in \Omega$, where we have used that if $(Y,s)\in \widetilde I$, then $\delta(Y,s)\approx \ell(I)$,  as well as the fact that the elements of the family $\{\widetilde I\}_{I\in \W}$ have bounded overlap. The estimate for $|\partial_t \Psi_N(Y,s)|$ proceeds analogously.

Finally, we prove \eqref{eq:fregtgtr}. By \eqref{PsiNsupport}, % \eqref{Coral}
we only need to verify \eqref{eq:fregtgtr} in
$\Omega^{***}_{\sbf^N_1}\setminus \bigcup_{I\in\Gamma_{\sbf^N_1}} I^{****}$.
Write $F_N:=\bigcup_{I\in\Gamma_{\sbf^N_1}} I^{****}$ and
let  $(Y,s)\in \Omega^{***}_{\sbf^N_1}\setminus F_N$.
Then $(Y,s)\in I^{****}$ for some
$I\in \widetilde{\mathcal{W}}_{\sbf^N_1}\setminus \Gamma_{\sbf^N_1}$.
By definition of $\Gamma_{\sbf^N_1}$, it follows that $J\in \widetilde{\mathcal{W}}_{\sbf^N_1}$,
for any $J\in\W$ such that $\partial J$ meets $\partial I$, and thus in particular for
any $J\in\W$ such that $(Y,s)\in \interior(\widetilde{J})$, by  \eqref{Whitneytouch}.
Consequently, by construction, we see that $\Psi_N\equiv 1$ in a neighborhood of
$(Y,s)$, and \eqref{eq:fregtgtr} follows.
\end{proof}

\begin{lemma}\label{partialADR} Let $\Psi=\Psi_N\in C_c^\infty(\mathbb R^{n+1})$ be as in Lemma \ref{lemma:approx-saw}. Then
\begin{align*}
\iiint_{\Omega} \big(|\nabla_X\Psi_N(Y,s)|+\delta(Y,s)|\partial_t\Psi_N(Y,s)|\big)\,
\d Y\d s\lesssim \sigma(Q_1),\end{align*}
with implicit constant depending only on $n$, $M$,  $\eps$, and $L$, but not on $N$.
\end{lemma}
\begin{proof} Using Lemma \ref{lemma:approx-saw} it follows that
\begin{align*}
\iiint_{\Omega } \big(|\nabla_X\Psi_N(Y,s)|+\delta(Y,s)|\partial_t\Psi_N(Y,s)|\big)\,
\d Y\d s\lesssim \sum_{I\in \Gamma_{\sbf^N_1}}\ell(I)^{n+1}.\end{align*}
Therefore, to prove the lemma is suffices to prove that
\begin{equation}\label{eq:fregtgtr+}
\sum_{I\in \Gamma_{\sbf^N_1}}\ell(I)^{n+1}\lesssim \sigma(Q_1),
\end{equation}
with implicit constant depending only on $n$, $M$,  $\eps$, and $L$, but not on $N$.

For each $I \in \Gamma_{\sbf^N_1}$ there exists $J = J(I)\in\W \setminus \widetilde{\mathcal{W}}_{\sbf^N_1}$ such that $\partial J \cap \partial I \neq \emptyset$ . For $I \in \Gamma_{\sbf^N_1}$ fixed, we take $Q_I$ to be any cube such that $Q_I \in \sbf^N_1$ and $I \in \W_{Q_I}^*$.

Pick $(Y,s)\in\Sigma$ so that $\dist(Y,s,J) = \dist(J, \Sigma)$ and let $\widetilde{Q}_I\in\dd$ be such that $\widetilde{Q}_I\ni (Y,s)$ and
$$
\ell(\widetilde{Q}_I) = \ell(J) \approx \ell(I) \approx \ell(Q_I).
$$
Our choice of $K_0$ guarantees that $J \in \W_{\widetilde{Q}_I} \subset \W_{\widetilde{Q}_I}^*$. As a consequence (and examining the definition of $\widetilde{\mathcal{W}}_{\sbf^N_1}$), we see that
$\widetilde{Q}_I \not\in \sbf^N_1$,
as otherwise it would hold that $J \in \widetilde{\mathcal{W}}_{\sbf^N_1}$.

\noindent
{\bf Claim:} For each $I \in \Gamma_{\sbf^N_1}$ there exists a $Q^\star_I \in \dd$ such that
\begin{equation}\label{bbdovqstarpick.eq}
\ell(Q^\star_I) \approx \ell(I), \quad \dist(I, Q^\star_I) \lesssim \ell(I), \quad \text{ and } \quad \sum_{I \in \Gamma_{\sbf^N_1}} 1_{Q^\star_I} \lesssim 1,
\end{equation}
where again the implicit constants depend on $n$, $M$,  $\eps$, and $L$, but not on $N$.

Taking this claim for granted,
we observe that any such $Q_I^\star$ as in the claim necessarily further satisfies
\[
\ell(Q_I^\star) \lesssim \ell(Q_1)\,,\quad \text{and} \quad \dist(Q_I^\star,Q_1) \lesssim \ell(Q_1)\,.
\]
Consequently, by the Ahlfors-David regularity of $\Sigma$, and the bounded overlap property
of the cubes $Q_I^\star$, we have
\begin{align*}
\sum_{I\in \Gamma_{\sbf^N_1}}\ell(I)^{n+1} \approx \sum_{I\in \Gamma_{\sbf^N_1}}\sigma(Q^\star_I)
% = \iint_{C(X_{Q_0},t_{Q_0},c\ell(Q_0))}
% \sum_{I\in \Gamma_{\sbf^N_1}}{1}_{Q^\star_I}(X,t) \, \d\sigma(X,t)
\lesssim \sigma(Q_1)\,.
\end{align*}
This proves \eqref{eq:fregtgtr+}.
% In this deduction we also used that $\ell(I) \approx \ell(Q_I) \le \ell(Q_0)$ and,
% by the second property in \eqref{bbdovqstarpick.eq}, that
% \begin{align*}
% \dist(Q^\star_I, Q_0)
% &\le  \diam(I) + \dist(Q^\star_I, I) +  \dist(I, Q_I)
% \lesssim \ell(I) \le \ell(Q_0).
% \end{align*}

Hence, we have reduced matters to proving the claim, i.e., to
proving that for each $I \in \Gamma_{\sbf^N_1}$ there exists $Q^\star_I$ such that \eqref{bbdovqstarpick.eq} holds. In the following we enumerate the cubes in $\F^N$ by $\{Q^k\}$, and we then split $\Gamma_{\sbf^N_1}$ into three subsets, based on how the cube $\widetilde{Q}_I$ introduced above fails to be in
$\sbf_1^N$. Indeed, let $\F_1^N$ denote the collection
of maximal subcubes of $Q_1$ that not in $\sbf_1^N$, and set
\begin{align*}
\Gamma_{\sbf^N_1}^1 &:= \{I \in  \Gamma_{\sbf^N_1}: Q_1 \subsetneq \widetilde{Q}_I\}, \\
\Gamma_{\sbf^N_1}^2 &:= \{I \in  \Gamma_{\sbf^N_1}: \exists Q^k \in \F_1^N \text{ such that } \widetilde{Q}_I \in \dd_{Q^k}\}, \\
\Gamma_{\sbf^N_1}^3 &:= \{I \in  \Gamma_{\sbf^N_1}: Q_1 \cap \widetilde{Q}_I = \emptyset\}.
\end{align*}
To prove the claim we will split the argument into cases depending on whether $I \in \Gamma^1_{\sbf^N_1}$, $I \in \Gamma^2_{\sbf^N_1}$ or $I \in \Gamma^3_{\sbf^N_1}$. Before proceeding,  we note that  $\partial J \cap \partial I \neq \emptyset$ implies that
\begin{equation}\label{tQiclosetoi.eq}
\dist(\widetilde{Q}_I, I) % \le \dist(\widetilde{Q}_I, J) + \diam(J) + \diam(I)
\lesssim \ell(J) \approx \ell(I) \lesssim \ell(Q_1) .
\end{equation}
Consequently, $\#\Gamma_{\sbf^N_1}^1$ is uniformly bounded.
For $I \in \Gamma_{\sbf^N_1}^1$ we then
simply take $Q^\star_I = \widetilde{Q}_I$, so that
the bounded overlap property in \eqref{bbdovqstarpick.eq} is trivial for the
collection $\Gamma_{\sbf^N_1}^1$. The other two properties in
\eqref{bbdovqstarpick.eq} follow readily from the construction and \eqref{tQiclosetoi.eq}.

If $I\in \Gamma_{\sbf^N_1}^2$ or $I\in \Gamma_{\sbf^N_1}^3$, the idea is to pick $Q^\star_I$ as a (small) subcube of $\widetilde{Q}_I$ in such a way that the cubes possess a Whitney type property (in $Q^k$ or $\Sigma\setminus Q_1$ resp.),
and such that the other two properties in \eqref{bbdovqstarpick.eq} hold. From this we will establish the last property in \eqref{bbdovqstarpick.eq}. With this in mind, let $I \in \Gamma_{\sbf^N_1}^i$ with $i = 2$ or $3$, and let $N' \in \mathbb{N}$ be large and to be chosen later (independent of $N$). We take $Q^\star_I$ to be the cube in $\dd_{N'}(\widetilde{Q}_I)$ containing
the center $(X_{\widetilde{Q}_I}, t_{\widetilde{Q}_I})$ of $\widetilde{Q}_I$.

Before splitting into the cases $i = 2$ and $i=3$ we first establish
a preliminary estimate. By \eqref{cube-ball}
(see also property $(v)$ in Lemma \ref{cubes}),  and the ADR property, we have
\[\dist(X_{\widetilde{Q}_I}, t_{\widetilde{Q}_I}, \Sigma\setminus \widetilde{Q}_I) \approx  \ell(\widetilde{Q}_I).\]
Consequently, since $\diam(Q^\star_I) \approx 2^{-N'}\ell(\widetilde{Q}_I)$,
we may choose $N'$ large enough so that
\begin{equation}\label{stwhitqI.eq}
\dist(Y,s, \Sigma\setminus \widetilde{Q}_I) \approx  \diam(Q^\star_I), \quad  \forall (Y,s) \in Q^\star_I.
\end{equation}
With $N'$ now fixed, we observe that the first
two properties in \eqref{bbdovqstarpick.eq}
follow routinely from the construction of $Q^\star_I$.
% in particular, one can use
% \eqref{tQiclosetoi.eq} and the fact that $\diam(\widetilde{Q}_I) \approx \ell(I)$,
% to establish the second property.
Hence it remains to only verify the bounded overlap property.

We now divide the argument into the cases depending on whether $I \in \Gamma_{\sbf^N_1}^i$ for $i = 2$ or $i = 3$. First, suppose that $I \in \Gamma_{\sbf^N_1}^2$, so
that $\widetilde{Q}_I \subset Q^k$ for some $Q^k \in \F^N$. We claim that
\begin{equation}\label{stwhitqj.eq}
\ell(Q^\star_I) \lesssim \dist(Y,s, \Sigma\setminus Q^k  ) \lesssim \ell(Q^\star_I), \quad  \forall (Y,s) \in Q^\star_I.
\end{equation}
Since $\widetilde{Q}_I \subseteq Q^k$, we see that
$$
\dist(Y,s, \Sigma\setminus\widetilde{Q}_I) \le \dist(Y,s, \Sigma\setminus Q^k  ),
$$
and the lower bound in \eqref{stwhitqj.eq} holds by \eqref{stwhitqI.eq}.
To prove the upper bound in \eqref{stwhitqj.eq}, we consider two
subcases, either $Q^k \subseteq Q_I$ or $Q^k \cap Q_I= \emptyset$
(recall $Q_I \in \sbf_1^N$ and $Q^k \in \F_1^N$).
Suppose $Q^k \subseteq Q_I$. By
\eqref{cube-ball} and the ADR property there
must exist $(Z_\ast,\tau_\ast) \in \Sigma\setminus Q^k$ such
that $\dist(Z_\ast,\tau_\ast, Q^k ) \approx \ell(Q^k)$.  Since
$\widetilde{Q}_I \subseteq Q^k\subset Q_I$ in the present scenario,
we have by construction that
$
\ell(Q^\star_I)\approx \ell(\widetilde{Q}_I) \approx \ell(Q^k)  \approx \ell(Q_I),
$
and since also $Q^\star_I\subset \widetilde{Q}_I \subseteq Q^k$, we therefore have
\[
\dist(Z_\ast,\tau_\ast, Q^\star_I ) \leq \diam(Q^k) + \dist(Z_\ast,\tau_\ast, Q^k )
\approx \ell(Q^k) \approx \ell(Q^\star_I )\,.
\]
The upper bound in \eqref{stwhitqj.eq} then follows in the subcase $Q^k \subseteq Q_I$.

If $Q^k \cap Q_I= \emptyset$, then $(X_{Q_I}, t_{Q_I})\in \Sigma\setminus Q^k$.
On the other hand, $Q^\star_I\subset\widetilde{Q}_I \subset Q^k$ for
$I \in \Gamma_{\sbf^N_1}^2$, and furthermore
$\ell(Q^\star_I)\approx \ell(\widetilde{Q}_I) \approx \ell(Q_I) \gtrsim
\dist(Q_I,\widetilde{Q}_I)$, so
\begin{equation}\label{obvtqiqiclosebd.eq}
\dist(Q_I^\star, \Sigma\setminus Q^k  )\le \dist(Q_I^\star, X_{Q_I}, t_{Q_I} )
 \lesssim \ell(Q_I^\star)\,,
\end{equation}
which yields the
upper bound in \eqref{stwhitqj.eq} in the subcase $Q^k \cap Q_I= \emptyset$.
Hence, \eqref{stwhitqj.eq} holds in general for
$I \in \Gamma_{\sbf^N_1}^2$.

Now we use \eqref{stwhitqj.eq} to verify the bounded overlap property of $\{Q^\star_I\}_{I \in \Gamma_{\sbf^N_1}^2}$. Suppose that $Q^\star_I \cap Q^\star_{I'} \neq \emptyset$ for $I, I' \in \Gamma_{\sbf^N_1}^2$. Then, by construction,
$Q^\star_I \subseteq Q^k$ and $Q^\star_{I'} \subset Q^{k'}$,
for some $Q^{k},Q^{k'} \in \F^N$.
Thus, $Q^{k} \cap Q^{k'} \neq \emptyset$ and since the cubes in $\F^N$ are pairwise disjoint we conclude that $Q^k = Q^{k'}$.  Using \eqref{stwhitqj.eq} for $(Y,s) \in Q^\star_I \cap Q^\star_{I'}$ it holds that $\ell(Q^\star_I) \approx \ell(Q^\star_{I'})$, and hence
$$
\dist(I, Q^\star_I)\approx
\ell(I) \approx \ell(Q^\star_I) \approx \ell(Q^\star_{I'}) \approx \ell(Q_{I'})\approx \ell(I')
\approx \dist(I', Q^\star_{I'})\,,
$$
so that $\dist(I, I') \lesssim \ell(I)\approx \ell(I')$.  Since there can be only a uniformly
bounded number of such Whitney cubes $I'$ relative to any given $I$,
the bounded overlap property follows.

We are left with showing the bounded overlap property of the
collection $\{Q^\star_I\}_I$
in the case that $I \in \Gamma_{\sbf^N_1}^3$. We claim that if $I \in \Gamma_{\sbf^N_1}^3$, then
\begin{equation}\label{stwhitstarqnau.eq}
\ell(Q^\star_I) \lesssim \dist(Y,s, Q_1) \lesssim \ell(Q^\star_I), \quad  \forall (Y,s) \in Q^\star_I.
\end{equation}
To prove this, we note, by definition of $\Gamma_{\sbf^N_1}^3$,
that $\widetilde{Q}_I \cap Q_1 = \emptyset$.
Hence,
$$
\dist(Y,s, \Sigma\setminus \widetilde{Q}_I)  \le \dist(Y,s, Q_1),
$$
and the lower bound in \eqref{stwhitstarqnau.eq} holds by \eqref{stwhitqI.eq}.
To obtain the upper bound in \eqref{stwhitstarqnau.eq}, we recall that
$\dist(\widetilde{Q}_I,Q_I)\lesssim \ell(I)$, by construction, and therefore,
% $ \subset \Sigma\setminus \widetilde{Q}_I$,
\[
\dist(Q^\star_I,Q_I)\lesssim\dist(\widetilde{Q}_I,Q_I) + \ell(\widetilde{Q}_I)
\lesssim  \ell(Q_I^\star)\,,
\]
since $Q_I^\star\subset \widetilde{Q}_I$,
 with $\ell(Q_I^\star)\approx \ell(\widetilde{Q}_I)\approx\ell(I)$.
 Thus, since $Q_I \subset Q_1$,
\[
\dist(Q_I^\star, Q_1) \leq
 \dist(Q_I^\star, Q_I) \lesssim \ell(Q_I^\star)\,,
\]
and the claim \eqref{stwhitstarqnau.eq} follows.

We can now establish the
bounded overlap property of $\{Q^\star_I\}$ for
$I \in \Gamma_{\sbf^N_1}^3$, much as
we did for $\{Q^\star_I\}$ in the case when $I \in \Gamma_{\sbf^N_1}^2$,
but using  \eqref{stwhitstarqnau.eq} in place of \eqref{stwhitqj.eq}.  In fact, in the
present case, the argument is similar, but simpler, and we omit the details.
We have therefore established all of the
properties in \eqref{bbdovqstarpick.eq} and thus have proved the lemma.
\end{proof}

Recall the set $\mathbb{D}_{\F,Q_0}^*$ introduced in \eqref{subsetcubes}.
% \begin{equation*}
% \mathbb{D}_{\F,Q_0}^*
% :=\{Q\in \mathbb{D}_{\F,Q_0}: \ell(Q)\le \eps^{10} \,\ell(Q_0)\},
% \end{equation*}
 % We introduce related domains
% \begin{equation*}% \label{sawtooth-appen-HMM--}
% \Omega^{*}_{\F,Q_0} :=
% \interior\bigg(\bigcup_{Q\in \mathbb{D}_{\F,Q_0}^* } U_Q^*\bigg)=
% \interior\bigg(\bigcup_{Q\in \mathbb{D}_{\F,Q_0}^* }
% \bigcup_{I\in\W^{*}_{Q}} I^{**} \bigg)
% \end{equation*}
% (where as usual, $\interior(E)$ denotes the
% topological interior of a set $E$), and we let
We define a related Whitney collection
\begin{equation*}% \label{sawtooth-appen-HMM--+}
\widetilde{\mathcal{W}}_{\F,Q_0}
:=
\bigcup_{Q\in \mathbb{D}_{\F,Q_0}^* } \W^{*}_{Q}\,,
\end{equation*}
along with corresponding domains
\begin{equation*}% \label{moredomains}
\Omega^{**}_{\F,Q_0}:= \interior\bigg(\bigcup_{I\in \widetilde{\mathcal{W}}_{\F,Q_0}} I^{***}\bigg), \quad  \Omega^{***}_{\F,Q_0}:= \interior\bigg(\bigcup_{I\in \widetilde{\mathcal{W}}_{\F,Q_0}} I^{****}\bigg),
\end{equation*}
where $\tau>0$ is the same small parameter fixed above, and as before,
$I^{***} = (1 + 4\tau)I$ and $I^{****} = (1 + 8\tau)I$.

\begin{remark}\label{remark.sawtoothcontain}
Note that for all
$\sbf_1$ as above, $\Omega^{**}_{\sbf_1}\subset \Omega^{**}_{\F,Q_0}$,
and $\Omega^{***}_{\sbf_1}\subset \Omega^{***}_{\F,Q_0}$,
since by construction, $\sbf_1\subset \mathbb{D}_{\F,Q_0}^*$, and thus
$\widetilde{\mathcal{W}}_{\sbf_1}\subset \widetilde{\mathcal{W}}_{\F,Q_0}$.
\end{remark}

\begin{lemma}\label{bdsforibpinuq.lem} The normalized Green
function $\widehat u=\widehat u(Y,s)$, introduced in \eqref{normgreeb}, satisfies
\begin{equation}\label{upperb1}
\delta(Y,s)|\nabla_Y^2\widehat u(Y,s)|+|\nabla_Y
\widehat u(Y,s)| +(\delta(Y,s))^{-1}\widehat u(Y,s)\lesssim 1\,,
\quad\forall  (Y,s)\in \Omega^{***}_{\mathcal{F}, Q_0},
\end{equation}
and more generally,
\begin{equation} \label{upperb}
 \delta(Y,s)^{|\beta| + 2j - 1}|\nabla_Y^{\beta} \partial_t^j
 (\widehat u(Y,s))| \lesssim_{\beta,j} 1\,,
 \qquad \forall\, (Y,s)\in \Omega^{***}_{\mathcal{F}, Q_0},
  \end{equation}
 where $\beta= (\beta_1, \dots, \beta_n)$ is a multi-index with
 $\beta_i \ge 0$, $j\geq 0$, and
 \[\nabla_Y^\beta \widehat u(Y,s) =
 \partial_{y_1}^{\beta_1}\dots \partial_{y_n}^{\beta_n}\widehat u(Y,s),
 \qquad |\beta| = \sum_{i= 1}^n \beta_i.\]
 In \eqref{upperb1}, the implicit constant depends only
 on $\eps$, $L$ and the usual allowable parameters, and in
 \eqref{upperb}, it depends on these same quantities,
 as well as on $\beta$ and $j$.
\end{lemma}
\begin{proof}
%	In addition to the sets introduced in \eqref{moredomains} we here also work with
%\begin{equation}\label{moredomains+}  \Omega^{****}_{\F,Q_0}:= \interior\bigg(\bigcup_{I\in \widetilde{\mathcal{W}}_{\F,Q_0}} I^{***}\bigg),
%\end{equation}
%where $I^{***}=(1+3\tau)I$ and the corresponding truncated set $\Omega^{****}_{\mathcal{F}^N, Q_0}$. Recall \eqref{eq2.3*g+}. By construction
%$$
%\Omega^{***}_{\F,Q_0}\subset  \Omega^{****}_{\F,Q_0},\qquad \Omega^{***}_{\mathcal{F}^N, Q_0}\subset \Omega^{****}_{\mathcal{F}^N, Q_0}.
%$$
Given the definition of $\Omega^{***}_{\mathcal{F}, Q_0}$ and
standard energy and Schauder estimates for the (adjoint) heat equation
it suffices to see that
\begin{equation}\label{boundu}
(\delta(Y,s))^{-1}\widehat u(Y,s)\lesssim 1, \ \forall  (Y,s)\in I^{*****}:=(1 + 16\tau)I, \text{ with } I\in \widetilde{\mathcal{W}}_{\F,Q_0}\,.
% =  \bigcup_{Q \in \mathbb{D}_{\F,Q_0}^*} \W^{*}_{Q}.
\end{equation}
Fix such an $I$, hence there exists
$Q \in \mathbb{D}_{\F,Q_0}^*$ such that $I\in \W^{i,*}_Q$ for some $i$. Thus, recalling \eqref{Ustar} we have that
\begin{align}\label{nota2uu}
I^*\cap  C_{Z,\tau}\neq\emptyset, \quad \text{where $(Z,\tau) \in C_{X,t}$ for some $(X,t)\in\tilde U^i_{Q}$.}
\end{align}
By the construction of $\tilde U^i_Q$,
\begin{equation*}%\label{eq5.27aprev}
\eps^3\ell(Q)\lesssim \delta(X,t)
\lesssim \eps^{-3}\ell(Q),\qquad
\diam(\tilde U_Q^i)\lesssim \eps^{-4}\ell(Q)\,.
\end{equation*}
Let $C_Q^{**}:=C(X_Q,t_Q, 2\eps^{-5}\ell(Q))$. Using that $\ell(Q)\le \eps^{10}\,\ell(Q_0)$ by the definition of $\mathbb{D}_{\F,Q_0}^*$,  we may apply Lemma \ref{lemma:G-aver}
with $(X_0,t_0)=(X_Q,t_Q)$ and $r=\eps^{-5}\ell(Q)$, and then
Lemma \ref {doublingintree.lem}/\eqref{muaveragebound} with $a=1$, to obtain
\begin{align}\label{nota3uu}
\widehat u(\tilde Z,\tilde \tau)\lesssim  \bariiint_{C_Q^{**}}\,
\widehat u(X, t)\,\d X\d t\lesssim \eps^{-5}\ell(Q)\,
\frac{\mu(\Delta(X_Q,t_Q,c\eps^{-5}\ell(Q)))}{\sigma(\Delta(X_Q,t_Q,c\eps^{-5}\ell(Q)))}
\lesssim \eps^{-5}\ell(Q)\,,
\end{align}
whenever $(\tilde Z,\tilde \tau)\in \Omega\cap C(X_Q,t_Q, \eps^{-5}\ell(Q))$.
% Recall that $Q\in \mathbb{D}^{*}_{\F,Q_0}$, hence
% $\ell(Q)\le \eps^{10}\,\ell(Q_0)$, and therefore
% $\Delta(X_Q,t_Q,c\eps^{-5}\ell(Q)) \subset 2Q_0=2Q(\sbf)$.  Thus,
% using \eqref{eq4.8} of Lemma \ref{initialcorona1.lem},
% we see that
% have $\ell(Q)\le \eps^{10}\,\ell(Q_0)\ll K_0^{-1}\ell(Q_0)$ (recall that $K_0\ll \eps^2$)
% \begin{align*}
 % \frac{\mu(\Delta(X_Q,t_Q,c\eps^{-5}\ell(Q)))}{\sigma(\Delta(X_Q,t_Q,c\eps^{-5}\ell(Q)))}
% \lesssim 1.
%\end{align*}
% Putting the estimates together we deduce that
% \begin{align}\label{nota3uu}
% \widehat u(\tilde Z,\tilde \tau)\lesssim \eps^{-5}\ell(Q),
% \end{align}
% whenever $(\tilde Z,\tilde \tau)\in \Omega\cap C(X_Q,t_Q, \eps^{-5}\ell(Q))$.
% For $(X,t)$ as in \eqref{nota2uu},  and

Let $(Y,s)\in I^{*****}$,
recall that $I\in \W^{i,*}_Q$ for some $i$, and observe that
\eqref{defsforlwbnd.eq} holds with $I^{*****}$  in place of $I$.
% Recall that $I\in \W^{i,*}_Q$ for some $i$.  Hence, for $(Y,s)\in I^{*****}$,
Since $\tau$ is small, it then follows % from \eqref{defsforlwbnd.eq} % the construction
that there exists $\upsilon\geq 3$, depending only on $n$, $M$, and $L$,
such that such $\eps^\upsilon\ell(Q)\lesssim \delta(Y,s)$.
Furthermore, $(Y,s)\in \Omega\cap C(X_Q,t_Q, \eps^{-5}\ell(Q))$,
and therefore that \eqref{nota3uu} holds with
$(\tilde Z,\tilde \tau)$ replaced by $(Y,s)$.
% by the construction, and similar to \eqref{defsforlwbnd.eq}, we see
 Combining these facts we can conclude that if $(Y,s)\in I^{*****}$, then
\begin{align}\label{nota3uu+}
(\delta(Y,s))^{-1}\widehat u(Y,s)\lesssim \eps^{-5-\upsilon},
\end{align}
with implicit constant depending  on the allowable parameters, $\eps$, and $L$. This proves as desired \eqref{boundu}.
 \end{proof}

\begin{lemma}\label{squarefunction} Let
$\sbf_1$ be as above, i.e., $\sbf_1\subset \sbf$
is a semi-coherent tree, with maximal cube
$Q_1=Q(\sbf_1)$ contained in $\mathbb{D}_{\F,Q_0}^*$ (see \eqref{subsetcubes}).
Then the normalized Green function $\widehat u$,
introduced in \eqref{normgreeb}, satisfies  the square function estimate
\begin{align*}
\iiint_{\Omega^{**}_{\sbf_1}} \big(|\nabla_X^2 \widehat u(Y,s)|^2\widehat u(Y,s)
+|\partial_t \widehat u(Y,s)|^2\widehat u(Y,s)\big) \,
\d Y\d s\lesssim \sigma(Q_1),\end{align*}
with implicit constant depending  on $K_0$,
$\eps$, $L$, and the usual allowable parameters.
\end{lemma}

\begin{proof}
By \eqref{eq:Omega_N-TCM}, it is enough to prove the stated estimate
with $\sbf_1$ replaced by $\sbf_1^N$, with a bound depending
on $K_0$,
$\eps$, $L$, and the allowable parameters, but not on $N$.

Let $\Psi=\Psi_N\in C_c^\infty(\mathbb R^{n+1})$ be as in Lemma \ref{lemma:approx-saw}. By \eqref{cuttoff},
\begin{align*}\Psi_N(Y,s)=
\sum_{I\in \widetilde{\mathcal{W}}_{\sbf_1^N}} \eta_I(Y,s)\,,
\end{align*}
for all $(Y,s)\in\Omega$. Based on this we introduce
\begin{align*}
A:= \sum_{i,j}\iiint \widehat u ( \partial_{x_i}\partial_{x_j} \widehat u)^2
\,\Psi_N \, \d Y\d s\,, \qquad
B:= \iiint \widehat u (\partial_t\widehat u)^2\, \Psi_N \, \d Y\d s\,.
\end{align*}
% where  we use summation convention with respect to $i$, $j$; and
% $I\in \widetilde{\mathcal{W}}_{\F^N,Q_0}$.
Proceeding as in  \cite[Proof of Theorem 9]{LeNy}
(see also \cite[Lemma 5.38]{BHMN}), and using Lemma \ref{bdsforibpinuq.lem}
(and Remark \ref{remark.sawtoothcontain}),
one can prove that
\begin{align*}
A+2B\lesssim \iiint_{\Omega}
\big(|\nabla_X\Psi_N(Y,s)|+\delta(Y,s)|\partial_t\Psi_N(Y,s)|\big)\, \d Y\d s,
\end{align*}
with implicit constant depending  on the allowable parameters, $\eps$, and $L$,
but not on $N$.
We refer the reader to \cite[Lemma 5.38]{BHMN} for the details.
Therefore, using Lemma \ref{partialADR} we can conclude that
\begin{align*}
A+2B\lesssim \sigma(Q_1),
\end{align*}
with implicit constant depending  on the allowable parameters, $\eps$, and $L$, but not on $N$.
\end{proof}

\section{Proof of Theorem \ref{WHSA.thrm}}\label{Sec6}

In this section we prove Theorem \ref{WHSA.thrm}.  In particular, we intend to prove that $\Sigma$ satisfies the $(\eps,K_0)$-WHSA condition (see Definition \ref{def2.14}), for all $\eps\leq\eps_0$, $\eps_0=K_0^{-100}$.  We note that $\eps$ and $K_0$ are (respectively, sufficiently small and large) degrees of freedom, whose values may eventually be fixed depending only on the allowable parameters.
As in Subsection \ref{initialstopanlys.sect} we work with a single
stopping time tree $\sbf$ from Lemma \ref{initialcorona1.lem},
and more precisely, with a single sub-tree
$\sbf'\in \mathcal{S}'_\sbf=\mathcal{S}'_\sbf(\xi) $
from Lemma \ref{refp+},
where $\xi$ is chosen so that Lemma \ref{refp++} holds.
For notational convenience we will continue to use the notation $Q_0 = Q(\sbf)$ and we recall that $\sbf = \mathbb{D}_{\F,Q_0}$ for a pairwise disjoint collection of subcubes of $Q_0$, denoted $\mathcal{F}$.
Note that, as a consequence of Lemma \ref{refp+},
in order to prove Theorem \ref{WHSA.thrm},
it suffices to show
that the cubes in $\sbf'$ for
which {\it the $(\eps,K_0)$-local WHSA condition fails},
satisfy a Carleson packing condition.

\subsection{Partitioning the cubes in $\sbf'$ % $\mathbb{D}_{\F,Q_0}$
into classes} \label{sspart}
Recall that we are working with the normalized caloric measure $\mu$, which
satisfies
\begin{equation}\label{eq6.1}
1/2 \le \frac{\mu(Q)}{\sigma(Q)} \le \bariint_Q \mathcal{M}(\mu\big|_{2Q(\sbf)})(X,t) \, \d\sigma(X,t) \le a_2,  \quad \forall Q \in \dd_{\F,Q_0}
\end{equation}
(see Lemma \ref{initialcorona1.lem}).
We have also shown in Lemma \ref{refp++} that there exists,
for each $$Q\in \sbf'\cap \dd_{\F,Q_0}^*
=\{Q\in\sbf': \ell(Q)\le \eps^{10}\,\ell(Q_0)\}\,,$$
a point $(Y^*_Q,s^*_Q)\in U_Q$ such that
\begin{align}\label{eq6.5ll}
 1\lesssim |\nabla_X \widehat u(Y^*_Q,s^*_Q)|\,,
\end{align}
where the implicit constant is independent of $\eps$ and $K_0$.
Recall that the Whitney region $U_{Q}$ has a uniformly bounded number
of connected components, which we have enumerated by $\{U^i_{Q}\}$.
Given $Q$ we now fix {\it the particular $i$ for which}
\begin{equation}\label{Y*s*component}
(Y^*_Q,s^*_Q)\in U^i_Q\subset \tilde U^i_Q \subset U^{i,*}_{Q},
\end{equation}
where the latter two sets are the enlarged Whitney regions
constructed in Definition \ref{def2.11a} and in \eqref{Ustar2}, respectively. In the following the reference point  $(Y^*_Q,s^*_Q)$ will play a prominent role and so will the direction defined by $\nabla_X \widehat u(Y^*_Q,s^*_Q)$.

% Recall that by Lemma \ref{bdsforibpinuq.lem},
% \begin{equation} \label{upperb}
%  \delta(Y,s)^{|\beta| + 2j - 1}|\nabla_Y^{\beta} \partial_t^j (\widehat u(Y,s))| \lesssim_{\beta,j} 1,\qquad \mbox{for all } (Y,s)\in \Omega^{***}_{\mathcal{F}, Q_0},
 % \end{equation}
% where $\beta= (\beta_1, \dots, \beta_n)$ is a multi-index, with
 % $\beta_i \ge 0$, and $j\geq 0$.

We divide the cubes in $\sbf'$ %$\mathbb{D}_{\F,Q_0}$
into four classes/collections by considering four different cases.

\noindent{\bf Case 0}: $Q\in % \mathbb{D}_{\F,Q_0}
\sbf' $, with $\ell(Q)> \eps^{10}\,\ell(Q_0)$.

\smallskip
\noindent{\bf Case 1}: $Q\in % \mathbb{D}_{\F,Q_0}
\sbf'$, with $\ell(Q)\le \eps^{10}\,\ell(Q_0)$, and
\begin{equation}\label{eq5.3}
%\max_{Q'\in\mathbb{D}_\eps(Q)}
\sup_{(X,t)\in \tilde U^i_{Q}}\,\sup_{(Z,\tau)\in C_{X,t}} |\nabla_X \widehat u(Z,\tau)
 -\nabla_X \widehat u(Y^*_Q,s^*_Q)|\,>\,\eps^{2L}\,.
\end{equation}

\smallskip
\noindent{\bf Case 2}: $Q\in % \mathbb{D}_{\F,Q_0}
\sbf'$,
with $\ell(Q)\le \eps^{10}\,\ell(Q_0)$, and
\begin{equation}\label{eq5.2}
\sup_{(X,t)\in \tilde U^i_{Q}}\,\sup_{(Z,\tau)\in C_{X,t}} |\nabla_X \widehat u(Z,\tau)
 -\nabla_X \widehat u(Y^*_Q,s^*_Q)| \leq  \eps^{2L},
\end{equation}
but
\begin{equation}\label{eq5.3+}
%\max_{Q'\in\mathbb{D}_\eps(Q)}
\sup_{(X,t)\in \tilde U^i_{Q}}\,\sup_{(Z,\tau)\in C_{X,t}} |\widehat u(Z,\tau)-\widehat u(X,t)-\nabla_X\widehat u(X,t)\cdot(Z-X)|>\eps^{2 L}\ell(Q).
\end{equation}

\smallskip
\noindent{\bf Case 3}: $Q\in % \mathbb{D}_{\F,Q_0}
\sbf'$, with $\ell(Q)\le \eps^{10}\,\ell(Q_0)$, and
\begin{equation}\label{eq5.2a}
\sup_{(X,t)\in \tilde U^i_{Q}}\,\sup_{(Z,\tau)\in C_{X,t}} |\nabla_X \widehat u(Z,\tau)
 -\nabla_X \widehat u(Y^*_Q,s^*_Q)| \leq  \eps^{2L},
\end{equation}
as well as
\begin{equation}\label{eq5.3+a}
%\max_{Q'\in\mathbb{D}_\eps(Q)}
\sup_{(X,t)\in \tilde U^i_{Q}}\,\sup_{(Z,\tau)\in C_{X,t}} |\widehat u(Z,\tau)-\widehat u(X,t)-\nabla_X\widehat u(X,t)\cdot(Z-X)|\leq \eps^{2L}\ell(Q).
\end{equation}

For cubes in {Case 0}, we use the trivial fact that in particular,
$Q\in\sbf= \mathbb{D}_{\F,Q_0}$.  Observe that for every $R \in \dd$,
\begin{equation}\label{eqn-case0-pack}
\sum_{\substack{Q\in \mathbb{D}_{\F,Q_0} \\ {\rm Case\, 0 \, holds } \\ Q \subseteq R }}\sigma(Q)
\le
\sum_{k = 0}^{\log_2(\eps^{-10}) + 1} \sum_{Q\in\mathbb{D}_k(Q_0)} \sigma(Q \cap R)
\lesssim
(\log_2 \eps^{-1})\,\sigma(R).
\end{equation}
This shows that the cubes in {Case 0} satisfy the packing condition.

To prove Theorem \ref{WHSA.thrm} we will, in the following three sections, prove that the cubes in {Case 1} (Section \ref{subq1}) and {Case 2} (Section \ref{subq2}) satisfy the packing condition, and that cubes in {Case 3} (Sections \ref{subq3} and \ref{subq4}) satisfy the $(\eps, K_0)$-local WHSA
condition in the sense of Definition \ref{def2.13}. We
remind the reader that the constant appearing in the
packing statement of the WHSA condition may depend on $\eps$,
and we do not make this dependence precise.
% Note that obviously not all cubes in Case 0-Case 3 can pack,
% and there must be an ample collection of cubes that fall into Case 3.

\subsection{The cubes in {Case 1} satisfy the packing condition}\label{subq1}
Consider a cube $Q\in % \mathbb{D}_{\F,Q_0}
\sbf'$ as in the statement of {Case 1}, so that in particular,
$\ell(Q)\le \eps^{10} \,\ell(Q_0)$.
In the scenario of Case 1, there exist $(X,t)\in \tilde U^i_{Q}$ and $(Z,\tau)\in C_{X,t}$ such that
\begin{equation}\label{eq5.3bla}
 |\nabla_X \widehat u(Z,\tau)
 -\nabla_X \widehat u(Y^*_Q,s^*_Q)|\,>\,\eps^{2L}\,.
\end{equation}
By the construction of $\tilde U^i_{Q}$, and \eqref{Y*s*component},
there exists a chain $\{C_{Y_j,s_j}\}_{1\le j\le l}$
 connecting $(X,t)$ and $(Y^*_Q,s^*_Q)$ so that $(Y_1,s_1)=(Y^*_Q,s^*_Q)$ and
$(Y_l,s_l)=(X,t)$ and $l\lesssim\eps^{-1}$.
Clearly, we can construct the chain in such a way that
\begin{equation}\label{Yjsjcontain}
(Y_{j},s_{j})\in C_{Y_{j-1},s_{j-1}}, \quad 2\leq j \leq l\,.
\end{equation}
Reformulating \eqref{eq5.3bla}  with this notation,
we may assume both that
\begin{equation}\label{eq5.3ha}
\sup_{(Z,\tau)\in C_{Y_l,s_l}}|\nabla_X \widehat u(Z,\tau)
 -\nabla_X \widehat u(Y_1,s_1)|\,>\,\eps^{2L}\,.
\end{equation}
and that
\begin{equation}\label{eq5.3haa}
\sup_{(Z,\tau)\in C_{Y_j,s_j}}|\nabla_X \widehat u(Z,\tau)
 -\nabla_X \widehat u(Y_1,s_1)|\leq\,\eps^{2L},\quad \mbox{for $1\leq j\leq l-1$, provided $l>1$},
\end{equation}
since otherwise we shorten the
chain and work with the first $(Y_j,s_j)$ for which \eqref{eq5.3ha} holds.
We claim that
\begin{equation}\label{eq5.3hagg---}
|\nabla_X \widehat u(Y_j,s_j)|\geq |\nabla_X \widehat u(Y_1,s_1)|-\eps^{2L},\quad \mbox{for $1\leq j\leq l$}.
\end{equation}
Indeed, this is trivial when $l=1$ and for $l>1$ it
follows at once from \eqref{eq5.3haa} and
the fact that $(Y_{l},s_{l})\in C_{Y_{l-1},s_{l-1}}$ (see \eqref{Yjsjcontain}).
Furthermore, by the triangle inequality,
\begin{align}\label{eq5.3hagg-}
\eps^{2L}\leq \sup_{(Z,\tau)\in C_{Y_l,s_l} }|\nabla_X\widehat u(Z,\tau)
 -\nabla_X \widehat u(Y_l,s_l)|
+\sum_{j=1}^{l-1}|\nabla_X \widehat u(Y_{j+1},s_{j+1})
 -\nabla_X \widehat  u(Y_j,s_j)|.
\end{align}
Consequently, since $l\lesssim \eps^{-1}$, we have one of the following options
\begin{list}{$(\theenumi)$}{\usecounter{enumi}\leftmargin=1cm \labelwidth=1cm \itemsep=0.2cm \topsep=.2cm \renewcommand{\theenumi}{\roman{enumi}}}
\item $\sup_{(Z,\tau)\in C_{Y_l,s_l} }|\nabla_X\widehat u(Z,\tau)
 -\nabla_X \widehat u(Y_l,s_l)|\geq \eps^{2L+2}$,

\item  $|\nabla_X \widehat u(Y_{j+1},s_{j+1})
 -\nabla_X \widehat  u(Y_j,s_j)|\geq\eps^{2L+2}$, for some $1\leq j\leq l-1$.
\end{list}

We treat both scenarios simultaneously. In scenario $(i)$,
we set $(Y,s) = (Y_l, s_l)$, and let $(Z,\tau)$ be
the point in $C_{Y_l,s_l}$ at which the sup is attained
(recall that we have taken the cylinders $C_{Y,s}$ to be closed).
If we are in scenario $(ii)$, then we set,  with an abuse of notation,
$(Y,s) = (Y_j,s_j) $ and $(Z,\tau) =(Y_{j+1},s_{j + 1})$. In either scenario,
$(Z,\tau)$ and $(Y,s)$ lie in the cylinder
$C_{Y,s}$, which has parabolic diameter
$\approx_\varepsilon\ell(Q)$.  Note also that
there is a constant $\kappa=\kappa_{\eps,L} > 1$ such that
$2\kappa C_{Y,s}$ (the $2\kappa$-dilate of $C_{Y,s}$) is
contained in $U_Q^*$. In particular, we
may use \eqref{upperb} in this dilated
cylinder and we note that $|Z - Y| \lesssim_\varepsilon \ell(Q)$ and $|\tau - s| \lesssim_\varepsilon \ell(Q)^2$. Using $(i)$ or $(ii)$ we then have for some $k$
\[
|\partial_{x_k} \widehat{u}(Z,\tau) - \partial_{x_k} \widehat{u}(Y,s)| \ge  \eps^{2L+2}/\sqrt{n}.
\]
Hence by Taylor's theorem we see there is a point $(\widehat{Y}, \widehat{s}\,) \in C_{Y,s}$
such that
\[
|\partial_t \partial_{x_k} \widehat{u} (\widehat{Y}, \widehat{s}) (\tau- s) + \nabla_X \partial_{x_k} \widehat{u} (\widehat{Y}, \widehat{s})\cdot (Z - Y)| = |\partial_{x_k} \widehat{u}(Z,\tau) - \partial_{x_k} \widehat{u}(Y,s)| \ge  \eps^{2L+2}/\sqrt{n}
\]
\
Then, using that $|Z - Y| \lesssim_{\eps,L} \ell(Q)$ and $|\tau - s| \lesssim_{\eps,L} \ell(Q)^2$,
and that $\widehat{u}$ is a solution to the adjoint equation, we see that either
\begin{equation}\label{c1dtnablg.eq}
|\nabla_X^3 \widehat{u} (\widehat{Y}, \widehat{s}\,)|  \ell(Q)^2 \gtrsim
|\partial_t \partial_{x_k} \widehat{u} (\widehat{Y}, \widehat{s}\,)|  \ell(Q)^2
\ge a_{\eps,L} % \gtrsim_\eps \eps^{2L+2},
\end{equation}
or
\begin{equation}\label{c1dnab2lg.eq}
|\nabla_X^2 \widehat{u} (\widehat{Y}, \widehat{s}\,)| \ell(Q)
\ge
|\nabla_X \partial_{x_k}  \widehat{u} (\widehat{Y}, \widehat{s}\,)| \ell(Q)
 \geq a_{\eps,L}\,, % _\eps \eps^{2L+2}.
\end{equation}
where $a_{\eps,L}$ is a small uniform
constant depending only on $\eps, L$ and allowable parameters.
Moreover, by standard interior estimates for adjoint solutions,
\eqref{c1dtnablg.eq} implies \eqref{c1dnab2lg.eq} with $a_{\eps,L}$ replaced by a
somewhat smaller constant with the same dependence, and with
$(\widehat{Y}, \widehat{s}\,)$ replaced by
$(\widehat{Y}, \widehat{s}-\eta\ell(Q))$, where $\eta=\eta_{\eps,L}$ is a small, fixed constant
which we may choose small enough that the point
$(\widehat{Y}, \widehat{s}-\eta\ell(Q))$ belongs to the
dilated cylinder $\kappa C_{Y,s}$.  After a suitable relabeling,
we may therefore assume that \eqref{c1dnab2lg.eq} holds for some
$(\widehat{Y}, \widehat{s}\,)\in \kappa C_{Y,s}$.

We then have
% and the local Hölder continuity estimate for solutions,
that for a sufficiently small uniform constant
$\zeta=\zeta_{\eps,L}$, depending only on $\eps,L$ and allowable parameters,
the cylinder $C(\widehat{Y}, \widehat{s}, \zeta \ell(Q))\subset 2\kappa C_{Y,s}
\subset U_Q^*$, and by \eqref{upperb},
% \begin{equation}\label{c1dtnablgnhd.eq}
% |\nabla_X^3 \widehat{u} (X,t)|  \ell(Q)^2  \geq \frac{a_{\eps,L}}2\,, % \gtrsim_\eps \eps^{2L+2},
% \quad \forall (X,t) \in C(\widehat{Y}, \widehat{s}, \zeta \ell(Q)),
% \end{equation}
% or
\begin{equation}\label{c1dnab2lgnhd.eq}
|\nabla_X^2 \widehat{u} (X,t)| \, \ell(Q) \, \geq \frac{a_{\eps,L}}2 \,, % \gtrsim_\eps \eps^{2L+2},
\quad \forall (X,t) \in C(\widehat{Y}, \widehat{s}, \zeta \ell(Q)).
\end{equation}
Consequently, by standard interior estimates for adjoint solutions, we have
\begin{equation}\label{ulowerbound}
 \frac{\widehat{u}(X,t)}{\ell(Q)}  \,\gtrsim \,a_{\eps,L} \,, % \gtrsim_\eps \eps^{2L+2},
\quad \forall (X,t) \in C^-(\widehat{Y}, \widehat{s}, \zeta \ell(Q))\,,
\end{equation}
where we recall that $C^-$ denotes the time-backwards half of a cylinder (see \eqref{Cpm}).

Combining \eqref{c1dnab2lgnhd.eq} and \eqref{ulowerbound}, we obtain that
 % in $C^-(\widehat{Y}, \widehat{s}, \zeta \ell(Q))$
\begin{equation}\label{c1dnab2lgavgpre.eq}
a^3_{\eps,L}\,\lesssim\,
\ell(Q) \bariiint_{C^-(\widehat{Y}, \widehat{s}, \zeta \ell(Q))}
|\nabla_X^2 \widehat{u}|^2\,\widehat{u}\, \d X\d t
\,\approx_{\eps,L}\,
\ell(Q)^{-n-1} \!\iiint_{C^-(\widehat{Y}, \widehat{s}, \zeta \ell(Q))}
|\nabla_X^2 \widehat{u}|^2\,\widehat{u}\, \d X\d t
\,.
\end{equation}
As noted above, $C^-(\widehat{Y}, \widehat{s}, \zeta \ell(Q))\subset U_Q^*$, and therefore, since
$\widehat u$ is an adjoint solution,
\begin{equation} \label{case1finallwbd.eq}
 \sigma(Q)\approx \ell(Q)^{n+1} \lesssim_{\eps,L}
\iiint_{U_Q^*}\big(|\nabla_X^2 \widehat u(X,t)|^2 + |\partial_t \widehat u(X,t)|^2\big) \widehat{u}(X,t) \, \d X\d t\,.
\end{equation}

Now let $R \in \dd$ be such that there exist
$Q\in\sbf'\cap\dd(R)$ with $\ell(Q)\le \eps^{10}\,\ell(Q_0)$,
let $\dd^*(R)$ denote the collection of all such $Q$, and let
$\dd_{\tt max}^*(R)$ denote the collection of maximal cubes (with respect to containment) in $\dd^*(R)$.
Our goal is to show that
\[
\sum_{Q\in\dd^*(R)}
\sigma(Q) \lesssim \sigma(R)\,.
\]
Observe that the maximal cubes are disjoint, and thus
$\sum_{Q\in\dd_{\tt max}^*(R)}
\sigma(Q) \leq \sigma(R)$.
Hence, it suffices to consider the case that $R=Q_1\in \sbf'$, with
$\ell(Q_1)\le \eps^{10}\,\ell(Q_0)$.

To this end, set $\sbf_1:= \sbf'\cap\dd(Q_1)$, so that
$\sbf_1$ is a semi-coherent tree with maximal cube
$Q_1=Q(\sbf_1)$ belonging to $\mathbb{D}_{\F,Q_0}^*$
(see \eqref{subsetcubes}).
Therefore, since the elements in $\{U^*_Q\}$
have bounded overlaps,
using \eqref{case1finallwbd.eq}, followed by
\eqref{sawtooth-appen-HMM--} and \eqref{doma++},
 we see that
\begin{multline*}
\sum_{\substack{Q\in \sbf' \cap \dd(Q_1) \\ {\rm Case\, 1 \, holds} }} \! \!\! \sigma(Q)
 \, \lesssim_{\eps,L} \, \sum_{Q\in \sbf_1}\,
 \iiint_{U_Q^*}(|\nabla_X^2 \widehat u(X,t)|^2 + |\partial_t \widehat u(X,t)|^2) \widehat{u}(Y,s)  \, \d X \, \d t
\\  \lesssim_{\eps,L} \, \iiint_{\Omega^{**}_{\sbf_1}}
(|\nabla_X^2 \widehat u(X,t)|^2\widehat u(X,t)+|\partial_t \widehat u(X,t)|^2\widehat u(X,t)) \, \d X\d t\, \lesssim \sigma(Q_1)\,,
 \end{multline*}
 where the last step follows by Lemma \ref{squarefunction}.
This completes the proof of the packing of  the cubes in Case 1.

\subsection{The cubes in  {Case 2} satisfy the packing condition}\label{subq2}
We fix $Q\in % \mathbb{D}_{\F,Q_0}
\sbf'$ as in the statement of {Case 2}, so that in particular,
$\ell(Q)\le \eps^{10} \,\ell(Q_0)$.
 Recall,
as we observed in Lemma \ref{refp++}
(see also \eqref{eq6.5ll}), that there exists $a > 0$ independent of $\eps$ such that
\[|\nabla_X \widehat u(Y_Q^*, s_Q^*)| \ge 2a.\]
Hence, by \eqref{eq5.2} (perhaps enforcing further smallness on $\eps$) we have
\[ |\nabla_X \widehat u(Z,\tau)| \ge a, \quad \mbox{whenever } (Z,\tau)\in C_{X,t},\ (X,t)\in \tilde U^i_{Q}.\]
As a consequence, if $(X,t)\in \tilde U^i_{Q}$, then combining the latter estimate
and \eqref{upperb}, % and  local Hölder continuity estimates,
we see that for some $\zeta =\zeta_{\eps,L} \ll 1$ sufficiently small, depending
only on $\eps$, $L$ and allowable parameters,
\[
|\nabla_X \widehat u(Z,\tau)| \ge a/2, \quad  \forall (Z,\tau) \in \mathcal{N}_{X,t, \eps}^{(1)}\,,
\]
where
\begin{equation}\label{N1def}
\mathcal{N}_{X,t, \eps}^{(1)} := \{(Z,\tau): \dist(Z,\tau, C_{X,t}) < \zeta \ell(Q)\}
\subset U_Q^*\,,
\end{equation}
 % and $\mathcal{N}_{X,t, \eps}^{(1)} \subset U_Q^*$
 (see \eqref{Ustar} and \eqref{Ustar2}). Consequently, using interior estimates for adjoint solutions
 % the Harnack inequality and the Caccioppoli inequality,
 we have
\begin{equation}\label{case2nondegen.eq}
 \widehat u(Z,\tau) \gtrsim_{\eps,L} \ell(Q),
 \quad \forall (Z,\tau) \in \mathcal{N}_{X,t, \eps}^{(2)},
 \end{equation}
 where for the same $\zeta$ as in \eqref{N1def},
\begin{equation}\label{N2def}
\mathcal{N}_{X,t, \eps}^{(2)} := \{(Z,\tau): \dist(Z,\tau, C_{X,t}) < \zeta^2 \ell(Q)\}.
\end{equation}
% In the following we will always ensure that our estimates/integrals
% are taken inside $\mathcal{N}_{X,t, \eps}^{(2)}$.

By the definition of {Case 2} there exist $(X,t)\in \tilde U^i_{Q}$ and  $(Z,\tau)\in C_{X,t}$ such that
\begin{equation*}% \label{eq5.3+gg}
 |\widehat u(Z,\tau)-\widehat u(X,t)-\nabla_X\widehat u(X,t)\cdot(Z-X)|>\eps^{2L}\ell(Q).
\end{equation*}
Using this and the triangle inequality
\begin{equation}\label{case2prelimest.eqML}
\eps^{2L}\ell(Q) \le  I_1 + I_2,
\end{equation}
where
\begin{equation*} % \label{case2prelimest.eq}
I_1:= |\partial_t \widehat u(X,t)(\tau - t)|, \quad
I_2:=|\widehat u(Z,\tau)-\widehat u(X,t)- \partial_t \widehat  u(X,t)(\tau - t) -\nabla_X\widehat u(X,t)\cdot(Z-X)|.
\end{equation*}

We now make the following claim.

\noindent{\bf Claim}:  There is a point $(\widehat Y,\widehat s\,)\in C_{X,t}$, a cylinder
$C(\widehat  Y,\widehat  s,\zeta^2\ell(Q))\subset \mathcal{N}_{X,t, \eps}^{(2)}\subset U_Q^*$, and a uniform positive constant
$a_{\eps,L}$ such that for $\zeta>0$ sufficiently small, depending
only on $\eps$, $L$ and allowable parameters
\begin{equation}  \label{grad2lower}
|\nabla_Y^2 \widehat{u} (Y,s)| \, \ell(Q) \, \geq a_{\eps,L} \,, % \gtrsim_\eps \eps^{2L+2},
\quad \forall (Y,s) \in C(\widehat{Y}, \widehat{s}, \zeta^2 \ell(Q))\,,
\end{equation}
and
\begin{equation} \label{ulowerbound2}
 \frac{\widehat{u}(Y,s)}{\ell(Q)}  \,\geq \,a_{\eps,L} \,, % \gtrsim_\eps \eps^{2L+2},
\quad \forall (Y,s) \in C(\widehat{Y}, \widehat{s}, \zeta^2 \ell(Q))\,.
\end{equation}
% where as usual, $C^-$ denotes the time-backwards half of the cylinder.

Taking the claim for granted momentarily, we observe
that the packing property of the cubes in Case 2 now follows almost
exactly as in Case 1:  indeed,
estimates \eqref{grad2lower} and  \eqref{ulowerbound2} are essentially similar to
\eqref{c1dnab2lgnhd.eq} and \eqref{ulowerbound}, and the respective cylinders
under consideration are each contained in $U_Q^*$.

Thus, it suffices to prove the claim.  Note first that
for every $(\widehat Y,\widehat s\,)\in C_{X,t}$, we have by construction
(see \eqref{N1def} and \eqref{N2def}) that
$C(\widehat  Y,\widehat  s,\zeta^2\ell(Q))\subset \mathcal{N}_{X,t, \eps}^{(2)}\subset U_Q^*$,
and hence that \eqref{ulowerbound2} holds by \eqref{case2nondegen.eq}.
We therefore need only find a point $(\widehat Y,\widehat s\,)\in C_{X,t}$
verifying \eqref{grad2lower}.

Suppose first that $I_1\,\geq\,\eps^{2L}\,\ell(Q)/2$.
Since $|t- \tau| \lesssim_{\eps,L} \ell(Q)^2$, and $\widehat u$ is an adjoint solution,
we have that
\begin{equation}\label{case2I1est.eq}
\eps^{2L}\,\ell(Q)\lesssim I_1\,=\, |\Delta \widehat u(X,t)||\tau - t|  \,
\lesssim_{\eps,L} \,|\nabla^2_X\widehat u(X,t)|\,\ell(Q)^2\,.
% \ell(Q)^{-(n-2)/2}\left(\iiint_{C(X,t, \zeta^3 \ell(Q))}
% |\nabla^2 \widehat u(\tilde{X},\tilde{t}\,)|^2\, \d \tilde X\d \tilde t\right)^{1/2}.
\end{equation}
Consequently, setting $(\widehat Y,\widehat s\,)=(X,t)$, and using
\eqref{upperb}, we see that \eqref{grad2lower} holds in the present scenario,
for $\zeta$ chosen
suitably small.
Thus, the claim holds if $I_1\,\geq\,\eps^{2L}\,\ell(Q)/2$.

% On the other hand if $I_1\,<\,\eps^{2L}\,\ell(Q)/2$, then necessarily
% $I_2\,\geq\,\eps^{2L}\,\ell(Q)/2$,
% by \eqref{case2prelimest.eqML}. % either $I_1\,\geq\,\eps^{2L}\,\ell(Q)/2$ or
% $I_2\,\geq\,\eps^{2L}\,\ell(Q)/2$.
% \begin{equation*}\label{case2I1est.eq}
% \max(I_1,I_2)\,\geq\,\eps^{2L}\,\ell(Q)\,.
% \end{equation*}

On the other hand, if $I_1\,<\,\eps^{2L}\,\ell(Q)/2$, then necessarily
$I_2\,\geq\,\eps^{2L}\,\ell(Q)/2$,
by \eqref{case2prelimest.eqML}.
% To handle term $I_2$
In this case, we use Taylor's theorem to obtain
\[I_2 = |R(X,t,Z,\tau)|,\]
where $R(X,t,Z,\tau)$ denotes the Taylor remainder whose coefficients are second order derivatives, in $X$ and $t$, evaluated at a point on the line segment between $(X,t)$ and $(Z,\tau)$. In particular, since $C_{X,t}$ is convex, there exists $(\widehat Y, \widehat s\,) \in C_{X,t}$ such that
\begin{multline*}% \label{I2bound}
\eps^{2L}\,\ell(Q)\,\lesssim \,I_2
% \\[4pt]
\, \lesssim \, % _{\eps,L}
|\nabla_X^2 \widehat{u}(\widehat Y,\widehat s\,)|\,|Z - X|^2 +
|\partial_t \nabla_X \widehat{u}(\widehat Y,\widehat s\,)|\,|\tau-t||Z- X| +
|\partial^2_t \widehat{u}(\widehat Y,\widehat s)|\,|\tau-t|^2
 \\[4pt]
\lesssim_{\eps,L}
|\nabla_X^2 \widehat{u}(\widehat Y,\widehat s\,)|\,\ell(Q)^2 +
|\partial_t \nabla_X \widehat{u}(\widehat Y,\widehat s\,)|\,\ell(Q)^3 +
|\partial^2_t \widehat{u}(\widehat Y,\widehat s\,)|\,\ell(Q)^4
\end{multline*}
For $\zeta$ small enough, the latter estimate plus \eqref{upperb} then
yields the bound
\[
1\,\lesssim_{\eps,L}\,
|\nabla_Y^2 \widehat{u}(Y,s)|\,\ell(Q) +
|\partial_s \nabla_Y \widehat{u}(Y,s)|\,\ell(Q)^2 +
|\partial^2_s \widehat{u}(Y,s)|\,\ell(Q)^3\,,\quad (Y,s)\in C(\widehat{Y}, \widehat{s}, \zeta \ell(Q))\,,
\]
and in turn, we obtain \eqref{grad2lower}
from standard interior estimates for adjoint solutions.
This concludes the treatment of the cubes in {Case 2}.

\subsection{Analysis of cubes in  {Case 3}: technical steps}\label{subq3}

Cubes in {Case 3} remain to be treated. Here our goal is to show that the cubes in {Case 3} satisfy the WHSA condition. In this subsection we carry out some preliminary analysis of the {Case 3} cubes.
Modifying as needed to treat the parabolic case, we
follow the corresponding arguments in \cite{HLMN}, which in turn are based on those of \cite{LV}.

Recall that if $Q\in \sbf'$ is a cube belonging to {Case 3}, then $\ell(Q)\le \eps^{10}\,\ell(Q_0)$, and \eqref{eq5.2a} and \eqref{eq5.3+a} hold. We also recall that
$\widehat u$ is the normalized Green function with pole far into the future, see Lemma \ref{nondegcubelem.lem} part $(a)$,  and that $\widehat u$ solves the adjoint heat equation. In the following, the implicit constants are allowed to depend on  the allowable parameters and $K_0$, unless otherwise stated, but not on $\eps$ or $L$. Dependence on $\eps$ and $L$ will be stated explicitly. Recall the cylinders $C_{Y,s}$ and $\widetilde C_{Y,s}$, defined in \eqref{eq5.1}.

Consider
$(\widehat Z,\widehat \tau\,)\in \partial\widetilde  C_{Y,s}\cap \Sigma$.
Then, by our assumptions and Lemma \ref{touchp},
$$
(\widehat Z,\widehat \tau\,)\in \partial B(Y,d)\times [s-d^2,s+d^2],
$$
where $d=\delta(Y,s)$. Given $\widehat Z$ and $Y$, we let $\tilde Z$ denote the radial projection, along the line defined by $\widehat Z$ and $Y$, of
$\widehat Z $ onto $\partial B(Y, (1-\eps^{2L/\alpha})d)$.  Furthermore, if
$$
\widehat \tau\in [s-(1-\eps^{2L/\alpha})^2d^2,s+(1-\eps^{2L/\alpha})^2d^2],
$$
then we define $\tilde\tau:=\widehat\tau$. If
\begin{equation}\label{hattauinterval}
\widehat \tau\notin [s-(1-\eps^{2L/\alpha})^2d^2,s+(1-\eps^{2L/\alpha})^2d^2],
\end{equation}
then we simply set $\tilde \tau$ equal to the closest
endpoint of the interval in \eqref{hattauinterval}; e.g., if
\[
s+(1-\eps^{2L/\alpha})^2d^2 <\hat \tau \leq s+d^2\,,
\]
then we set $\tilde \tau := s+(1-\eps^{2L/\alpha})^2d^2$.
This construction maps   $(\widehat Z,\widehat \tau)\in \partial\widetilde C_{Y,s}\cap\Sigma$ to a point
$(\tilde Z,\tilde \tau)\in C_{Y,s}$ and
\begin{align}\label{distf}
\dist(\widehat Z,\widehat \tau,\tilde Z,\tilde \tau)\lesssim \eps^{2L/\alpha}d=\eps^{2L/\alpha}\delta(Y,s).
\end{align}

\begin{lemma}\label{l5.14}  Consider $(Y,s)\in U^i_Q$, $(X,t)\in \tilde U^i_Q$,
$(\widehat Z,\widehat \tau\,)\in \partial\widetilde  C_{Y,s}\cap \Sigma$, and
let $(\widehat Z,\widehat \tau\,)$ be mapped to $(\tilde Z,\tilde \tau)\in C_{Y,s}$ as above. Then
\begin{equation}\label{eq5.15a}
\widehat u(\tilde Z,\tilde \tau)\lesssim \eps^{2L-5} \delta(Y,s)\,.
\end{equation}
Similarly, if  $(\widehat Z,\widehat \tau\,)\in \partial\widetilde  C_{X,t}\cap \Sigma$, and
$(\widehat Z,\widehat \tau\,)$ is mapped to $(\tilde Z,\tilde \tau)\in C_{X,t}$ through the construction above, then
\begin{equation}\label{eq5.16a}
\widehat u(\tilde Z,\tilde \tau)
\lesssim
\eps^{2L-5} \ell(Q)\,.
\end{equation}
\end{lemma}
\begin{proof}  As $K_0^{-1}\ell(Q)\lesssim \delta(Y,s)\lesssim K_0\, \ell(Q)$ whenever $(Y,s)\in U_Q^i$, it is enough to prove
\eqref{eq5.16a}. To prove \eqref{eq5.16a}, we first
note that
\begin{align}\label{distf+}
\delta(\tilde Z,\tilde \tau)\leq
\dist(\widehat Z,\widehat \tau,\tilde Z,\tilde \tau)\lesssim \eps^{2L/\alpha}\delta(X,t) \lesssim   \eps^{2L/\alpha}\eps^{-3}\ell(Q),
\end{align}
by \eqref{distf} (with $(X,t)$ in place of $(Y,s)$), and the fact that
by the construction of $\tilde U^i_Q$,
\begin{equation}\label{eq5.27a}
\eps^3\ell(Q)\lesssim \delta(X,t)
\lesssim \eps^{-3}\ell(Q),\quad \forall \, (X,t)\in\tilde U_Q^i.
\end{equation}
In addition, again by the construction of $\tilde U_Q^i$,
\begin{equation}\label{eq5.29a}
\diam(\tilde U_Q^i)\lesssim \eps^{-4}\ell(Q)\,.
\end{equation}

Let $C_Q^{**}:=C(X_Q,t_Q, \eps^{-5}\ell(Q))$. Using that $\ell(Q)\le \eps^{10}\,\ell(Q_0)$, we see that $\widehat u$ solves the adjoint heat equation in $C_Q^{**}\cap\Omega$.  Recall also that
$\widehat u$ is a normalized Green function, whose pole is a time forward corkscrew point relative to
$Q_0$ (see \eqref{normgreeb} and Lemma \ref{initialcorona1.lem}).
By construction,
$(\tilde Z,\tilde \tau)\in \frac12 C_Q^{**}$, so
using Corollary \ref{continuity2} and \eqref{distf+}, we deduce
\begin{align*}
\widehat u(\tilde Z,\tilde \tau) &\lesssim  \left(\frac{\eps^{2L/\alpha}\eps^{-3}\ell(Q)}{\eps^{-5}\ell(Q)}\right)^\alpha \bariiint_{C_Q^{**}}\, \widehat u(X,t)\,\d X\d t\lesssim \eps^{2L+2\alpha}  \bariiint_{C_Q^{**}}\, \widehat u(X,t)\,\d X\d t.
\end{align*}
Furthermore, using Lemma \ref{lemma:G-aver},
the fact that $Q\in \mathbb{D}^{*}_{\F,Q_0}$
(hence, $\ell(Q) \le \eps^{10} \ell(Q_0)$), and then
Lemma \ref {doublingintree.lem}/\eqref{muaveragebound} with $a=1$, we obtain
\begin{align*}
 \bariiint_{C_Q^{**}}\, \widehat u(X,t)\,\d X\d t\lesssim \eps^{-5}\ell(Q)\,\frac{\mu(\Delta(X_Q,t_Q,c\eps^{-5}\ell(Q)))}{\sigma(\Delta(X_Q,t_Q,c\eps^{-5}\ell(Q))}
 \lesssim \eps^{-5}\ell(Q)\,.
\end{align*}
% Using  \eqref{eq6.1} (or
% Lemma \ref{doublingintree.lem}), and the fact that $Q\in \mathbb{D}^{*}_{\F,Q_0}$
% (hence, $\ell(Q) \le \eps^{10} \ell(Q_0)$), % $ \leq K_0^{-1}\ell(Q_0)$,
% we have
% \begin{align*}
%  \frac{\mu(\Delta(X_Q,t_Q,c\eps^{-5}\ell(Q)))}{\sigma(\Delta(X_Q,t_Q,c\eps^{-5}\ell(Q))}\lesssim 1.
% \end{align*}
Hence, putting the estimates together we deduce
\begin{align*}
\widehat u(\tilde Z,\tilde \tau) &\lesssim \eps^{2L+2\alpha-5}\ell(Q) \leq  \eps^{2L -5}\ell(Q)\,.
\end{align*}
\end{proof}

By the definition of {Case 3} we have
\begin{equation*}
%\max_{Q'\in\mathbb{D}_\eps(Q)}
\sup_{(X,t)\in \tilde U^i_{Q}}\,\sup_{(Z,\tau)\in C_{X,t}} |\widehat u(Z,\tau)-\widehat u(X,t)-\nabla_X\widehat u(X,t)\cdot(Z-X)|\leq \eps^{2L}\ell(Q).
\end{equation*}
In particular, if $(Y,s)\in U^i_Q$, $(\widehat Z,\widehat \tau\,)\in \partial\widetilde C_{Y,s}\cap \Sigma$, and if $(\widehat Z,\widehat \tau)$ is mapped to $(\tilde Z,\tilde \tau)\in C_{Y,s}$ as above, then
\begin{equation}\label{eq6.15}
|\widehat u(\tilde Z,\tilde \tau) -\widehat u(Y,s) -\nabla_X \widehat u(Y,s) \cdot(\tilde Z-Y)|\leq \eps^{2 L}\ell(Q)\lesssim \eps^{2 L} \delta(Y,s),
\end{equation}
where we in the last step have used that  $l(Q)\lesssim K_0\,\delta(Y,s)$ by the construction of
$U_Q$ (see \eqref{eq2.1} and \eqref{eq2.3}).
Note that $\widehat u(\widehat Z,\widehat \tau\,)=0$, since $(\widehat Z,\widehat \tau\,)\in\Sigma$,
and that by Lemma \ref{refp+++},
\begin{align}\label{reffa}
|\nabla_X \widehat u(Y,s)|\lesssim 1.
\end{align}
We then observe that
\begin{multline*}
|\widehat u(\widehat Z,\widehat \tau\,)-\widehat u(Y,s) -\nabla_X \widehat u(Y,s) \cdot(\widehat Z-Y)|
% \\ \leq % | \widehat u(\widehat Z,\widehat \tau)-
% \widehat u(\tilde Z,\tilde \tau) % |
% \,+\,|\widehat u(\tilde Z,\tilde \tau)-\widehat u(Y,s) -\nabla_X \widehat u(Y,s) \cdot(\widehat Z-Y)|
\\ \leq \, \widehat u(\tilde Z,\tilde \tau)\,+\,
|\widehat u(\tilde Z,\tilde \tau)-\widehat u(Y,s) -\nabla_X \widehat u(Y,s) \cdot(\tilde Z-Y)|
 \,+\, |\nabla_X \widehat u(Y,s)|\, |\widehat Z-\tilde Z|\\
 \lesssim\,  \eps^{2L-5} \delta(Y,s)+ \eps^{2 L} \delta(Y,s)+\eps^{2 L/\alpha} \delta(Y,s)\lesssim
 \eps^{2 L-5} \delta(Y,s)\,,
\end{multline*}
where in the next-to-last step we have used Lemma \ref{l5.14},
\eqref{eq6.15}, \eqref{reffa}, and \eqref{distf}.

Thus, we have proved that if $(Y,s)\in U^i_Q$, $(\widehat Z,\widehat \tau\,)\in
\partial \widetilde C_{Y,s}\cap \Sigma$, then
 \begin{align}\label{eq6.17}
 |\widehat u(\widehat Z,\widehat \tau\,)-\widehat u(Y,s)
 -\nabla_X \widehat u(Y,s) \cdot(\widehat Z-Y)|\lesssim \eps^{2L-5} \delta(Y,s).
\end{align}

From now on, we let $(Y,s)$ be the reference point $(Y_Q^*,s_Q^*) \in U_Q^i$ introduced in Lemma \ref{refp++}, so that by
\eqref{eq6.5},
\begin{equation}\label{eq6.5+}
 |\nabla_X \widehat u(Y_Q^*,s_Q^*)| \approx 1.
\end{equation}
Using \eqref{eq6.5+}  and \eqref{eq5.2a} we deduce,  for $\eps\ll 1$  and $L$ large enough, that
\begin{equation}\label{eq6.5aa}
|\nabla_X \widehat u(Z,\tau)| \approx 1,\qquad \forall\, (Z,\tau)\in \tilde U^i_Q\,.
\end{equation}
Let $(\widehat Z,\widehat \tau)\in \partial\widetilde C_{Y_Q^*,s_Q^*}\cap \Sigma$.
After a possible translation in all coordinates, and a rotation in the spatial coordinates,
we may assume that $(\widehat Z,\widehat \tau)=(0,0)\in \partial\widetilde C_{Y_Q^*,s_Q^*}\cap \Sigma$, and that $Y_Q^*/|Y_Q^*|=(1,0,...,0)=:e_{\perp}\subset \rn$
(recall that when we wish to distinguish a particular ``vertical" direction in space,
we use the notation $X=(x_0,x_1,...,x_{n-1})$, with $x_0$ viewed as the distinguished variable).
In the following, we will use the notation
 $d_Q:=\delta(Y_Q^*,s_Q^*)$ and we note that by construction $Y_Q^*=d_Q e_{\perp}$. We first prove the following lemma.

\begin{lemma}\label{claim6.19}  We have
\begin{equation}\label{eq6.20}
\big|\, \nabla_X\widehat u(Y_Q^*,s_Q^*)\cdot e_{\perp} -|\nabla_X\widehat u(Y_Q^*,s_Q^*)|\,\big|\lesssim \eps^{2L-5}\,.
\end{equation}
\end{lemma}
\begin{proof} Applying \eqref{eq6.17}, with $\widehat u(\widehat Z,\widehat \tau\,) =0$ (since
$(\widehat Z,\widehat \tau\,)\in\Sigma)$, and with
$(\widehat Z,\widehat \tau\,)=(0,0)$, we obtain
\begin{align}\label{eqimp}
|\widehat u(Y_Q^*,s_Q^*) -\nabla_X \widehat u(Y_Q^*,s_Q^*)\cdot Y_Q^*|\lesssim
\eps^{2L-5}d_Q.
\end{align}
Consider $(Z,\tau)\in C_{Y_Q^*,s_Q^*}$.
By the definition of {Case 3} and the construction of $U^i_Q$ we have
\begin{equation*}
	 |\widehat u(Z ,\tau )-\widehat u(Y_Q^*,s_Q^*)-\nabla_X\widehat u(Y_Q^*,s_Q^*)\cdot(Z -Y_Q^*)|
	 \leq
	 \eps^{2L}\ell(Q)
	 \approx
	 \eps^{2L}\dist(Y_Q^*,s_Q^*)
	 =
	 \eps^{2L}d_Q.
\end{equation*}
This estimate and \eqref{eqimp} yield
\begin{equation}\label{eq6.21}
|\widehat u(Z ,\tau )  -\nabla_X \widehat u(Y_Q^*,s_Q^*)\cdot Z |
\,\lesssim \, \eps^{2L-5}d_Q \,,\qquad \forall \, (Z,\tau)\in C_{Y_Q^*,s_Q^*}\,.
\end{equation}

In particular, we now fix $(Z,\tau) = (Z^\ast ,\tau^\ast )
\in C_{Y_Q^*,s_Q^*}$, with $Z^\ast \in \partial B(Y_Q^*,(1-\eps^{2L/\alpha})d_Q)$
chosen so that
\begin{align}\label{Zastdef}
\nabla_X \widehat u(Y_Q^*,s_Q^*)\cdot\frac {(Z^\ast -Y_Q^*)}{|Z^\ast -Y_Q^*|}=-|\nabla_X \widehat u(Y_Q^*,s_Q^*)|.
\end{align}
That is, $Z^\ast $ is a point on $\partial B(Y_Q^*,(1-\eps^{2L/\alpha})d_Q)$ such that $(Z^\ast -Y_Q^*)/{|Z^\ast -Y_Q^*|}$ points in the direction opposite to $\nabla_X \widehat u(Y_Q^*,s_Q^*)/|\nabla_X \widehat u(Y_Q^*,s_Q^*)|$.
By \eqref{Zastdef}, % the construction of $Z^\ast $,
and the fact that $\widehat u\geq 0$,
\begin{multline*}
0\leq \,|\nabla_X \widehat u(Y_Q^*,s_Q^*)|-\nabla_X \widehat u(Y_Q^*,s_Q^*)\cdot e_{\perp}
\\
% \leq |\nabla_X \widehat u(Y_Q^*,s_Q^*)|-\nabla_X \widehat u(Y_Q^*,s_Q^*)\cdot e_{\perp}
% +\frac{\widehat u(Z^\ast ,\tau^\ast )}{d_Q} \\
\leq \,-\nabla_X \widehat u(Y_Q^*,s_Q^*)\cdot\frac {(Z^\ast -Y_Q^*)}{|Z^\ast -Y_Q^*|}\,-\,\nabla_X \widehat u(Y_Q^*,s_Q^*)\cdot e_{\perp} \,+\,\frac{\widehat u(Z^\ast ,\tau^\ast )}{d_Q}\,.
\end{multline*}
Using that
$|Z^\ast -Y_Q^*|= (1-\eps^{2L/\alpha})d_Q$ and $Y_Q^*=d_Q e_{\perp}$,  % and then re-ordering terms,
we rewrite the last estimate as
\begin{multline*}
0\leq \,|\nabla_X \widehat u(Y_Q^*,s_Q^*)|-\nabla_X \widehat u(Y_Q^*,s_Q^*)\cdot e_{\perp}
\\[4pt] % \leq |\nabla_X \widehat u(Y_Q^*,s_Q^*)|-\nabla_X \widehat u(Y_Q^*,s_Q^*)\cdot e_{\perp}
\leq \, \frac1{d_Q}\left(-\nabla_X \widehat u(Y_Q^*,s_Q^*)\cdot\frac {(Z^\ast -Y_Q^*)}{1-\eps^{2L/\alpha}}-\nabla_X\widehat u(Y_Q^*,s_Q^*)\cdot Y_Q^* +\widehat u(Z^\ast ,\tau^\ast )\right)
\\ \lesssim \, \eps^{2L-5} \,+\, \eps^{2L/\alpha}\, \lesssim \, \eps^{2L-5}\,,
\end{multline*}
by \eqref{eq6.21} and \eqref{eq6.5}.
\end{proof}

% Using Lemma \ref{claim6.19} we have
% \begin{equation}\label{ala}
% \big|\,|\nabla_X \widehat u(Y_Q^*,s_Q^*)|-\nabla_X \widehat u(Y_Q^*,s_Q^*)\cdot e_{\perp}\,\big| \lesssim
% \eps^{2L-5}\,.
% \end{equation}
Now set
$$
\nabla_\| \widehat u(Y_Q^*,s_Q^*):= \nabla_X \widehat u(Y_Q^*,s_Q^*) - (\nabla_X \widehat u(Y_Q^*,s_Q^*)\cdot e_{\perp}) e_{\perp}.
$$
Then
\begin{multline*}
|\nabla_\| \widehat u(Y_Q^*,s_Q^*)|^2
=
|\nabla_X\widehat u(Y_Q^*,s_Q^*)|^2 -\big(\nabla_X \widehat u(Y_Q^*,s_Q^*)\cdot e_{\perp}\big)^2\\
=
\big(|\nabla_X \widehat u(Y_Q^*,s_Q^*)|-\nabla_X \widehat u(Y_Q^*,s_Q^*)\cdot e_{\perp}\big)\,
\big(|\nabla_X \widehat u(Y_Q^*,s_Q^*)|+\nabla_X \widehat u(Y_Q^*,s_Q^*)\cdot e_{\perp}\big)
\lesssim
\eps^{2L-5},
\end{multline*}
by \eqref{eq6.20} and \eqref{eq6.5}.  % the fact that $|\nabla_X \widehat u(Y_Q^*,s_Q^*)|\lesssim 1$.
% In particular,
Combining the latter estimate with \eqref{eq6.20},
we see  that for $L>5/2$,
\begin{equation}\label{eq6.25}
\big|\, |\nabla_X \widehat u(Y_Q^*,s_Q^*)|e_{\perp}-\nabla_X \widehat u(Y_Q^*,s_Q^*)\,\big|
\lesssim
\eps^{L-\frac52}\,.
\end{equation}

With $(Y_Q^*,s_Q^*)=(d_Qe_{\perp},s_Q^*)\in U_{Q}^i$ still
fixed as above, we next consider an arbitrary  point
$(X,t)\in\tilde U_Q^i$. By definition of $\tilde U_Q^i$,
there is a polygonal path in $\tilde U^i_Q$, joining $(Y_Q^*,s_Q^*)$ to $(X,t)$, with vertices
$$(Y_0,s_0):=(Y_Q^*,s_Q^*), (Y_1,s_1),\dots,(Y_N,s_N):= (X,t),\qquad N\lesssim \eps^{-4},$$
such that %the total length of the path, as well as
$(Y_{k+1},s_{k+1})\in C_{Y_k,s_k}\cap C(Y_k,s_k,\ell(Q)),\, 0\leq k\leq N-1$, and such that
the distance between consecutive vertices is at most  $c\ell(Q)$. Indeed,
by definition of $\tilde U_Q^i$, we may connect $(Y_Q^*,s_Q^*)$ to $(X,t)$ by a polygonal
path connecting the centers of at most $\eps^{-1}$ Whitney cylinders, such that the distance between consecutive vertices % is between $\eps^3\,\ell(Q)/2$ and $\eps^{-3}\,\ell(Q)/2$.
is no larger than $\eps^{-3}\,\ell(Q)$.
If any such distance is greater than $\ell(Q)$, we
take at most $C\eps^{-3}$ intermediate vertices %(lying in the same segment)
with distances on the order of $\ell(Q)$. The total length of the path is thus
on the order of $N\ell(Q)$  with $N\lesssim \eps^{-4}$. Note that since
$(Y_{k+1},s_{k+1})\in C_{Y_k,s_k}$, with $(Y_k,s_k)\in \tilde U_Q^i$, by \eqref{eq5.3+a} we have
\begin{equation}\label{eq5.3+ak}
%\max_{Q'\in\mathbb{D}_\eps(Q)}
|\widehat u(Y_{k+1},s_{k+1})-\widehat u(Y_k,s_k)-\nabla_X\widehat u(Y_k,s_k)
\cdot (Y_{k+1}-Y_k)|\leq \eps^{2L}\ell(Q)\,.
\end{equation}
Note also that for any
$(Y_*,s_*)\in \tilde U^i_Q$ and $(Y,s)\in C_{Y_*,s_*}$, we obtain from  \eqref{eq5.2a} and \eqref{eq6.25} that
\begin{multline}\label{eq6.26}
\big|\nabla_X\widehat u(Y,s) -|\nabla_X \widehat u(Y_Q^*,s_Q^*)|e_{\perp}\big|
\\[4pt]
 \leq |\nabla_X \widehat u(Y,s)-\nabla_X \widehat u(Y_Q^*,s_Q^*)| \,
+ \,\big|\nabla_X \widehat u(Y_Q^*,s_Q^*) -|\nabla_X \widehat u(Y_Q^*,s_Q^*)|e_{\perp}\big|
\\[4pt] \lesssim\,
\eps^{2L}+\eps^{L-\frac52}
\lesssim \eps^{L-3}\,.
\end{multline}
% by \eqref{eq5.2a} and \eqref{eq6.25}.

To proceed we prove the following lemma.

\begin{lemma}\label{claim6.27}
Assume $L>11$.  Let $(X,t)\in \tilde U_Q^i$, and define $\{Y_k,s_k\}_{k=0}^N$ as above.
Then for each $k=0,1,2,\dots,N$,
\begin{equation}\label{eq6.28}
\big|\widehat u(Y_{k},s_k)-|\nabla_X \widehat u(Y_Q^*,s_Q^*)|(Y_k\cdot e_{\perp})\big|  \le   (1+k)\,\eps^{L-7}\ell(Q)\,.
\end{equation}
Furthermore,
\begin{equation}\label{eq6.28a}
\big|\widehat u(Z,\tau)-|\nabla_X \widehat u(Y_Q^*,s_Q^*)|(Z\cdot e_{\perp})\big|  \le  \eps^{L-11}\ell(Q),
\qquad \forall\, (Z,\tau)\in C_{X,t}\,.
\end{equation}
% whenever $(Z,\tau)\in C_{X,t}\,$,  for some $(X,t)\in \tilde U^i_Q$.
\end{lemma}

\begin{proof}
We claim that for every $0\le k\le N$, % and for every $(Z,\tau)\in C_{Y_k,s_k}$
\begin{align}\label{eq6.29--kk:induc}
	\big |\widehat u(Z,\tau)-|\nabla_X \widehat u(Y_Q^*,s_Q^*)|(Z\cdot e_{\perp})\big |
\le (1+k)\,\eps^{L-7}\ell(Q)\,,\quad \forall \, (Z,\tau)\in C_{Y_k,s_k}\,.
\end{align}	
Assuming \eqref{eq6.29--kk:induc} momentarily, we obtain
 \eqref{eq6.28} at once since
$(Y_0,s_0):=(Y_Q^*,s_Q^*)$, and $(Y_{k},s_{k})\in C_{Y_{k-1},s_{k-1}}$ for
$1\le k\le N$. Applying \eqref{eq6.29--kk:induc} with $k=N$, we deduce \eqref{eq6.28a},
since $(X,t)=(Y_N,s_N)$ and $N\lesssim \eps^{-4}$.

It therefore suffices to verify \eqref{eq6.29--kk:induc}.	
We proceed by induction.   To treat the case $k=0$,
since $(Y_0,s_0):=(Y_Q^*,s_Q^*)$,
 we suppose that $(Z,\tau)\in C_{Y_Q^*,s_Q^*}$.
Set $Z_\perp:=Z\cdot e_{\perp}$. We write
\begin{multline}\label{eq6.29--}
\big|\widehat u(Z,\tau)-|\nabla_X \widehat u(Y_Q^*,s_Q^*)|Z_{\perp}\big|
\leq |\widehat u(Z,\tau)-\nabla_X \widehat u(Y_Q^*,s_Q^*)\cdot Z|\\
+\big|(\nabla_X \widehat u(Y_Q^*,s_Q^*)-|\nabla_X \widehat u(Y_Q^*,s_Q^*)|e_{\perp})\cdot Z\big|.
\end{multline}
Recall that by translation and spatial rotation, we have in particular that
$Y_Q^* =d_Q e_\perp$, so for $(Z,\tau)\in C_{Y_Q^*,s_Q^*}$, we have $|Z|\lesssim d_Q$.
Thus, combining \eqref{eq6.21} and \eqref{eq6.25},
and using that $\eps$ is small, we obtain the desired estimate in the case $k=0$:
\begin{align*}%\label{eq6.29--}
\big|\widehat u(Z,\tau)-|\nabla_X \widehat u(Y_Q^*,s_Q^*)|Z_{\perp}\big|
\leq \eps^{2L-6}\,d_Q+\eps^{L-3}\,d_Q
\le \eps^{L-7}\ell(Q)\,,\quad (Z,\tau)\in C_{Y_0,s_0}\,.
\end{align*}

 Now let $1\leq k\leq N$, and suppose that  \eqref{eq6.29--kk:induc} holds for $k-1$.
Let $(Z,\tau)\in C_{Y_k,s_k}$. We set
 \[
 R(Z,\tau,Y_k,s_k):= \widehat u(Z,\tau)-\widehat u(Y_k,s_k)-\nabla_X\widehat u(Y_k,s_k)\cdot(Z-Y_k)\,,
 \]
 and note that $| R(Z,\tau,Y_k,s_k)| \leq \eps^{2L}\ell(Q)$
 by the definition of Case 3; see \eqref{eq5.3+a}.
Moreover, by construction, $(Y_k,s_k)\in C_{Y_{k-1},s_{k-1}}$, so
by the induction hypothesis
\begin{align*} % \label{eq6.29-++}
|\widehat u(Y_k,s_k)-|\nabla_X \widehat u(Y_Q^*,s_Q^*)|(Y_k\cdot e_{\perp})|\le k\eps^{L-7}\ell(Q).
\end{align*}
 Thus,
  \begin{multline} \label{eq6.29-}
\big |\widehat u(Z,\tau)-|\nabla_X \widehat u(Y_Q^*,s_Q^*)|Z_{\perp}\big|
\\[4pt]
\leq \, | R(Z,\tau,Y_k,s_k)|
\,+\,\big |\widehat u(Y_k,s_k)+\nabla_X\widehat u(Y_k,s_k)\cdot(Z-Y_k) -
|\nabla_X \widehat u(Y_Q^*,s_Q^*)|Z_{\perp}\big|
\\[4pt]
\leq \eps^{2L}\ell(Q) + k\eps^{L-7}\ell(Q)
+\big |\nabla_X\widehat u(Y_k,s_k)\cdot(Z-Y_k) -
|\nabla_X \widehat u(Y_Q^*,s_Q^*)|(Z-Y_k)\cdot e_{\perp}\big|\,.
\end{multline}
We now apply \eqref{eq6.26} with $(Y_*,s_*)=(Y_{k-1},s_{k-1})$, and
$(Y,s)= (Y_k,s_k)$, to obtain
\begin{equation}\label{eq6.29--kk}
\big |\nabla_X\widehat u(Y_k,s_k)\cdot(Z-Y_k) -
|\nabla_X \widehat u(Y_Q^*,s_Q^*)|(Z-Y_k)\cdot e_{\perp}\big|
\,\lesssim \, \eps^{L-3} |Z-Y_k| \lesssim \eps^{L-6}\ell(Q)\,,
\end{equation}
since for $(Z,\tau)\in C_{Y_k,s_k}$, we have $|Z-Y_k|<\delta(Y_k,s_k)\lesssim \eps^{-3}\ell(Q)$,
by \eqref{eq5.27a}.
Plugging \eqref{eq6.29--kk} into \eqref{eq6.29-}, we conclude that for all $(Z,\tau)\in C_{Y_k,s_k}$,
\begin{equation*}% \label{eq6.29--kk}
\big |\widehat u(Z,\tau)-|\nabla_X \widehat u(Y_Q^*,s_Q^*)|Z_{\perp}\big |
\le \eps^{2L}\ell(Q)+ k\eps^{L-7}\ell(Q)+ C\eps^{L-6}\ell(Q) \le (1+k)\eps^{L-7}\ell(Q)\,,
\end{equation*}
by choice of $\eps$ small enough.
\end{proof}

\begin{lemma}\label{claim6.32} Let $(X,t)\in \tilde U^i_Q$, and
let $(\widehat Z,\widehat\tau)\in \Sigma\cap \partial\widetilde C_{X,t}$.  Then
\begin{equation}\label{eq6.33}
|\nabla_X \widehat u(Y_Q^*,s_Q^*)|\, |\widehat Z_\perp|
\lesssim
\eps^{L/2}\ell(Q)\,.
\end{equation}
\end{lemma}
\begin{proof}
	 Given $(\widehat Z,\widehat\tau)$ we construct $(\tilde Z,\tilde\tau)\in \partial C_{X,t}$ as before,
satisfying (see \eqref{distf})
\begin{equation*}
\dist(\widehat Z,\widehat\tau,\tilde Z,\tilde\tau)\lesssim\eps^{2L/\alpha}\delta(X,t) \lesssim \eps^{(2L/\alpha) -3}\ell(Q)\,,
\end{equation*}
where in the last step we have used \eqref{eq5.27a}.

Consequently, using also \eqref{eq6.5}, \eqref{eq6.28a}, and \eqref{eq5.16a}, we have
\begin{multline*}
|\nabla_X \widehat u(Y_Q^*,s_Q^*)|\, |\widehat Z_\perp|
\leq
|\nabla_X \widehat u(Y_Q^*,s_Q^*)|\, |\widehat Z-\tilde Z|+
\big|\widehat u(\tilde Z,\tilde\tau)-
|\nabla_X \widehat u(Y_Q^*,s_Q^*)|\,|\tilde Z_\perp|\big|+\widehat u(\tilde Z,\tilde\tau)\\
\lesssim
\eps^{(2L/\alpha)-3}\ell(Q)+\eps^{L-11}\ell(Q)+\eps^{2L-5} \ell(Q).
\end{multline*}
For $L$ chosen large enough, we obtain \eqref{eq6.33}.
\end{proof}

\subsection{The Cubes in {Case 3} satisfy the $(\eps,K_0)$-local WHSA condition} \label{subq4}

From \eqref{eq6.33} and \eqref{eq6.5} we see that
\begin{equation}\label{eq5.35}
|\widehat Z_\perp| \lesssim  \eps^{L/2} \ell(Q),
\end{equation}
whenever $(\widehat Z,\widehat\tau)\in \Sigma\cap \partial\widetilde C_{X,t}$,
for some $(X,t)\in \tilde U^i_Q$. Recalling that $(X_Q,t_Q)$ is the ``center" of $Q$ (see
Remark \ref{remarkscube}), we set
\begin{equation}\label{eq5.41a}
\mathcal{O}:= C(X_Q,t_Q,2\eps^{-5/2}\ell(Q)) \cap \left\{ (Z,\tau): Z_\perp> \eps^2 \ell(Q)\right\}.
\end{equation}

\begin{lemma}\label{claim5.40} For every point  $(X,t)=(x_0,x,t)\in \mathcal{O}$, we have
$(X,t)\sim_{\eps,Q} (Y_Q^*,s_Q^*)$ (see Definition \ref{def2.11a}).
Thus, in particular, $\mathcal{O}\subset \tilde U^i_Q.$
\end{lemma}
\begin{proof} For arbitrary $(X,t)\in \mathcal{O}$,
we need to connect $(X,t)$ to $(Y_Q^*,s_Q^*)$ by a chain of
at most $\eps^{-1}$ Whitney cylinders $C(Y^k,s^k,\delta(Y^k,s^k)/2)$, with
$\eps^3\ell(Q)\leq\delta(Y^k,s^k)\leq \eps^{-3}\ell(Q)$. For convenience,
we will simply refer to such cylinders as
admissible.   Recalling that $Y_Q^* =d_Q e_\perp$, with $d_Q:= \delta(Y_Q^*,s_Q^*)$,
we first observe that
if $X = \eta e_{\perp}$, with $\eps^{3} \ell(Q) \leq \eta\leq \eps^{-3}\ell(Q)$,
then we may join $(X,s_Q^*)$ to $(Y_Q^*,s_Q^*)$ by at most $C\log(1/\eps)$ admissible cylinders.
Indeed, for $\eta \in [\eps^{3} \ell(Q),d_Q]$, this fact holds trivially, and for
$\eta \in (d_Q, \eps^{-3}\ell(Q)]$, it holds
by an iteration argument using \eqref{eq5.35}
(with $L$ chosen large enough).
  The point $((2\eps)^{-3} \ell(Q) e_{\perp},s_Q^*)$ may then be joined to
any point of the form $((2\eps)^{-3} \ell(Q),x,t)$
by a chain of at most $C$ admissible cylinders, whenever
$(x,t)\in \mathbb R^{n-1}\times \mathbb{R}$ with $|x| +|t-s_Q^*|^{1/2}\leq (2\eps)^{-3}\ell(Q)$.
In turn, the latter point, with $(x,t)$ now fixed, may then be joined
to $(x_0,x,t)$, for any $x_0 \in [\eps^{3} \ell(Q), (2\eps)^{-3} \ell(Q)]$,
 by at most $C\log(1/\eps)$ admissible cylinders.
\end{proof}

Lemma \ref{claim5.40} implies that
\begin{equation}\label{eq5.41}
 \Sigma\cap \mathcal{O} =\emptyset\,.
\end{equation}
Indeed, $\mathcal{O}\subset \tilde U^i_Q\subset \Omega$. We introduce the hyperplane
$$P_0:=\left\{ (Z,\tau):  Z_\perp=0\right\}.$$

To prove the next lemma, we shall make use of the following geometric observation.
\begin{remark}\label{remarkgeo}
Note that if $(X,t)\in \overline{C(Z,\tau,r)}$, then
$$ % d  \leq
\dist(X,t,Z,\tau\,)
= |X-Z| + |t-\tau|^{1/2}\leq 2r\,.
$$
Recall that $\delta(X,t)$ is the smallest value of $d$ such
that $\Sigma$ meets
$\overline{C(X,t,d)}$.  In particular,
if $\delta(X,t)=d$, and if
$(\hat{X},\hat{t}\,)\in \Sigma\cap\overline{C(X,t,d)}$
is a point that realizes $\delta(X,t)$, then
$d \leq \dist(X,t,\Sigma) \leq \dist(X,t,\hat{X},\hat{t}\,) \leq 2d$.
Combining these observations, we see that if
$(X,t)\in \overline{C(Z,\tau,r)}$, and $\delta(X,t) = d$, then
$\dist(Z,\tau,\Sigma) \leq 2(r + d)$.
\end{remark}

\begin{lemma}\label{claim5.42}
If $(Z,\tau)\in P_0$, with
$\dist (Z,\tau,X_Q,t_Q)\leq \,\eps^{-5/2}\ell(Q)$,  then
\begin{equation}\label{eq5.43}
\dist(Z,\tau, \Sigma) \leq \,32\, \eps^2\ell(Q)\,.
\end{equation}
\end{lemma}
\begin{proof}
Observe that $C(Z,\tau, 2 \eps^2\ell(Q))$ meets $\mathcal{O}$. Hence,
by Lemma \ref{claim5.40},
there is a point $(X,t)$ in $\tilde U^i_Q \cap C(Z,\tau, 2 \eps^2\ell(Q))$.
Suppose now that \eqref{eq5.43} is false: then Remark \ref{remarkgeo}
implies that $\delta(X,t) \geq14\eps^2\ell(Q).$
Consequently, $C(Z,\tau, 4 \eps^2\ell(Q))\subset C_{X,t}$, and thus by \eqref{eq6.28a}, we have
\begin{equation}\label{eq5.44}
\big|\widehat u(\tilde Z,\tilde \tau)-|\nabla_X \widehat u(Y_Q^*,s_Q^*)|\tilde Z_{\perp}\big|
\lesssim \eps^{L-11}\ell(Q),\quad \forall \, (\tilde Z,\tilde \tau)\in
C(Z,\tau, 4 \eps^2\ell(Q))\,.
\end{equation}
In particular, since $Z_{\perp}=0$, we may
choose $(\tilde Z,\tilde \tau)$ such that $\tilde Z_{\perp} = -\eps^2\ell(Q)$,
to obtain that
$$|\nabla_X\widehat u(Y_Q^*,s_Q^*)|\, \eps^2\ell(Q) \lesssim \eps^{L-11}\ell(Q),$$
since $\widehat u\geq 0$.  But for $\eps<1/2$, and $L$ large enough, this contradicts \eqref{eq6.5}.
\end{proof}

We conclude that $Q$ satisfies the $(\eps^{9/8},K_0)$-local WHSA
condition  in the sense of Definition \ref{def2.13}, with respect to
\begin{equation}\label{eq5.44++}P=\hat P_Q:= \{(Z,\tau):\,Z_{\perp} = \eps^2\ell(Q)\}\},\quad
H=\hat H_Q:= \{(Z,\tau): \, Z_{\perp}>\eps^2\ell(Q)\} \,.
\end{equation}
In particular, $(ii)$ in Definition \ref{def2.13} holds by construction of $\hat P_Q$ and
Remark \ref{cnchoic.rmk}.

Note further that $(Y_Q^*,s_Q^*) \in \hat H_Q$.  Indeed,
$Y_Q^*=d_Q e_\perp$, with
$d_Q=\delta(Y_Q^*,s_Q^*)$, and since
$(Y_Q^*,s_Q^*)\in U_Q$, we therefore have  $d_Q\gtrsim K_0^{-1} \ell(Q) \gg \eps\ell(Q)$.

\begin{remark}\label{nablaXnormal}  Choosing $L$ large enough, and using \eqref{eq6.5} and
\eqref{eq6.25}, we see that the (purely spatial) $n$-dimensional
vector $\nabla_X \widehat{u}(Y_Q^*,s_Q^*)$  lies within
a (spatial) vertical upward cone of aperture $\eps^{100}$, with central axis $e_\perp$.
Now let $ P_Q$ and $ H_Q$ denote the images of $\hat P_Q$ and $\hat H_Q$ after
a rotation (about the origin) in the spatial variables only,
by an angle of at most $\eps^{100}$, such that the space-time vector
$(\nabla_X \widehat{u}(Y_Q^*,s_Q^*),0)$
is orthogonal to $ P_Q$, and points into $ H_Q$.
We then have that $Q$ satisfies the $(\eps,K_0)$-local WHSA
with respect to $ P_Q$ and $ H_Q$, and furthermore, $(Y_Q^*,s_Q^*) \in  H_Q$.
\end{remark}

We summarize our findings in the following proposition.
\begin{proposition}\label{WHSAplorient.lem}
Let $\sbf'$ be a stopping time tree
from Lemma \ref{refp+}, with $\xi \in (0,1)$ fixed so that
Lemma \ref{refp++} holds. Then every Case 3 cube $Q\in \sbf'$
satisfies the $(\eps,K_0)$-local WHSA (see Definition \ref{def2.13}).  Moreover,
the associated plane $P_Q$ and half-space $H_Q$ may be chosen
to have the additional property that the unit normal vector
to $P_Q$, pointing into $H_Q$, is given by the space-time vector
$(\nabla_X \widehat{u}(Y_Q^*,s_Q^*),0)/|\nabla_X \widehat{u}(Y_Q^*,s_Q^*)|$.
Furthermore, $(Y_Q^*,s_Q^*)\in H_Q$.
\end{proposition}
% This concludes the proof of the fact that every cube in $Q\in \sbf$, which falls into Case 3,
% satisfies the $(\eps,K_0)$-local WHSA.
Thus, using the coronization in Lemma \ref{initialcorona1.lem}, Theorem \ref{WHSA.thrm} is proved.

\section{Proof of Theorem \ref{main.thrm}: preliminaries}\label{Sec7}

The rest of the paper is devoted to the proof of Theorem \ref{main.thrm}, with
Theorem \ref{WHSA.thrm} (more precisely Proposition \ref{WHSAplorient.lem}),
and the associated circle of ideas, taken
as the starting point.  Therefore, in the following we will
consistently assume the hypotheses of  Theorem \ref{main.thrm}. The proof will rely heavily on Section \ref{Sec5} and Section \ref{Sec6}, and in the following we consequently let $\eps\leq\eps_0=K_0^{-100}$. All implicit
constants will depend at most on the allowable parameters, i.e., on $n,M,\gamma,p$ and $C$,  and on $K_0$, unless otherwise stated.

By Lemma \ref{refp+} we have that $\dd = \dd(\Sigma)$ can be written as a disjoint decomposition $\dd=\cG' \cup \cB'$, with the collection $\cG'$ being further decomposed into semi-coherent stopping time trees $\{\sbf'\}$, and such that the maximal cubes $\{Q(\sbf')\}$, and the cubes in $\cB'$, pack, i.e.,
\begin{align}\label{Sec71}
&\sum_{Q(\sbf'): Q(\sbf') \subseteq R } \sigma(Q(\sbf'))
+\sum_{Q \in \cB': Q \subseteq R } \sigma(Q) \lesssim \sigma(R),
\quad \forall R \in \dd.
\end{align}

The following proposition is the key to the proof
of Theorem \ref{main.thrm}.
\begin{proposition}\label{finalprop} Let $\sbf'$ be a (single) stopping time tree from
Lemma \ref{refp+}. Then $\sbf'$ can be divided into a collection of disjoint semi-coherent stopping time trees $\{\sbf^*\}$,  and a set of  cubes $ \B_{\sbf'}$, such that the following holds. First, the maximal cubes $\{Q(\sbf^*)\}$ and the cubes in
$\B_{\sbf'}$ satisfy a packing condition,
\begin{equation}\label{graphclosepack}
\sum_{Q(\sbf^*): Q(\sbf^*) \subseteq R} \sigma(Q(\sbf^*)) + \sum_{Q' \in \B_{\sbf'}: Q \subseteq R} \sigma(Q) \lesssim \sigma(R) \quad \forall R \in \dd.
\end{equation}
Second, there are constants $\eta > 0, K  > 1$,  such that for each $\sbf^*$ there
exists a {\bf regular} Lip(1,1/2) graph $\Gamma_{\sbf^*}$ in such  a way that if
$Q\in {\sbf^*}$, then
\begin{equation}\label{graphclose}
\sup_{(X,t) \in KQ} \dist(X,t,\Gamma_{{\sbf^*}} ) \leq \eta\,\diam(Q)\,.
\end{equation}
 Here the constants of the regular Lip(1,1/2) graphs, and the constants $\eta$ and $K$, depend only on the allowable parameters and $K_0$.
 Moreover, the Lip(1,1/2) constant $b_1$ in
 Subsection \ref{ssRLip} can be made small, so that in particular $b_1<\eta$.
\end{proposition}

Taking Proposition \ref{finalprop} for granted momentarily, we may deduce
Theorem \ref{main.thrm} in short order.

\begin{proof}[Proof of Theorem \ref{main.thrm}]
By Theorem \ref{coronaisUR.thrm}, it is enough to
show that $\Sigma$ has a semi-coherent
$(\eta,K)$-Corona decomposition in the sense of
Definition \ref{semiCorona}, by regular Lip(1,1/2) graphs with constant
$b_1\leq\eta$.  Observe that Proposition \ref{finalprop} says that
this is indeed true locally,
within any given $\sbf'$ from
Lemma \ref{refp+}, so that in particular
\eqref{graphclosepack} holds with $R\subset Q(\sbf')$.
We can then ``globalize" the packing condition
\eqref{graphclosepack} by using the packing condition for
the maximal cubes $Q(\sbf')$, and for the bad cubes (the collection $\cB'$) from
Lemma \ref{refp+}.
We omit the routine details.
% \begin{remark}\label{TH1.1conclude}
% Let us note that
% \end{remark}
\end{proof}

Thus, it remains to prove Proposition \ref{finalprop}.

\begin{remark}\label{prop8.2proofoutline}
The proof of Proposition \ref{finalprop} will proceed in four stages.  The preliminary stage, to be carried out in the present section, is
to decompose $\sbf'$ as described, via a stopping time procedure, and to
reduce the packing estimate \eqref{graphclosepack}
to a certain key lemma
(Lemma \ref{keylemma1}).
In Section \ref{sec: graph}, we then
construct, for each $\sbf^*$,
a Lip(1,1/2) graph $\Gamma_{\sbf^*}$ with small Lip(1,1/2) constant,
satisfying \eqref{graphclose}.  The third stage, to be carried out in
Section \ref{sec: reggraph}, is to show that this
graph is regular.  Finally, in Section \ref{sec:pack}, we prove
Lemma \ref{keylemma1}, thus completing the proof of the packing
condition \eqref{graphclosepack}, and hence also the proof of Proposition
\ref{finalprop}.
\end{remark}

In the sequel we fix $\sbf'$ as in Lemma \ref{refp+} (a `caloric measure tree'),
with the relevant parameters chosen small enough that Lemma \ref{refp++} holds.
Our goal is produce the decomposition in Proposition \ref{finalprop}.
We first observe that $\sbf'\subset \sbf=\dd_{\F,Q_0}$ for some stopping-time
family $\F$ with $Q_0=Q(\sbf)$ (we mention that this property of $\sbf$
follows from the constructions in Lemma \ref{initialcorona1.lem},
or just from the fact that $\sbf$ is semi-coherent).
Recall that we had also defined
\[ \mathbb{D}_{\F,Q_0}^*=\{Q\in \mathbb{D}_{\F,Q_0}: \ell(Q)\le \eps^{10} \,\ell(Q_0)\}\]
and we know that the cubes in $\B_1 := \sbf' \setminus \mathbb{D}_{\F,Q_0}^*$ satisfy a packing condition
(these were the ``Case 0" cubes treated in Subsection \ref{sspart}). Therefore, we define $\sbf'' := \sbf'\cap \mathbb{D}_{\F,Q_0}^*$ and to prove Proposition \ref{finalprop} we can work with $\sbf''$ in place of $\sbf'$.
The collection $\sbf''$ is not necessarily a semi-coherent stopping time tree (it is a finite, disjoint union of such), but this fact will be harmless.

Recall that for each $Q\in\sbf'$, there exists $(Y^*_Q, s_Q^*)\in U_Q$ satisfying
\eqref{Y*Qest}, such that the normalized Green function satisfies the non-degeneracy condition
\begin{equation}\label{nondegenstoppack.eq}
a_\star^{-1} \le |\nabla_X\widehat u(Y^*_Q, s_Q^*)| \le a_\star,
\end{equation}
where $a_\star > 1$ is a constant depending only the allowable parameters and $K_0$, but not on $\eps$ (see Lemma \ref{refp++},
Lemma \ref{refp+++}, and \eqref{eq6.5};  in fact, by Lemma \ref{refp++},
the lower bound can be taken to be independent of
$K_0$ as well, but this is not really needed for our purposes).

In order to prove Proposition \ref{finalprop}, we shall need to refine Proposition \ref{WHSAplorient.lem} slightly,
in order to give ourselves more flexibility in the choice of the reference point $(Y^*_Q, s_Q^*)$.
In principle, a single cube could have two WHSA planes whose half-spaces point in (roughly)
opposite directions, and we would like to allow ourselves the choice of either of those in our construction.

Based on \eqref{nondegenstoppack.eq},
we make the following definition,
recalling that we have enumerated the components of $U_Q$ as $U_Q = \cup_j U_Q^j$.
\begin{definition}[Indices with non-degenerate gradient]
Given $Q\in\sbf'$, we let $\mathcal{J}_Q$ be the set of indices  $j$  such that  component $U_Q^j$ contains a point $$(\tilde{Y}_Q^{*,j},\tilde{s}_Q^{*,j}) \in U_Q,$$ satisfying
\begin{equation}\label{nondegenweakened.eq}
  (a_\star^{-1}/2)\leq |\nabla_X \widehat u(\tilde{Y}_Q^{*,j},\tilde{s}_Q^{*,j})| \leq 2a_\star.
\end{equation}
\end{definition}
Note that by taking $\eps$ sufficiently small, we can ensure that $\eps^{2L} \ll a_\star/200$, and by construction $\mathcal{J}_Q\neq\emptyset$ (see the discussion before \eqref{nondegenstoppack.eq}). One further definition is required to state our refined version of Proposition \ref{WHSAplorient.lem}.

\begin{definition}[$\eps^{2L}$-blue and $\eps^{2L}$-red cubes]\label{defosc}
Let $Q \in \sbf''$.
We say that $Q$ is an $\eps^{2L}$-blue cube if for all $j\in \mathcal{J}_Q$ and every point $(\tilde{Y}_Q^{*,j},\tilde{s}_Q^{*,j})$ in $U^j_{Q}$ for which \eqref{nondegenweakened.eq} holds the following estimates hold
\begin{equation}\label{smalloscubdef1.eq}
\sup_{(X,t)\in \tilde U^j_{Q}}\,\sup_{(Z,\tau)\in C_{X,t}} |\nabla_X \widehat u(Z,\tau)
 -\nabla_X \widehat u(\tilde{Y}_Q^{*,j},\tilde{s}_Q^{*,j})| \leq  \eps^{2L},
\end{equation}
and
\begin{equation}\label{smalloscubdef2.eq}
\sup_{(X,t)\in \tilde U^j_{Q}}\,\sup_{(Z,\tau)\in C_{X,t}} |\widehat u(Z,\tau)-\widehat u(X,t)-\nabla_X\widehat u(X,t)\cdot(Z-X)|\leq\eps^{2L}\ell(Q).
\end{equation}
If $Q$ is not $\eps^{2L}$-blue we say $Q$ is $\eps^{2L}$-red.
\end{definition}

Our refinement of Proposition \ref{WHSAplorient.lem} is as follows.
\begin{proposition}\label{WHSAplorientrefine.lem}
The $\eps^{2L}$-red cubes in $\sbf''$ satisfy a packing condition. If $Q$ is an
$\eps^{2L}$-blue cube, then the $(\eps,K_0)$-local WHSA holds for $Q$ (see Definition \ref{def2.13}). Moreover, for all $\eps^{2L}$-blue cubes and for each $j\in \mathcal{J}_Q$ the WHSA hyperplane $P_Q$ and half-space $H_Q$
can be chosen so that the unit normal to $P_Q$, pointing into $H_Q$, is given by the space-time vector
$ (\nabla \widehat u(\tilde{Y}_Q^{*,j},\tilde{s}_Q^{*,j}),0)/|\nabla \widehat u(\tilde{Y}_Q^{*,j},\tilde{s}_Q^{*,j})| $,
where $(\tilde{Y}_Q^{*,j},\tilde{s}_Q^{*,j})$ is {\bf any} point in $U^j_{Q}$ for which \eqref{nondegenweakened.eq} holds.  In addition, $(\tilde{Y}_Q^{*,j},\tilde{s}_Q^{*,j})\in H_Q$.
\end{proposition}

The proof is essentially the same as that of Proposition \ref{WHSAplorient.lem}.
Indeed, one may simply replace
$U_Q^i$, and the reference point $(Y_Q^{*},s_Q^{*})\in U_Q^i$, by any
$U_Q^j$, and  $(\tilde{Y}_Q^{*,j},\tilde{s}_Q^{*,j})\in U_Q^j$ such that \eqref{nondegenweakened.eq} holds, and then repeat the same argument.  The only difference is that we now have \eqref{nondegenweakened.eq}, rather than
\eqref{nondegenstoppack.eq}, but the slightly worse constants in \eqref{nondegenweakened.eq}
are harmless.  In particular, the $\eps^{2L}$-red cubes fall into either Case 1
(if \eqref{smalloscubdef1.eq} fails) or Case 2 (if \eqref{smalloscubdef1.eq} holds but
\eqref{smalloscubdef2.eq} fails).

 We define
\[
\B_2 := \{Q \in \sbf'' : Q \text{ is $\eps^{2L}$-red}\}, \qquad \B_{\sbf'}: = \B_1 \cup \B_2,\qquad
\sbf''' :=  \sbf'' \setminus\B_{\sbf'}=\sbf' \setminus\B_{\sbf'}\,.
\]
We already have that $\B_{\sbf'}$ satisfies a packing condition;
therefore, to prove Proposition \ref{finalprop}, it is enough to
divide $\sbf'''$ into semi-coherent stopping time trees with
maximal cubes that pack, and to
produce the associated graph in each stopping time.

We now describe the construction of the stopping times. If $\sbf'''$ is empty,
then Proposition \ref{finalprop} holds vacuously, so we may assume that this
is not the case. To construct the first stopping time tree $\sbf^*_1$, let
$Q_1\in  \sbf''$ be a maximal (with respect to generation) $\eps^{2L}$-blue cube,
and let $(Y^*_{Q_1}, s^*_{Q_1})$ be a point satisfying \eqref{nondegenstoppack.eq}
and \eqref{Y*Qest} (with $Q_1$ in place of $Q$).
We subdivide $Q_1$ dyadically,
and continue the subdivision until we encounter a cube $Q$ such that
any of the following occur.

\begin{condition}\label{stopping} {\bf (Stopping time conditions).}

\begin{enumerate}
%\begin{list}{$(\theenumi)$}{\usecounter{enumi}\leftmargin=1cm \labelwidth=1cm
% \itemsep=0.2cm \topsep=.2cm \renewcommand{\theenumi}{\alph{enumi}}}
\item  $Q\not\in \sbf'$;
\smallskip
\item  $Q\in \sbf'$ is $\eps^{2L}$-red;
\smallskip
\item  $Q\in\sbf'$ is $\eps^{2L}$-blue, but
\begin{equation}\label{stoppc.eq}
|\nabla_X \widehat u(Y,s) - \nabla_X\widehat u(Y^*_{Q_1}, s^*_{Q_1})| > 10\eps^{2L}, \quad \forall (Y,s) \in U_Q.
\end{equation}
% \end{list}
\end{enumerate}
\end{condition}

The cubes on which we stop comprise a pairwise-disjoint family $\F_{Q_1}^*$. We then introduce $\sbf_1^*:=\dd_{\F_{Q_1}^*,Q_1}$, a
semi-coherent stopping tree with $Q(\sbf^*_1):=Q_1$.
In particular, none of the three conditions $(a)$--$(c)$ hold for $Q \in \sbf^*_1$, by construction.
Next,  let $Q_2\in \sbf'''\setminus \sbf^*_1$ be a
maximal (by generation) $\eps^{2L}$-blue cube,
and let $(Y^*_{Q_2}, s^*_{Q_2})$ be a point satisfying  \eqref{nondegenstoppack.eq}
and \eqref{Y*Qest} (with $Q_2$ in place of $Q$).  Based on this we build the second stopping time tree $\sbf^*_2$ as before: we subdivide $Q_2$ dyadically,
and stop the first time we encounter a cube $Q$ for which at least one of $(a)$--$(c)$ holds,
with $Q_1$ replaced by $Q_2$ in $(c)$/\eqref{stoppc.eq}.
As a result of this process we obtain again a
pairwise-disjoint stopping time family $\F_{Q_2}^*$,
and we introduce $\sbf_2^*:=\dd_{\F_{Q_2}^*,Q_2}$.
Note that $\sbf_2^*$ is a semi-coherent stopping tree with $Q(\sbf^*_2)= Q_2$.
Continuing in this way we get a sequence of
 semi-coherent stopping time trees
$\{\sbf^*_k\}_k$, such that
$\sbf^*_{k+1}\subset \sbf'''\setminus \cup_{j=1}^k\sbf^*_{j} $
for $k\ge 1$ and $\sbf'''=\cup_{k\geq 1}  \sbf^*_{k}$.

\begin{remark}
We carefully note that we have chosen our reference points $(Y^*_{Q(\sbf^*_k)}, s^*_{Q(\sbf^*_k)})$ so that they satisfy \eqref{nondegenstoppack.eq} (and not the weaker condition \eqref{nondegenweakened.eq}), as well as \eqref{Y*Qest}.
\end{remark}

Notice that our construction is a `corona within a corona'. It is therefore useful to think of the stopping time trees $\sbf'$ as {\bf caloric measure trees}, which we have further divided up into `bad cubes' $\B_{\sbf'}$
and {\bf graph trees} $\{\sbf^*_k\}$,
so that $\sbf' = \B_{\sbf'} \cup(\cup_k\sbf_k^*)$ and all these families are pairwise disjoint. Our goal below is to construct, for each graph tree $\sbf^*_k$, a regular graph $\Gamma_{\sbf^*_k}$ that is close to $\Sigma$ in the sense that \eqref{eq2.2a} holds and
such that the maximal cubes $\{Q(\sbf^*_k)\}$ satisfy a packing condition.

Let us analyze the maximal cubes  $\{Q(\sbf^*_k)\}$ further.  Given
a particular tree $\sbf^* = \sbf^*_k$ for some fixed $k$, we decompose
its maximal cube $Q(\sbf^*)$ as follows.  Let $\F^*=\F^*(\sbf^*)$
denote the collection of
stopping time cubes, that is, the sub-cubes of $Q(\sbf^*)$ satisfying one of the stopping time conditions % $(a)$, $(b)$, or $(c)$
(1), (2), or (3) in Condition \ref{stopping} above, and set
\begin{equation}\label{F0def}
F^0_{\sbf^*}:= Q(\sbf^*)\setminus \left(\cup_{Q_j\in \F^*}\,Q_j\right)\,.
\end{equation}
Let $\F_1^*=\F_1^*(\sbf^*)$, $\F_2^*=\F_2^*(\sbf^*)$, and
$\F_3^*=\F_3^*(\sbf^*)$ denote
the disjoint subcollections of $\F^*$
such that conditions % $(a)$, $(b)$, and $(c)$
(1), (2), and (3) hold (respectively), so that
\begin{equation}\label{Fcalidef}
\F^*= \F_1^* \cup \F_2^*\cup \F_3^*\,,
\end{equation}
and set
\begin{equation}\label{Upidef}
\Upsilon^i_{\sbf^*}:= \cup_{Q\in \F_i^*}\,Q\,,\quad i=1,2,3\,.
\end{equation}
We then decompose $Q(\sbf^*)$ as follows:
\begin{equation}\label{QS*decomp}
Q(\sbf^*)= % E_{\sbf^*}\cup \left(\cup_{Q_j\in \F_2^*}\,Q_j\right)\cup
% \left(\cup_{Q_j\in \F_3^*}\,Q_j\right)  =:
\left(F^0_{\sbf^*} \cup \Upsilon^1_{\sbf^*}\right) \cup
\Upsilon^2_{\sbf^*}
\cup \Upsilon^3_{\sbf^*} =:
E_{\sbf^*}\cup \Upsilon^2_{\sbf^*}
\cup \Upsilon^3_{\sbf^*} \,,
\end{equation}
i.e.,  $E_{\sbf^*}:=  % \left(\cup_{Q_j\in \F_1^*}\,Q_j\right)
F^0_{\sbf^*} \cup \Upsilon^1_{\sbf^*}$.

In Section \ref{sec:pack}, we shall prove the following key lemma.

\begin{lemma}\label{keylemma1}
 There are uniform constants
$\zeta_0 \in (0,1)$, and $M_0\in (1,\infty)$,
depending on $K_0, \eps$ and the allowable parameters,
such that for every
graph tree $\sbf^*=\sbf^*_k$ as above, there are
disjoint decompositions $\F_3^*:=\F_{3,good}^*\cup \F_{3,bad}^*\,$, and
 $\Upsilon^3_{\sbf^*} := \Upsilon^3_{\sbf^*,\, good}
\cup \Upsilon^3_{\sbf^*,\, bad}\,$  satisfying
\begin{equation}\label{Up3compample}
\sigma\big(Q(\sbf^*)\setminus \Upsilon^3_{\sbf^*,\, bad} \big) \,\geq \,
\zeta_0 \,\sigma\big(Q(\sbf^*)\big)\,,
\end{equation}
and such that, for each caloric measure tree $\sbf'$, we have the packing condition
 \begin{equation}\label{Fp3pack}
 \sum_{\sbf^*:\,\sbf^*\subset \sbf' \,,\,Q(\sbf^*)\subset R }\,\,
 \sum_{Q\in \F_{3,good}^*(\sbf^*)} \sigma(Q) \,
 \leq \, M_0\, \sigma(R)\,,\qquad \forall \, R\in \dd\big(Q(\sbf')\big).
 \end{equation}
% \begin{equation}\label{Fp3pack}
% \sum_{\sbf^*:\,\sbf^*\subset \sbf' }\,\,
% \sum_{Q\in \F_{3,good}^*(\sbf^*)\,,\, Q\subset R} \sigma(Q) \,
% \leq \, M_0\, \sigma(R)\,,\qquad \forall \, R\in \dd\big(Q(\sbf')\big).
% \end{equation}
Here, $\Upsilon^3_{\sbf^*,\,good}:= \cup_{Q\in \F_{3,good}^*(\sbf^*)}\,Q$, and
$\Upsilon^3_{\sbf^*,\,bad}:= \cup_{Q\in \F_{3,bad}^*(\sbf^*)}\,Q$,
and the first sum in
\eqref{Fp3pack} runs over graph trees $\sbf^*$ contained in the given caloric measure tree $\sbf'$.
\end{lemma}

We defer the proof of Lemma \ref{keylemma1} until Section \ref{sec:pack}.
The proof will rely in part on material to be developed in
Sections \ref{sec: graph} and \ref{sec: reggraph}.

Let us for now take Lemma \ref{keylemma1} for granted, and use it to prove
the packing condition \eqref{graphclosepack}.
Recall that in Lemma \ref{refp+},
we have already established packing of the maximal cubes
$Q(\sbf')$ of the caloric measure trees.  Thus, it is enough to verify \eqref{graphclosepack}
% pack the maximal cubes of the graph trees $\sbf^*$ lying
{\em within} a given caloric measure tree
$\sbf'$.  We therefore fix such an $\sbf'$, and note that we have shown that
$\sbf'$ has the disjoint decomposition $\sbf'= (\cup_k \sbf^*_k)\cup \B_{\sbf'}$,
where the collection $\B_{\sbf'}$ satisfies a packing condition.
Consider now any given $\sbf^*=\sbf^*_k$ in this decomposition.
Recalling \eqref{QS*decomp}, we see
that the set
$Q(\sbf^*)\setminus \Upsilon^3_{\sbf^*,\, bad}$ may be decomposed as
\[
Q(\sbf^*)\setminus \Upsilon^3_{\sbf^*,\, bad} =
E_{\sbf^*}\cup \Upsilon^2_{\sbf^*}
\cup \Upsilon^3_{\sbf^*,\, good} \,,
\]
hence, by \eqref{Up3compample}, it follows that at least one of the
following bounds holds:
\[
\sigma(E_{\sbf^*})\, \ge \,(\zeta_0/3)\, \sigma(Q(\sbf^*)), \leqno(i)
\]
\[
\sigma\left(\Upsilon^2_{\sbf^*} \right) \,\ge \,(\zeta_0/3) \,\sigma(Q(\sbf^*)) \quad \text { or } \leqno(ii)
\]
\[
\sigma\left(\Upsilon^3_{\sbf^*,\, good}\right) \,\ge\, (\zeta_0/3)
\sigma(Q(\sbf^*))\,. \leqno(iii)
\]
We may now deduce the packing condition \eqref{graphclosepack} as follows.  As we have already observed, the bad cubes $\B_{\sbf'}$ pack.
% Let $\G_{\sbf'}^1,\,
% \G_{\sbf'}^2$ and $\G_{\sbf'}^3$ denote, the respective subcollections of those
% $\sbf^*_k\in\G_{\sbf'}$ for which $(i)$, $(ii)$, or $(iii)$ holds.
For $\sbf'$ fixed, the sets in the collection
$\{E_{\sbf^*_k}: (i) \text{ holds for } \sbf_k^*\}$
% $\{E_{\sbf^*_k}\}_{\sbf_k^* \in \G_{\sbf'}}$
are pairwise disjoint, so
for such $\sbf_k^*$, we may trivially
pack the associated maximal cubes $Q(\sbf_k^*)$.
The maximal cubes $Q(\sbf_k^*)$ for which $\sbf_k^*$ satisfies $(ii)$ may be packed, because the red cubes pack.  For those $\sbf_k^*$ satisfying $(iii)$,
the desired packing follows from \eqref{Fp3pack}, thus, we have
reduced the proof of \eqref{graphclosepack} to that of Lemma \ref{keylemma1}.

As remarked above (see Remark \ref{prop8.2proofoutline}),
the proof of Proposition \ref{finalprop} will proceed in stages,
three of which remain.  Let us now
discuss some preliminary matters in connection with these three steps,
before completing
the argument in subsequent sections.

First, for each graph tree $\sbf^*=\sbf^*_k$,
we shall construct a Lip(1,1/2) graph $\Gamma_{\sbf^*}$ that is close to $\Sigma$ in the sense that \eqref{graphclose} holds, and such that
the graph has suitably small Lip(1,1/2) constant $b_1$.
For the construction of the graph we have a lot of information at our disposal. In particular, for any graph tree $\sbf^*$, the cube $Q(\sbf^*)$ satisfies the WHSA condition with a half space $H_{Q(\sbf^*)}$, whose boundary is the plane $P_{Q(\sbf^*)}= \partial H_{Q(\sbf^*)}$, such that the unit normal vector to the plane $P_{Q(\sbf^*)}$, pointing into $H_{Q(\sbf^*)}$, is given by the space-time vector
\[ \big(\nabla_X \widehat u(Y^*_{Q(\sbf^*)}, s^*_{Q(\sbf^*)}),0\big)/|\nabla_X \widehat u(Y^*_{Q(\sbf^*)}, s^*_{Q(\sbf^*)})|.\]
Furthermore, every $Q \in \sbf^*$ is a $\eps^{2L}$-blue cube satisfying the statement reverse  to that in Condition \ref{stopping} (3) %$(c)$,
i.e.,
\begin{align}\label{decomp1}
&|\nabla_X \widehat u(Y,s) - \nabla_X\widehat u(Y^*_{Q(\sbf^*)}, s^*_{Q(\sbf^*)})| \leq 10\eps^{2L},\ \mbox{for some } (Y,s) \in U_Q.
\end{align}
Recall that we always choose $(Y^*_{Q(\sbf^*)}, s^*_{Q(\sbf^*)})$ satisfying \eqref{nondegenstoppack.eq}, and hence by \eqref{decomp1}, % and  the definition of $U_Q$,
for $\eps$ sufficiently small, there is
a $j\in\mathcal{J}_Q$, and a corkscrew point $(\tilde{Y}_Q^*,\tilde{s}_Q^*):=
(\tilde{Y}_Q^{*,j},\tilde{s}_Q^{*,j})$ in $U_Q^j$,  such that
\begin{align}\label{decomp1+}
&|\nabla_X \widehat u(\tilde{Y}_Q^*,\tilde{s}_Q^*) - \nabla_X\widehat u(Y^*_{Q(\sbf^*)}, s^*_{Q(\sbf^*)})| \leq 10\eps^{2L}.
\end{align}
Indeed, we simply choose $(\tilde{Y}_Q^*,\tilde{s}_Q^*)$ to be the point $(Y,s)$  in \eqref{decomp1} (or if $Q=Q(\sbf^*)$, we just set
$(\tilde{Y}_{Q(\sbf^*)}^*,\tilde{s}_{Q(\sbf^*)}^*):=(Y^*_{Q(\sbf^*)}, s^*_{Q(\sbf^*)})$).
It then follows that
\eqref{nondegenweakened.eq} holds for % $(\tilde{Y}_Q^{*,j},\tilde{s}_Q^{*,j}) =$
$(\tilde{Y}_Q^*,\tilde{s}_Q^*)$, since \eqref{nondegenstoppack.eq}
 holds for $(Y^*_{Q(\sbf^*)}, s^*_{Q(\sbf^*)})$, and we have insisted that $\eps^{2L} \ll a_\star/200$.
Consequently, by Proposition \ref{WHSAplorientrefine.lem}, $Q\in \sbf^*$
satisfies the $(\eps,K_0)$-local WHSA condition with a half space $H_Q$
containing $(\tilde{Y}_Q^*,\tilde{s}_Q^*)$, and such that
the normal to the plane $P_Q = \partial H_Q$ that points into
$H_Q$ is in the direction $(\nabla_X\widehat u(\tilde{Y}_Q^*,\tilde{s}_Q^*),0)$.  Since
\[|\nabla_X \widehat u(\tilde{Y}_Q^*,\tilde{s}_Q^*)|\approx_{a_\star} 1\approx_{a_\star}  |\nabla_X \widehat u(Y^*_{Q(\sbf^*)}, s^*_{Q(\sbf^*)})|,\] it follows from \eqref{decomp1+} that
the angle between $P_Q$ and $P_{Q(\sbf^\ast)}$, denoted $\angle(P_Q, P_{Q(\sbf^\ast)})$, satisfies
\[\angle(P_Q, P_{Q(\sbf^\ast)}) \lesssim \eps^{2L}.\]
Therefore, after rotating in the spatial coordinates so that $(e_\perp,0)=(1,0,...,0,0)$
is the normal to $P_{Q(\sbf^*)}$ pointing into $H_{Q(\sbf^*)}$, we see by
a simple geometric argument  % at hand, using the coordinates generated by
% $P_{Q(\sbf^*)}^\perp \times P_{Q(\sbf^*)}$,
that if we take $(x_0, x,t) = (X,t) \in C_{Q,\eps}^{*} \cap P_Q$ and let
\[
\widetilde{H}_Q := \{(y_0, y, s): y_0 > x_0 + \eps\ell(Q)\ \},
\quad \widetilde{P}_Q := \{(y_0, y, s):  y_0 = x_0 + \eps\ell(Q)\},
\]
then the $(\teps,K_0)$-WHSA condition, with $\teps:= 10\eps$, holds for $Q$ with the plane $\widetilde{P}_Q$ and the half-space $\widetilde{H}_Q$. Here we used that $\eps\leq\eps_0=K_0^{-100}$,
and that $L$ is large and at our disposal. In particular, note that we have deduced
that every $Q \in \sbf^*$ satisfies the $(\teps,K_0)$-WHSA condition with a plane $\widetilde{P}_Q$ satisfying
\begin{align}\label{ffina-}
\widetilde{P}_Q \parallel \widetilde{P}_{Q(\sbf^*)}\,,
\end{align}
the normals to the planes $\{\widetilde{P}_Q\}_{Q\in \sbf^*}$ are all
pointing in the same direction, and we have set
\begin{align}\label{ffina}\widetilde{P}_{Q(\sbf^*)}: = P_{Q(\sbf^*)}\,,
\qquad \widetilde{H}_{Q(\sbf^*)}:={H}_{Q(\sbf^*)}\,.
\end{align}  Given $Q \in \sbf^*$ , we will refer to $\widetilde{P}_{Q}$ and  $\widetilde{H}_{Q}$ as the {canonical $\teps$-WHSA plane}, and the
canonical $\teps$-WHSA half-space associated to $Q$, respectively. The next section is devoted to rigorously (and painstakingly) proving that one can glue these planes together to form a Lip(1,1/2) graph, and to providing additional properties needed for the remaining two steps.
The next step, to be carried out in Section \ref{sec: reggraph},
is to show the graphs are, in fact, regular.
The final step, to be carried out in Section \ref{sec:pack}, is to prove Lemma \ref{keylemma1}, which then completes the proof of the packing condition \eqref{graphclosepack}, and thus also the proofs of Proposition \ref{finalprop} and of our main result, Theorem \ref{main.thrm}.

\section{Constructing Lip(1,1/2) graphs $\{\Gamma_{{\sbf^*}}\}$ associated to the graph trees $\{\sbf^*\}$}\label{sec: graph}
In this section, following the strategy
of  \cite{DS1}, we construct Lip(1,1/2) graphs adapted to the stopping time trees $\{\sbf^*\}$ introduced in the preliminary argument
for Proposition \ref{finalprop} outlined in Section \ref{Sec7}. The construction is similar to that in \cite{BHHLN-Corona}, which in itself is modeled on \cite{DS1}. Compared to \cite{BHHLN-Corona}, the main difference is that we have a poorer estimate in Lemma \ref{DSLem8.4mod.lem} below compared to the estimate in \cite[Lemma 4.2]{BHHLN-Corona} (and compared to the estimate in \cite[Lemma 8.4]{DS1}). We are also going to need additional properties of the graphs in order to, in the subsequent sections, improve the regularity of the graphs and to pack the maximal cubes.

In the following we consider a fixed stopping time tree $\sbf^*$.

\begin{remark}\label{*WHSA}
Recall that by construction,
(see \eqref{decomp1+}-\eqref{ffina}),
every $Q \in \sbf^*$ satisfies the $(\teps,K_0)$-WHSA condition,
$\teps=10\eps$, with a half-space $ \widetilde{H}_Q$,
and a plane $\widetilde{P}_Q=\partial \widetilde{H}_Q$
satisfying $\widetilde{P}_Q \parallel \widetilde{P}_{Q(\sbf^*)}$
(see \eqref{ffina-}), such that the
normals to the planes $\{\widetilde{P}_Q\}_{Q\in \sbf^*}$ are all
pointing in the same direction.
In particular (see Definition \ref{def2.13}), % there is a half-space
% $ \widetilde{H}_Q$, a hyperplane
% $\widetilde{P}_Q =\partial \widetilde{H}_Q$, parallel to the time axis,  satisfying the following:
\begin{list}{$(\theenumi)$}{\usecounter{enumi}\leftmargin=1cm \labelwidth=1cm \itemsep=0.2cm \topsep=.2cm \renewcommand{\theenumi}{\roman{enumi}}}
\item $\dist(Z,\tau,\Sigma)\leq\teps\ell(Q),$ for every $(Z,\tau)\in \widetilde{P}_Q \cap C_{Q,\teps}^{*}$,
\item $\dist(Q,\widetilde{P}_Q)\leq c_nK_0 \ell(Q),$
\item $\widetilde{H}_Q\cap C_{Q,\teps}^{*}\cap \Sigma=\emptyset.$
\end{list}
\end{remark}

 The cylinder $C^*_{Q,\teps}$, of size
 roughly $\teps^{-2}\diam(Q)$, was introduced in \eqref{eq2.bstarstar}.

We are going to construct a coordinate system and a Lip(1,1/2)  function $ \psi_{\sbf^*}:=\psi = \psi ( x, t ) : \mathbb R^{n-1}\times\mathbb R\to \mathbb R$ with constant $b_1\lesssim \eps^{1/2}$, such if we define
 $$
 \Gamma_{\sbf^*}:=\{(\psi_{\sbf^*}(x,t),x,t): (x,t)\in \mathbb R^{n-1}\times\mathbb R\},
 $$
 then for $(\eta,K)$ as in \eqref{eq2.2ahala} below,
\begin{equation}\label{eq2.2aha}
\sup_{(X,t) \in KQ} \dist(X,t,\Gamma_{{\sbf^*}} )\leq\eta\,\diam(Q),
\end{equation}
 for every $Q\in {\sbf^*}$. We will prove that this can be done with
 \begin{equation}\label{eq2.2ahala}
(\eta,K)=(cK_0,2),
\end{equation}
where $c\geq 1$ is a constant that only depends on the allowable parameters. In particular, the results of this section and the following section, will in the end result in a $(cK_0,2)$-Corona decompositions in the sense of Definition \ref{unilateralcorona.def}.

Recall that $\teps := 10\eps$.
Fix $\sbf^*$, with maximal cube $Q(\sbf^*)$, so that Remark \ref{*WHSA} applies to
each $Q \in \sbf^*$.
% By Remark \ref{*WHSA}, each $Q \in \sbf^*$
% satisfies the  $(\teps,K_0)$-WHSA with respect to
% a plane $\widetilde{P}_Q$ and half-space $\widetilde{H}_Q$, with
% $\widetilde{P}_Q \parallel \widetilde{P}_{Q(\sbf^*)}$,
% such that the unit normals associated to the elements in the
% set $\{\widetilde{P}_Q\}$ are all pointing in the same direction.
We will work in the coordinates induced by
$\widetilde{P}_{Q(\sbf^*)}^\perp \times \widetilde{P}_{Q(\sbf^*)},$ i.e.,
% $(0,x,t) \in \widetilde{P}_{Q(\sbf^*)}$, and
\begin{equation}\label{PHdef}
\widetilde{P}_{Q(\sbf^*)} = \{(0,x,t)\in \mathbb{R}\times \mathbb{R}^{n-1}\times\mathbb{R}\}\,,
\quad \text{ and } \quad \widetilde{H}_{Q(\sbf^*)}=\{(x_0,x,t): x_0 > 0\}.
\end{equation}
We let $\pi$ denote the (orthogonal) projection onto $\widetilde{P}_{Q(\sbf^*)}$ and we let $\pi^{\perp}$ denote the (orthogonal) projection onto the inner normal of $\widetilde{H}_{Q(\sbf^*)}$, i.e., given $(X,t) = (x_0,x,t)$ we have
\[
\pi(X,t) = (x,t) \cong (0,x,t), \quad \pi^\perp(X,t) =x_0 \cong (x_0,0,0).
\]
Given $(x,t),(y,s)\in\mathbb R^n$ we  let
$$\dist(x,t,y,s):=|x-y|+|t-s|^{1/2}.$$
Furthermore,  given $(x,t)\in\mathbb R^n$, $r>0$, % in the following we use the
we define the $n$-dimensional parabolic cylinder
\begin{equation}\label{ndimcylinder}
C_r'(x,t):=\{(y,s)\in \mathbb R^n: y\in B_r'(x),\ |s-t|<r^2\}\,,
\end{equation}
where now $B_r'(x)$ is the Euclidean ball in $\mathbb R^{n-1}$ centered at $x$ and with radius $r$.

Following \cite{DS1, BHHLN-Corona}, we introduce (two) stopping time distances.
For $(X,t) \in \ree$ we define
\begin{equation}\label{DSDfn.eq--}d(X,t):= d_{\sbf^*}(X,t):= \inf_{Q \in \sbf^*}[\dist(X,t, Q) + \diam(Q)],
\end{equation}
and for $(x,t) \in \rn$ we define
\begin{equation}\label{DSDfn.eq}
D(x,t):= \inf_{(Y,s) \in \pi^{-1}(x,t)} d(Y,s) =  \inf_{Q\in \sbf^*} [\dist(x,t, \pi(Q)) + \diam(Q)].
\end{equation}
We also set
\begin{equation}\label{DSDfn.F}
F=F_{\sbf^*}:= \{(X,t) \in \Sigma: d(X,t) = 0\}.
\end{equation}
The set $F$ is the set where the graph $\Gamma_{\sbf^*}$, to be constructed, makes contact with $\Sigma$, i.e., the set where the graph $\Gamma_{\sbf^*}$ coincides with $\Sigma$.  It may be that the contact set $F$ is empty.

\begin{remark}\label{remarkF0vsF}
We note that the set $F^0_{\sbf^*}$ defined in \eqref{F0def} agrees with
$F_{\sbf^*}$ up to a set of $\sigma$-measure zero.  More precisely,
$F^0_{\sbf^*}\subset F_{\sbf^*}$, and
\begin{equation}\label{eqF0vsF}
\sigma(F_{\sbf^*}\setminus F^0_{\sbf^*}) = 0\,.
\end{equation}
Indeed, by definition, $(X,t)\in F^0_{\sbf^*}$ means that whenever
$(X,t)\in Q\subset Q(\sbf^*)$, then $Q\in \sbf^*$.  Letting
$Q\searrow (X,t)$, we see that $d(X,t)=0$, i.e., $F^0_{\sbf^*}\subset F_{\sbf^*}$.
To verify \eqref{eqF0vsF}, we observe that
$F_\sbf^*\subset \overline{Q(\sbf^*)}$, and that
$F_\sbf^*$ is disjoint from the interior of every stopping time
cube $Q_j\in\F^*(\sbf^*)$, i.e.,
$F_\sbf^*\subset \overline{Q(\sbf^*)}\setminus (\cup_{Q_j\in\F^*} \interior(Q_j)$.
Thus, the set $F_{\sbf^*}\setminus F^0_{\sbf^*}$ is
contained in a union of boundaries
of cubes in $\dd(\Sigma)$, whence \eqref{eqF0vsF} follows by the thin boundary property of the dyadic cubes (Lemma \ref{cubes} $(vi)$).
\end{remark}

Below, we will construct $\psi_{\sbf^*}$, the function which will define the graph $\Gamma_{\sbf^*}$, and in the construction we are going to choose a number $\kappa$ according to
\begin{align}\label{constant1}
\mbox{$\kappa = c K_0\,$, $\,\,\, c\ge 100$},
\end{align}
where $c$ is a constant depending only on $n$ and $M$, and chosen so that we can justify \eqref{IiQi.eq}.

\subsection{Constructing a graph parametrization $(\widehat{\psi}(x,t),x,t)$ of the contact set $F$} We first construct a graph on the contact set $F$.  If the latter is empty, then this step is moot.

\begin{lemma}\label{DSLem8.4mod.lem}
Let ${\sbf^*}$ be as above, let $F=F_{\sbf^*}$ be as in \eqref{DSDfn.F},
and let $\widetilde{P}_{ Q({\sbf^*})}=\mathbb R^n$.
Then  $\pi$ is one-to-one on $F$
and we can define a function $\widehat\psi$ on $\pi(F)\subset \mathbb R^n$
as follows. Given
$(X,t)=(x_0,x,t)\in F$, we set  $\widehat\psi(x,t):=\pi^\perp(X,t)=x_0$.  Then
$\widehat \psi$ is a well-defined Lip(1,1/2) function on $\pi(F)$,
with constant bounded by $K_0\eps$,  i.e.,
\begin{align}\label{hatlip}
|\widehat \psi(x,t)-\widehat \psi(y,s)|\lesssim \,K_0\eps  \dist(x,t,y,s),
\end{align}
whenever $(x,t), (y,s)\in \pi(F)$.  More generally, for
$\kappa$ as in \eqref{constant1}, there is a constant $\beta_0$, with  % $ \ll K_0^{-1}$
$\beta_0K_0\ll_{n,M} 1$, % depending only on $n$ and $M$,
such that if
$(X,t),\ (Y,s)\in 16 \kappa Q({\sbf^*})$ and
\begin{equation}\label{ddpsep.eq}
\max\{d(X,t),d(Y,s)\} \le \,\beta
\dist(X,t, Y,s)
\end{equation}
for some $\beta\in (\eps^{3/2},\beta_0)$, then
\begin{align}\label{DSl84cons.eq}
|\pi^\perp(X,t)-\pi^\perp(Y,s)|\lesssim \, % \left(\eps^2\beta^{-1} + \beta K_0 \right)
\beta K_0 \dist(\pi(X,t),\pi(Y,s)).
\end{align}
\end{lemma}

\begin{proof} The lemma is similar to \cite[Lemma 4.1]{BHHLN-Corona},
but the proof is a bit different, since
% \eqref{ddpsep.eq}  is stated with a maximum, and
we are working with the $\teps$-WHSA condition.
Recall that $\eps\leq \eps_0=K_0^{-100} \ll (1+K_0)^{-2}$ and that $\teps = 10\eps$.
It suffices to prove that \eqref{ddpsep.eq} implies \eqref{DSl84cons.eq}:
indeed,  if $(X,t),\,(Y,s)\in F$, then by definition
\eqref{ddpsep.eq} holds for all $\beta \geq 0$. In turn, % assuming  \eqref{DSl84cons.eq}
\eqref{DSl84cons.eq} implies that $\pi$ is 1-1 on $F$, and
we obtain \eqref{hatlip} by  taking $\beta=\eps$ in \eqref{DSl84cons.eq}.

To prove \eqref{DSl84cons.eq}, we consider
$(X,t)=(x_0,x,t),\, (Y,s)=(y_0,y,s)\in 16 \kappa Q({\sbf^*})$, and
we set $r:= \dist(X,t,Y,s)$.
It suffices to prove  that
\begin{equation}\label{eq2.7}
|\pi^\perp(X,t)-\pi^\perp(Y,s)|\lesssim  \, % \eps^2\beta^{-1}\, \dist(\pi(X,t),\pi(Y,s))+
\beta K_0r\,.
\end{equation}
Indeed, for $\beta K_0\ll 1$, \eqref{eq2.7}
implies in particular that
$|\pi^\perp(X,t)-\pi^\perp(Y,s)| <r/2$.  Consequently,
$ \dist(\pi(X,t),\pi(Y,s))\geq r/2$, and \eqref{DSl84cons.eq} follows.
We may assume,  without loss of generality, that $x_0>y_0$ (if equality holds, there is nothing to prove).  Given $(Y,s)$, let  $Q_1\in{\sbf^*}$ be such that
\[
d(Y,s)\leq \dist(Y,s,Q_1) +\diam (Q_1) \leq 2\beta r,
\]
where in the last step we have used \eqref{ddpsep.eq} and the definition of $r$.
Given $Q_1\in{\sbf^*}$, since $r \lesssim \kappa \diam(Q(\sbf^*)) \approx K_0 \diam(Q(\sbf^*))$,
we may choose $Q\in {\sbf^*}$, with $Q_1\subseteq Q$, such that
\begin{equation}\label{Qcompr}
c\beta r\leq \diam (Q) \leq 2\beta r,
\end{equation}
where $c$ can be chosen to depend only on the parameters in the dyadic cube construction,
and hence ultimately only on $n$ and $M$.  We  then have
\[
\dist(Y,s,Q)\, \leq \,\dist(Y,s,Q_1)\, \leq\, 2\beta r\, \lesssim_{n,M}\, \diam (Q).
\]
Consequently,  $(Y,s) \in \lambda Q$, with $\lambda$ depending only on $n$ and $M$.  By the $\teps$-local WHSA condition,
we have (see Remark \ref{*WHSA})
\begin{equation}\label{eq2.10}
\dist(Q,\widetilde{P}_Q)\leq c_n K_0 \ell(Q) \lesssim \beta K_0 r,
\end{equation}
and also that
$\widetilde{H}_Q\cap C_{Q,\teps}^{*}\cap \Sigma=\emptyset$, where by \eqref{ffina-}
and \eqref{PHdef},
\[
\widetilde{H}_Q= \big\{(z_0,z,\tau): \, z_0 >z_Q\big\},\quad
 \widetilde{P}_Q = \big\{(z_0,z,\tau): \, z_0 =z_Q\big\},
\]
 for some suitable fixed height $z_Q$.  In particular,
 $ (X,t) \notin \widetilde{H}_Q \cap C_{Q,\teps}^{*}\,$, since $(X,t)\in \Sigma$.
 We further claim that $ (X,t) \in C_{Q,\teps}^{*}$, and thus
 $(X,t)\notin \widetilde{H}_Q$, i.e., $x_0\leq z_Q$.
 Indeed,  $(Y,s)\in \lambda Q$, % (see \eqref{dilatecube}),
with $\lambda=\lambda(n,M)$, so
\begin{multline*}
\dist(X,t,X_Q,t_Q)\, \leq\, \dist(X,t,Y,s) + \dist(Y,s,X_Q,t_Q)\\[4pt]
\leq \,r +\lambda \diam(Q)\, \lesssim\,
\big(\beta^{-1}+\lambda\big) \diam(Q)\,\lesssim \eps^{-3/2} \diam(Q)\,\ll
\eps^{-2} \diam(Q)\,,
\end{multline*}
by \eqref{Qcompr} and our assumed constraints on $\beta$.  Thus,
$(X,t)\in C_{Q,\teps}^{*}$, as claimed, hence $x_0\leq z_Q$, so also
$y_0<x_0\leq z_Q\,$.
Moreover, by
 \eqref{eq2.10}, \eqref{Qcompr}, and the fact that $(Y,s)\in \lambda Q$,
 \begin{equation*}% \label{eq2.10uu}
 0< z_Q -y_0 =
\dist(Y,s,\widetilde{P}_Q)\lesssim (\lambda + K_0)\, \diam(Q) \lesssim \beta (\lambda + K_0) r
\lesssim \beta K_0 r\,,
\end{equation*}
since $ K_0\gg 1$ and $\lambda=\lambda(n,M)$.
Consequently, using again that $y_0<x_0\leq z_Q\,$, we have
\[
|\pi^\perp(X,t) -\pi^\perp(Y,s)| = x_0-y_0 \lesssim \beta K_0 r\,,
\]
so that \eqref{eq2.7} holds and we are done.
\end{proof}

\subsection{Extending $\widehat\psi$ off $\pi(F)$: constructing $\psi$}
With Lemma \ref{DSLem8.4mod.lem} as the starting point we next construct a function  $\psi$ on
$\mathbb R^n \cong \widetilde{P}_{ Q({\sbf^*})}$
which coincides with $\widehat\psi$ on $\pi(F)$. We will
extend ${\widehat\psi}$ off $\pi( F)$ using an
appropriate Whitney type extension introduced in
\cite{DS1}, and also used in \cite{BHHLN-Corona}.  In contrast to the classical Whitney decomposition, the present version makes sense not only on the complement of a non-empty closed set, but even on all of $\rn$; thus, in the sequel,
we allow the possibility that the set $F$ is empty.
Given $(x,t)\in  \mathbb R^n\setminus\pi(F)$ not on the boundary of any parabolic
dyadic cube in $\re^n$, let $I_{(x,t)}$ be the largest (closed)
parabolic dyadic cube in
$\mathbb R^n$ containing $(x,t)$, and satisfying
$$\diam(I_{(x,t)})\leq \frac 1 {20}\inf_{(z,\tau)\in I_{(x,t)}}D(z,\tau).$$
Let $\{I_i\}$ be a labeling of the set of all these cubes $\{I_{(x,t)}\}$ without repetition. By construction, the collection
$\{I_i\}$ consists of pairwise non-overlapping closed parabolic
dyadic cubes in $\mathbb R^n$.
These cubes cover $ \mathbb R^{n} \sem  \pi( F) $ and they do not intersect
$\pi( F)$.   Notice that this covering is well-defined, even when $F=\emptyset$.

We will need the following lemma, the proof of
which follows that of \cite[Lemma 8.7]{DS1} (and \cite[Lemma 4.6]{BHHLN-Corona}) very closely.
We include the proof for the reader's convenience.
\begin{lemma}\label{10-60lemma}
Given $I_i$ we have
\begin{align}\label{10-60.eq}
10 \diam I_i \leq D(y,s) \leq 60 \diam I_i, \indent \forall (y,s) \in 10I_i.
\end{align}
In particular, if  $10I_i \cap 10I_j \neq \emptyset$, then
\begin{align}\label{eq1}
\diam I_i \approx \diam I_j.
\end{align}
\begin{proof}
Fix $i$ and let $(y,s) \in 10 I_i$.  As $D(\cdot,\cdot)$ is Lip$(1,1/2)$ with norm $1$, we have
\[D(y,s) \geq \inf_{(x,t)\in I_i} D(x,t) - 10\diam I_i \geq 10 \diam I_i,\]
where we used that $\diam I_i \leq 20^{-1} \inf_{(x,t)\in I_i} D(x,t)$.  On the other hand, if $I$ is the dyadic parent of $I_i$, then there exists $(z,\tau) \in I$ such that $D(z,\tau) <20 \diam I = 40 \diam I_i$.  Thus,
\[D(y,s) \leq D(z,\tau) + 20 \diam I_i \leq 60\diam I_i.\]
This proves \eqref{10-60.eq}, and hence also \eqref{eq1}.
\end{proof}
\end{lemma}
We continue to follow the constructions in \cite[Chapter 8]{DS1} and \cite{BHHLN-Corona}.
For the sake of notational convenience, we set
\begin{align}\label{nota2}
	R_{{\sbf^*}}:=\diam ( Q({\sbf^*})).
\end{align}
Given the center $(X_{ Q({\sbf^*})},t_{ Q({\sbf^*})})$ of $ Q({\sbf^*})$, we let
$$(x_{ Q({\sbf^*})},t_{ Q({\sbf^*})}):=\pi(X_{ Q({\sbf^*})},t_{ Q({\sbf^*})}),$$ and
we introduce the index set $\Lambda$,
\begin{equation}\label{la-def}
 \La \, := \{ i :
 I_i \cap C'_{2\kappa R_{\sbf^*}}(x_{ Q({\sbf^*})},t_{ Q({\sbf^*})}) \not = \emptyset \}.
\end{equation}
Recall that the parameter $\kappa$ is defined
in \eqref{constant1}, and observe that % since $\kappa \approx K_0$,
in the statement of Lemma \ref{DSLem8.4mod.lem},
estimates \eqref{hatlip} and \eqref{DSl84cons.eq}
hold with $K_0$ replaced by $\kappa$,
since $\kappa \approx K_0$.

We shall require the following.

\begin{lemma}\label{lemmaQi}
There exists, for each $i \in \Lambda$, a cube $Q(i) \in {\sbf^*}$ such that
\begin{align}\label{eq1+}
\kappa^{-1} \diam I_i \lesssim  \diam Q(i) \leq 120 \diam I_i, \quad \dist(\pi(Q(i)),I_i) \leq 120 \diam I_i.
\end{align}
\end{lemma}
\begin{proof}[Proof of Lemma \ref{lemmaQi}]
Note that given $I_i$ and $(x,t)\in I_i$, there exists a cube $Q\in {\sbf^*}$ such that
$$\dist(x,t,\pi(Q))+\mbox{diam}(Q)\leq 2D(x,t)\approx\diam (I_i).$$
Moreover, $D(x,t) \le c \kappa \diam(Q({\sbf^*}))$ for a constant $c\geq 1$ depending
only on $n$ and $M$, and
\[
\dist(x,t,\pi(Q^*)) \le 2D(x,t)\leq 120 \diam I_i,
\]
for all $Q  \subseteq Q^* \subseteq Q({\sbf^*})$. Therefore we can choose $Q(i)$ so that  $Q  \subseteq Q(i) \subseteq Q({\sbf^*})$ and
\begin{equation}\label{IiQi.eq}
(c \kappa)^{-1}D(x,t) \le \diam(Q(i)) \le 2D(x,t) \leq 120 \diam I_i,
\end{equation}
provided that the constant $c=c(n,M)$ in \eqref{constant1}
(i.e., for which $\kappa=cK_0$) is chosen large enough.
\end{proof}

In the sequel, we will, at instances, use the notation
\begin{align}\label{nota1}
r_i:=\diam I_i,\,\, \text{ so that } \,   r_i/K_0\approx  r_i/\kappa    \lesssim
\diam Q(i) \lesssim r_i,
\end{align}
by \eqref{eq1+}.

Let $ \{ \tilde \nu_i \} $ be a class of infinitely differentiable functions on $\mathbb R^{n}$ such that
\begin{align*}
 \tilde \nu_i  \equiv 1  \mbox{ on }  2I_i
 \mbox{ and }   \tilde \nu_i \equiv 0  \mbox{ in } \mathbb R^{n} \sem 3I_i \mbox{ for all } i,
 \end{align*}
 and
\begin{align*}
 r_i^{|\beta|} \, \left| \frac{\partial^{|\beta|}}{\partial x^\beta }  \tilde \nu_i \right|
\, + \,    r_i^{ 2 l}  \left| \frac{\partial^l}{\partial t^l } \tilde \nu_i \right|  \,  \leq
\, c (|\beta|, l, n ) \mbox{ for } |\beta|, l \in\{0, 1, 2, \dots\},
\end{align*}
where $ \frac{\partial^{|\beta|}}{\partial x^\beta } $
denotes a partial derivative with respect to the space variable $ x $, according to the multiindex $\beta=(\beta_1,...,\beta_{n-1})$, of order $|\beta|$.  We also introduce a related collection
$ \{ \nu_i \} $, which defines a partition of unity on $\rn\setminus \pi(F)$,
 adapted to $ \{ 2I_i \}\,$:
 $$\nu_i(x,t):=\tilde \nu_i(x,t)\big (\sum_{j} \tilde \nu_j(x,t)\big)^{-1}.$$
Then $ \{  \nu_i \} $  is also a class of infinitely differentiable functions on $\mathbb R^{n}$ and
\begin{align}\label{2.16+}
r_i^{|\beta|} \, \left| \frac{\partial^{|\beta|}}{\partial x^\beta }  \nu_i \right|
\, + \,    r_i^{ 2 l}  \left| \frac{\partial^l}{\partial t^l } \nu_i \right|  \,  \leq
\, c ( |\beta|,l, n ) \mbox{ for } |\beta|, l \in\{ 0,1, 2, \dots\}.
\end{align}

We are now ready to construct the extension of $  \widehat \psi  $ off $ \pi( F) $. {If $i \in \Lambda$, then} we have a dyadic cube $Q(i)\in {\sbf^*}$, associated  to $I_i$ as in Lemma \ref{lemmaQi}, % \eqref{eq1+},
and an associated hyperplane $\widetilde{P}_{Q(i)}$  which we express as
\begin{align}\label{2.19-} \widetilde{P}_{Q(i)}=\big\{ (B_i  ( x,t),x,t ) : ( x,t) \in \mathbb R^n \big\},
\end{align}
i.e., $B_i:\mathbb R^n\to \mathbb R$ is the affine function whose graph is $\widetilde{P}_{Q(i)}$.
Note that $B_i$ is simply constant for each $i$, i.e.,
$B_i\equiv constant$ (depending on $i$), since
$\widetilde{P}_Q \parallel \widetilde{P}_{Q(\sbf^*)}$,
for all $Q \in \sbf^*$. If $i \not \in \Lambda$, we set $B_i = 0$ and note that in this case, $B_i$ is the affine function
whose graph is $\widetilde{P}_{Q({\sbf^*})}$. For these reasons  we will from now on drop the dependence of $B_i$ on $(x,t)$. We define
\begin{align}\label{2.19}
\psi ( x,t)&:=\widehat \psi ( x,t)\quad\quad\quad\mbox{ when } ( x,t) \in \pi (
F),\notag\\
\psi ( x,t) &:=
 { \ds \sum_{ i } } \,  B_i
\,   \nu_i ( x,t )\,\, \mbox{ when } ( x,t ) \in \mathbb R^n \sem \pi (
F). \end{align}
This defines $\psi $ on $\mathbb R^n$, with $\psi=\widehat\psi$ on
$\pi (F)$.   Note that by construction (see \eqref{la-def}),
\begin{equation}\label{supportpsi}
\psi\equiv 0 \text{ on } \mathbb R^n\setminus C'_{4\kappa R_{\sbf^*}}(x_{ Q({\sbf^*})},t_{ Q({\sbf^*})}).
\end{equation}

As we will see, we need to investigate the planes $\{\widetilde{P}_Q\}$ in more depth, and subsequently we will prove that the region above the graph defined by $\psi$, at the scale $Q(\sbf^*)$, is in the {\it interior} of $\Omega=\mathbb R^{n+1}\setminus\Sigma$. We will also need to show that the graph is
``close"\footnote{In the sense that \eqref{graphclose} holds.}
to the cubes in the stopping time tree, albeit not necessarily with a
small constant. For all of this, we will utilize certain unions of truncated cones which we introduce below.

Since $\widetilde{P}_Q \parallel \widetilde{P}_{Q(\sbf^*)}$, we have as above, in the coordinates induced by $\widetilde{P}_{Q(\sbf^*)}$, that
\begin{equation}\label{eq.BQdef}
\widetilde{P}_Q  = \{(B_Q, x,t): (x,t) \in \rn\},\ \widetilde{H}_Q =
 \{(x_0, x,t): x_0 > B_Q,\ (x,t) \in \rn\},
 \end{equation}
for some $B_Q \in \mathbb{R}$, and in particular, $B_{Q(\sbf^*)}=0$.
Recall that $\teps = 10\eps$.
By Definition \ref{def2.13}, if $Q, Q^* \in \sbf^*$ are such that $Q^*$ is the parent of $Q$, then
\begin{equation}\label{graphcloseBQobs.eq}
B_Q - \teps\ell(Q) \le B_{Q^*} \le B_Q +  \teps\ell(Q^*).
\end{equation}
Note that by definition, $\ell(Q^*) = 2\ell(Q)$.
We then have that
\begin{equation}\label{BQBQS*.eq}
B_Q - \teps\ell(Q) \le B_{Q(\sbf^*)} =0\le B_Q +  2\teps\ell(Q(\sbf^*))\,,
\qquad \forall\, Q\in\sbf^*\,,
\end{equation}
where one obtains the right hand inequality by iterating \eqref{graphcloseBQobs.eq},
and summing a geometric series;
the left hand inequality holds
by Definition \ref{def2.13} (i) (for $Q$), and (iii) (for $Q(\sbf^*)$).
% \[ B_Q - \teps\ell(Q(\sbf^*)) \le B_{Q(\sbf^*)} \le B_Q +  \teps\ell(Q(\sbf^*)). \]

Given $Q\in \sbf^*$ we introduce the (union of) cones
\begin{align}\label{gammaQ'def}
\gamma_Q' &= \{(y_0,y,s): \exists (x_0,x,t) \in Q \text{ such that } y_0 - x_0 \ge \teps^{1/4}\dist(y,s,x,t)\},
\end{align}
and its truncated version
\begin{align} \label{gammaQdef}
\gamma_Q &= \{(y_0,y,s)\in \gamma_Q': \text{ such that }
B_Q < y_0 <  \eps^{-1/2} \ell(Q(\sbf^*))\}.
\end{align}
We will use the following lemma several times in the sequel.
\begin{lemma}\label{gqconelem.lem}
If $Q \in \sbf^*$, then % $\gamma_Q \cap \Sigma = \emptyset$. More precisely,
\[
\gamma_Q \subseteq
\bigcup_{\widetilde{Q}: Q \subseteq \widetilde{Q} \subseteq Q({\sbf^*})}
\tilde{U}_{\widetilde{Q}}^i\,.
\]
In particular, $\gamma_Q \cap \Sigma = \emptyset$.
Here, for each $Q\in \sbf^*$, $\tilde{U}_Q^i$, defined as in Definition \ref{def2.11a}, is the component of
$\tilde{U}_Q$ which contains the corkscrew point $({Y}^*_Q, {t}^*_Q)$.
% (see also Lemma \ref{claim5.40})
\end{lemma}

\begin{proof} To begin, recall that for each $Q\in \sbf^*$, we have constructed the
$\teps$-WHSA plane $\widetilde{P}_Q$
 in a two stage process.  First,
 starting with the plane $\hat P_Q$ (see \eqref{eq5.44++}),
with respect to which $Q$ satisfies the $(\eps^{9/8},K_0)$-local WHSA condition,
we produced a plane $P_Q$
(see Remark \ref{nablaXnormal}), with normal pointing in the direction
of $(\nabla_X\widehat u(Y_Q^*,s_Q^*),0)$, with respect
to which $Q$ satisfies the $(\eps,K_0)$-local WHSA condition. Second,
we constructed the plane $\widetilde P_Q$ with
respect to which $Q$ satisfies the $(\tilde \eps,K_0)$-local WHSA condition with
$\tilde\eps=10\eps$, such that $\widetilde P_Q$ is
parallel to $\widetilde P_{Q(\sbf^*)}$ (see \eqref{ffina-}),
and where after a possible translation and spatial rotation, we have assumed that
$\widetilde P_{Q(\sbf^*)}$ is simply the hyperplane $x_0=0$ in $\ree$ (see \ref{PHdef}).
In particular, by construction, we have
\begin{equation}\label{qfavsrrv}
	C_{Q,\teps}^{*} \cap \widetilde{H}_Q \,\subset\, C_{Q,\eps}^{*} \cap H_Q
	 \,\subset\, C_{Q,\eps^{9/8}}^{*} \cap \hat H_Q \, \subset\, \tilde{U}_Q^i\,,
	 \qquad \forall \,Q \in \sbf^*\,.
\end{equation}
For the last inclusion we used Lemma \ref{claim5.40}.
Thus, it suffices to prove
\begin{equation}\label{gqconelemred.eq}
\gamma_Q \subseteq \bigcup_{\widetilde{Q}: Q \subseteq \widetilde{Q} \subseteq Q({\sbf^*})} (C_{\widetilde{Q},\teps}^{*} \cap
\widetilde{H}_{\widetilde{Q}}).
\end{equation}

We claim that if $(y_0,y,t) \in \gamma_Q$, then
there exists $\widetilde{Q}$ such that $Q \subseteq \widetilde{Q} \subseteq Q(\sbf^*)$,
and  such that
\begin{equation}\label{ynintervalbqt.eq}
y_0 \in (B_{\widetilde{Q}} ,B_{\widetilde{Q}} + \teps^{-3/2} \ell(\widetilde{Q})).
\end{equation}
Indeed, by \eqref{gammaQdef}, $B_Q < y_0 <  \eps^{-1/2} \ell(Q(\sbf^*))$, so
if $y_0>0=B_{Q(\sbf^*)}$, we just set $\widetilde{Q} = Q(\sbf^*)$, and use that
$ \eps^{-1/2}\ll  \teps^{-3/2}$.  Otherwise, since $y_0>B_Q$, we let $Q_1$
be the smallest cube in $\sbf^*$, containing $Q$, for which $y_0\leq B_{Q_1}$ (note that
$Q_1$ exists in the present scenario, and properly contains $Q$), and then let $\widetilde{Q}$ be the child of
$Q_1$ that contains $Q$.  Then $y_0>B_{\widetilde{Q}}$, and by
\eqref{graphcloseBQobs.eq},
\[
y_0\leq B_{Q_1} \leq B_{\widetilde{Q}} +\teps\ell(Q_1) =
B_{\widetilde{Q}} +2\teps\ell(\widetilde{Q}) \ll
B_{\widetilde{Q}} +\teps^{-3/2}\ell(\widetilde{Q})\,.
\]
This proves the claim.

% Using \eqref{graphcloseBQobs.eq} repeatedly, and that $\teps\leq K_0^{-100} \ll 1$,
% we see that  if $(y_0,y,t) \in \gamma_Q$, then
% there exists $\widetilde{Q}$ such that $Q \subseteq \widetilde{Q} \subseteq Q(\sbf^*)$,
% and  such that
% \begin{equation*}% \label{ynintervalbqt.eq}
% y_0 \in (B_{\widetilde{Q}} ,B_{\widetilde{Q}} + \teps^{-3/2} \ell(\widetilde{Q})).
% \end{equation*}

Now let $(Y,s) = (y_0,y,t) \in \gamma_Q$, let $\widetilde{Q}\supseteq Q$ be as above, and let
$(X,t) = (x_0,x,t)\in Q$ be the vertex of a cone containing $(y_0,y,t)$,
as in the definition of $\gamma'_Q$.
Then $(x_0,x,t) \in \widetilde{Q}$, and
by the $\teps$-WHSA condition for $\widetilde{Q}$
(see Definition \ref{def2.13} part $(ii)$), it holds that
\[-\diam(\widetilde{Q})  - cK_0\ell(\widetilde{Q}) \le B_{\widetilde{Q}} - x_0 < \diam(\widetilde{Q})  + cK_0\ell(\widetilde{Q}).\]
This combined with \eqref{ynintervalbqt.eq} gives
\[|y_0 - x_0| \lesssim \teps^{-3/2}\ell(\widetilde{Q}),\]
and therefore by the definition of $\gamma'_Q$,
\[\dist(y,s,x,t) \le \teps^{-1/4} (y_0 - x_0) \lesssim \teps^{-7/4}\ell(\widetilde{Q}).\]
Therefore,
\[\dist(X,t, Y,s) \lesssim \teps^{-7/4}\ell(\widetilde{Q}) \ll \teps^{-2}\ell(\widetilde{Q}),\]
which shows that $(Y,s) \in C_{\widetilde{Q},\teps}^{*}$. Then, using \eqref{ynintervalbqt.eq}  we have $(Y,s) \in \widetilde{H}_{\widetilde{Q}}$ and it follows that \eqref{gqconelemred.eq} holds.
\end{proof}

We next prove that the levels defining the planes $\{\widetilde{P}_Q\}$, i.e., the constant levels $\{B_Q\}$,  are well behaved in the following sense.
Recall that the set $F$ is defined in \eqref{DSDfn.F}.

\begin{lemma}\label{bqbqpclosediam.lem}
Let $Q, Q' \in \sbf^*$. If
\begin{equation}\label{bqbqpclsimphyp.eq}
\dist(Q,Q') \le \eps^{-1}\max\{\diam(Q),\diam(Q')\},
\end{equation}
then
\begin{equation}\label{bqbqpclsimpconc.eq}
|B_Q - B_{Q'}| \lesssim \eps \max\{\diam(Q),\diam(Q')\}.
\end{equation}
If
\begin{equation}\label{bqbqpclnotsimphyp.eq}
\dist(Q,Q') \ge \eps^{-1}\max\{\diam(Q),\diam(Q')\},
\end{equation}
then
\begin{equation}\label{bqbqpclnotsimpconc.eq}
|B_Q - B_{Q'}| \lesssim K_0^2\eps\dist(\pi(Q), \pi(Q')) + \eps \max\{\diam(Q),\diam(Q')\}.
\end{equation}
Moreover, if $Q \in \sbf^*$ and $(Z,\tau)\in F$, then
\begin{equation}\label{bqFclsdiamdst.eq}
|B_Q - \pi^\perp(Z,\tau)| \lesssim K_0^2\eps\dist(\pi(Q), \pi(Z,\tau)) + \eps\diam(Q).
\end{equation}
\end{lemma}
\begin{proof} % To estimate $|B_Q - B_{Q'}|$,
We may assume without loss of generality that $\diam(Q) \ge \diam(Q')$.

We first assume that \eqref{bqbqpclsimphyp.eq} holds
and we verify \eqref{bqbqpclsimpconc.eq}. To this end,
we argue by contradiction.  Our goal is to rule out the following
two cases:

\smallskip

\noindent
\textbf{Case 1:} $B_Q \ge 40 \eps \ell(Q) + B_{Q'}$.

\smallskip

\noindent
\textbf{Case 2:} $B_Q \le  B_{Q'} -  40 \eps \ell(Q)$.

\noindent
Indeed, if Cases 1 and 2 both fail, then $|B_Q-B_{Q'}|\leq 40 \eps\ell(Q)\approx \eps \diam(Q)$ as desired.

To begin, suppose that \textbf{Case 1} holds.
By the $\teps$-WHSA condition for $Q$
(see Definition \ref{def2.13} part $(ii)$), we have that
\begin{equation}\label{Bqxtclstxnxt.eq}
|x_0-B_Q|=\dist((x_0,x,t), (B_Q,x,t)) \lesssim K_0 \ell(Q), \qquad \forall\, (x_0,x,t)\in Q.
\end{equation}
Let $(X',t') = (x_0',x',t') \in Q'$.
Using \eqref{Bqxtclstxnxt.eq}
and \eqref{bqbqpclsimphyp.eq}, we obtain
\begin{equation}\label{Bqxtincylinders.eq}
(B_Q, x',t') \in \,\widetilde{P}_Q \cap C_{Q,\eps^{3/4}}^*
 \,\subset\, \widetilde{P}_Q \cap C_{Q,\teps}^*.
\end{equation}
Thus, by the $\teps$-WHSA condition for $Q$
(Definition \ref{def2.13} part $(i)$),
there is a point
$(\tilde{X},\tilde{t}) \in\Sigma$
% $= (\tilde{x}_0, \tilde{x}, \tilde{t}) \in \Sigma$
such that $\dist(\tilde{X},\tilde{t}, (B_Q,x',t')) \!\le \teps\ell(Q)\!=
\!10\eps \ell(Q)\,.$
% \begin{equation}\label{XtBqdist.eq}
% \dist(\tilde{X},\tilde{t}, (B_Q,x',t')) \le \teps\ell(Q)= 10\eps \ell(Q)\,.
% \end{equation}
% In particular, $(\tilde{X},\tilde{t}) \in\Sigma \cap C_{Q,\teps}^*$,
% by \eqref{Bqxtincylinders.eq}, and therefore
% $(\tilde{X},\tilde{t}) \notin \widetilde{H}_Q$,
% by Definition \ref{def2.13} part $(iii)$. Consequently,
Writing $(\tilde{X},\tilde{t}\,) = (\tilde{x}_0, \tilde{x}, \tilde{t}\,)$, we then have
\begin{equation}\label{Bqxtclstxnxt.equu1}
\dist(\tilde{x}, \tilde{t}, x',t') \le 10\eps \ell(Q)\,,
\end{equation}
and also $|\tilde{x}_0 - B_Q|\leq \teps\ell(Q)=10\eps \ell(Q)$, so that
\begin{equation}\label{Xtbelow.eq}
\tilde{x}_0 \leq B_Q+ \teps\ell(Q) \leq 2\teps\ell(Q) \ll\eps^{-1/2}\ell(Q(\sbf^*))\,.
\end{equation}
by \eqref{BQBQS*.eq}, and furthermore
\[
\tilde{x}_0 \ge B_Q - 10\eps \ell(Q) \ge B_{Q'} + 30\eps\ell(Q)\,,
\]
where the last step is just a restatement of Case 1.
In particular,
\begin{equation}\label{Bqxtclstxnxt.equu2}
\tilde{x}_0 - B_{Q'} \ge 30\eps \ell(Q).
\end{equation}
Note that $x_0'\leq B_{Q'},$ by Definition \ref{def2.13} part $(iii)$ applied
to $Q'$.
Thus, using \eqref{Bqxtclstxnxt.equu1},
\eqref{Bqxtclstxnxt.equu2}, and \eqref{Xtbelow.eq}, we see that
\[ \tilde{x}_0-x_0'\geq 3\dist(\tilde{x}, \tilde{t}, x',t'),\qquad  B_{Q'}<\tilde{x}_0
<\eps^{-1/2}\ell(Q(\sbf^*)).\]
Thus, $(\tilde{X},\tilde{t}) \in \gamma_{Q'}$, but
this is a contradiction, since $\gamma_{Q'}\cap \Sigma = \emptyset$ and
$(\tilde{X},\tilde{t}\,) \in \Sigma$, so Case 1 cannot hold.

Next, assume that \textbf{Case 2} holds. Fix any $(X',t') = (x'_0,x',t')\in Q'$. By the $\teps$-WHSA condition for $Q'$
(see Definition \ref{def2.13} part $(i)$),
there exists
$(\tilde{X},\tilde{t}\,) = (\tilde{x}_0, \tilde{x}, \tilde{t}\,) \in \Sigma$
such that
\begin{equation}\label{Xtildclscsb.eq}
\dist(\tilde{X},\tilde{t}, (B_{Q'},x',t')) \le \teps\ell(Q')=10\eps\ell(Q') \le 10\eps\ell(Q).
\end{equation}
As above it follows that
\[\tilde{x}_0 \ge B_{Q'} - 10\eps \ell(Q) \ge B_{Q} + 30\eps\ell(Q),\]
which implies in particular that
$(\tilde{X},\tilde{t})\in \widetilde{H}_{Q}$. Moreover, the $\teps$-WHSA condition for $Q'$ (see Definition \ref{def2.13} part $(ii)$), gives
\[
\dist((B_{Q'},x',t'), (X',t')) \lesssim K_0 \diam(Q') \approx K_0\ell(Q),
\]
which, when combined with \eqref{bqbqpclsimphyp.eq}
and \eqref{Xtildclscsb.eq}, % and the fact that $\eps\leq \epsilon_0= K_0^{-100}$
implies that
$(\tilde{X},\tilde{t}\,) \in C_{Q,\teps}^*$. Hence, we see that
$(\tilde{X},\tilde{t}\,)\in \widetilde{H}_{Q}\cap C_{Q,\teps}^*$,  and
$(\tilde{X},\tilde{t}\,) \in \Sigma$, which contradicts
the $\teps$-WHSA condition for $Q$ (see Definition \ref{def2.13} part $(iii)$).  Thus, Case 2 cannot hold, and
we have proved that  \eqref{bqbqpclsimphyp.eq} implies \eqref{bqbqpclsimpconc.eq}.

Next we assume that \eqref{bqbqpclnotsimphyp.eq} holds and we prove \eqref{bqbqpclnotsimpconc.eq}. By \eqref{bqbqpclnotsimphyp.eq}
\begin{equation}\label{bqbqpntclshyp.eq}
\max\{\diam(Q),\diam(Q')\} = \diam(Q) \le \eps \dist(Q,Q').
\end{equation}
% We first assume $\dist(\pi(Q),\pi(Q')) > 0$.
Let $(X,t)= (x_0,x,t) \in Q$ and let $(X',t') = (x_0',x',t') \in Q'$ be such that
\begin{equation}\label{distpiQ}
\dist(\pi(X,t),\pi(X',t')) \le 2 \dist(\pi(Q),\pi(Q')) +\eta\,,
\end{equation}
where we set $\eta=0$ if
$\dist(\pi(Q),\pi(Q')) > 0$, and $\eta := \eps\diam(Q)$
if $\dist(\pi(Q),\pi(Q')) = 0$.

As above, using Definition \ref{def2.13} part $(ii)$,
% in the proof that \eqref{bqbqpclsimphyp.eq} implies \eqref{bqbqpclsimpconc.eq},
we see that
\begin{equation}\label{distQPQ}
\dist((B_Q,x,t), (X,t)), \dist((B_{Q'},x',t'), (X',t')) \lesssim K_0 \diam(Q)\,,
\end{equation}
and for suitable
$(\tilde{X},\tilde{t}\,)= (\tilde{x}_0,\tilde{x}, \tilde{t}\,)\in \Sigma$,
$(\tilde{X}',\tilde{t}') = (\tilde{x}_0',\tilde{x}', \tilde{t}') \in \Sigma$,
by Definition \ref{def2.13} part $(i)$,
\begin{equation}\label{txtxpcl1.eq}
\dist((B_Q,x,t), (\tilde{X},\tilde{t}\,)), \dist((B_{Q'},x',t'), \tilde X',\tilde t') \le
\teps \ell(Q)= 10 \eps \ell(Q)\,.
\end{equation}
% and such that, by the $\teps$-WHSA condition for $Q$ and $Q'$,
% see Definition \ref{def2.13} part $(ii)$,
% \[\dist((B_Q,x,t), (X,t)), \dist((B_{Q'},x',t'), (X',t')) \lesssim K_0 \diam(Q).\]
Combining \eqref{distQPQ} and \eqref{txtxpcl1.eq},
and then using that $\eps^{-1}\gg K_0$, we obtain
\begin{equation}\label{distXtQ.eq}
\dist(\tilde{X},\tilde{t},Q), \dist(\tilde{X}',\tilde{t}',Q') \lesssim K_0 \diam(Q)
\ll \eps^{-1} \diam(Q).
\end{equation}
% Note also that by  Furthermore, this implies, first using
Furthermore, using \eqref{bqbqpntclshyp.eq} and \eqref{distXtQ.eq}, and
hiding (relatively) small terms, we have
\begin{equation}\label{disttildeXt.eq}
\dist(\tilde{X},\tilde{t}, \tilde{X}',\tilde{t}') \ge (2\eps)^{-1} \diam(Q)\,.
\end{equation}
Hence, by
the definition of $d$ (see \eqref{DSDfn.eq--}),
the first inequality in \eqref{distXtQ.eq}, and then \eqref{disttildeXt.eq},
\[\max\{d(\tilde{X},\tilde{t}\,), d(\tilde{X}',\tilde{t}')\} \lesssim K_0 \diam(Q)
\lesssim K_0 \eps  \dist(\tilde{X},\tilde{t}, \tilde{X}',\tilde{t}').\]
Thus, Lemma \ref{DSLem8.4mod.lem} applies with $\beta \approx K_0 \eps$ and we obtain
\begin{equation}\label{bqbqpclnotsimpalmst1.eq}
|\pi^{\perp}(\tilde X,\tilde t\,) - \pi^{\perp}(\tilde X',\tilde t')| \lesssim K_0^{2}\eps \dist(\pi(\tilde X,\tilde t\,), \pi(\tilde X',\tilde t')).
\end{equation}

Note that \eqref{txtxpcl1.eq} implies
\[
|B_Q - \pi^\perp(\tilde X,\tilde t\,)|,\ |B_{Q'} - \pi^\perp(\tilde X',\tilde t')|
 \le10 \eps \ell(Q) \approx \eps \diam(Q),
 \]
 \[\dist(\pi(\tilde X,\tilde t\,), (x,t) )= \dist(\pi(\tilde X,\tilde t), \pi(X,t))
 \lesssim 10 \eps \ell(Q) \approx \eps \diam(Q),\]
 and
 \[\dist(\pi(\tilde X',\tilde t'), (x',t') )= \dist(\pi(\tilde X',\tilde t'), \pi(X',t')) \lesssim
 10 \eps \ell(Q) \approx \eps \diam(Q)\,.\]
Combining the latter three estimates with \eqref{bqbqpclnotsimpalmst1.eq}
and \eqref{distpiQ}, fixing $\eta$ as above, and using that $\eps\ll K_0^{-1}$,
we obtain \eqref{bqbqpclnotsimpconc.eq}.

Statement \eqref{bqFclsdiamdst.eq} in the lemma remains to be proven. To this
end, we assume $Q \in \sbf^*$ and $(Z,\tau)\in F$.
The triangle inequality and the definition of $d$ give us a cube
$Q' \in \sbf^*$ such that
\begin{equation}\label{dfzslctcb.eq}
 \dist(Z,\tau, X,t)
 \leq \dist(Z,\tau, Q') + \diam(Q')  \le K_0^{-1}\eps\diam(Q)
 \,,\qquad \forall \, (X,t)\in Q'\,.
\end{equation}
Writing $(Z,\tau) = (z_0, z, \tau)$ and using the  $\teps$-WHSA condition for $Q'$ (see Definition \ref{def2.13} part $(ii)$), we then have
\begin{equation}\label{prlmznbqp.eq}
|z_0 - B_{Q'}| = \dist(Z,\tau, \widetilde{P}_{Q'}) \le \dist(Z,\tau, Q') + \diam(Q') + \dist(Q', \widetilde{P}_{Q'})
\lesssim \eps \diam(Q).
\end{equation}
% \begin{multline}\label{prlmznbqp.eq}
% |z_0 - B_{Q'}| = \dist(Z,\tau, \widetilde{P}_{Q'}) \le \dist(Z,\tau, Q') + \diam(Q') +
% \dist(Q', \widetilde{P}_{Q'})\\
% \lesssim \eps \diam(Q).
% \end{multline}
Note that \eqref{dfzslctcb.eq} also implies
\begin{equation}\label{projzprgqp.eq}
\dist(\pi(Z,\tau),\pi(X,t)) \leq  K_0^{-1}\eps \diam(Q)
\le \eps \diam(Q)\,, \quad \forall \,(X,t) \in Q'.
\end{equation}
Now either \eqref{bqbqpclsimphyp.eq}
or \eqref{bqbqpclnotsimphyp.eq} holds for $Q$ and $Q'$  and therefore we may apply the weaker bound \eqref{bqbqpclnotsimpconc.eq}, along with
\eqref{prlmznbqp.eq}, and then \eqref{projzprgqp.eq}, using that
$K_0^2\eps^2 \ll \eps$,
to conclude that
\begin{multline*}
|z_0 - B_{Q}| \le |B_{Q'} - B_{Q}| + |B_{Q'} - z_0|
\lesssim K_0^2\eps\dist(\pi(Q), \pi(Q')) + \eps \diam(Q)
\\[4pt]
\lesssim\,  K_0^2\eps\dist(\pi(Q), \pi(Z,\tau)) \,+ \,\eps\diam(Q).
\end{multline*}
%This competes the proof of the lemma.
\end{proof}

Recall the construction of $\psi$ (see \eqref{2.19}); in particular,
$B_i=B_{Q(i)}$ for $i\in\Lambda$, and $B_i=0$, otherwise.
Recall also that $r_i:= \diam(I_i)$ (see \eqref{nota1}).

\begin{lemma}\label{lemma2.21}
We have
\begin{equation}\label{oldii.eq}
 | B_i - B_j  |\lesssim  \eps  \min ( r_i, r_j ),  \mbox{ if  }
  10  I_i  \cap 10  I_j  \not =\emptyset.
 \end{equation}
Also,
\begin{equation}\label{psilipestonIi.eq}
|\psi(y,s) - \psi(z,\tau)|  \lesssim \eps \dist(y,s,z,\tau), \qquad
\forall \,(y,s), (z,\tau) \in 10 I_i\,,
\end{equation}
\begin{equation}\label{psigradpestonIi.eq}
|\nabla_y\psi(y,s)|  \lesssim \eps\,, \qquad
\forall \,(y,s)  \in 10 I_i\,,
\end{equation}
\begin{equation}\label{psiderestonIi.eq}
|\nabla_y^2 \psi(y,s)| + |\partial_s \psi(y,s)| \lesssim \eps r_i^{-1},
\qquad \forall\, (y,s)
\in 10 I_i\,,
\end{equation}
and
\begin{equation}\label{B-psi}
 |B_i -  \psi(y,s)|  \lesssim \eps r_i\,,
\qquad  \forall\, (y,s) \in 10 I_i\,.
\end{equation}
Moreover, if $i \in \Lambda$, then
\begin{equation}\label{closeIiforpsi.eq}
\dist((\psi(y,s),y,s), \Sigma) \le \eps^{1/2}\diam(I_i), \qquad \forall \,(y,s) \in 10 I_i,
\end{equation}
\end{lemma}

\begin{proof}	
	The proofs of \eqref{oldii.eq}, \eqref{psilipestonIi.eq}, \eqref{psiderestonIi.eq}, and \eqref{closeIiforpsi.eq}, will follow those of,
respectively, \cite[Proof of $(ii)$ Lemma 4.3]{BHHLN-Corona}, \cite[(4.23)]{BHHLN-Corona}, \cite[Lemma 4.7]{BHHLN-Corona}, and
\cite[Lemma 4.6]{BHHLN-Corona}.  The quoted results in \cite{BHHLN-Corona}
are, in turn, based on analogous estimates, in the elliptic case, in
\cite[Chapter 8]{DS1}.

To start the proof of \eqref{oldii.eq}, let $I_i$ and $I_j$ be two cubes such that $10 I_i  \cap 10 I_j  \not =
\emptyset$. Thus $\diam I_i=r_i \approx r_j=\diam I_j$ by Lemma \ref{10-60lemma}.
 If $i \notin \Lambda$ and $j \notin \Lambda$, then
$B_i \equiv 0\equiv B_j$, so the estimate holds trivially.  Otherwise, at least one of $i$ or $j$ belongs to $\Lambda$.
Without loss of generality, we suppose that
$j\in \Lambda$.
In this case, since $10 I_i $ meets $10 I_j $, we may still define $Q(i)$ satisfying
\eqref{eq1+}, \eqref{nota1}, even if
$i\notin \Lambda$:  indeed in that case, we may simply set $Q(i) := Q({\sbf^*})$, and observe that, recalling the notation in \eqref{nota2}, and the definition of the index set $\Lambda$ (see \eqref{la-def}), we
then have $r_i \approx r_j \approx \kappa R_{\sbf^*} = \kappa \diam (Q(i))=\kappa\diam (Q({\sbf^*}))$.  Of course, if $i\in \Lambda$, then $Q(i)$
satisfying
\eqref{eq1+} has already been defined.
In any case, with $j\in \Lambda$, we claim that
\begin{align}\label{6++-uu}\dist(Q(i),Q(j))\lesssim K_0 r_i\approx  K_0 r_j.
\end{align}

To prove the claim, let $(X_j,t_j) \in Q(j)$, $(X_i,t_i) \in Q(i)$.  Set $\beta_0:= (N K_0)^{-1}$, with $N$ a suitably large
constant depending only on the implicit constants in \eqref{eq2.7},  so that Lemma \ref{DSLem8.4mod.lem}
holds with this choice of $\beta_0$.  If
\[
\dist(X_j,t_j ,X_i,t_i ) \leq \beta_0^{-1} \max (\diam Q(j),\diam Q(i)),\]
then  by \eqref{nota1} we see immediately that \eqref{6++-uu} holds. Otherwise,
by definition of $d$,
\[
d(X_j,t_j) \leq \diam Q(j) \leq \,\beta_0 \,\dist(X_j,t_j,X_i,t_i ),\]
and
\[
d(X_i,t_i) \leq \diam Q(i) \leq\, \beta_0 \,\dist(X_j,t_j,X_i,t_i)\,.
\]
Therefore, by Lemma \ref{DSLem8.4mod.lem} we have
\begin{equation*}%\label{eq2.28}
|\pi^\perp(X_j,t_j) -\pi^\perp(X_i,t_i)| \, \lesssim \, \beta_0 K_0
\dist(\pi(X_j,t_j),\pi(X_i,t_i))\,
\lesssim \, r_j\, \approx  \,r_i,
\end{equation*}
where in the last step, % we have taken $K_0$ large enough
% so that $\eps \ll \beta_0$ if $\eps\leq\eps_0=K_0^{-100}$,
we have used \eqref{eq1+} and the fact that
$\dist(I_i,I_j) \lesssim r_j\approx r_i$.  Thus  \eqref{6++-uu} holds, and in this case with constant $\approx 1$.

Having proved \eqref{6++-uu}, we see that \eqref{oldii.eq} follows from Lemma \ref{bqbqpclosediam.lem} since $K_0^2 \ll \eps^{-1}$. With \eqref{oldii.eq} in hand, we prove the remaining estimates.
Recall that by construction, $\nu_j$ is supported in $3I_j$, so in particular, $\nu_j$
vanishes on $10I_i$, unless $10I_j$ (indeed, $3I_j$) meets $10I_i$.

We first prove \eqref{psilipestonIi.eq}. Let $(y,s), (z,\tau) \in 10 I_i$. In this case, by Lemma \ref{10-60lemma}, \eqref{2.16+} and \eqref{oldii.eq}, using
that $B_i$ is constant, and that $\sum_{j}  (\nu_j(y,s) - \nu_j(z,\tau)) =0$,
we have
\begin{multline}\label{2.21++}
| \psi( y, s ) - \psi ( z,\tau) | =
\big|\sum_{j}  (B_j-B_i)(\nu_j(y,s) - \nu_j(z,\tau))\big|\\
\lesssim
  \eps r_i \dist(y,s,z,\tau)r_i^{-1}
\approx \eps\dist(y,s,z,\tau)\,.
\end{multline}
Note that \eqref{psigradpestonIi.eq} now follows immediately from
\eqref{psilipestonIi.eq}.

Next, we prove \eqref{psiderestonIi.eq}. To do this, let $\partial_y^2$ denote any second order spatial derivative,
and note that $\nabla_y \sum_{j} \nu_j = 0$ on $10I_i$. Then, for $(y,s) \in 10I_i$,
again using  \eqref{oldii.eq}, we have
\begin{equation*}
|\partial_y^2 \psi(y,s)|
=
\Big|\partial_y^2 \sum_{j} B_j \nu_j(y,s)\Big|
=
\Big|\sum_{i}(B_j - B_i) \partial_y^2 \nu_j(y,s)\Big|
\lesssim
\eps r_i \sum |\partial_y^2  \nu_j(y,s)|
\lesssim
\eps r_i^{-1},
\end{equation*}
where we used that $|\partial_y^2  \nu_j(y,s)| \lesssim r_j^{-2}$,
 by \eqref{2.16+}, and that $10I_j$ meets $10I_i$ for
all non-trivial terms appearing in the sum, hence
$r_j \approx r_i$, and the number of non-vanishing terms is uniformly bounded.
The estimate for the $t$-derivative may be obtained in the same way.

By similar reasoning, we may verify \eqref{B-psi}:  if
$(y,s) \in 10I_i$, then
by \eqref{oldii.eq},
\[|B_i -  \psi(y,s)| = \Big|\sum_{j} (B_i - B_j) \nu_j(y,s)\Big| \lesssim \eps r_i.\]

Finally, we prove \eqref{closeIiforpsi.eq}. Let
$(y,s) \in 10I_i$ for $i \in \Lambda$.
By the $\teps$-WHSA condition (see Definition \ref{def2.13} parts $(i)$ and $(ii)$),
and \eqref{eq1+}, \eqref{nota1}, there exists $(X,t)\in \Sigma$ such that
\[\dist((B_i,y,s), (X,t)) \le \teps \ell(Q(i)). \]
Thus, using \eqref{B-psi} and \eqref{nota1}, and taking $\eps$ sufficiently small,
we see that
\[\dist((\psi(y,s), y,s), X,t) \le |B_i - \psi(y,s)| +  \teps \ell(Q(i))
\le c \eps r_i \le \eps^{1/2}\diam(I_i)\,.\]
\end{proof}

\subsection{The function $\psi$  and the constructed graph are Lip(1,1/2)} We now prove that $\psi$ is Lip(1,1/2).

\begin{lemma}\label{graphlip.lem}
	Let $\psi:\mathbb R^n\to\mathbb R$ be defined as in
\eqref{2.19}.  Then $\psi: \mathbb R^n\to\mathbb R$ is Lip(1,1/2) with
constant $b_1$ on the order of $\eps^{1/2}$, i.e., \begin{equation}\label{lemma4.26est}
|\psi(y,s)- \psi(z,\tau)| \lesssim \eps^{1/2}\, \dist(y,s,z,\tau), \quad
\forall \, (y,s),(z,\tau) \in \mathbb{R}^n \,.
\end{equation}
\end{lemma}
\begin{proof} The proof is similar to that of \cite[Lemma 4.5]{BHHLN-Corona}
(which in turn was based on arguments in \cite[Chapter 8]{DS1}).
As usual, we set $r_i:= \diam I_i$.

Let us make a preliminary observation.  Recall that we use the notation $C'$
to denote a parabolic cylinder
in $\rn$, see \eqref{ndimcylinder}.
Note that by \eqref{supportpsi} we may suppose that
$(y,s)$ and $(z,\tau)$ both lie in the closure of
$C_*':=C'_{4\kappa R_{\sbf^*}}(x_{ Q({\sbf^*})},t_{ Q({\sbf^*})})$.   Indeed, if both points lie
outside of this closed cylinder, then \eqref{lemma4.26est} is trivial. On the other hand, if one (say $(y,s)$) lies outside,
and the other inside, then we may replace $(y,s)$ in \eqref{lemma4.26est}
by another point $(y',s') \in \partial C'_*$, such that
$\dist(z,\tau,y',s') \leq \dist(z,\tau,y,s)$, since $\psi(y,s) = 0 = \psi(y',s')$ for any
such $(y,s)$ and $(y',s')$.
For example, we may take $(y',s')$ to be
the point on the line segment from $(y,s)$ to $(z,\tau)$ that meets $\partial C'_*$.

Turning to the proof of the lemma, there are four cases to consider:

\smallskip
\noindent
\textbf{Case 1:} $(y,s),(z,\tau) \in \pi(F)$.

\smallskip
\noindent
\textbf{Case 2:}
$(y,s)\in  \mathbb{R}^n\setminus \pi(F) $, with $(y,s) \in 2I_j$,
and  $(z,\tau) \in \pi(F)$.

\smallskip
\noindent
\textbf{Case 3:} $(y,s), (z,\tau) \in  \mathbb{R}^n\setminus \pi(F)$, with $(y,s) \in 2I_j$,
$(z,\tau) \in 2I_k$, and,
without loss of generality, $r_k \leq r_j$, $(z,\tau)\in 10I_j$.

\smallskip
\noindent
\textbf{Case 4:} $(y,s), (z,\tau) \in  \mathbb{R}^n\setminus \pi(F)$, with $(y,s) \in 2I_j$,
$(z,\tau) \in 2I_k$, and, without loss of generality, $r_k \leq r_j$, $(z,\tau)\notin 10I_j$.

Let us now discuss the various cases. \textbf{Case 1} follows immediately from Lemma \ref{DSLem8.4mod.lem} and the definition of $\psi$, and
in this case we obtain \eqref{lemma4.26est} with a smaller constant:
see \eqref{hatlip} and \eqref{2.19}.
  % if, say, $\eps \lesssim \eta$.
%(see Remark \ref{r2.9}).
\textbf{Case 3} (again with an even smaller constant) follows directly from Lemma \ref{lemma2.21}, see \eqref{psilipestonIi.eq}.

It remains to treat \textbf{Case 2} and \textbf{Case 4}, which we shall do more or less simultaneously.
 We remark that if $F$ is nonempty, one can reduce \textbf{Case 4} to \textbf{Case 2}, as in the
 proof of the classical Whitney extension theorem. However, in the present setting, it may be that
$F$ is empty, so we shall treat \textbf{Case 4}  directly.  For the sake of specificity, let us do this first,
as the proof in \textbf{Case 2}  will be similar.

 In \textbf{Case 4} , we decompose
$$|\psi(y,s) - \psi(z,\tau)|
\leq
|\psi(y,s) - B_j|+|B_j - B_k|+|B_k - \psi(z,\tau)|=:
E_1 + E_2 + E_3,$$
In the scenario of Case 4,
$(z,\tau)\in 2I_k\setminus 10I_j$,
and $(y,s)\in 2I_j$, with $r_k\leq r_j$, so that
\begin{equation}\label{4.29}
r_k\leq r_j \lesssim \dist(y,s,z,\tau),
\end{equation}
where the implicit constant only depend on $n$.  By \eqref{B-psi}, we then have
\[
 E_1 + E_3 \,  \lesssim \, \eps  (r_j + r_k)\, \lesssim \,
\eps \dist(y,s,z,\tau)\,.
\]

We now turn to $E_2$. We first recall that at the start of the proof,
we had reduced matters to the case that
$(y,s)$ and $(z,\tau)$ both lie in the closure of
$C_*':=C'_{4\kappa R_{\sbf^*}}(x_{ Q({\sbf^*})},t_{ Q({\sbf^*})})$.
Hence, we may assume that
there exist cubes $Q(j)$ and $Q(k)$ satisfying \eqref{eq1+},
relative to $I_j$ and $I_k$:
indeed, for $j$ or $k$ in $\Lambda$, we have already constructed such cubes, and
otherwise, for example if $j\notin \Lambda$, then we may simply set
$Q(j)=Q(\sbf^*)$, and note that \eqref{eq1+} holds for this choice of $Q(j)$,
as in the proof of Lemma \ref{lemma2.21}.

We now observe that by \eqref{eq1+} and \eqref{4.29}
\[
\dist\big((y,s), \pi(Q(j))\big) \,+ \,
\dist\big((z,\tau), \pi(Q(k))\big) \, \lesssim \,\dist(y,s,z, \tau),
\]
and therefore
\[\dist\big(\pi(Q(j)), \pi(Q(k))\big) \lesssim \dist(y,s,z, \tau).\]
Moreover, again using \eqref{eq1+} and \eqref{4.29} we have
\[\diam(Q(j))\,+\,\diam(Q(k))\, \lesssim \, \dist(y,s,z, \tau).\]
The previous two estimates and
Lemma \ref{bqbqpclosediam.lem}
(specifically, the weaker estimate \eqref{bqbqpclnotsimpconc.eq}) give
\[|B_j - B_k| \lesssim K_0^2\eps \dist(y,s,z, \tau) \lesssim \eps^{1/2}\dist(y,s,z, \tau).\]
Combining our estimates for $E_1,E_2,E_3$, we complete \textbf{Case 4}.

To handle \textbf{Case 2}  we have by the Whitney property of $I_j$ that
\begin{equation}\label{rjwhitc2.eq}
\dist(y,s,z,\tau) \gtrsim r_j.
\end{equation}

Next, we note that since $(z,\tau) \in \pi(F)$, Lemma \ref{DSLem8.4mod.lem} shows that $z_0 = \psi(z,\tau) = \widehat\psi(z,\tau)$ is the
unique real number so that $(Z,\tau)=(z_0,z,\tau) \in F$.
% and that  $d(Z,\tau) = 0$.
Now we have
\[
|\psi(y,s) - \psi(z,\tau)| = |\psi(y,s)-B_j|+  |B_j-z_0 |   =: E_1 + E_2.
\]
As in \textbf{Case 4},
by \eqref{B-psi} and \eqref{rjwhitc2.eq},
we have $E_1\lesssim \eps r_j \lesssim \eps \dist(y,s,z,\tau)$.
To deal with $E_2$ we use \eqref{eq1+} and \eqref{rjwhitc2.eq} to conclude that
\[\dist((y,s), \pi(Q(j))) \lesssim \dist(y,s,z, \tau)\, \text{ and }\, \diam(Q(j)) \lesssim  \dist(y,s,z, \tau).\]
Thus,
\[\dist(z,\tau, \pi(Q(j)) \lesssim \dist(y,s,z, \tau)\, \text{ and }\, \diam(Q(j)) \lesssim  \dist(y,s,z, \tau).\]
Based on this,  Lemma \ref{bqbqpclosediam.lem}
(specifically, \eqref{bqFclsdiamdst.eq}) yields
\[
E_2 = |B_j-z_0 | \lesssim K_0^2\eps \dist(y,s,z, \tau) \lesssim \eps^{1/2}\dist(y,s,z, \tau).
\]
This handles \textbf{Case 2}  and proves the lemma.
\end{proof}

We now define
\[\Gamma_{\sbf^*} := \{(\psi(y,s),y,s): (y,s) \in \rn\}.\]
Next we verify \eqref{eq2.2a}, with $(\eta,K)$ as in \eqref{eq2.2ahala}, for the graphs we have constructed.
\begin{lemma}\label{graphclosecrna.lem}
If $(X,t)\in 2Q$ with $Q\in \sbf^*$, then
\[\dist(X,t, \Gamma_{\sbf^*}) \lesssim K_0 \ell(Q).\]
\end{lemma}
\begin{proof} Since $K_0>1$, it suffices to prove that if $(X,t)\in Q \in \sbf^*$, then
\[\dist(X,t, \Gamma_{\sbf^*}) \lesssim K_0\ell(Q).\]
We write $(X,t) = (x_0,x,t)$ and we divide the argument into two cases, depending on whether $(x,t)\in\rn\setminus \pi(F)$ (equivalently, $D(x,t) >0$) or $(x,t)\in \pi(F)$
 (i.e., $D(x,t) = 0$).

First, suppose that $(x,t)\in\rn\setminus \pi(F)$, so
$D(x,t) > 0$.
Then $(x,t) \in I_i$ for some $i$. Note that
$i \in \Lambda$, since
$(X,t)\in Q\subset Q(\sbf*)$.
Hence, there is a corresponding $Q(i) \in \sbf^*$ with
\begin{equation}\label{grfclsqiqcmp.eq}
\diam(Q(i)) \le 2D(x,t) \le 2d(X,t) \le 2\diam(Q),
\end{equation}
where we have used the definitions of $D$ and $d$. Additionally,
using \eqref{eq1+}, we have
\[\dist(\pi(Q),\pi(Q(i)) \le \dist(x,t,\pi(Q(i)) \lesssim K_0\diam(Q(i)) \lesssim K_0\diam(Q).\]
In turn, by the previous estimate, \eqref{grfclsqiqcmp.eq}, and Lemma \ref{bqbqpclosediam.lem},
\[|B_Q - B_{Q(i)}| \lesssim K_0^3 \eps \diam(Q) \lesssim K_0 \diam(Q)\,,\]
since $\eps\ll K_0^{-2}$.
 % By construction of $\psi$, \eqref{nota1}, and
By Lemma \ref{lemma2.21} (specifically \eqref{B-psi}), and \eqref{nota1},
\[
|B_{Q(i)} - \psi(x,t)| \lesssim \eps r_i \lesssim  \eps K_0 \ell(Q(i)) \ll K_0 \diam(Q).
\]
Finally, by the $\teps$-WHSA condition, see Definition \ref{def2.13} part $(ii)$,
\[|x_0 - B_Q| \lesssim K_0  \diam(Q).\]
Combining all these, we find that
\[\dist(X,t, \Gamma_{\sbf^*}) \le \dist((x_0,x,t),(\psi(x,t), x,t)) = |x_0 - \psi(x,t)| \lesssim K_0 \diam(Q) \approx K_0 \ell(Q).\]
This proves the lemma in the case $(x,t)\in\rn\setminus \pi(F)$. % when  $D(x,t) > 0$.

Now suppose that $(x,t)\in \pi(F)$. % $D(x,t) = 0$.
In this case there is a unique $x_0'$ such that $(x'_0, x,t)\in F$.
% $(X',t') \in \overline{Q(\sbf^*)}$ with $(X',t') = (x'_0, x,t)$ and $d(X',t') = 0$.
% In particular, $(X',t')\in F$, so by construction $(X',t') =
In particular, $x_0'= \psi(x,t)$.
It then suffices to show that $|x_0 - x_0'| \lesssim  K_0 \diam(Q)$. To this end, we proceed as in \textbf{Case 2} of Lemma \ref{graphlip.lem}. Since $d(x_0',x,t) = 0$ and $\pi(x_0,x,t) = \pi(x_0',x,t)$,
we can apply directly
the last estimate in Lemma \ref{bqbqpclosediam.lem} (i.e.,
\eqref{bqFclsdiamdst.eq}) to obtain
\[
|\psi(x,t) - B_Q| \lesssim \eps \diam(Q) \ll K_0 \diam(Q).
\]
Then, using the $\teps$-WHSA condition for $Q$ (see Definition \ref{def2.13}
part $(ii)$), we see that
\[|x_0 - B_Q| \lesssim K_0  \diam(Q),\]
and therefore, % the triangle inequality then shows that
$|x_0 - x_0'| = |x_0 - \psi(x,t)| \lesssim  K_0 \diam(Q),$ as desired.
\end{proof}

\subsection{Additional properties of the construction} Having constructed $\psi_{\sbf^*}:=\psi$, we introduce the region above $\psi_{{\sbf^*}}$,
\begin{equation}\label{ddom}
\Omega_{{\sbf^*}} := \{(y_0,y,s) \in \re\times \re^{n-1}\times \re :
y_0 > \psi_{{\sbf^*}}(y,s)\}.
\end{equation}
We also introduce  a localized version of $\Omega_{{\sbf^*}}$, near $Q(\sbf^*)$,
\begin{equation}\label{ddom+}\Omega'_{\sbf^*}  := \{(y_0,y,s): \psi_{\sbf^*}(y,s) < y_0 < \eps^{-1/4} \ell(Q(\sbf^*)),\, (y,s) \in C'_{\kappa R_{\sbf^*}} (x_{Q(\sbf^*)}, t_{Q(\sbf^*)})\},
\end{equation}
where as we recall, see \eqref{constant1}, $\kappa \approx K_0$. Furthermore, by construction (see \eqref{supportpsi})
\begin{equation*}%\label{psisupport}
\psi_{{\sbf^*}}\equiv 0 \text{ on } \mathbb R^n\setminus C'_{4\kappa R_{\sbf^*}}(x_{ Q({\sbf^*})},t_{ Q({\sbf^*})}),
\end{equation*}
and since $\kappa \approx K_0$, we see by Lemma \ref{graphlip.lem} that
\begin{equation}\label{psiglobbd.eq}
\sup_{(y,s)\in\rn}|\psi_{{\sbf^*}}(y,s)| \lesssim K_0 \eps^{1/2}\ell(Q(\sbf^*)).
\end{equation}
We here prove an important result
concerning the containment of $\Omega'_{\sbf^*}$.

\begin{lemma}\label{intgraphdomlem.lem} Let $\Omega'_{\sbf^*}$ be as in \eqref{ddom+}. Then
\begin{equation}\label{ompsbdmn.eq}
\Omega'_{\sbf^*}  \subseteq \bigcup_{Q \in \sbf^*} \tilde{U}_Q^i,
\end{equation}
where as above $\tilde{U}_Q^i$ is the distinguished component of $\tilde{U}_Q$ which contains $({Y}_{Q}^*,{s}_{Q}^*)$ (see
Definition \ref{def2.11a}, Lemma \ref{refp++} , % Subsection \ref{sspart},
and \eqref{Y*s*component}).
\end{lemma}
\begin{proof}
Let $(y_0,y,s) \in \Omega'_{{\sbf^*}}$. We divide the argument into two cases. First, we assume that $(y,s) \in I_j$ for some $j$. By construction,
if $3I_k$ contains $(y,s)$, then $k \in \Lambda$. Among all such $k$,
we then fix $k \in \Lambda$ that minimizes $B_k=B_{Q(k)}$, i.e.,
$B_k \le B_j$ for all $j$ such that $(y,s) \in 3I_j$.
By definition of $\psi=\psi_{\sbf^*}$ (see \eqref{2.19}),
and of $ \Omega'_{{\sbf^*}}$,
it then follows that
\begin{equation}\label{psibkbdigl.eq}
0 \le \psi(y,s) - B_k \le y_0 - B_k\,,
\end{equation}
i.e., $y_0>B_k$, so that $(y_0, y,s) \in \widetilde{H}_{Q(k)}$.

Next, we observe that by construction, see \eqref{eq1+},
the cube $Q(k)$ satisfies
\[
\dist(I_k, \pi(Q(k)) \le 120 r_k = 120\diam(I_k) \lesssim K_0\ell(Q(k)),
\]
where we used $\kappa \approx K_0$. Therefore,
fixing any $(x_0,x,t) \in Q(k)$, since $(y,s)\in 3I_k$ we have
\begin{equation}\label{igdlemclsqk.eq}
\dist(x,t,y,s) \lesssim K_0\ell(Q(k))\,,
\end{equation}
% Fix any $(x_0,x,t) \in Q(k)$.
Also, by \eqref{psibkbdigl.eq} and
the $\teps$-WHSA condition for $Q(k)$,
we see that $x_0 < B_k < y_0$.
Consider now two subcases. If $y_0 - x_0 > \dist(x,t,y,s)$,
then from the definition of $\gamma_Q$ (see \eqref{gammaQ'def}, \eqref{gammaQdef})
it follows that $(y_0,y,s) \in \gamma_{Q(k)}$,
by \eqref{psibkbdigl.eq} and \eqref{ddom+}.
Therefore, in this subcase Lemma \ref{gqconelem.lem} gives the containment in \eqref{ompsbdmn.eq}. Otherwise, using \eqref{igdlemclsqk.eq}
\[|y_0 - x_0| = y_0 - x_0 \le \dist(x,t,y,s) \lesssim K_0 \ell(Q(k)),\]
and therefore
\[\dist((y_0,y,s), (x_0,x,t)) \lesssim K_0 \ell(Q(k)) \ll \teps^{-2}\ell(Q(k)),\]
hence $(y_0,y,s) \in C_{Q(k),\teps}^*$. Since we have established
that $(y_0, y,s) \in \widetilde{H}_{Q(k)}$, we see by
\eqref{qfavsrrv} that $(y_0, y,s) \in \tilde{U}_{Q(k)}^i$, as desired.
This handles the case that $(y,s) \in I_j$ for some $j$.

We are left with treating the case $(y,s) \in \pi(F)$. In this case,
by Lemma \ref{DSLem8.4mod.lem}, there is a unique point
$(X,s):= (x_0,y,s)\in F$
% be the unique (by Lemma \ref{DSLem8.4mod.lem}) point in $F$ for which
with $\pi(X,s) =(y,s)$, and
% $(X,s) = (\psi(y,s),y,s) = (\widehat\psi(y,s),y,s)$.
$x_0= \psi(y,s) =\widehat\psi(y,s)$.
In particular, $y_0 > x_0$,
by definition of $\Omega'_{{\sbf^*}}$.
Moreover, since $d(x_0,y,s) = 0$, for every $\delta \in (0,\ell(Q({\sbf^*}))$,
there is a cube $Q_\delta \in \sbf^*$ such that
\[
\dist(X,s, Q_\delta)+ \diam(Q_\delta) \lesssim \diam(Q_\delta) \lesssim \delta
\]
(as in the proof of Lemma \ref{DSLem8.4mod.lem}, we can, if necessary,
replace $Q_\delta$ by one of its ancestors to ensure that
$\dist(X,s, Q_\delta)\lesssim \diam(Q_\delta)$).
Then, provided $\delta $ is sufficiently small,
\begin{equation}\label{xfulqdeq.eq}
\dist(X,s, Q_\delta) \le |x_0 - y_0|/8,
\end{equation}
and using the $\teps$-WHSA condition (see Definition \ref{def2.13} part $(ii)$),
\begin{equation}\label{xnbqdeq.eq}
|x_0 - B_{Q_\delta}| \leq c K_0 \delta  < |x_0 - y_0|/2.
\end{equation}
Since $y_0 > x_0$, it follows from \eqref{xnbqdeq.eq} that
\begin{equation}\label{ynbqdeq.eq}
y_0 > B_{Q_\delta} + (y_0 -x_0)/2.
\end{equation}
By \eqref{xfulqdeq.eq}, there is a point
$(\tilde X,\tilde t\,) = (\tilde{x}_0,\tilde{x}, \tilde{t}\,) \in Q_\delta$
such that $\dist(X,s, \tilde{X},\tilde t) < |x_0 - y_0|/4$, so in particular,
since $(X,s) = (\psi(y,s),y,s)$,
\begin{equation*} % \label{txttyseqn.eq}
\dist(y,s,\tilde{x}, \tilde{t}\,) < |x_0 - y_0|/4 = (y_0-x_0)/4\,.
\end{equation*}
Hence, the $\teps$-WHSA condition for $Q_\delta$
(see Definition \ref{def2.13} part $(iii)$), and \eqref{ynbqdeq.eq}, give
\[
y_0 - \tilde{x}_0 \ge y_0 - B_{Q_\delta} > (y_0 - x_0)/2
> \dist( y,s,\tilde{x}, \tilde{t}\,).
\]
This string of inequalities, and the fact that $y_0 < \eps^{-1/4} \ell(Q(\sbf^*)$,
implies that $y_0 \in \gamma_{Q_\delta}$, and therefore
Lemma \ref{gqconelem.lem} gives the containment in \eqref{ompsbdmn.eq}.
\end{proof}

\section{The Lip(1,1/2) graphs $\{\Gamma_{{\sbf^*}}\}$ are {regular} Lip(1,1/2) graphs}\label{sec: reggraph}
In this section we prove that the constructed Lip(1,1/2) graph
$$
\Gamma_{{\sbf^*}}=\{(y_0,y,s): y_0=\psi_{{\sbf^*}}(y,s), (y,s) \in \rn\},
$$
is in fact a regular Lip(1,1/2) graph (recall that $\psi_{\sbf^*}=\psi$,
where the latter is defined in \eqref{2.19}).
We will achieve this by `pushing' the (weak)-$A_\infty$ estimates for the caloric measure associated to $\mathbb R^{n+1}\setminus \Sigma$ onto the approximating graphs,  and then use the main theorem proved \cite{BHMN}. We will use that although our graphs have poor approximation (i.e., with large constant)
in the sense of Lemma \ref{graphclosecrna.lem}, they  have excellent approximation in the sense that \eqref{closeIiforpsi.eq} holds. The latter
allows us to push the estimates efficiently and it enables us to avoid the connectivity hypotheses and additional complications encountered in \cite{HMMTZ}.

\begin{remark}\label{remarkq0}
We recall that by hypothesis, the caloric measure for $\Omega$ is absolutely continuous with respect to $\sigma$, and its density (the ``Poisson kernel") satisfies
the weak Reverse H\"older condition \eqref{wrhpdefeq.eq} for some $q>1$.  In the sequel, we shall let $q_0$ denote this weak reverse H\"older exponent.
\end{remark}

Recall the construction in \eqref{2.19}:
\begin{align}
\psi_{{\sbf^*}}( x,t)&=\widehat \psi_{{\sbf^*}} ( x,t)\mbox{ when } ( x,t) \in \pi (
F_{{\sbf^*}}),\notag\\\label{psiSdef}
\psi_{{\sbf^*}}( x,t) &=
 { \ds \sum_{ k } } \,  B_k
\,   \nu_k ( x,t ) \mbox{ when } ( x,t ) \in \mathbb R^n \sem \pi (
F_{{\sbf^*}}),\end{align}
where by definition $F_{{\sbf^*}}$ is the contact set between $\Gamma_{\sbf^*}$ and $\Sigma$, and $\widehat \psi_{{\sbf^*}}$ is the graph constructed on the contact set.  We also remind the reader that $ \{ \nu_i \} $ is the partition of unity adapted to
$ \{ 2I_i \}$, and that $ \{ I_i \}$ is the collection of
(adapted) Whitney-type parabolic dyadic
cubes decomposing
$\mathbb R^n\setminus \pi (F_{\sbf^*})$
(or simply covering $\rn$ if $F_{\sbf^*}$ is empty).
By Lemma \ref{graphlip.lem},  the Lip(1,1/2) constant of $\psi_{{\sbf^*}}$ is bounded by $c\eps^{1/2}$.  Recall that by construction (see \eqref{supportpsi}),
\begin{equation*}
\psi_{{\sbf^*}}\equiv 0 \text{ on } \mathbb R^n\setminus C'_{4\kappa R_{{\sbf^*}}}(x_{ Q({\sbf^*})},t_{ Q({\sbf^*})}),
\end{equation*}
where $\kappa \approx K_0$.   Given $I_i$ as above and $k\in \mathbb{N}$,
let $\dd_{k}(I_i,\rn)$
be the global grid of parabolic dyadic cubes in $\rn$
of (parabolic) side length $2^{-k}\ell(I_i)$,
and let $\dd_{k}(2I_i)$ denote the cubes
in $\dd_{k}(I_i,\rn)$ that meet $2I_i$.
We will need the  following technical lemma.

\begin{lemma}\label{localcoincwncgrf.lem}
There exist $k_0 \in \mathbb{N}$, % and $b$,
depending only on $n$, such that if $I \in \dd_{k_0}(2I_i)$, then $200I \subset 3I_i$,
and there exists a  {regular} Lip(1,1/2) function  $\psi_I$,
with constants $(b_1,b_2)$ satisfying $b_1+b_2\lesssim \eps$,
such that  $\psi_I=\psi_{{\sbf^*}}$ on $100I$.
\end{lemma}

\begin{proof}
Clearly, we obtain $200I \subset 3I_i$ by simply taking $k_0$ large enough.
The remainder of the proof will proceed via a standard localization argument.
Let $(x_I,t_I)$ denote the center of $I$, and
let $\Phi_I$ be a smooth bump function
with support in $200I$,
such that $\Phi_I\equiv 1$ on $100I$, with $r_i|\nabla_x\Phi_I(x,t)| + r_i^2\left(|\nabla_x^2\Phi_I(x,t)| +
\partial_t\Phi_I(x,t)| \right)\lesssim 1$, where as usual $r_i =\diam(I_i)$.
Set $\tilde{\psi}(x,t):= \left(\psi_{\sbf^*}(x,t) -c_I\right)\Phi_I(x,t)$,
where $c_I:= \psi(x_I,t_I)$, and define
\[
\psi_I(x,t):= \tilde{\psi}(x,t) +c_I\,.
\]
Then $\psi_I=\psi_{\sbf^*}$ on $100I_i$. Moreover,
$\tilde{\psi}$ is supported in $200I\subset 3I_i$, so using
 \eqref{psilipestonIi.eq},
\eqref{psigradpestonIi.eq} and \eqref{psiderestonIi.eq} for $\psi$, we
see that $\psi_I$ satisfies \eqref{psilipestonIi.eq},
\eqref{psigradpestonIi.eq} and \eqref{psiderestonIi.eq} globally
(indeed, $\psi_I$ is non-constant only in $3I_i$).
Now using \eqref{psiderestonIi.eq} for $\psi_I$, and Taylor's Theorem,
we may then routinely verify the
characterization (implicit in \cite[Section 2]{HLN2}) of regular
Lip(1,1/2) functions, based on showing that
\[
d\mu(x,t,r):=
\left[\inf_L r^{-n-1}\iint_{C'_r(x,t)}
\left(\frac{|\psi_I(y,s)-L(y)|}{r}\right)^2dyds\right]  dx dt \frac{dr}{r}
\]
is a Carleson measure on $\re^{n-1}\times \re \times(0,\infty)$, where the infimum runs over all linear functions $L(y)$ depending only on the space variables.
We omit the details.
\end{proof}

After these preliminaries, the section is devoted to the proof of the following proposition.

\begin{proposition}\label{graphregisreg.prop}
For each $\sbf^*$, the graph $\Gamma_{\sbf^*}$ is the graph of a regular Lip(1,1/2) function $\psi_{{\sbf^*}}$ with constants $(b_1,b_2)$, such that
$b_1\lesssim \eps^{1/2}$, and with $b_2$ bounded by a constant that only depends on the allowable parameters and $K_0$.
\end{proposition}

Note that the bound $b_1\lesssim \eps^{1/2}$
has already been established in Lemma \ref{graphlip.lem},
thus, it remains only to verify
that $\psi_{{\sbf^*}}$ is {\em regular} Lip(1,1/2), with bound $b_2$ as stated in the lemma.

Recall the region above $\psi_{{\sbf^*}}$, $\Omega_{{\sbf^*}}$, and its localized version $\Omega'_{{\sbf^*}}$,
introduced in \eqref{ddom} and \eqref{ddom+}, respectively. Also,  Lemma \ref{intgraphdomlem.lem} states that
\begin{equation}\label{contain}
\Omega'_{\sbf^*}  \subseteq \bigcup_{Q \in \sbf^*} \tilde{U}_Q^i.
\end{equation}
In the following we will use the notation that
\begin{align}
&\mbox{$\omega_{{\sbf^*}}$ is the caloric measure for $\Omega_{\sbf^*}$},
\notag\\
&\mbox{$\sigma_{{\sbf^*}}$  is the restriction of $\cH_{\text{par}}^{n+1}$ to $\Gamma_{\sbf^*}$, and} \label{notationsec10}
\\
&\mbox{$\tilde \omega_{{\sbf^*}}$ is the caloric measure for $\Omega'_{\sbf^*}$.}
\notag
\end{align}
 Given a surface ball/cylinder on $\Gamma_{\sbf^*}$ of the form
\begin{align}\label{surfball}
\Delta=\Delta_r = \{(\psi_{\sbf^*}(x,t), x,t): (x,t) \in C'_r(y,s)\},
\end{align}
we let $A_\Delta^+$ and $A_\Delta^-$ denote time-forward and time-backward corkscrew points, respectively, for the domain $\Omega_{{\sbf^*}}$, relative to $\Delta$ and defined as
\begin{align} \label{AD+def}
A_\Delta^+ &= (\psi_{\sbf^*}(y,s) + 200r, y, s + (8r)^2),\\
A_\Delta^- &= (\psi_{\sbf^*}(y,s) + 200r, y, s - (8r)^2). \notag
\end{align}

With $\sbf^*$, $\Gamma_{\sbf^*}$, and $\psi_{\sbf^*}$ fixed, the main theorem of \cite{BHMN} implies that to prove Proposition \ref{graphregisreg.prop} it suffices to prove the following proposition.

\begin{proposition}\label{pushprop.prop}
There exist constants $q> 1$ and $A_* \ge 1$, depending only on the allowable parameters and $K_0$, such that the following holds. For every surface ball
$\Delta$ on $\Gamma_{\sbf^*}$ as in \eqref{surfball} it holds that $\omega_{\sbf^*}^{A_\Delta^+} \ll \sigma_{\sbf^*}$ on $\Delta$, and  $k^{A_\Delta^+}_{{\sbf^*}}(Z,\tau) := {\d\omega^{A_\Delta^+}_{{\sbf^*}}}/{\d\sigma_{{\sbf^*}}}(Z,\tau)$ satisfies
\begin{equation}\label{RHsbourg.eq}
\iint_{\Delta} (k^{A_\Delta^+}_{{\sbf^*}}(Z,\tau))^q\, \d\sigma_{{\sbf^*}}(Z,\tau) \leq A_\ast (\sigma_{{\sbf^*}}(\Delta))^{1-q},
\end{equation}
% \begin{equation}\label{RHsforgraph.eq}
% \left(\bariint_{\Delta} (k^{A_\Delta^+}_{{\sbf^*}}(Z,\tau))^q\,
% \d\sigma_{{\sbf^*}}(Z,\tau) \right)^{1/q} \le
% A_*\bariint_{\Delta} k^{A_\Delta^+}_{{\sbf^*}}(Z,\tau)\, \d\sigma_{{\sbf^*}}(Z,\tau).
% \end{equation}
\end{proposition}

 We remark that by a well-known argument based on the
change of pole formula (see \eqref{comp2pre+++} below)
for caloric measure in Lip(1,1/2) domains,
\eqref{RHsbourg.eq} yields the reverse H\"older estimate
 \begin{equation}\label{RHq}
\left(\bariint_{\Delta_s} (k^{A_\Delta^+}_{{\sbf^*}}(Z,\tau))^q\,
\d\sigma_{{\sbf^*}}(Z,\tau) \right)^{1/q} \lesssim
 A_*\bariint_{\Delta_s} k^{A_\Delta^+}_{{\sbf^*}}(Z,\tau)\, \d\sigma_{{\sbf^*}}(Z,\tau)\,,
 \end{equation}
 for all $\Delta_s := \{(\psi_{\sbf^*}(z,\tau), z,\tau): (z,\tau) \in C'_s(x,t)\}$,
 with $C_s'(x,t)\subset C'_r(y,s)$ (thus $\Delta_s\subset \Delta=\Delta_r$).
 Conversely, we note that in the special case $\Delta_s =\Delta$, the
 reverse H\"older estimate directly implies \eqref{RHsbourg.eq},
 since $\omega^{A_\Delta^+}_{{\sbf^*}}(\Delta)\approx 1$, by Lemma \ref{Bourgain}
 and Harnack's inequality in $\Omega_{\sbf^*}$.

% To start the proof of Proposition \ref{pushprop.prop},
% we note that given $\Delta$ as in \eqref{surfball}, we have
% \[\bariint_{\Delta} k^{A_\Delta^+}_{{\sbf^*}}(Z,\tau)\,
% \d\sigma_{{\sbf^*}}(Z,\tau) =\frac {\omega^{A_\Delta^+}_{{\sbf^*}}(\Delta)}
% {\sigma_{{\sbf^*}}(\Delta)}\approx (\sigma_{{\sbf^*}}(\Delta))^{-1},\]
% as $\omega^{A_\Delta^+}_{{\sbf^*}}(\Delta)\approx 1$.
% Hence to verify \eqref{RHsforgraph.eq} it suffices to prove that
% \begin{equation}\label{RHsbourg.eq}
% \iint_{\Delta} (k^{A_\Delta^+}_{{\sbf^*}}(Z,\tau))^q\,
% \d\sigma_{{\sbf^*}}(Z,\tau) \leq A_\ast (\sigma_{{\sbf^*}}(\Delta))^{1-q},
% \end{equation}
% for $q> 1$ and $A_* \ge 1$ as stated.
To start the proof of Proposition \ref{pushprop.prop},
we divide the verification of \eqref{RHsbourg.eq} into four cases:
\begin{align*}
\mbox{{\bf Case 0: }}&C'_r(y,s) \subseteq F,\\
\mbox{{\bf Case 1: }}&\mbox{$C'_r(y,s) \subseteq 2I_{i}$ for some $I_i$},\\
\mbox{{\bf Case 2: }}&\mbox{$C'_r(y,s) \subseteq C'_{10R_{\sbf^*}}(x_{Q({\sbf^*})}, t_{Q({\sbf^*})})$ and Case 0 and 1 do not hold},\\
\mbox{{\bf Case 3: }}&\mbox{$C'_r(y,s) \not\subseteq C'_{10R_{\sbf^*}}(x_{Q({\sbf^*})}, t_{Q({\sbf^*})})$ and Case 0 and 1 do not hold.}
\end{align*}
Below we present the proof of \eqref{RHsbourg.eq} for each case in separate sections. However, the arguments  assume (significant) familiarity with boundary estimates for non-negative solutions,  caloric measure and the Green function, and we below first briefly review the estimates needed in the context of the Lip(1,1/2) domains under consideration.

\subsection{Boundary estimates for non-negative solutions in $\Omega_{{\sbf^*}}$ and $\Omega_{{\sbf^*}}'$} In general, we note that the boundary behavior of non-negative solutions,  for the heat equation but also for more general linear uniformly parabolic equations with space and time dependent coefficients,  have been studied intensively in Lipschitz cylinders and in Lip(1,1/2) domains over the years, see \cite{FGS,FS,FSY,LewMur,N1997}.
Results include Carleson type estimates, the relation between the associate parabolic measure and the Green function, the backward in time Harnack inequality, the doubling of parabolic measure, boundary Harnack principles (local and global), and  Hölder continuity up to the boundary of quotients of non-negative solutions vanishing on the lateral boundary.

When  verifying \eqref{RHsbourg.eq} in the
subsequent sections, the regions present in the arguments are
our original open set
$ \Omega\subset \mathbb{R}^{n+1}$ with boundary $\Sigma=\pom$,
the (global) Lip(1,1/2) graph domain $\Omega_{{\sbf^*}}$, its (local) counterpart $\Omega_{{\sbf^*}}'$, and some additional auxiliary Lip(1,1/2) graph domains. The Lip(1,1/2) constants,
for  the Lip(1,1/2) graph domains considered,
will be bounded by a constant that depends only
on the allowable parameters and $K_0$. Therefore,
in the following all implicit constants will depend only
on these constants unless otherwise stated. By definition,
$\omega_{{\sbf^*}}$ and $\tilde \omega_{{\sbf^*}}$  are the
caloric measures for $\Omega_{\sbf^*}$ and $\Omega'_{\sbf^*}$,
respectively, and we let $G_{{\sbf^*}}(\cdot,\cdot)$ and
$\tilde G_{{\sbf^*}}(\cdot,\cdot)$ be the Green's functions
for $\Omega_{\sbf^*}$ and $\Omega'_{\sbf^*}$, respectively.
In the following outline we consistently assume that
$$
\tilde\Delta_{\tilde r}=:\tilde\Delta\subseteq\Delta:=\Delta_r \subset \partial\Omega'_{{\sbf^*}},
$$
with $\pi(\Delta)\subseteq C'_{100R_{\sbf^*}}(x_{Q(\sbf^*)}, t_{Q(\sbf^*)})$ and
$\tilde r\leq r/10$, that is, $\tilde\Delta$ is a smaller surface ball/cylinder centered on
$\partial \Omega'_{{\sbf^*}}$, and contained in the large surface ball/cylinder $\Delta$.  Observe that by definition of $\Delta$, this means that we are restricting our attention to the portion of $\partial \Omega'_{{\sbf^*}}$ that coincides with
$\Gamma_{\sbf^*}$, and that even
a large (roughly $K_0$-fold) dilate of $\Delta$ is contained in
$\partial \Omega'_{{\sbf^*}}\cap \Gamma_{\sbf^*}$ (see \eqref{ddom+}).
Note that by construction,
$A_\Delta^\pm \in \Omega_{\sbf^*}'\subset \Omega$
(see \eqref{AD+def} and \eqref{contain}).

We then first note that
\begin{align}\label{comp2pre}
\omega^{A_\Delta^+}_{{\sbf^*}}(2\tilde\Delta)\approx\omega^{A_\Delta^+}_{{\sbf^*}}(\tilde\Delta)\approx \, \tilde r^{\,n}\,
G_{{\sbf^*}}(A_\Delta^+,A_{\tilde\Delta}^+)\approx \, \tilde r^{\,n}\,
G_{{\sbf^*}}(A_\Delta^+,A_{\tilde\Delta}^-),
\end{align}
and that
\begin{align}\label{comp2pre+}
\tilde\omega^{A_\Delta^+}_{{\sbf^*}}(2\tilde\Delta)\approx\tilde\omega^{A_\Delta^+}_{{\sbf^*}}(\tilde\Delta)\approx\, \tilde r^{\,n}\,
\tilde G_{{\sbf^*}}(A_\Delta^+,A_{\tilde\Delta}^+)\approx \, \tilde r^{\,n}\,
\tilde G_{{\sbf^*}}(A_\Delta^+,A_{\tilde\Delta}^-).
\end{align}
Furthermore,
\begin{align}\label{comp2pre++}
\frac{G_{{\sbf^*}}(A_\Delta^+,A_{\tilde\Delta}^-)}{\tilde G_{{\sbf^*}}(A_\Delta^+,A_{\tilde\Delta}^-)}\approx\frac{G_{{\sbf^*}}(A_\Delta^+,A_{\tilde\Delta}^+)}{\tilde G_{{\sbf^*}}(A_\Delta^+,A_{\tilde\Delta}^+)}
\approx\frac{G_{{\sbf^*}}(A_\Delta^+,A_{\frac 14\Delta}^+)}{\tilde G_{{\sbf^*}}(A_\Delta^+,A_{\frac 14\Delta}^+)}\approx 1,
\end{align}
\begin{align}\label{comp2pre+++}
\frac{\omega^{A_\Delta^+}_{{\sbf^*}}(\widehat\Delta)}{\omega^{A_{\tilde\Delta}^+}_{{\sbf^*}}(\widehat\Delta)}\approx \omega^{A_\Delta^+}_{{\sbf^*}}(\tilde\Delta), \mbox{ and } \frac{\tilde\omega^{A_\Delta^+}_{{\sbf^*}}(\widehat\Delta)}{\tilde\omega^{A_{\tilde\Delta}^+}_{{\sbf^*}}(\widehat\Delta)}\approx \tilde\omega^{A_\Delta^+}_{{\sbf^*}}(\tilde\Delta)\mbox{ whenever }\widehat\Delta\subseteq\tilde\Delta.
\end{align}
Observe that \eqref{comp2pre} and \eqref{comp2pre+} summarize the doubling property of the caloric measures, their relations to the Green functions, and the
time reversed Harnack inequality for $G_{{\sbf^*}}(A_\Delta^+,\cdot)$ and $\tilde G_{{\sbf^*}}(A_\Delta^+,\cdot)$.
Furthermore, \eqref{comp2pre++} is a
consequence of \eqref{comp2pre}, \eqref{comp2pre+},
the comparison/boundary Harnack principle for non-negative solutions,
and bounds for the Green function near the pole.
For the heat equation in Lip(1,1/2) domains, these ingredients (except for the
Green function bounds, which are standard)
are all established in \cite{FGS}.
Display \eqref{comp2pre+++} is
 sometimes referred to as the ``change of pole formula",
 and may be gleaned from
\eqref{comp2pre} (or \eqref{comp2pre+}), and the
the comparison/boundary Harnack principle.
The estimates in \eqref{comp2pre}-\eqref{comp2pre+++}
have been extended to the case of divergence form
parabolic operators with bounded measurable, time-dependent
variable coefficients in \cite{FS,FSY} (in Lipschitz cylinders)
and in \cite{N1997} (in Lip(1,1/2) domains).
We remark that these estimates remain valid in certain more
general classes of domains, e.g., see \cite[Section 3]{HLN2} in the case of the heat equation. % See also Section 3 in \cite{BHMN}.

\subsection{Pushing the (weak)-$A_\infty$ property of caloric measure to the graph: Case 0} In this case $\Delta \subset \partial\Omega'_{{\sbf^*}} \cap \Sigma$.
By \eqref{contain}, $\Omega'_{{\sbf^*}} \subset \Omega$.
The maximum principle therefore implies that
 \begin{align}\label{comp1-}\tilde\omega^{A_\Delta^+}_{{\sbf^*}}(\tilde\Delta)\leq
\omega^{A_\Delta^+}(\tilde\Delta)\mbox{ for all }\tilde\Delta\subseteq\Delta\,.
\end{align}
Hence, since
 $\omega^{A_\Delta^+} \ll \sigma$ on $\Delta$, we see that
 $\tilde\omega^{A_\Delta^+}_{{\sbf^*}}\ll\sigma_{\sbf^*}$ on $\Delta$,
 so $\tilde k^{A_\Delta^+}_{{\sbf^*}}(Z,\tau)
 :={\d\tilde\omega^{A_\Delta^+}_{{\sbf^*}}}/{\d\sigma_{{\sbf^*}}}(Z,\tau)$ exists
 $\sigma_{{\sbf^*}}$-a.e. on $\Delta$.
 Furthermore, \eqref{comp1-} implies that
\begin{align}\label{comp1}
\tilde k^{A_\Delta^+}_{{\sbf^*}}(Z,\tau)\leq
k^{A_\Delta^+}(Z,\tau),
\end{align}
for $\sigma_{{\sbf^*}}$-a.e $(Z,\tau)\in\Delta$. Let $\tilde\Delta\subseteq\Delta$,
with $\tilde{r} \leq r/10$.
Using \eqref{comp2pre}-\eqref{comp2pre++} we see that
\begin{align}\label{comp2}\frac{\omega^{A_\Delta^+}_{{\sbf^*}}(\tilde\Delta)}{\tilde\omega^{A_\Delta^+}_{{\sbf^*}}(\tilde\Delta)}\approx\frac
{G_{{\sbf^*}}(A_\Delta^+,A_{\tilde\Delta}^+)}{\tilde G_{{\sbf^*}}(A_\Delta^+,A_{\tilde\Delta}^+)}\approx1,
\end{align}
and as a consequence, $\omega^{A_\Delta^+}_{{\sbf^*}}\ll\sigma_{\sbf^*}$ on $\Delta$, and $k^{A_\Delta^+}_{{\sbf^*}}(Z,\tau)={\d\omega^{A_\Delta^+}_{{\sbf^*}}}/{\d\sigma_{{\sbf^*}}}(Z,\tau)$ exists $\sigma_{{\sbf^*}}$-a.e. on $\Delta$. Combining these facts we deduce that
\begin{align}\label{comp3}k^{A_\Delta^+}_{{\sbf^*}}(Z,\tau)\approx
\tilde k^{A_\Delta^+}_{{\sbf^*}}(Z,\tau)\leq k^{A_\Delta^+}(Z,\tau),
\end{align} for $\sigma_{{\sbf^*}}$-a.e $(Z,\tau)\in\Delta$.
Hence, recalling that $q_0>1$ is the weak reverse H\"older exponent for
$k^{A_\Delta^+}$, and using that $\Delta \in \partial\Omega_{\sbf^*}\cap\Sigma$ in the present scenario, we have, with $q=q_0$,
\begin{equation}\label{RHsbourg.eq+}
\iint_{\Delta} (k^{A_\Delta^+}_{{\sbf^*}}(Z,\tau))^q\, \d\sigma_{{\sbf^*}}(Z,\tau) \lesssim \iint_{\Delta} (k^{A_\Delta^+}(Z,\tau))^q\, \d\sigma_{{\sbf^*}}(Z,\tau)\leq A_\ast (\sigma_{{\sbf^*}}(\Delta))^{1-q},
\end{equation}
by the assumption on $k^{A_\Delta^+}$. In particular, this proves \eqref{RHsbourg.eq} with $q=q_0$.

\subsection{Pushing the (weak)-$A_\infty$ property of caloric measure to the graph: Case 1} In this case $C'_r(y,s) \subseteq 2I_{i}$ for some $I_i$. Given $I_i$ as above, we let $k_0$ be as in  Lemma \ref{localcoincwncgrf.lem}, and we
denote by $\{I_{i,j}\}$ the $2^{(n+1)k_0}$ cubes in $\dd_{k_0}(I_i)$.
Applying Lemma \ref{localcoincwncgrf.lem} we can construct functions  $\{\psi_{I_{i,j}}\}$, which for simplicity we denote  $\{\psi_{i,j}\}$.
We further split Case 1 into 2 sub-cases:  either  $r \ll \ell(I_i)$ so that $C'_r(y,s) \subseteq 10I_{i,j}$ for some $I_{i,j}$, or $r \gtrsim \ell(I_i)$.

 In this case $C'_r(y,s) \subseteq 2I_{i}$ for some $I_i$.
Let $k_0$ be as in  Lemma \ref{localcoincwncgrf.lem}, and
denote by $\{I_{i,j}\}$ the roughly $2^{(n+1)k_0}$
cubes of side length $2^{-k_0}\ell(I_i)$
in $\dd_{k_0}(2I_i)$. Applying Lemma \ref{localcoincwncgrf.lem} with $I=I_{i,j}$,
we construct the associated regular Lip(1,1/2) functions
$\{\psi_{I_{i,j}}\}$, which for
simplicity we denote  $\{\psi_{i,j}\}$.
We further split Case 1 into 2 subcases:  either  $r \ll \ell(I_i)$ so that $C'_r(y,s) \subseteq 10I_{i,j}$ for some $I_{i,j}$, or $r \gtrsim \ell(I_i)$.

First,  we assume that $r \ll \ell(I_i)$, with
\begin{equation}\label{ddom++a}
\mbox{$C'_r(y,s) \subseteq 10 I_{i,j}$ for some $I_{i,j}$}.
\end{equation}

\begin{remark}\label{remark-general-cyl}
Although we have previously
fixed $C'_r(y,s)$, and $\Delta =\Delta_r$ as in \eqref{surfball},
it will be convenient to allow
$C'_r(y,s)$ to be, for the moment, {\em any} cylinder in $\rn$
satisfying \eqref{ddom++a}, with $I_{i,j}\in\dd_{k_0}(2I_i)$ for any $I_i$, and with
$\Delta =\Delta_r$ defined relative to $C'_r(y,s)$ as
in \eqref{surfball}.
\end{remark}

By Lemma \ref{localcoincwncgrf.lem},
we have that $\psi_{{\sbf^*}} = \psi_{i,j}$ on $100I_{i,j}$, where $\psi_{i,j}$ is a {regular} Lip(1,1/2) function with uniformly bounded constants (independent of $i,j$ and ${\sbf^*}$). Set
\begin{equation}\label{ddom++}
\Omega_{i,j} := \{(y_0,y,s): y_0 > \psi_{i,j}(y,s), (y,s) \in \rn\},\quad
 \Gamma_{i,j}:=\partial \Omega_{i,j}\,,
\end{equation}
and note that by construction (using
\eqref{surfball}, \eqref{AD+def}, \eqref{ddom++a},
and the fact that $\psi_{{\sbf^*}} = \psi_{i,j}$ on $100I_{i,j}$),  we see that
$5\Delta:=\Delta_{5r}\subset \Gamma_{i,j}\cap \Gamma_{\sbf^*}$, and that
$A_\Delta^+$ is a time forward corkscrew
point relative to $\Delta$, in $\Omega_{i,j}$.
Let $\omega_{i,j}$ and $G_{i,j}$ denote the caloric measure and Green function
for $\Omega_{i,j}$, and let $\sigma_{i,j}$ be the restriction of $\cH_{\text{par}}^{n+1}$ to $\Gamma_{i,j}$.
Then by the main result of
\cite{LewMur}, we have $\omega_{i,j}\ll\sigma_{i,j}$, and in addition,
the associated Radon-Nikodym derivative
$k_{i,j}^{A_\Delta^+}:=\d \omega_{i,j}^{A_\Delta^+}/\d \sigma_{i,j}$ satisfies
(in particular), the analogue of \eqref{RHsbourg.eq}, i.e.,
\begin{equation}\label{RHsbourg.eq++}
\iint_{\Delta} (k_{i,j}^{A_\Delta^+}(Z,\tau))^{q_1}\, \d\sigma_{i,j}(Z,\tau) \lesssim (\sigma_{i,j}(\Delta))^{1-q_1}\,,
\end{equation}
  for some $q_1 > 1$ which is
  independent of $\Delta $ and $i,j$.

% We claim that in the present scenario,
% $\omega^{A_\Delta^+}_{{\sbf^*}}\ll\sigma_{\sbf^*}$ on $\Delta$,
% and that $k^{A_\Delta^+}_{{\sbf^*}}(Z,\tau)
% :={\d\omega^{A_\Delta^+}_{{\sbf^*}}}/{\d\sigma_{{\sbf^*}}}(Z,\tau)$ exists
% $\sigma_{{\sbf^*}}$-a.e.~on $\Delta$, with
% $k^{A_\Delta^+}_{{\sbf^*}} \approx k_{i,j}^{A_\Delta^+}$.  Indeed,

Recall that, as noted above,
$5\Delta \subset \Gamma_{i,j}\cap\Gamma_{\sbf^*}$, and
$A_\Delta^+$ is a time forward corkscrew
point relative to $\Delta$, in $\Omega_{i,j}$.  Hence,
\eqref{comp2pre}-\eqref{comp2pre++} continue to hold with
$\omega_{i,j}, G_{i,j}$ in place of either $\hm_{\sbf^*}, G_{\sbf^*}$
or $\tilde\hm_{\sbf^*}, \tilde{G}_{\sbf^*}$, for
$\tilde{\Delta}\subset \Delta$.  Consequently, the analogue of \eqref{comp2}, with
 $\omega_{i,j}, G_{i,j}$ in place of
$\tilde\hm_{\sbf^*}, \tilde{G}_{\sbf^*}$, also holds, so letting $\tilde{\Delta}$
shrink to a point, we find that
$\omega^{A_\Delta^+}_{{\sbf^*}}\ll\sigma_{\sbf^*}$ on $\Delta$, with
 \begin{align}\label{comp4}
 k^{A_\Delta^+}_{{\sbf^*}}(Z,\tau)
 \approx k_{i,j}^{A_\Delta^+}(Z,\tau)\,,\quad
 \sigma_{{\sbf^*}}\,\text{- a.e.~}(Z,\tau)\in\Delta\,.
\end{align}
Thus, using Remark \ref{remark-general-cyl},
we conclude that \eqref{RHsbourg.eq}
holds, with $q=q_1$ as in \eqref{RHsbourg.eq++}, for any cylinder satisfying
\eqref{ddom++a}.   In particular, we obtain the desired conclusion in the first subcase.

We now consider the second subcase, and we return
to our original assumption that $C_r'(y,s)$ and the corresponding
surface ball $\Delta=\Delta_r$ have been fixed.
In the present subcase, we have
$r \gtrsim \ell(I_i)$ (so in fact $r \approx \ell(I_i)$, since
$C_r'(y,s) \subset 2I_i$ in the scenario of Case 1).  In essence, we will reduce matters to the first subcase, via  Remark \ref{remark-general-cyl}.
We first establish an estimate (see \eqref{deltairhq.eq} below), that holds in general for any $I_i$.
Let
\begin{equation} \label{Aidef}
A_i^+ := \left(\psi_{\sbf^*}(x_i,t_i) + 20\diam(I_i), x_i, t_i + \big(20\diam(I_i)\big)^2\right),
\end{equation}
where $(x_i,t_i)$ is the center of $I_i$, and we let
\[
2\Delta_i := \{(\psi_{\sbf^*}(x,t),x,t): (x,t) \in 2I_i\}.
\]
As above, let $k_0=k_0(n)$
be as in Lemma \ref{localcoincwncgrf.lem}, and again let
$\{I_{i,j}\}$ denote the roughly $2^{(n+1)k_0}$ cubes in $\dd_{k_0}(2I_i)$.
We cover $2I_i$ by a collection of cylinders $\{C'_{i,k}\}_k$ of size
$2^{-N-k_0}\ell(I_i)$,
where $N$ is a purely dimensional constant chosen large enough that for each $k$,
$C'_{i,k}\subset 10I_{i,j}$ for some $j$, and where
the cardinality of the collection also depends only on dimension.  Let
$$
\Delta'_{i,k}:= \{(\psi_{\sbf^*}(x,t),x,t): (x,t) \in C'_{i,k}\}
$$
be the surface ball on $\Gamma_{{\sbf^*}}$ above $C'_{i,k}$,
and note that the point $A_i^+$ is forward in time
from $\Delta'_{i,k}$, at a parabolic distance roughly
$\diam(I_i)\approx 2^{N+k_0}\diam(C'_{i,k})$.
Let $A_{\Delta'_{i,k}}^+$ be a time forward
corkscrew point in $\Omega_{\sbf^*}$,
relative to $\Delta'_{i,k}$, and observe that by choosing $N$ a bit larger,
if need be, we may assume that $A_i^+$ is forward in time from
$A_{\Delta'_{i,k}}^+$, also by a parabolic distance on the order of
$\diam(I_i)\approx 2^{N+k_0}\diam(C'_{i,k})$.
Consequently, using the change of pole inequality in
\eqref{comp2pre+++}, with
$(A_\Delta^+,A_{\tilde \Delta}^+)$
replaced by
$(A_i^+,A_{\Delta'_{i,k}}^+)$, and with $\widehat\Delta\subseteq \Delta'_{i,k}$,
we find that for each $k$,
\begin{align}\label{comp5}
k^{A_i^+}_{{\sbf^*}}(Z,\tau) \, \approx\,
\omega_{{\sbf^*}}^{{A_i^+}}(\Delta'_{i,k}) \,
k^{A_{\Delta'_{i,k}}^+}_{{\sbf^*}}(Z,\tau) \,\approx\,
k^{A_{\Delta'_{i,k}}^+}_{{\sbf^*}}(Z,\tau)\,,
\quad \sigma_{{\sbf^*}}\text{-a.e.~}(Z,\tau)\in \Delta'_{i,k}\,.
\end{align}
where in the second step we have used \eqref{Bourgain} and Harnack's inequality in
$\Omega_{\sbf^*}$.
Thus, since $2\Delta_i \subset \cup_k \Delta'_{i,k}$, we have,
with $q_1>1$ based on the result of
\cite{LewMur} as in the first subcase,
\begin{multline}\label{deltairhq.equu}
\iint_{2\Delta_i} \big(k^{A_i^+}_{{\sbf^*}}(Z,\tau)\big)^{q_1}\,
\d\sigma_{{\sbf^*}}(Z,\tau) \lesssim
\sum_k \iint_{\Delta'_{i,k}} \big(k^{A_i^+}_{{\sbf^*}}(Z,\tau)\big)^{q_1}\,
\d\sigma_{{\sbf^*}}(Z,\tau)
\\
\approx\sum_k \iint_{\Delta'_{i,k}}
\big(k^{A_{\Delta'_{i,k}}^+}_{{\sbf^*}}(Z,\tau)\big)^{q_1}\,
\d\sigma_{{\sbf^*}}(Z,\tau).
\end{multline}
 Recall that $C'_{i,k}\subset 10I_{i,j}$ for some $j$, by construction.  Hence, by
 Remark \ref{remark-general-cyl} and the argument in the first subcase,
 we deduce that \eqref{RHsbourg.eq} holds with
 $\Delta_{i,k}'$ in place of $\Delta$, and with $q=q_1$.
 Since the constants $N$ and $k_0$, as well as
the cardinality of the set of indices $k$, depend only on dimension, and since
 $\diam(\Delta_{i,k}')\approx_{k_0,N} \diam(2\Delta_i)$, we may sum
 over $k$ in \eqref{deltairhq.equu} to conclude that
\begin{equation}\label{deltairhq.eq}
\iint_{2\Delta_i} (k^{A_i^+}_{{\sbf^*}}(Z,\tau))^{q_1}\, \d\sigma_{{\sbf^*}}(Z,\tau) \lesssim (\sigma_{{\sbf^*}}(\Delta_i))^{1-{q_1}}.
\end{equation}

\begin{remark}\label{Deltaiest}
Note for future reference that the preceding argument shows that
\eqref{deltairhq.eq} holds for every $I_i$, with $q_1>1$ arising
from the result of
\cite{LewMur}, as in the first subcase.
\end{remark}

Recall that in the present scenario, $C_r'(y,s)\subset 2I_i$ with
$r\approx \diam(I_i)$.  Hence $\Delta=\Delta_r\subset 2\Delta_i$, with
$\sigma_{\sbf^*}(\Delta)\approx \sigma_{\sbf^*}(2\Delta_i)$, and
by Lemma \ref{Bourgain} and Harnack's inequality
\[
\hm_{\sbf^*}^{A_\Delta^+}(\Delta)\approx 1\approx \hm_{\sbf^*}^{A_i^+}(\Delta)\,.
\]
Hence, by the change of pole formula in \eqref{comp2pre+++}
we also have
\begin{align*} % \label{comp5ll}
k^{A_\Delta^+}_{{\sbf^*}}(Z,\tau)  \approx  k^{A_{i}^+}_{{\sbf^*}}(Z,\tau)\,,
\qquad \sigma_{{\sbf^*}}\text{-a.e.} \, (Z,\tau)\in \Delta\,.
\end{align*}
Combining these observations with \eqref{deltairhq.eq}, we obtain
\begin{multline*}
\iint_{\Delta} (k^{A_\Delta^+}_{{\sbf^*}}(Z,\tau))^{q_1}\,
\d\sigma_{{\sbf^*}}(Z,\tau) \approx
\iint_{\Delta} (k^{A_i^+}_{{\sbf^*}}(Z,\tau))^{q_1}\, \d\sigma_{{\sbf^*}}(Z,\tau)
\\[4pt] \leq
\iint_{2\Delta_i} (k^{A_i^+}_{{\sbf^*}}(Z,\tau))^{q_1}\, \d\sigma_{{\sbf^*}}(Z,\tau)
 \lesssim (\sigma_{{\sbf^*}}(\Delta_i))^{1-q} \approx
 (\sigma_{{\sbf^*}}(\Delta))^{1-q_1}\,,
\end{multline*}
thus establishing \eqref{RHsbourg.eq}, in this case with $q=q_1$, as above, arising from the application of the result in \cite{LewMur}.

\subsection{Pushing the (weak)-$A_\infty$ property of caloric measure to the graph: Case 2} Given $\Delta= \{(\psi_{\sbf^*}(x,t), x,t): (x,t) \in C'_r(y,s)\}$, in this case we have $C'_r(y,s) \not\subseteq 2I_{i}$ for all $I_i$. Let
\[\mathcal{I}_\Delta = \{i:  C'_r(y,s)\cap I_i \neq \emptyset\}.  \]
A bit of geometry shows that if $C'_r(y,s)\cap I_i\neq\emptyset$, and
$4r < \ell(I_i)$, then  $C'_r(y,s) \subset 2I_i$. Thus,
if  $i \in \mathcal{I}_\Delta$, then $4r \ge \ell(I_i)$.
Let $i \in \mathcal{I}_\Delta$, and consider again
$$
\Delta_i =  \{(\psi_{\sbf^*}(x,t),x,t): (x,t) \in I_i\}.
$$
The time coordinate of every
point in $\Delta_i$ lies in the interval
$[t_i - \ell(I_i)^2/2, t_i + \ell(I_i)^2/2]$.  Thus, since
$C'_r(y,s)\cap I_i\neq\emptyset$ it must be the case that $s \ge t_i -  \ell(I_i)^2/2 - r^2$. Hence,  using also that $r \ge \ell(I_i)/4$, we have
\begin{equation*}
s + (8r)^2 \ge t_i + 63r^2 - \frac12\ell(I_i)^2 \ge \big(t_i +  \frac12\ell(I_i)^2\big) +
 63r^2 - \ell(I_i)^2\ge
 \big(t_i +  \frac12\ell(I_i)^2\big) +
 47r^2 \,.
\end{equation*}
Thus, $A^+_\Delta$ is forward in time from each $\Delta_i$,
at a parabolic distance of the scale $r\gtrsim \ell(I_i)$.
Therefore, by the change of pole formula \eqref{comp2pre+++},
\begin{equation}\label{djpolechange.eq}
k^{A^+_\Delta}_{{\sbf^*}}(Z,\tau)  \approx \omega_{{\sbf^*}}^{A^+_\Delta}(\Delta_i) k^{A_{i}^+}_{{\sbf^*}}(Z,\tau)\,,\quad \sigma_{{\sbf^*}}\text{-a.e. } (Z,\tau)\in \Delta_i\,,
\end{equation}
where $A_i^+$ is the time-forward corkscrew point relative to $\Delta_i$ (or $2\Delta_i$) defined in \eqref{Aidef}.  To proceed, we note that
\begin{multline*}
\iint_{\Delta} (k^{A_\Delta^+}_{{\sbf^*}}(Z,\tau))^q\, \d\sigma_{{\sbf^*}}(Z,\tau)
\\[4pt]
\le \, \sum_{i \in \mathcal{I}_\Delta}\iint_{\Delta_i} (k^{A_\Delta^+}_{{\sbf^*}}(Z,\tau))^q\, \d\sigma_{{\sbf^*}}(Z,\tau)\,+\,
\iint_{\Delta \cap F} (k^{A_\Delta^+}_{{\sbf^*}}(Z,\tau))^q\, \d\sigma_{{\sbf^*}}(Z,\tau)
%\\[4pt]
\,=:\,I + II\,.
\end{multline*}
Recalling that $F$ is the contact set
between $\Sigma$ and  $\Gamma_{{\sbf^*}}$, and
arguing as in Case 0,
we can use the maximum principle and \eqref{comp2}
to deduce that \eqref{comp3} holds $\sigma_{\sbf^*}$-a.e.~ on $\Delta\cap F$, whence by the weak-$A_\infty$ property of $\omega$ (see Remark \ref{remarkq0}),
and the ADR property of $\Sigma$, we see that
$$
II \lesssim (\sigma_{{\sbf^*}}(\Delta))^{1-q}\,,\,\,\mbox{ with }q=q_0\,.
$$
Thus, to verify \eqref{RHsbourg.eq}, it suffices to prove that also
$$
I \lesssim (\sigma_{{\sbf^*}}(\Delta))^{1-q}\,,\,\,\mbox{ with } q=\min(q_0,q_1)\,,
$$
where $q_1>1$ is the exponent in \eqref{deltairhq.eq}
(see also Remark \ref{Deltaiest}).

To this end, let $(X_i,t_i) = (\psi_{{\sbf^*}}(x_i,t_i), x_i,t_i)$ be the center of
 the ``surface cube" $\Delta_i$.
 Using \eqref{closeIiforpsi.eq} we see that
 there exists $(Z_i, \tau_i) \in \Sigma$ such
 that $\|(X_i, t_i) - (Z_i,\tau_i)\| \lesssim \eps^{1/2} \ell(I_i)$.
 Given a small constant $\delta > 0$ to be chosen, we let
$$\widetilde{\Delta}_i :=  C(Z_i, \tau_i,{\ell(I_i)/100}) \cap \Sigma,$$
and $$\Delta'_i: = \{(\psi_{{\sbf^*}}(x,t),x,t): (x,t) \in C_{\delta\ell(I_i)}'(x_i,t_i)\}.$$
 We emphasize that
$\widetilde{\Delta}_i$ is a surface ball on $\Sigma$, while $\Delta'_i$ is a surface ball on $\Gamma_{{\sbf^*}}$. Furthermore, if $\delta\ll 1/100$, then $\widetilde{\Delta}_i$ is much larger than $\Delta'_i$.  Recalling that $\psi_{{\sbf^*}}$ is Lip(1,1/2) with small constant, of the order $\eps^{1/2}$ (see Lemma \ref{graphlip.lem}), we now choose $\delta=\delta(\eps)$ small enough that
 \begin{equation}\label{xtztclose}
\|(X, t) - (Z_i,\tau_i)\| \lesssim \eps^{1/2} \ell(I_i), \quad \forall (X, t) \in \Delta'_i.
 \end{equation}
Consequently, for $\eps$ sufficiently small, recalling that
$\Omega_{\sbf^*}'\subset \Omega$ (see \eqref{contain}),
we find that
\[
\omega^{(X,t)}(\widetilde{\Delta}_i) \gtrsim 1, \quad \forall (X, t) \in \Delta'_i\,,
\]
by Lemma \ref{Bourgain}, and
therefore, by the maximum principle,
\[
\tilde{\omega}_{{\sbf^*}}^{(X,t)}(\Delta'_i) \lesssim \omega^{(X,t)}(\widetilde{\Delta}_i), \quad \forall (X, t) \in \Omega'_{{\sbf^*}}\,,
\]
where we recall that $\tilde{\omega}_{{\sbf^*}}$ is the
caloric measure for $\Omega'_{{\sbf^*}}$, the localizeded
version of $\Omega_{{\sbf^*}}'$. Moreover,
as $\Omega'_{{\sbf^*}}$ coincides with
$\Omega_{{\sbf^*}}$ in a large neighborhood of
$\Delta_i$, we can use \eqref{comp2pre}-\eqref{comp2pre++} to obtain that
$$
\tilde{\omega}_{{\sbf^*}}^{A_\Delta^+}(\Delta'_i) \approx
\omega_{{\sbf^*}}^{A_\Delta^+}(\Delta'_i).
$$
 The doubling property of $\omega_{\sbf^*}$ implies that
 $\omega_{{\sbf^*}}^{A_\Delta^+}(\Delta_i)\lesssim_\delta
 \omega_{{\sbf^*}}^{A_\Delta^+}(\Delta'_i)$,
so combining the inequalities in the last two displays, we deduce that
\begin{equation}\label{cmpofomsomdj.eq}
\omega_{{\sbf^*}}^{A_\Delta^+}(\Delta_i) \,\lesssim_\eps\, \omega^{A_\Delta^+}(\widetilde{\Delta}_i)\,,
\end{equation}
since $\delta=\delta(\eps)$.
Recall that the cubes in the collection $\{I_i\}$ have disjoint interiors,
and note that by construction,
the projection of $\widetilde{\Delta}_i$ is contained in the interior of
$I_i$.  Thus, the elements of the collection $\{\widetilde{\Delta}_i\}$ are disjoint.

We are now ready to estimate term $I$.
Using \eqref{djpolechange.eq}, \eqref{cmpofomsomdj.eq}, \eqref{deltairhq.eq}
(and Remark \ref{Deltaiest}), the ADR property of $\Sigma$,
and our choice of $q=\min(q_0,q_1)$, we have
\begin{multline*}
I = \sum_{i \in \mathcal{I}_\Delta} \iint_{\Delta_i}(k^{A_\Delta^+}_{{\sbf^*}}(Z,\tau))^q\, \d\sigma_{{\sbf^*}}(Z,\tau)
\\ \lesssim \sum_{i \in \mathcal{I}_\Delta} (\omega_{{\sbf^*}}^{A^+_\Delta}(\Delta_i))^q \iint_{\Delta_i} (k^{A_i^+}_{{\sbf^*}}(Z,\tau))^q\, \d\sigma_{{\sbf^*}}(Z,\tau)
 \lesssim
 \sum_{i \in \mathcal{I}_\Delta} \omega^{A_\Delta^+}(\widetilde{\Delta}_i)^q (\sigma_{{\sbf^*}}(\Delta_i))^{1-q}
\\
 \approx \sum_{i \in \mathcal{I}_\Delta} \sigma(\widetilde{\Delta}_i) \left(\bariint_{\widetilde{\Delta}_i}k^{A_\Delta^+}(Z, \tau) \, \d\sigma(Z,\tau) \right)^q
  \lesssim\iint_{\cup_{i\in\mathcal{I}_\Delta} \widetilde{\Delta}_i}(k^{A_\Delta^+}(Z, \tau))^q
  \, \d\sigma(Z,\tau),
\end{multline*}
where we have used that the surface balls $\widetilde{\Delta}_i$ are disjoint,
and where our use of \eqref{cmpofomsomdj.eq} means that
some implicit constants may depend on $\eps$.
We now claim that $\cup_{i\in \mathcal{I}_\Delta}\widetilde{\Delta}_i$
is contained in a surface ball
$\Delta_{Nr}(Z^*,\tau^*):=C_{Nr}(Z^*,\tau^*)\cap \Sigma$, where $N$ is a purely dimensional constant.  To see this, note first that
$\cup_{i\in \mathcal{I}_\Delta}I_i \subset C_{N'r}'(y,s)$,
for a dimensional constant $N'$,
by the definition of $\mathcal{I}_\Delta$ and the fact that
$4r\geq \ell(I_i)$.  Hence,
$$
\cup_{i\in \mathcal{I}_\Delta}\Delta_i
\subset \Delta_{N'r}:=
\Delta_{N'r}\big((\psi_{\sbf^*}(y,s),y,s\big)
:=  \{(\psi_{\sbf^*}(x,t), x,t): (x,t) \in C'_{N'r}(y,s)\}\,.
$$
For any $i\in \mathcal{I}_\Delta$, let $(Z_{i},\tau_{i})\in\Sigma$
be the point selected in \eqref{xtztclose}, for the cube $I_{i}$.  Then
$$
\dist\big((Z_{i},\tau_{i}),\Delta_{N'r}\big)\leq
\dist\big((Z_{i},\tau_{i}),\Delta_i\big)\lesssim \eps^{1/2} \ell(I_i)\lesssim
\eps^{1/2} r\,,
$$
and therefore, fixing $(Z^*,\tau^*):=(Z_i,\tau_i)$ for an
arbitrary $i\in\mathcal{I}_\Delta$, using the definition of
$\widetilde{\Delta}_i$, and choosing $N$ somewhat larger than $N'$, we obtain the claim.

Hence, by the weak reverse Hölder inequality for $k^{A_\Delta^+}$, and the ADR property of $\Sigma$, we see that
\[\iint_{\cup_i \widetilde{\Delta}_i}(k^{A_\Delta^+}(Z, \tau))^q
\, \d\sigma(Z,\tau) \lesssim (\sigma_{\sbf^*}(\Delta))^{1-q}.\]
This concludes the proof in Case 2 of \eqref{RHsbourg.eq} with
$q=\min(q_0,q_1)$.  We remark that our
use of \eqref{cmpofomsomdj.eq} means that the constant $A_*$ in
\eqref{RHsbourg.eq}, and hence the constant $b_2$ in \eqref{1.7}, for the function
$\psi=\psi_{\sbf^*}$, may depend on $\eps$. This dependence is harmless,
since $\eps$ is a sufficiently small degree of freedom, whose value may be fixed
depending only on allowable parameters and $K_0$.

\subsection{Pushing the (weak)-$A_\infty$ property of caloric measure to the graph: Case 3} In this case $C'_r(y,s) \not\subseteq C'_{10R_{\sbf^*}}(x_{Q({\sbf^*})}, t_{Q({\sbf^*})})$. We divide Case 3 into three subcases.

\noindent
{\bf Case 3a:} $C'_{4r}(y,s) \cap C'_{4\kappa R_{\sbf^*}}(x_{Q({\sbf^*})}, t_{Q({\sbf^*})}) = \emptyset$. In this case, $\psi_{\sbf^*}\equiv 0$ on
$C'_{4r}(y,s)$ since $\psi_{\sbf^*}$ is zero outside
$C'_{4\kappa R_{\sbf^*}}(x_{Q({\sbf^*})}, t_{Q({\sbf^*})})$.
Set $\Omega_0:= \{(x_0,x,t): x_0>0,\, (x,t)\in \re^{n-1}\times \re\}$, for which solvability is classical.
We can then use \eqref{comp2pre}-\eqref{comp2pre++}, with $\Omega_0$ and its associated caloric measure and Green function replacing
$\Omega_{\sbf^*}'$, $\tilde{\hm}_{\sbf^*}$ and $\tilde{G}_{\sbf^*}$, to compare
$k_{\sbf^*}$ to the Poisson kernel for $\Omega_0$.
We omit the remaining details, which are by now familiar.

\noindent
{\bf Case 3b:} $C'_{4r}(y,s) \cap C'_{4\kappa R_{\sbf^*}}(x_{Q({\sbf^*})}, t_{Q({\sbf^*})}) \neq \emptyset$ and $r \le 10^5\kappa R_{\sbf^*}$.
Let $R_{\sbf^*}' = c_n\min\{r,R_{\sbf^*}\}$
where $c_n \in (0,1/50)$ is a constant only depending on $n$ and to be chosen below. Note that $50R_{\sbf^*}' \le r \lesssim_{n,\kappa} R_{\sbf^*}'$ and $R_{\sbf^*}' \le c_n R_{\sbf^*}$.   We now cover $C'_{r}(y,s)$ by a finite collection of cylinders
$\{C'_k\}_{k \in \mathcal{K}} := \{C'_{R_{\sbf^*}'}(z_k,\tau_k)\}_{k \in \mathcal{K}}$, of uniformly bounded cardinality depending,
at most, only on $n$ and $\kappa$, such that
 \begin{equation*} % \label{Ck'contain}
 C_k' \subset C'_{2r}(y,s)\,,\qquad \forall \, k \in \mathcal{K}\,.
 \end{equation*}
We divide the index set $\mathcal{K}$
into two disjoint sets $\mathcal{K}_1$ and $\mathcal{K}_2$.
We say that
 $$
\mbox{$k \in \mathcal{K}_1$, if $C'_k \subset C'_{10R_{\sbf^*}}(x_{Q({\sbf^*})}, t_{Q({\sbf^*})})$,}
$$
and we say that
$$
\mbox{$k \in \mathcal{K}_2,$
if $C'_k \cap \left(\rn\setminus
C'_{10R_{\sbf^*}}(x_{Q({\sbf^*})}, t_{Q({\sbf^*})})\right) \neq \emptyset$}.
$$
Note that, since $R_{\sbf^*}'\leq (1/50)R_{\sbf^*}$, it follows by the definition of
$C_k'$ that
$$
\mbox{$k \in \mathcal{K}_2 \implies$
$C'_k \cap C'_{5R_{\sbf^*}}(x_{Q({\sbf^*})}, t_{Q({\sbf^*})}) = \emptyset$}.
$$
We now let $\Delta'_k := \{(\psi_{\sbf^*}(x,t),x,t): (x,t) \in C'_k\}$.
Observe that if $k \in \mathcal{K}_1$, then
we can apply either Case 0, Case 1, or Case 2 to
$\Delta'_k$. Consequently, for $k \in \mathcal{K}_1$,
we have  \eqref{RHsbourg.eq} with
$\Delta = \Delta'_k$.

\noindent {\bf Claim:} If $k \in \mathcal{K}_2$, then $C'_k$ is contained in $2I_i$, for some $I_i$.

\noindent Assuming this claim, we see that Case 1 applies to each
$C_k'$ with $k\in\mathcal{K}_2$,
so that \eqref{RHsbourg.eq} now holds with $\Delta = \Delta'_k$,
for every $k\in \mathcal{K}$.
Moreover, by construction, $r\lesssim_{n,\kappa} \diam(\Delta_k')< r/25$, and furthermore,
$\Delta\subset \cup_{k\in\mathcal{K}}\Delta_k'\subset 2\Delta =\Delta_{2r}$.
Thus, by a now familiar argument using
the change of pole formula \eqref{comp2pre+++} and Lemma \ref{Bourgain},
we see that
\[
k^{A_\Delta^+}_{{\sbf^*}}(Z,\tau)\,\approx\, k^{A^+_{\Delta_k'}}_{{\sbf^*}}(Z,\tau)\,,
\qquad \sigma\text{-a.e.~}(Z,\tau)\in \Delta_k'\,.
\]
Therefore, assuming the claim, and using the fact that
$\#\mathcal{K}\lesssim_{n,\kappa}1$, we conclude that
\begin{multline*}
\iint_{\Delta} \big(k^{A_\Delta^+}_{{\sbf^*}}(Z,\tau)\big)^q\,
\d\sigma_{{\sbf^*}}(Z,\tau) \,
 \leq\,
\sum_{k\in\mathcal{K}}
\iint_{\Delta'_k} \big(k^{A_{\Delta}^+}_{{\sbf^*}}(Z,\tau)\big)^q\,
\d\sigma_{{\sbf^*}}(Z,\tau)
\\
 \approx\,
\sum_{k\in\mathcal{K}}
\iint_{\Delta'_k} \big(k^{A_{\Delta_k'}^+}_{{\sbf^*}}(Z,\tau)\big)^q\,
\d\sigma_{{\sbf^*}}(Z,\tau)\, \lesssim \,
\sum_{k\in\mathcal{K}} \big(\sigma_{{\sbf^*}}(\Delta'_k)\big)^{1-q} \,\approx \,
\big(\sigma_{{\sbf^*}}(\Delta)\big)^{1-q}\,.
\end{multline*}

Thus, to conclude our analysis in subcase 3b, it remains only to prove the claim, i.e., it suffices to prove that if $c_n$ is sufficiently small, and $k \in \mathcal{K}_2$, then $C'_k$ is contained in $2I_i$ for some $I_i$.
To this end, note that
$\pi(Q({\sbf^*})) \subset C'_{3R_{\sbf^*}}(x_{Q({\sbf^*})}, t_{Q({\sbf^*})})$,
and that, since
$k \in \mathcal{K}_2$, % and the definition of $D(x,t)$
$$(z_k,\tau_k) \in \mathbb R^{n}\setminus
C'_{5R_{\sbf^*}}(x_{Q({\sbf^*})}, t_{Q({\sbf^*})})\,.
$$
Therefore, by definition of $D(x,t)$, we see that $D(z_k,\tau_k) \ge R_{\sbf^*}$.
Thus, $(z_k,\tau_k) \in I_i$ for some $I_i$. Furthermore, by \eqref{10-60.eq},
% using that $60 \diam(I_i) \ge D(z_k,\tau_k)$,
it follows that
$$60c_n\diam(I_i) \ge c_nD(z_k,\tau_k)\ge c_nR_{\sbf^*}
\ge R_{\sbf^*}'\,.$$
% as $R_{\sbf^*}' \le c_nR_{\sbf^*}$.
Therefore, since  $C'_k = C'_{R_{\sbf^*}'}(z_k,\tau_k)$, we conclude that
 $C'_k \subset 2I_i$,
for $c_n$ sufficiently small (depending on $n$). This proves the claim.

\noindent
{\bf Case 3c:} $C'_{4r}(y,s) \cap C'_{4\kappa R_{\sbf^*}}(x_{Q({\sbf^*})}, t_{Q({\sbf^*})}) \neq \emptyset$ and $r > 10^5\kappa R_{\sbf^*}$. In this case we use a well-known characterization of $A_\infty$ weights, based on which it is enough to prove there exist two uniform constants $\upsilon, \upsilon' \in (0,1)$ such that
\begin{equation}\label{ainftildebb.eq}
\sigma_{{\sbf^*}}(E) \ge \upsilon\sigma_{{\sbf^*}}(\Delta') \implies \omega_{\sbf^*}^{A_\Delta^+}(E) \ge \upsilon'\omega_{\sbf^*}^{A_\Delta^+}(\Delta'), \quad \forall E \subset \Delta', \Delta' \subset \Delta.
\end{equation}
To this end, we first note that if $C_r'(y,s)$ and its corresponding
$\Delta=\Delta_r$ satisfy
any of the previous cases or subcases (Cases 0, 1, 2, 3a, or 3b), then every
$C_{s}'(x,t)\subset C_r'(y,s)$ and corresponding
$\Delta_{s}\subset \Delta_r$ will also satisfy one of those cases or subcases
(not necessarily the same one), as may be verified by inspection of the case and subcase definitions.   Hence, by the analysis in the previous cases,
we see that \eqref{RHsbourg.eq} holds with $\Delta$ replaced by $\Delta_s$,
for every $\Delta_s \subset \Delta_r$, whenever
any of those previous cases applies to $\Delta_r$ itself.
Therefore, as noted earlier,
we may use a change of pole argument to deduce that the reverse H\"older inequality \eqref{RHq} holds for all $\Delta_s$ contained
in any such $\Delta=\Delta_r$.
Moreover, if for  given $\Delta',\Delta''$, with
$\Delta''\subseteq \Delta'$,
we have the reverse H\"older inequality
 \begin{equation*} % \label{RHq}
\left(\bariint_{\Delta''} (k^{A_{\Delta'}^+}_{{\sbf^*}}(Z,\tau))^q\,
\d\sigma_{{\sbf^*}}(Z,\tau) \right)^{1/q} \lesssim \,
 \bariint_{\Delta''} k^{A_{\Delta'}^+}_{{\sbf^*}}(Z,\tau)\, \d\sigma_{{\sbf^*}}(Z,\tau)\,,
 \end{equation*}
 then for any $\Delta\supset \Delta'$,
 we may immediately obtain the analogous inequality with
$k_{\sbf^*}^{A_{\Delta'}^+}$ replaced by $k_{\sbf^*}^{A_{\Delta}^+}$,
by a now familiar use of the pole change formula \eqref{comp2pre+++}.
In particular, given $\Delta,\Delta'$, with $\Delta'\subset \Delta$, we have that
\eqref{RHq} holds with $\Delta_s=\Delta'$ if the latter satisfies
any of the previous cases or subcases, even if $\Delta$ itself does not.
In turn, it then holds that for each
$\Delta'$ falling into one of the previous cases,
there exist uniform constants
$\upsilon_0,\upsilon_0' \in (0,1)$
such that
\begin{equation}\label{ainftildebb0.eq}
\sigma_{{\sbf^*}}(E) \ge \upsilon_0\,\sigma_{{\sbf^*}}(\Delta') \implies \omega_{\sbf^*}^{A_{\Delta'}^+}(E) \ge \upsilon_0' \,\omega_{\sbf^*}^{A_{\Delta'}^+}(\Delta'), \quad \forall E \subset \Delta', \Delta' \subset \Delta\,,
\end{equation}
as one may verify by a well known argument, taking complements and using the reverse H\"older estimate \eqref{RHq}.

This means that we need only
 consider the case when $\Delta'$ itself is in Case 3c. Our
 strategy will be to reduce matters to Case 3a, to complete the argument.
Writing
$$
\Delta' = \{(\psi_{\sbf^*}(x,t), x,t): (x,t) \in C'_{r'}(z,\tau)\}\,,
$$
and noting that we have $r' > 10^5\kappa R_{\sbf^*}$, we
see that in this case $C'_{r'}(z,\tau)$
contains a cylinder $C'_{r''}(z_1,t_1)$
with $r'' = 10^{-4}r'$  so that $C'_{4r''}(z_1,t_1)$
does not meet $C'_{4\kappa R_{\sbf^*}}(x_{Q({\sbf^*})}, t_{Q({\sbf^*})})$. Setting
$$
\Delta'' =\{(\psi_{\sbf^*}(x,t), x,t): (x,t) \in C'_{r''}(z_1,t_1)\}\,,
$$
we have  $\sigma_{{\sbf^*}}(\Delta'') \approx \sigma_{{\sbf^*}}(\Delta')$, and therefore, if $\upsilon$ is close enough to $1$,
\[
\sigma_{{\sbf^*}}(E) \ge \upsilon\sigma_{{\sbf^*}}(\Delta') \implies \sigma_{{\sbf^*}}(E\cap \Delta'') \ge \upsilon_0\,\sigma_{{\sbf^*}}(\Delta'') \implies \omega_{\sbf^*}^{A_{\Delta''}^+}(E\cap\Delta'') \ge \upsilon'_0 \, \omega_{\sbf^*}^{A_{\Delta''}^+}(\Delta''),
\]
where we have used Case 3a for $\Delta''$ and \eqref{ainftildebb0.eq}.
Finally, a change of pole argument, combined with the local doubling property of $\omega_{{\sbf^*}}$, gives that for some uniform $\upsilon'$,
\[\omega_{\sbf^*}^{A_{\Delta''}^+}(E\cap\Delta'') \ge \upsilon_0' \,\omega_{\sbf^*}^{A_{\Delta''}^+}(\Delta'') \implies \omega_{\sbf^*}^{A_{\Delta}^+}(E) \ge \upsilon'\omega_{\sbf^*}^{A_{\Delta}^+}(\Delta').\]
% Here we have used that $\omega_{\sbf^*}^{A_{\Delta}^+}(\Delta') \approx \omega_{\sbf^*}^{A_{\Delta}^+}(\Delta'')$.
This concludes the proof of the last subcase, and hence the proof of Case 3 and the proposition.

\section{Proof of Lemma \ref{keylemma1}: packing the maximal cubes}\label{sec:pack}

In this section, we prove Lemma \ref{keylemma1}.  Once this step has been completed, we will have concluded the proof of
Proposition \ref{finalprop} (see Remark \ref{prop8.2proofoutline}),
and hence also that of Theorem \ref{main.thrm}.
 In the sequel, for $Q\in\dd(\Sigma)$,
we shall use the notation $\ch(Q)$ to denote the
collection of dyadic children of $Q$.
A key ingredient is the following.

\begin{lemma}\label{keylemma2}
There exists a constant $\zeta_0 \in (0,1)$, depending only on $K_0, \eps$,
and the allowable parameters,
 such that for any given caloric measure tree $\sbf'$, and for each
 graph tree $\sbf^*\subset \sbf'$, % the following holds: % for every $\sbf^*$
 there exists a set $A_{\sbf^*} \subset Q(\sbf^*)$ satisfying the following properties:
\begin{equation}\label{ASsetamp.eq}
\sigma(A_{\sbf^*}) \,\ge \,\zeta_0\, \sigma(Q(\sbf^*))\,,
\end{equation}
and if $Q\in\sbf^*$ with $Q \cap A_{\sbf^*} \neq \emptyset$, then
\begin{equation}\label{childcap.eq}
\big(\psi_{\sbf^*}(x_{Q'}, t_{Q'}) + \ell(Q'), x_{Q'}, t_{Q'}\big) \in U_{Q'}\,,
\qquad \forall \, Q' \in \ch(Q)\,,
\end{equation}
where $(X_{Q'}, t_{Q'}) = (x_{Q'}^0, x_{Q'}, t_{Q'})$ is the center of $Q'$ in the coordinates generated by $P_{Q(\sbf^*)}$.
\end{lemma}

\begin{proof}
We may assume (indeed, we have done so already)
 that $P_{Q(\sbf^*)} = \{0\} \times \mathbb{R}^n$; see \eqref{ffina} and \eqref{PHdef}.
 Let $(X_{Q(\sbf^*)}, t_{Q(\sbf^*)}) = (x^0_{Q(\sbf^*)}, x_{Q(\sbf^*)},t_{Q(\sbf^*)})$ be the center of $Q(\sbf^*)$.
We remind the reader that $r_Q\approx \ell(Q)$
is the radius of the surface cylinder
$\Delta_Q=\Delta(X_Q,t_Q,r_Q)\subset Q$
defined in \eqref{cube-ball}, \eqref{cube-ball2}.
For convenience of notation, set $r_{\sbf^*}:= r_{Q(\sbf^*)}$,
so in particular,
\[
\Delta_{Q(\sbf^*)}=\Delta(X_{Q(\sbf^*)},t_{Q(\sbf^*)},r_{\sbf^*})\,.
\]
 We define two cylinders in $P_{Q(\sbf^*)}$:
%\[C'_0:= C_{10^{-100} c_0\ell(Q)}(x_{Q(\sbf^*)},t_{Q(\sbf^*)}),\]
\[C'_1:= C'_{10^{-10} r_{\sbf^*}}(x_{Q(\sbf^*)},t_{Q(\sbf^*)})\]
and
\[C'_2: = C'_{10^{-5} r_{\sbf^*}}(x_{Q(\sbf^*)},t_{Q(\sbf^*)})\]

% Continuing with the proof of Lemma \ref{keylemma2}, we
Recall that \eqref{Y*Qest}
holds, in particular, with $Q= Q(\sbf^*)$, and therefore
by the $\eps$-WHSA plane construction it follows that
\begin{equation}\label{PQcls2cent.eq}
\dist\big((x^0_{Q(\sbf^*)}, x_{Q(\sbf^*)},t_{Q(\sbf^*)}), P_{Q(\sbf^*)}\big) \le
10^{-100} r_{\sbf^*}\,.
\end{equation}
% Since $\psi = \psi_{\sbf^*}$ is $\eps^{1/2}$ parabolic Lipschitz
% with support in a ball of parabolic diameter roughly
% $K_0r_{\sbf^*} \le \eps^{-1/100}r_{\sbf^*}$, then for
For notational convenience, set $\psi:=\psi_{\sbf^*}$.
By \eqref{psiglobbd.eq}, for $\eps$ sufficiently small we have
\[
\sup_{(y,s)\in\rn}|\psi(y,s)| \ll 10^{-100} r_{\sbf^*}\,.\]
Combining the latter bound with
\eqref{PQcls2cent.eq}, we see that
if $(X,t), (Y,s) \in \Sigma$, and $(z,\tau) \in C'_2$, are a triple of points satisfying
\begin{equation}\label{xtysztaurelate.eq}
\dist((Y,s),(X,t)) \lesssim \eps^{1/4}\ell(Q(\sbf^*)),
\, \text{ and } \,\dist((X,t), (\psi(z,\tau), z,\tau)) \lesssim \eps^{1/4}\ell(Q(\sbf^*))\,,
\end{equation}
 then it holds that $(Y,s) \in \Delta_{Q(\sbf^*)}$
 (since $\ell(Q(\sbf^*)) \approx r_{\sbf^*}$),
 provided that $\eps$ is small enough depending on the implicit constants
 in \eqref{xtysztaurelate.eq}.
When we apply \eqref{xtysztaurelate.eq} in the sequel, these implicit constants
may depend on the allowable parameters and $K_0$, but of course not
on $\eps$.

We recall that $\psi$ is constructed off the contact set $F_{\sbf^*}$
(see \eqref{DSDfn.F}) using a Whitney-type construction involving a collection
$\{I\}$ of non-overlapping parabolic cubes in $P_{Q(\sbf^*)}\cong \rn$, that cover
$P_{Q(\sbf^*)}\setminus \pi(F_{\sbf^*})$.  Throughout the present proof, the letter ``$I$" will always be used to denote a cube in this collection.

Let us make a simple observation that will be useful in the sequel.
We shall assume henceforth that $K_0\gg 100$.
Let $Q\in \sbf^*$. %, and let $Q'\in\ch(Q)$.
Then for $\eps$ small enough,
  \begin{equation}\label{eq.proximity4}
  \delta\big(\psi(x_{Q'},t_{Q'})+\ell(Q'), x_{Q'}, t_{Q'}\big)
 %   \dist\big((X_{Q'},t_{Q'}),(\psi(x_{Q'},t_{Q'})+\ell(Q'), x_{Q'}, t_{Q'})\big)
\, \geq\,  \ell(Q')/2\,, \qquad \forall \, Q'\in\ch(Q)\,,
   \end{equation}
since by construction $100 Q(\sbf^*)$ % $(X_{Q'},t_{Q'})\in Q$, and $Q$
lies below the graph of $\psi$ (see, e.g., Lemma
\ref{intgraphdomlem.lem}), and $\psi$ is parabolically Lipschitz of order
roughly $\eps^{1/2}$, by Lemma \ref{graphlip.lem}.

We record the following pair of lemmata, for future reference.

\begin{lemma}\label{lemma.I}
Suppose that $I$ meets $C'_1$,
and consider any point $(x_I,t_I)\in I$. Then
% by \eqref{closeIiforpsi.eq}, we have that
\begin{equation}\label{gc9.53}
\dist((\psi(x_I,t_I), x_I, t_I), \Sigma) \leq \eps^{1/2} \diam(I)\,.
\end{equation}
Let $(\hat{X}_I,\hat{t}_I)$ be the
nearest point in $\Sigma$ to $(\psi(x_I,t_I), x_I, t_I)$, and define
\[
\Delta_I := C\big((\hat{X}_I,\hat{t}_I), \eps^{1/4}\diam(I)\big) \cap \Sigma.
\]
If $\eps$ is small enough, depending only on allowable parameters, then
$\pi(\Delta_I) \subset 2I$, and in the special case
that $(x_I,t_I)$ is the {\em center} of $I$, we have
$\pi(\Delta_I) \subset \interior(I)$.
\end{lemma}

\begin{proof}[Proof of Lemma \ref{lemma.I}]
Note first that $I=I_i$ for some $i\in \Lambda$ (see \eqref{la-def}),
since $I$ meets $C_1'$.  The estimate \eqref{gc9.53} then holds immediately, by
\eqref{closeIiforpsi.eq}.  Combining \eqref{gc9.53} with
the definitions of $(\hat{X}_I,\hat{t}_I)$ and $\Delta_I$,
we have in particular that
\[
\dist\big((x,t),(x_I,t_I)\big) \leq 2 \eps^{1/4} \diam(I)\,,\qquad \forall\,
(X,t)=(x^0,x,t) \in \Delta_I\,.
\]
Consequently, $\pi(\Delta_I)\subset \interior(I)$ if $(x_I,t_I)$ is
the center of $I$, and $\pi(\Delta_I) \subset 2I$ when
$(x_I,t_I)$ is an arbitrary point in $I$,
provided in each case that $\eps$ is sufficiently small.
\end{proof}

\begin{lemma}\label{lemma.proximity} Let $Q\in\sbf^*$, and fix
$1\leq N<\infty$.  Suppose that there is a cube $I$ with $\diam(I)\leq N\diam(Q)$,
along with a triple of points $(X,t)=(x^0,x,t)\in Q$, $(x_I,t_I)\in I$, and
 $(\hat{X}_I,\hat{t}_I)\in\Sigma$, such that
 \begin{equation}\label{eq.proximity1}
 \dist\big((\hat{X}_I,\hat{t}_I),(\psi(x_I,t_I), x_I, t_I)\big) \leq \eps^{1/4} \diam(I)\,,
 \end{equation}
 and
  \begin{equation}\label{eq.proximity2}
   \dist\big((X,t),(\hat{X}_I,\hat{t}_I)\big) \leq \eps^{1/4} \diam(I)\,.
 \end{equation}
 Then \eqref{childcap.eq} holds for every $Q'\in \ch(Q)$, provided that
 $\eps$ is chosen small enough, depending on $N$, and $K_0$ is
 large enough, depending on allowable parameters.
\end{lemma}

\begin{proof}
Combining \eqref{eq.proximity1}-\eqref{eq.proximity2}, we have
 \[
 \dist\big((X,t),(\psi(x_I,t_I), x_I, t_I)\big) \leq 2\eps^{1/4} \diam(I)\,,
 \]
thus in particular, $\dist((x,t),(x_I, t_I)) \leq 2\eps^{1/4} \diam(I)$, so
$|\psi(x,t)-\psi(x_I,t_I)| \lesssim \eps^{3/4} \diam(I)$,
since $\psi$ is parabolically Lipschitz of order
roughly $\eps^{1/2}$, by Lemma \ref{graphlip.lem}.  Consequently,
 \[
 \dist\big((X,t),(\psi(x,t), x, t)\big) \leq 3\eps^{1/4} \diam(I)
 \leq 3N \eps^{1/4} \diam(Q)\leq \eps^{1/8} \diam(Q)\,,
 \]
 for $\eps\leq \eps(N)$ sufficiently small.  Since $(X,t) \in Q$,
 it then follows that
   \begin{equation}\label{eq.proximity3}
    \dist\big((X_{Q'},t_{Q'}),(\psi(x_{Q'},t_{Q'})+\ell(Q'), x_{Q'}, t_{Q'})\big)
    \lesssim  \ell(Q')\,, \qquad \forall \, Q'\in\ch(Q)\,,
   \end{equation}
 where as usual $(X_{Q'},t_{Q'})=(x^0_{Q'},x_{Q'},t_{Q'})$ is the ``center" of
 $Q'$ as in \eqref{cube-ball}-\eqref{cube-ball2}, and where the implicit constant depends only on $n$ and $ADR$, and in particular is independent of $N$, $K_0$ and $\eps$.
 % Furthermore, again for $\eps$ small enough, and for $K_0\gg 100$,
  % \begin{equation}\label{eq.proximity4}
  % \delta\big(\psi(x_{Q'},t_{Q'})+\ell(Q'), x_{Q'}, t_{Q'}\big)
 %   \dist\big((X_{Q'},t_{Q'}),(\psi(x_{Q'},t_{Q'})+\ell(Q'), x_{Q'}, t_{Q'})\big)
% \, \geq\,  \ell(Q')/2\,,
%   \end{equation}
% since by construction $100 Q(\sbf^*)$ % $(X_{Q'},t_{Q'})\in Q$, and $Q$
% lies below the graph of $\psi$ (see, e.g., Lemma
% \ref{intgraphdomlem.lem}), and $\psi$ is parabolically Lipschitz of order
% roughly $\eps^{1/2}$ (Lemma \ref{graphlip.lem}).
For $K_0$ large enough, depending
on $n$ and $ADR$, \eqref{eq.proximity3} and \eqref{eq.proximity4} imply
\eqref{childcap.eq}.
\end{proof}

We now return to the proof of Lemma \ref{keylemma2}.
There are two cases to consider.

\medskip

\noindent{\bf Case 1:}  $\diam(I) \le 10^{-100}r_{\sbf^*}$ for every $I$ that meets $C'_1$.

Suppose that $I$ meets $C'_1$,
and let $(x_I,t_I)$ be the center of $I$.
Note that $(x_I,t_I) \in C'_2$
since $\diam(I) \le 10^{-100}r_{\sbf^*}$.
% By Lemma \ref{lemma.I}, we see that \eqref{gc9.53} holds, and
Defining $(\hat{X}_I,\hat{t}_I)$ and $\Delta_I$ as in Lemma \ref{lemma.I},
we find that
$\pi(\Delta_I)\subset \interior(I)$
if $\eps$ is sufficiently small. Thus, the surface cylinders
$\{\Delta_I\}_{I\cap C_1'\neq \emptyset}$ are pairwise disjoint, since the
cubes $\{I\}$ are pairwise non-overlapping. Furthermore,
by Lemma \ref{10-60lemma},
\begin{equation}\label{DpiXt}
\diam(I)\leq 10\diam(I)\leq D(x,t)\leq 60 \diam(I)\,,\qquad \forall\, (X,t)=(x^0,x,t) \in \Delta_I\,.
\end{equation}
 Moreover, since $(x_I,t_I) \in C'_2$, and since in Case 1,
$\diam(I) \ll r_{\sbf^*}= r_{Q(\sbf^*)} \approx \ell(Q(\sbf^*)$,
we may use \eqref{gc9.53} and the definition of $\Delta_I$ to
see that \eqref{xtysztaurelate.eq} holds with
$(X,t)= (\hat{X}_I,\hat{t}_I)$ and $(z,\tau)=(x_I,t_I)$, and
for every  $(Y,s)\in \Delta_I$.  Hence,
\[\Delta_I \subseteq \Delta_{Q(\sbf^*)}.\]
Now set
\[A^1_{\sbf^*} := \bigcup_{I \cap C'_1 \neq \emptyset} \Delta_I\]
and note that $\sigma(\Delta_I) \gtrsim_\eps |I|$, since $\Sigma$ is ADR.
We define
\[
A^2_{\sbf^*}:= F_{\sbf^*} \cap \pi^{-1}(C_1') \cap
C(X_{Q(\sbf^*)},t_{Q(\sbf^*)},r_{\sbf^*})\,,
\]
and note that in particular, $A_{\sbf^*}\subset \Delta_{Q(\sbf^*)}\subset
Q(\sbf^*)$ (see \eqref{cube-ball}-\eqref{cube-ball2}).  Set
\[
A_{\sbf^*} = A^1_{\sbf^*} \cup A^2_{\sbf^*}\,,
\]
and observe that
\[
C'_1 \subseteq \left(\bigcup_{I \cap C'_1 \neq \emptyset} I\right) \cup  \pi(A^2_{\sbf^*})\,.
\]
Then either $|\pi(A_{\sbf^*})|\ge |C'_1|/2$, or
\[
\left|\bigcup_{I \cap C'_1 \neq \emptyset} I\right| = \sum_{I \cap C'_1 \neq \emptyset} \left|I\right|\, \ge \,\frac{|C'_1|}{2}\,,
\]
in which case, by pairwise disjointness of the surface cylinders $\{\Delta_I\}_I$,
\[
\sigma(A^1_{\sbf^*})=
\sum_{I \cap C'_1 \neq \emptyset} \sigma(\Delta_I) \gtrsim_\eps \sum_{I \cap C'_1 \neq \emptyset} \left|I\right| \gtrsim_\eps  |C'_1|.
\]
Consequently, since the parabolic Hausdorff
measure does not increase under projections,
it holds that $\sigma(A_{\sbf^*}) \gtrsim_\eps|C'_1|
\approx_\eps \sigma(Q(\sbf^*))$.
In particular, \eqref{ASsetamp.eq} holds in Case 1.

Now let us verify \eqref{childcap.eq} in Case 1. First, suppose that $Q \in \sbf^*$ meets $A^1_S$, i.e., $Q \cap \Delta_I \neq \emptyset$ for some $I$ that meets $C'_1$. Let $(X,t) = (x^0,x,t) \in Q \cap \Delta_I$.  Since
$Q \in \sbf^*$, the definition of $D$ gives
$D(x,t) \leq \diam(Q)$, and thus by
\eqref{DpiXt}, we have $\diam(I) \leq \diam(Q)$. Using Lemma \ref{lemma.I},
we may combine the latter fact with
\eqref{gc9.53} and the definition of $\Delta_I$, to see that
the hypotheses of Lemma \ref{lemma.proximity} are satisfied (with $N=1$).
Hence, \eqref{childcap.eq} holds if $Q$ meets $A^1_{\sbf^*}$.

% \begin{equation}\label{packargtrpstr.eq}
% \dist((X,t), (\psi(x,t),x,t)) \lesssim \eps^{1/4}\diam(Q).
% \end{equation}

Suppose now that $Q$ meets $A^2_{\sbf^*}$. Then
$Q$ meets $F_{\sbf^*}$, by definition of
$A^2_{\sbf^*}$, and thus also $Q$ meets $\Gamma_{\sbf^*}$,
the graph of $\psi=\psi_{\sbf^*}$ (indeed, recall that
$F_{\sbf^*}\subset \Sigma \cap \Gamma_{\sbf^*}$, by the construction in
Section \ref{sec: graph}).  Consequently, there is a point
\begin{equation}\label{intersectionpoint}
(Z,\tau) = (z^0,z,\tau) = (\psi(z,\tau),z,\tau) \in \Gamma_{\sbf^*} \cap Q\,,
\end{equation}
and in turn, it follows that
\begin{equation}\label{eq.proximity5}
\dist\big((X_{Q'},t_{Q'}),(Z,\tau)\big) \leq \diam(Q)\,,\qquad \forall\, Q'\in\ch(Q)\,,
\end{equation}
hence also, by Lemma \ref{graphlip.lem}, for $\eps$ chosen small enough,
\[
|\psi(x_{Q'},t_{Q'})-\psi(z,\tau)|\leq  \diam(Q)\,,
\]
where as usual $(X_{Q'},t_{Q'}) = (x^0_{Q'},x_{Q'},t_{Q'})$ denotes the ``center"
of $Q'$.  Combining these observations, we see that for every
$Q'\in\ch(Q)$,
 \begin{equation}\label{eq.proximity6}
    \dist\big((X_{Q'},t_{Q'}),(\psi(x_{Q'},t_{Q'})+\ell(Q'), x_{Q'}, t_{Q'})\big)
    \lesssim \diam(Q) \approx \diam(Q')\,,
   \end{equation}
where the implicit constants depend only on $n$ and $ADR$ (so of course
{\em not} on $K_0$ or $\eps$).
Note that, in addition, \eqref{eq.proximity4} continues to hold.
For $K_0$ large enough, % depending on $n$ and $ADR$,
\eqref{eq.proximity6} and \eqref{eq.proximity4} together imply
\eqref{childcap.eq}, as desired.

\medskip

\noindent{\bf Case 2:} There exists $I$ meeting $C'_1$ for which
$\diam(I) \ge 10^{-100}r_{\sbf^*}$.

Fix any such $I$, let $(x_I,t_I) \in I\cap C'_1$,
and define $(\hat{X}_I,\hat{t}_I)$ and $\Delta_I$ as in Lemma
\ref{lemma.I}.  Since $I$ meets $C'_1$, it follows that
$I =I_i$ for some $i\in \Lambda$ (see \eqref{la-def}),
hence by Lemma \ref{lemmaQi}, we have
\begin{equation}\label{diamIbound.eq}
\diam(I)  \lesssim \kappa \diam(Q(\sbf^*)) \approx K_0 \diam(Q(\sbf^*))\,.
% \approx K_0 r_{\sbf^*}\,.
\end{equation}
Furthermore, by Lemma \ref{lemma.I}, we have
\begin{equation}\label{eq.piDeltaI2I}
\pi(\Delta_I)\subset 2I\,,
\end{equation}
and also
\begin{equation}\label{eq.proximityrepeat}
\dist\big((\psi(x_I,t_I), x_I, t_I),(\hat{X}_I,\hat{t}_I)\big) \leq \eps^{1/2} \diam(I)
\leq \eps^{1/4} \diam(Q(\sbf^*))\,,
\end{equation}
provided $\eps$ is small enough, depending on $K_0$.
Combining the latter estimate with
the definition of $\Delta_I$ and \eqref{diamIbound.eq},
we
see that \eqref{xtysztaurelate.eq} holds with
$(X,t)= (\hat{X}_I,\hat{t}_I)$ and $(z,\tau)=(x_I,t_I)$, and
for every  $(Y,s)\in \Delta_I$.  Hence,
since $(x_I,t_I)\in I\cap C'_1\subset C'_2$, we have
\[
\Delta_I \subseteq \Delta_{Q(\sbf^*)}.
\]
We now define
\[
A_{\sbf^*} := \Delta_I\,,
\]
and observe that \eqref{ASsetamp.eq} holds by ADR and the fact that
$\diam(I) \gtrsim r_{\sbf^*} \approx \diam(Q(\sbf^*))$, in the scenario of Case 2.

It remains to prove \eqref{childcap.eq} in Case 2.  Suppose that $Q$ meets
$A_{\sbf^*}$, say $(X,t) \in Q\cap A_{\sbf^*}$.
Since $A_{\sbf^*}=\Delta_I$ in the present case, we see that
 \eqref{eq.proximity2} holds for this $(X,t)$.
Note also that the first inequality in \eqref{eq.proximityrepeat} trivially implies
\eqref{eq.proximity1}.   We claim also that $\diam(I)\leq\diam(Q)$,
in which case the hypotheses of Lemma \ref{lemma.proximity} hold, with $N=1$.
Momentarily taking the claim for granted, we therefore
deduce that \eqref{childcap.eq} holds.  To complete the proof of Lemma \ref{keylemma2}, it now remains only to verify the claim.
To this end, set $(x,t):=\pi(X,t)$, and note that \eqref{DpiXt} continues to hold,
by \eqref{eq.piDeltaI2I} and Lemma \ref{10-60lemma}.  On the other hand,
$D(x,t) \le \diam(Q)$, by definition of $D(x,t)$ (see \eqref{DSDfn.eq}),
since $(x,t) \in \pi(Q)$.
In particular, these observations establish the claim.
\end{proof}

\begin{remark}\label{proximityremark} We record the
following observation for future reference.
Let $Q\in \sbf^*$, and suppose that $Q$ meets $A_{\sbf^*}$. Then
 \begin{equation}\label{eq.proximity7}
    \dist\big((X_{Q'},t_{Q'}),(\psi(x_{Q'},t_{Q'})+\ell(Q'), x_{Q'}, t_{Q'})\big)
    \lesssim  \diam(Q')\,, \qquad \forall \, Q'\in\ch(Q)\,,
   \end{equation}
   where the implicit constants depend only on $n$ and ADR.  Indeed, we have already proved this fact in the previous argument:  in Case 2, or if $Q$ meets $A_{\sbf^*}^1$ in Case 1, then we have observed that the hypotheses of Lemma
   \ref{lemma.proximity} hold with $N=1$,
   in which case the claimed bound is simply \eqref{eq.proximity3}.
   If $Q$ meets $A_{\sbf^*}^2$ in Case 1, then the claimed bound is
   \eqref{eq.proximity6}.
\end{remark}

With Lemma \ref{keylemma2} in hand, we deduce Lemma \ref{keylemma1}.
Given a cube $Q\in\dd(\Sigma)$, let us for the moment use the notation
$\P(Q)$ to denote the dyadic parent of $Q$.  We define
\[
\F_{3,good}^*(\sbf^*):=
\left\{Q\in \F_{3}^*(\sbf^*): \P(Q) \cap A_{\sbf^*} \neq \emptyset\right\}\,,
\]
and set $\F_{3,bad}^*(\sbf^*):=\F_{3}^*(\sbf^*)\setminus \F_{3,good}^*(\sbf^*)$, i.e.,
a cube $Q\in \F_{3}^*(\sbf^*)$ belongs to $\F_{3,bad}^*(\sbf^*)$ if and
only if its parent
$\P(Q)\subset Q(\sbf^*)\setminus A_{\sbf^*}$.  Thus,
\[
\Upsilon^3_{\sbf^*,\,bad}:= \bigcup_{Q\in \F_{3,bad}^*(\sbf^*)}\,Q \subset
Q(\sbf^*)\setminus A_{\sbf^*}\,,
\]
and therefore \eqref{Up3compample} follows directly from \eqref{ASsetamp.eq}.

It now remains only to verify the packing condition \eqref{Fp3pack}.

\begin{proof}[Proof of \eqref{Fp3pack}]
We fix a tree $\sbf$ as in Lemma \ref{initialcorona1.lem},
a caloric measure tree $\sbf'\subset\sbf$ as in Lemma \ref{refp+},
and a cube $R \in \dd(Q(\sbf'))$.
In the sequel, given a graph tree $\sbf^*$, we shall use $Q$ to
denote cubes in $\sbf^*$, and $Q'$ to denote cubes in $\F_{3,good}^*(\sbf^*)$.
We then seek to prove the following estimate, which we reproduce here:
 \begin{equation*}% \label{Fp3pack}
 \sum_{\sbf^*:\,\sbf^*\subset \sbf' \,,\,Q(\sbf^*)\subset R }\,\,
 \sum_{Q'\in \F_{3,good}^*(\sbf^*)} \sigma(Q') \,
 \leq \, M_0\, \sigma(R)\,,\qquad \forall \, R\in \dd\big(Q(\sbf')\big)\,,
\leqno{\eqref{Fp3pack}}
 \end{equation*}
 where $M_0$ is a uniform constant depending only on $\eps, K_0$ and the allowable parameters.

 \begin{remark}\label{R=QS*0}
 Note that the first generation of
 cubes $Q(\sbf^*)\subset R$ (i.e., those
 that are maximal with respect to containment), are disjoint, so
 we may assume that $R=Q_1=Q(\sbf^*_1)$ for some fixed
 $\sbf^*_1 \subset \sbf'$.
 \end{remark}

 Consider now a graph tree $\sbf^*$ with maximal cube
 $Q(\sbf^*)\subset Q(\sbf^*_1)$, let $Q'\in\F^*_{3,good}(\sbf^*)$,
 and let $Q=\P(Q')$ be the dyadic parent of $Q'$.
By definition of $F^*_{3,good}(\sbf^*)$,
it follows that $Q\in\sbf^*$, $Q$ meets $A_{\sbf^*}$, and
 $Q'$ satisfies Condition \ref{stopping} (3), i.e.,
 \begin{equation*} % \label{stoprepeat.eq}
\big|\nabla_X \widehat u(Y,s) -
\nabla_X\widehat u\big(Y^*_{Q(\sbf^*)}, s^*_{Q(\sbf^*)}\big)\big|
> 10\eps^{2L}, \quad \forall (Y,s) \in U_{Q'}.
\end{equation*}
Moreover, by Lemma \ref{keylemma2}, we see that
\eqref{childcap.eq} holds for $Q'$, so in particular,
 \begin{equation}\label{stopparticular.eq}
\big|\nabla_X \widehat u (X^+_{Q'}, t_{Q'}) -
\nabla_X\widehat u\big(Y^*_{Q(\sbf^*)}, s^*_{Q(\sbf^*)}\big)\big|
> 10\eps^{2L}\,,
\end{equation}
where
 \begin{equation}\label{X+def}
(X^+_{Q'}, t_{Q'}) := \big(\psi_{\sbf^*}(x_{Q'}, t_{Q'}) + \ell(Q'), x_{Q'}, t_{Q'}\big)
=: (x^{0,+}_{Q'}, x_{Q'},t_{Q'})\,.
\end{equation}
For future reference, let us also define
 \begin{equation}\label{Xstardef}
(X^\star_{Q'}, t_{Q'}) := \big(\psi_{\sbf^*}(x_{Q'}, t_{Q'}), x_{Q'}, t_{Q'}\big)
=: (x^{0,\star}_{Q'}, x_{Q'},t_{Q'})\,,
\end{equation}
and note that of course, $(X^\star_{Q'}, t_{Q'})\in\Gamma_{\sbf^*}$
(the graph of $\psi_{\sbf^*}$).

Let $\Omega'_{\sbf^*}$ be the localized graph domain
introduced in \eqref{ddom+}.  For $(X,t)= (x_0,x,t) \in \Gamma_{{\sbf^*}}$, we introduce the narrow truncated (parabolic) cone
\[
\tilde{\gamma}_{X,t} := \tilde{\gamma}_{X,t,{\sbf^*}} := \{(y_0,y,s): y_0 - \psi_{\sbf^*}(x,t) > \kappa^{10}\dist(y,s,x,t),\  y_0 < \kappa^5 \ell(Q(\sbf^*))\}.
\]
\begin{remark}\label{remark.narrowcone}
The factor $\kappa^{10}$ ensures that if $\pi(X,t) \in
C_{\sqrt{\kappa}R_{\sbf^*}}'(x_{Q(\sbf^*)},t_{Q(\sbf^*)})$, then
$\tilde{\gamma}_{X,t}$ is (well) contained in $\Omega'_{\sbf^*}$.
Here, we recall that $R_{\sbf^*}:=\diam(Q(\sbf^*))$.
\end{remark}

We then define the localized non-tangential
maximal function at $(X,t)\in \Gamma_{\sbf^*}$,
for a (possibly $\rn$-valued) function
$F$ defined on $\tilde{\gamma}_{X,t}$, by
\[
\nn_{\sbf^*}(F)(X,t) := \sup_{(Y,s)\in \tilde{\gamma}_{X,t}} |F(Y,s)|.
\]
We define
\[
\Delta_{\star,\eps}(Q') := % \Delta_{Q',\Gamma_{\sbf^*}}:=
C\big(X^\star_{Q'}, t_{Q'}, \eps\diam(Q')\big)
\cap \Gamma_{\sbf^*}\,,
\]
and observe that by ADR,
\begin{equation}\label{eq.comparablemeasure}
\sigma(Q') \approx_\eps \sigma_{\Gamma_{\sbf^*}}\big(\Delta_{\star,\eps}(Q')\big)\,,
\end{equation}
where $\sigma_{\Gamma_{\sbf^*}}:=  \cH^{n+1}_{\text{par}}|_{\Gamma_{\sbf^*}}$
is the parabolic ``surface measure" on $\Gamma_{\sbf^*}$.
Note also that, since $\psi_{\sbf^*}$ is Lip(1/1/2) with small constant
(see Lemma \ref{graphlip.lem}),
using \eqref{stopparticular.eq}-\eqref{X+def} we obtain
\[
10\eps^{2L}\, \leq \,
\nn_{\sbf^*}\!\left(\nabla_X \widehat u (\cdot,\cdot) -
\nabla_X\widehat u\big(Y^*_{Q(\sbf^*)}, s^*_{Q(\sbf^*)}\big)\right)(X,t)\,,
\qquad \forall \, (X,t) \in \Delta_{\star,\eps}(Q')\,,
\]
provided that $\eps$ is small enough, depending on $K_0$.
Combining the latter estimate with \eqref{eq.comparablemeasure}, we find that
\begin{equation}\label{eq.Q'lowerbound}
\sigma(Q') \lesssim_\eps \iint_{\Delta_{\star,\eps}(Q')}
\Big(\nn_{\sbf^*}\!\left(\nabla_X \widehat u (\cdot,\cdot) -
\nabla_X\widehat u\big(Y^*_{Q(\sbf^*)}, s^*_{Q(\sbf^*)}\big)\right)\Big)^2
d\sigma_{\Gamma_{\sbf^*}}\,.
\end{equation}
Since $Q\in\sbf^*$, and $Q\cap A_{\sbf^*}\neq\emptyset$, we can apply Remark
\ref{proximityremark} to deduce that
\begin{equation}\label{eq.proximity8}
\dist(\Delta_{\star,\eps}(Q'), Q') \lesssim \diam(Q')\,,
\end{equation}
where the implicit constant depends only on $n$ and $ADR$.  Hence, by a standard covering lemma argument, we may extract a refinement of $\F^*_{3,good}$, call it
$\F^*_{3,gs}$ (``gs" = ``good, separated"), such that
the collection $\{\Delta_{\star,\eps}(Q')\}_{Q'\in \F^*_{3,gs}}$ is pairwise disjoint, and
\[
 \sum_{Q'\in \F_{3,good}^*(\sbf^*)} \sigma(Q')  \lesssim
  \sum_{Q'\in \F_{3,gs}^*(\sbf^*)} \sigma(Q')\,,
\]
where the implicit constant depends only on $n$ and $ADR$. Combining the latter
fact with \eqref{eq.Q'lowerbound}, we obtain
\begin{multline}\label{eq.Q'bound1}
 \sum_{Q'\in \F_{3,good}^*(\sbf^*)} \sigma(Q')
 \\[4pt] \lesssim_\eps \,
   \sum_{Q'\in \F_{3,gs}^*(\sbf^*)}
   \iint_{\Delta_{\star,\eps}(Q')}
\Big(\nn_{\sbf^*}\!\left(\nabla_X \widehat u (\cdot,\cdot) -
\nabla_X\widehat u\big(Y^*_{Q(\sbf^*)}, s^*_{Q(\sbf^*)}\big)\right)\Big)^2
d\sigma_{\Gamma_{\sbf^*}} \\[4pt]
\lesssim_\eps  \,  \iint_{\partial\Omega'_{\sbf^*} \cap\Gamma'_{\sbf^*}}
\Big(\nn_{\sbf^*}\!\left(\nabla_X \widehat u (\cdot,\cdot) -
\nabla_X\widehat u\big(Y^*_{Q(\sbf^*)}, s^*_{Q(\sbf^*)}\big)\right)\Big)^2
d\sigma_{\Gamma'_{\sbf^*}}
 \\[4pt]
\lesssim_\eps  \,
\iint_{\Omega_{\sbf^*}'} |\nabla^2_X \widehat{u} (X,t)|^2 \,
\delta_{\Omega_{\sbf^*}'}(X,t)\,dXdt\,,
\end{multline}
where $\delta_{\Omega_{\sbf^*}'}(X,t):= \dist((X,t),\partial\Omega'_{\sbf^*})$.
Here, in the next-to-last step, we have introduced the localized graph:
\[
\Gamma'_{\sbf^*}:=
\big\{\big(\psi_{\sbf^*}(y,s),y,s\big):
(y,s) \in C'_{\sqrt{\kappa} R_{\sbf^*}} \big(x_{Q(\sbf^*)}, t_{Q(\sbf^*)}\big)\big\}\,.
\]
and in the last step, we have used
Remark \ref{remark.narrowcone},
and the results of R. Brown \cite{B89}
combined with \eqref{RHq}.
Note that trivially, $\delta_{\Omega_{\sbf^*}'}(X,t)\leq \delta(X,t)$, hence by
Corollary \ref{cor.deltahatu} and Lemma \ref{intgraphdomlem.lem}, we deduce from
\eqref{eq.Q'bound1} that
\begin{equation}\label{eq.3goodbound}
 \sum_{Q'\in \F_{3,good}^*(\sbf^*)} \sigma(Q')
 \,\lesssim_\eps \,\sum_{Q \in \sbf^*}
 \iint_{\tilde{U}_Q^i} |\nabla^2_X \widehat{u} (X,t)|^2 \,
\widehat{u}(X,t)\,dXdt\,.
\end{equation}
Note further that $\sbf^*\subset\sbf'\subset\sbf$, for every
$\sbf^*$ currently under consideration, and therefore, with $Q_1=Q(\sbf^*_1)$
(see Remark \ref{R=QS*0}), setting $\sbf_1:= \sbf\cap \dd(Q_1)$, and
using the containment \eqref{UQcontain}, along with
\eqref{sawtooth-appen-HMM--},
\eqref{doma++}, and the disjointness of distinct trees in the collection
$\{\sbf^*\}$, we see that
\begin{equation}\label{eq.treecontain}
\interior\left(\bigcup_{\sbf^*:\,\sbf^*\subset \sbf' \,,\,Q(\sbf^*)\subset Q_1} \,\,
\bigcup_{Q\in\sbf^*} \tilde{U}_Q \right)\,\subset \, \Omega_{\sbf_1}^{**}\,.
\end{equation}
Since the fattened Whitney regions $\tilde{U}_Q$ have bounded overlaps
(depending on $K_0$ and $\eps$), and have boundaries with Lebesgue measure
zero in space-time $\ree$, we may sum \eqref{eq.3goodbound}
 over $\sbf^*$ with $Q(\sbf^*)\subset Q_1$, and use \eqref{eq.treecontain} and then
Lemma \ref{squarefunction},  to obtain
\begin{multline*}
 \sum_{\sbf^*:\,\sbf^*\subset \sbf' \,,\,Q(\sbf^*)\subset Q_1}\,\,
 \sum_{Q'\in \F_{3,good}^*(\sbf^*)} \sigma(Q') \\
 \lesssim\,\,
  \sum_{\sbf^*:\,\sbf^*\subset \sbf' \,,\,Q(\sbf^*)\subset Q_1}\,\,
  \sum_{Q \in \sbf^*}
 \iint_{\tilde{U}_Q^i} |\nabla^2_X \widehat{u} (X,t)|^2 \,
\widehat{u}(X,t)\,dXdt
 \\ \lesssim\, \iint_{\Omega^{**}_{\sbf_1}} |\nabla^2_X \widehat{u} (X,t)|^2 \,
\widehat{u}(X,t)\,dXdt
\, \lesssim\, \sigma\big(Q_1\big)\,,
\end{multline*}
where the implicit constants depend on
$\eps$, $K_0$, and the allowable parameters.
As observed in Remark \ref{R=QS*0}, this concludes the proof of
\eqref{Fp3pack}, and hence also the proofs of Lemma \ref{keylemma1},
Proposition \ref{finalprop}, and ultimately Theorem \ref{main.thrm}.
\end{proof}

\newcommand{\etalchar}[1]{$^{#1}$}

\end{document}